\documentclass[pdflatex,sn-basic]{sn-jnl}

\setcitestyle{aysep={}} 
\usepackage{graphicx}%
\usepackage{multirow}%
\usepackage{amsmath,amssymb,amsfonts}%
\usepackage{amsthm}%
\usepackage{mathrsfs}%
\usepackage[title]{appendix}%
\usepackage{xcolor}%
\usepackage{textcomp}%
\usepackage{manyfoot}%
\usepackage{booktabs}%
\usepackage{algorithm}%
\usepackage{algorithmicx}%
\usepackage{algpseudocode}%
\usepackage{listings}%

\newtheoremstyle{mytheoremstyle}
{5pt} 
{5pt} 
{\itshape} 
{} 
{\bfseries} 
{.} 
{.5em} 
{} 

\newtheoremstyle{mydefinitionstyle}
{5pt} 
{5pt} 
{} 
{} 
{\bfseries} 
{.} 
{.5em} 
{} 

\theoremstyle{mytheoremstyle}
\newtheorem{theorem}{Theorem}[section]

\newtheorem{lemma}[theorem]{Lemma}
\newtheorem{proposition}[theorem]{Proposition}
\newtheorem{corollary}[theorem]{Corollary}

\theoremstyle{mydefinitionstyle}
\newtheorem{remark}[theorem]{Remark}
\newtheorem{notation}[theorem]{Notation}
\newtheorem{definition}[theorem]{Definition}

\newtheorem{assumption}[theorem]{Assumption}

\usepackage{amscd} 
\usepackage{amsthm} 
\usepackage{latexsym} 
\usepackage{mathrsfs} 
\usepackage{mathtools} 
\usepackage{slashed}  

\usepackage{verbatim}
\usepackage{cases}
\usepackage{enumitem} 

\usepackage{pifont} 
\usepackage{upgreek} 
\usepackage{calligra} 
\usepackage{marvosym} 
\usepackage[percent]{overpic} 
\usepackage{pict2e} 
\usepackage{wasysym} 
\usepackage{accents}

\numberwithin{equation}{section}
\allowdisplaybreaks

\newcommand{\norm}[1]{\Vert#1\Vert}
\newcommand{\abs}[1]{\vert#1\vert}

\newcommand{\dive}{\mbox{\upshape div}\mkern 1mu}
\newcommand{\curl}{\mbox{\upshape curl}\mkern 1mu}
\newcommand{\tr}{\operatorname{tr}}
\newcommand{\vol}{\operatorname{vol}}

\usepackage{tensor}

\usepackage{cancel}

\newcommand{\proj}{\mathsf{\Pi}}

\newcommand{\Err}{\mathcal{R}}
\newcommand{\textdef}[1]{\textbf{#1}}
\newcommand{\Ric}{\operatorname{Ric}}
\newcommand{\Riem}{\operatorname{Riem}}
\newcommand{\vort}{\operatorname{vort}}
\newcommand{\C}{\mathcal{C}}
\newcommand{\D}{\mathcal{D}}
\newcommand{\B}{\mathbf{B}}
\newcommand{\Md}{\mathscr{D}}
\newcommand{\Fb}{\mathsf{\Gamma}}
\newcommand{\Hspace}{\mathscr{H}}
\newcommand{\SoundCone}{\mathscr{C}}

\newtheorem{innercustomthmproblem}{Problem}

\setlist[itemize]{topsep=3pt}
\setlist[enumerate]{topsep=3pt}

\begin{document}

\title[Relativistic fluids]{Recent developments in mathematical aspects of relativistic fluids}


\author*[1]{\fnm{Marcelo} \sur{Disconzi}}\email{marcelo.disconzi@vanderbilt.edu}

\affil*[1]{\orgdiv{Mathematics Department}, \orgname{Vanderbilt University}, \orgaddress{\street{1326 Stevenson Center Ln}, \city{Nashville}, \postcode{37240}, \state{TN}, \country{United States of America}}}


\abstract{We review some recent developments in mathematical aspects
of relativistic fluids. The goal is to provide a quick entry point to some research topics 
of current interest that is 
accessible to graduate students and researchers from adjacent fields, as well as to researches working on broader aspects of relativistic fluid dynamics interested in its mathematical formalism. Instead 
of complete proofs, which can be found in the published 
literature, here we focus on the proofs' main ideas
and key concepts. After an introduction to the relativistic Euler equations, we cover the following topics:
a new wave-transport formulation of the relativistic Euler equations tailored to applications; the problem
of shock formation for relativistic Euler;
rough (i.e., low-regularity) solutions to the relativistic Euler equations;
the relativistic Euler equations with a physical vacuum boundary;
relativistic fluids with viscosity. We finish with a discussion of open problems and future directions of research.}


\maketitle

\setcounter{tocdepth}{3} 
\tableofcontents

\section{Introduction}
\label{S:Introduction}

The field of relativistic fluid dynamics is concerned with the study of fluids in situations when effects pertaining to the theory of relativity cannot be neglected. It is an essential tool in high-energy nuclear physics, cosmology, and astrophysics (see references in Sect.~\ref{S:Background}).
Relativistic effects are manifest in models of relativistic fluids through the geometry of spacetime. This can be done in two ways: (a) by letting the fluid evolve on fixed spacetime geometry, or (b) by coupling the fluid equations to Einstein's equations. In (a), we are neglecting the effect of the fluid's matter and energy on the curvature of spacetime\footnote{In applications, often the fixed geometry is chosen to satisfy the vacuum Einstein equations.}, while in (b) such effects are taken into account. 
Both situations will be investigated here.

Recent years have witnessed a great deal of progress in understanding the mathematical structure of relativistic fluid theories. Such understanding underscores the role played by geometric-analytic tools in the study of fluids, tools that have been largely developed in the study of the Cauchy problem for the vacuum Einstein equations. The use of techniques of this sort to solve challenging problems in fluid dynamics is a testimony to their depth and versatility. More precisely, one might naively think 
that invoking ideas from general relativity is an obvious approach to study \emph{relativistic} fluids.
Nevertheless, the core of the geometric-analytic techniques we will discuss is fundamentally linked to the \emph{wave} aspect of Einstein's equations, and not so much to other notorious features of general relativity, such as its covariance or the thorny problem of singularities. In fact, 
geometric-analytic techniques originally developed in the context of Einstein's equations can be, and have been, successfully applied to the study of the classical (i.e., non-relativistic\footnote{Physicists often use the terminology ``classical'' to refer to ``non-quantum.'' Here, we will use ``classical'' to mean ``non-relativistic.'' Since we will not discuss any quantum system, no confusion should arise from this terminology.\label{FN:Classical_NR}}) compressible Euler equations as well. The bedrock feature of these techniques encompassing both the study of fluids and Einstein's equations is the \emph{propagation of waves:} \emph{gravitational waves} in the case of Einstein's equations and \emph{sound waves} in the case of fluids. The role played by sound waves is captured by the key concept of the \emph{acoustical metric,} which is a Lorentzian metric constructed out of the fluid variables whose characteristic sets track the propagation of sound. The characteristics of the acoustical metric will be called \emph{sound cones,} very much like the characteristics of the Minkowski metric, tracking propagation of light, are called lightcones. More generally, the \emph{techniques we will present are intrinsically tied to the behavior of the characteristic sets (or simply characteristics) of the equations describing the 
evolution of relativistic fluids.} 

When a fluid is perfect (i.e., non-viscous), irrotational, and isentropic (i.e., has constant entropy), its dynamics is entirely governed by the propagation of sound waves, as the corresponding (classical compressible or relativistic) Euler equations reduce to a system of quasilinear wave equations. Thus,  perfect irrotational isentropic fluids is where techniques originated in the study of Einstein's equations first found an application in the study of fluids, most notably in the breakthrough monograph 
by \cite{Christodoulou-Book-2007}. We should stress, however, that
the study of a perfect irrotational isentropic fluid presents its own difficulties, and one should not think that it suffices to quote general relativity techniques as a black-box in order to study the fluid problem treated by Christodoulou. We also remark that, prior to Christodoulou, \cite{Alinhac-1999-1, Alinhac-1999-2, Alinhac-2001-1, Alinhac-2001-2} also employed some geometric-analytic ideas to study quasilinear wave equations.
 
When a fluid is not irrotational nor isentropic, a fundamental new dynamical aspect emerges. Aside from  sound waves, the fluid also exhibits propagation of vorticity and entropy. Vorticity and entropy are transported along the fluid's \emph{flow lines} (which are the integral curves of the fluid's velocity), corresponding to a different set of characteristics than the sound cones. When viscosity is added, further characteristics are present, such as shear waves and the so-called second sound (tied to the propagation of temperature disturbances), corresponding to additional modes of propagation in the fluid. (As we will discuss in Sect.~\ref{S:Relativistic_viscous_fluids}, differently than the classical setting where the standard Navier--Stokes--Fourier equations are parabolic, in the relativistic setting a viscous fluid is described by hyperbolic equations\footnote{In view of the requirement, in relativity, that propagation speeds be finite and at most the speed of light.}, thus one can meaningfully talk about its characteristics.) In other words, the equations of relativistic fluid dynamics form a hyperbolic system with \emph{multiple characteristics speeds.} Understanding systems of this type in more than one spatial dimension is one of the current challenges in the theory of hyperbolic partial differential equations.

Therefore, for the general study of relativistic fluids (or the classical compressible 
Euler equations), it is crucial
to understand the \emph{interactions} among their different characteristics. More than
adapting ideas employed in the study of the vacuum Einstein equations or perfect irrotational isentropic fluids, which are tailored to systems with a single characteristic speed (the gravitational waves in the case of vacuum Einstein's equations and sound waves for the perfect irrotational isentropic  Euler equations), it is necessary to develop \emph{new techniques tailored to the additional characteristics and their interactions.} In order to do so, one has to uncover several hidden aspects of the dynamics.
Furthermore, unlike the case of the Minkowski metric which appears in the standard wave equation, the characteristics
of relativistic fluid equations depend on the solution variables (this is a feature of the quasilinear nature of the problem). An interplay between the regularity of different characteristics, of several geometric quantities associated with them, and of the fluid variables themselves, leads to 
many intricate questions and very rich mathematics. 

Having alluded to some of the general aspects and ideas permeating the study of relativistic fluids,
we will next focus on the specific goals and topics of this work. After an introduction to the relativistic Euler equations (Sect.~\ref{S:Relativistic_Euler}), we will present the following results: a new formulation of the relativistic Euler equations that exhibits some remarkable structural and regularity properties (Sect.~\ref{S:New_formulation}); we provide a discussion of how such new
formulation can be used in the study of shock formation for relativistic fluids (Sect.~\ref{S:Study_of_shocks});
we also use the new formulation to investigate existence of rough (i.e., low-regularity) solutions to the relativistic Euler equations (Sect.~\ref{S:Rough_solutions}); the free-boundary relativistic Euler equations, more specifically, the relativistic Euler equations with a physical vacuum boundary (Sect.~\ref{S:Vacuum_bry}); finally, we discuss relativistic fluids with viscosity (Sect.~\ref{S:Relativistic_viscous_fluids}). We finish with a discussion of open problems and future directions of research (Sect.~\ref{S:Open_problems}). We have included appendices that will be useful to some readers, especially graduate students.

\subsection{Aims and scope}

The primary goal of this article is to provide a quick entry point to some research topics of current interest in relativistic fluids.
While there is no lack of references addressing the topics we will present, including review articles (some of which are provided below), the key words here are ``quick'' and ``entry point."
With the field of relativistic fluids moving rather fast and the amount of background needed to understand many of its modern aspects being substantially large, 
it is often difficult for graduate students, researchers from adjacent fields, and
researches working on broader aspects of relativistic fluid dynamics interested in its mathematical formalism, to keep up with recent developments. We are especially hopeful that our presentation
will encourage students and researchers trained in traditional methods of fluid dynamics to appreciate the geometric-analytic formalism originated in the study of Einstein's equations applied to fluid problems.

Like most review papers, we sought to strike a balance between big picture ideas versus precise statements, as well as
high-level conceptual discussions and nitty-gritty technicalities. 
We took the inspiration for our style of presentation from years of effective discussion with members of the community. Most readers probably have had the experience of 
discussing a topic on a blackboard with a colleague, wherein one takes several liberties in order to get to the main point of a mathematical argument. For example, we write equations in a schematic form, ignore lower-order terms, abuse notations, appeal to toy-models, and often invoke things that are not exactly true but are ``close enough to be true'' for the sake of that discussion. This type of informal-yet-informative form of presentation is particularly useful when we explain new ideas and results to each other. Proceeding in this fashion, we can cut through pages of work and get to key ideas in a proof. This approach is particularly useful when explaining things to graduate students, since they often do not yet have the experience to see the forest for the trees, and can thus get bogged down on some secondary technical point and miss the fundamental aspects of some proof.

Therefore, \emph{we have made a deliberate attempt to present a discussion that is often schematic but retains the core of the arguments.} We hope that we have accomplished this \emph{without compromising clarity or precision, while avoiding many technical aspects.} Our goal is not to completely de-emphasize technical arguments. The results we will discuss are, unfortunately, quite technical, and pointing
out some of the technical difficulties is crucial to understand the role of the techniques we employ. For example, only after clarifying some key technical difficulties can one fully appreciate the relevance of the geometric-analytic formalism
borrowed from the study of Einstein's equations as applied to fluid problems.
But these technical points can often be highlighted in a simplified setting wherein one avoids much of the baggage needed for the complete argument. In this way, we often carry out the presentation at 
a high level, trying to point out the main difficulties and key insights to overcome them. 
Thus, we will not shy away from using schematic notation, plainly ignore certain terms in the estimates, use derivative counting, and equate different terms that differ by some ``unimportant'' factor. We will, of course, always warn the reader when this is being done. 

In many parts of this paper, we have decided to present several intermediate steps in computations. Some of the topics discuss involve some heavy computations.
At first sight, it may look like presenting detailed computations goes against our philosophy of being schematic whenever possible. However, this is not the case for two reasons. First and foremost, one of the key results we will discuss, a new formulation of the relativistic Euler equations enjoying remarkable properties (Sect.~\ref{S:New_formulation}), is derived through some very delicate computations that exhibit miraculous cancellations. Although it is not feasible to reproduce here all the arguments (which are the bulk of the nearly hundred-pages paper by \citealt{Disconzi-Speck-2019}), readers would not be able to appreciate the miraculous cancellations of this new formulation unless one example of the exact computations is provided. In particular, presenting such a calculations in schematic form would entirely miss the point of the cancellations. Second, while we can always omit intermediate computations and present only the end results, we believe that there
is only that much one can learn if many important calculations are omitted. Although readers should ultimately reproduce such computations for themselves in order to fully understand them, it can be  frustrating to be stuck in some computations while the goal should be to get to the main key ideas of a proof. Thus, we present several of the calculations needed for some important formulas, although we leave them to the appendix when they are too long or cumbersome. 

An important aspect of relativistic fluid dynamics is that the mathematical structures present in the equations of motion can be, depending on the problem one has in mind, substantially different than those present in classical (i.e., non-relativistic) fluids. Thus, often results from relativistic fluids cannot, in general, be obtained by a simple tweak of techniques used for classical fluids. This is an important point that is not always fully appreciated, and thus we will routinely reminded readers of it. That said,
there will be one instance in which we will illustrate our main points using a non-relativistic fluid. In the discussion of rough solutions to the relativistic Euler equations (Sect.~\ref{S:Rough_solutions}), we will dedicate a considerable part of the presentation to establish a similar result for the classical compressible Euler equations. The reason for this is twofold. First, the proof in the simpler setting of the classical compressible Euler equations is already quite demanding, thus discussing the proof in the relativistic setting would add several layers of difficulty to the exposition in a way that some important points common to both the classical and the relativistic setting could be missed. Second, \emph{we want to give at least one example of how geometric-analytic techniques imported from general relativity can be used in the study of fluids in general, thus emphasizing that what makes such techniques useful to the study of relativistic fluids is not so much the fact that we are considering a relativistic theory but rather the fact that we are studying a model with propagation of waves,} as pointed out earlier in the Introduction.

We will be primarily interested in mathematical questions motivated by physics. Thus, our goal is to establish rigorous results under physically relevant assumptions. This is not always possible, in which case
we settle for results that can be proven under physically-inspired scenarios. Results of this kind
should be viewed as a first step toward the eventual goal of proving results under realistic hypotheses.
In this spirit, physical aspect will be invoked when helpful to justify certain mathematical assumption or to clarify basic concepts and terminology. 

There are several different topics covered in this review and readers might have different interest in each of them. Thus, we try as much as possible to make each of the main sections self-contained. Connections between sections, however, will be pointed out whenever it is pertinent. The level of detail and to some extent the style of presentation also varies among sections. Section~\ref{S:Relativistic_Euler} provides a relatively self contained introduction to the relativistic Euler equations from a partial differential equations (PDEs) perspective. More precisely, it is not feasible to give a self-contained introduction to the topic from scratch within the scope of these notes. But readers interested primarily in mathematical aspects of the relativistic Euler equations with focus on the evolution problem from a PDE perspective should be able to quickly absorb the basics of the Cauchy problem needed to understand the more advanced applications of the subsequent sections. Sect.~\ref{S:New_formulation} focuses on hidden structures of the relativistic Euler equations and explains how such structures are key for investigating several questions of interest. Sect.~\ref{S:New_formulation} also introduces some features of the geometric-analytic framework originally developed in studies of Einstein's equations. Such formalism is then expanded in Sect.~\ref{S:Rough_solutions}. Among all different topics presented in this review, those of Sect.~\ref{S:Rough_solutions}, where rough solutions to the Euler equations are studied, is where the discussion will be most schematic. This is due mainly to the extensive amount of background needed for the results of that section. Nevertheless, Sect.~\ref{S:Rough_solutions} should provide readers with a good snapshot of the geometric-analytic formalism. Sect.~\ref{S:Vacuum_bry}, on the other hand, is where estimates are presented in considerably more detail in comparison to the other sections. This is due to the nature of the topic there presented, which allows us to focus primarily on energy estimates, which aside from constituting the basis of the overall  argument, can be well-illustrated with some not-very-complicated calculations. Finally, Sect.~\ref{S:Relativistic_viscous_fluids} will provide little discussion of proofs, focusing primarily in introducing the problems and stating the results. This is because the topics covered in Sect.~\ref{S:Relativistic_viscous_fluids}, relativistic viscous fluids, have been less investigated among mathematicians than those of previous sections, despite a great deal of activity in the physics community which leads to several interesting mathematical questions. In addition, Sect.~\ref{S:Relativistic_viscous_fluids} will be by far the part of this review with more physical discussion, the reason being that the mathematical results of that section are directly motivated by physical questions.
This difference in focus and presentation among different sections is reflected in Sect.~\ref{S:Open_problems}, where we discuss open problems and future directions of research. Some of the open problems will be presented as precise conjectures following directly from the material here presented, whereas other problems will be more tentative, pointing toward interesting yet less specific directions of research. 

For the sake of brevity, we will not provide a thorough literature review. 
Hence, works will be cited only when directly connected to the results being presented. The topics here surveyed all belong to very active fields of research, with important contributions from many authors, 
and a detailed literature review would be outside our scope. 
This, unfortunately, will leave several interesting and related works unmentioned.
In this sense, our manuscript is more properly understood as a review -- in the sense that we review specific results and techniques -- rather than a survey -- where one would provide a comprehensive 
overview of the status of the field. That said, in an attempt to convey the richness of the field of mathematical relativistic fluid dynamics, in Sect.~\ref{S:Further_important_results} we provide some discussion of important works in relativistic fluids that are not directly related to the topics that are the main focus of our presentation.

\subsection{Background\label{S:Background}}

We will assume that readers are familiar with the main ideas of hyperbolic partial differential equations (PDEs), 
such as energy estimates, finite speed of propagation, characteristic sets, etc.
Basic elliptic theory will also be assumed. Some further tools from analysis will also be employed in some
parts of the text (e.g., Littlewood--Paley theory in Sect.~\ref{S:Rough_solutions} or weighted Sobolev embeddings in Sect.~\ref{S:Vacuum_bry}). Readers not familiar with such tools should have no difficulty finding their basic results in standard textbooks, although readers should still be able to follow our main arguments if they simply take our applications of these tools at face value.

We will assume familiarity with Lorentzian geometry and the basics of relativity theory. 
We stress that not much background in these subjects will in fact be needed for our discussion. 
After appealing to ideas from geometry and general relativity at the very beginning (e.g., introducing concepts such as an energy-momentum tensor), we will quickly write down equations of motion and focus on their study from a PDE perspective. More precisely, much of the Lorentzian geometry needed for the proofs we will discuss will be hidden
in our high-level presentation. 
Thus, with a bit of patience, 
students and researchers trained in traditional methods of fluid dynamics should still be able to follow most of our arguments. In fact, readers acquainted only with basic ideas from geometry, such as the 
definition of a Riemannian metric and the notion that a Lorentzian metric is like a Riemannian metric that is not positive definite, should have no difficulty reading the paper. We will talk about 
covariant derivatives, but for the most part readers not trained in geometry can think of covariant derivatives as roughly regular partial derivatives without compromising much of the PDE aspects which will be our main focus. The concept of curvature will also be invoked but, once again, readers untrained in geometry can take for granted that the curvature tensors we will use are certain combination of derivatives of the solution variables (of the metric in the case of Einstein's equations; of the fluid variables in the case of the curvature of the acoustical metric). An instrumental introduction to the most basic aspects of Lorentzian geometry that could help readers with no geometry background follow our presentation can be found in Sect.~3 of \cite{Disconzi-USCFluidNotes}.

There are many books and monographs providing extensive background on relativistic fluids, Einstein's equations, and Lorentzian geometry. Here we present a few suggestions to the interested reader, stressing that this is an incomplete list. 

\begin{itemize}
\item A thorough introduction to relativistic fluids can be found
in the book by \cite{Rezzolla-Zanotti-Book-2013}. It covers the physics of relativistic fluids and its basic mathematical aspects. It also contains a short introductory chapter on Lorentzian geometry that 
can be helpful to readers unfamiliar with the topic.
\item A short introduction to the mathematics of relativistic perfect fluids can be found in \cite{Christodoulou-2007}.
\item The book by \cite{Romatschke:2017ejr} summarizes much of what is known about the physics of relativistic viscous fluids, with emphasis on connections with experiments and numerical simulations.
\item The review by \cite{Rocha:2023ilf} gives a comprehensive outlook of most theories of relativistic viscous fluids, with an emphasis on kinetic theory aspects.
\item Applications of relativistic fluids to cosmology can be found in \cite{Weinberg-Book-2008}.
\item The microscopic foundations of relativistic fluid dynamics are presented in 
\cite{Denicol-Rischke-Book-2021}. Although microscopic theory will not be discussed in our presentation and we will only make some very brief references to it when commenting on some physical aspects, it constitutes an important aspect of research on relativistic fluids.
\item Taken together, the above references will give readers a very good view of the importance and applications of relativistic fluid dynamics.
\item A more mathematical take on relativistic perfect fluids can be found in 
\cite{Anile-Book-1990,Lichnerowicz-Book-1967}.
\item The monograph by \cite{Choquet-Bruhat-Book-2009} is provides a very comprehensive study of general relativity and the Einstein equations, focusing mostly on mathematical results. It stars with an introduction to Lorentzian geometry and covers a tremendous amount of material, including a chapter on relativistic fluids.
\item The books by \cite{Hawking-Ellis-Book-1973,Wald:1984rg} are standard introductions to general relativity.
\item A mathematical introduction to the Cauchy problem for Einstein's equations, which also provides much of the necessary background in Lorentzian geometry, can be found in \cite{Ringstrom-Book-2009}.
\item An introduction to the geometric-analytic tools we will employ, albeit focusing strictly on wave
equations and thus with no discussion of fluids, can be found in
\cite{Alinhac-Book-2010}. See also the introductory online lecture notes by \cite{Aretakis-NotesGR}.
\item A recent review on the problem of shock formation for the relativistic Euler equations can be found in \cite{Abbrescia-Speck-2023}. While there is some intersection between this work and our exposition in Sect.~\ref{S:Study_of_shocks}, our treatment of the problem of shocks and \cite{Abbrescia-Speck-2023} complement each other.
\end{itemize}

\subsection{Notation and conventions}
\label{S:Notation_conventions}

We will assume familiarity with Lorentzian geometry\footnote{See, however, our comments in Sect.~\ref{S:Background} about pre-requisites in geometry.} and the basics of partial differential equations (PDEs). Unless stated otherwise, we will always assume as given a differentiable four-dimensional manifold $M$
equipped with a Lorentzian metric $g$ --- so $(M,g)$ will be the spacetime --- where our objects (tensors etc.) will be defined.
For much of the discussion, precise details about $(M,g)$ will not be important, and we will use them simply to define our basic objects. Precise assumptions on $(M,g)$ will be given for the statement of the theorems we will discuss. The tangent and cotangent bundles of $M$ will be denoted by $TM$ and $T^*M$, respectively.

All fluids will be relativistic, unless explicitly mentioned otherwise. Thus, sometimes we refer to them simply as fluids instead of relativistic fluids. The word classical will be used to refer to non-relativistic objects, e.g., a classical (non-relativistic) fluid\footnote{See Footnote \ref{FN:Classical_NR}.}.

Unless stated otherwise, we adopt the following notation, abbreviations, and conventions:
\begin{itemize}
\item Greek indices run from $0$ to $3$, Latin indices from $1$ to $3$, and repeated indices appearing once upstairs and once downstairs are summed over their range.
\item $\{ x^\alpha \}_{\alpha=0}^3$ denotes local coordinates in spacetime, with $t := x^0$ denoting a time coordinate and
$\{x^i\}_{i=1}^3$ denoting spatial coordinates. 
In such coordinates, we often write $(t,x)$ for a spacetime point, with $x=(x^1,x^2,x^3)$. We write
$\{\frac{\partial}{\partial x^\alpha} \}_{\alpha=0}^3$ 
or $\{ \partial_\alpha \}_{\alpha=0}^3$ for the corresponding basis of coordinate vectors.
\item Our signature convention for Lorentzian metrics is $-+++$.
\item We use the terms one-form and co-vector interchangeably. ``One-form'' seems more common in the math literature, but physicists seem to prefer ``co-vector."
\item Indices are raised and lowered with the spacetime metric. We will silently 
use the spacetime metric to identify vectors and one-forms. 
We follow a standard convention in relativity of using the same letter for vectors and one-forms related by the metric, e.g., $u_\alpha = g_{\alpha\beta} u^\beta$, where $g$ is the spacetime metric.
In particular, we do not use the notation of ``musical operators,'' e.g., $u^\flat$ for a one-form associated with a vectorfield $u$ and $\omega^\sharp$ for a vectorfield associated with a one-form $\omega$.
\item $\nabla$ is the covariant derivative associated with the spacetime metric.
\item In Sect.~\ref{S:Characteristics_Euler}, Definition \ref{D:Acoustical_metric}, we will introduce the so-called acoustical metric, which, as we will show, is a Lorentzian metric that plays an important role in the mathematical study of relativistic fluids. We stress that, unless stated otherwise, indices are not raised and lowered with the acoustical metric, and $\nabla$ will not be the covariant derivative associated with the acoustical metric. Similarly, other geometric concepts, like ``orthogonality,'' etc., will be with respect to the spacetime metric, and not the acoustical metric, unless stated otherwise.  
\item We use units where $c_\ell = 8\pi G = 1$, where 
$c_\ell$ is the speed of light (in vacuum) and $G$ is Newton's gravitational constant.
\item $H^N$ denotes the Sobolev space with norm $\norm{\cdot}_N$. The relevant domain for maps in $H^N$, i.e., $H^N(M)$, $H^N(\mathbb{R}^3)$, etc., will be specified in the appropriate context and when needed we will also specify it in the norm, e.g., $\norm{\cdot}_{H^N(\mathbb{R}^3)}$.
\item $\Sigma_t$ will denote a constant-time hypersurface, i.e., if $(x^0,x^1, x^2, x^3) = (t, x)$, then 
$\Sigma_t := \{ (\tau,x) \, | \, \tau = t \}$.
\item We use the standard notation $A \lesssim B$ to mean $A \leq C B$ where $C$ is some universal constant depending only on ultimately fixed parameters (e.g., the dimension of space or the size of a time interval).
\item L.O.T. stands for lower-order terms. Which terms can be treated as lower order depends on the context and the terms subsumed under L.O.T. can vary from line to line.
\item RHS = right-hand side, LHS = left-hand side.
\item $\updelta$ denotes the Euclidean metric and $\updelta^\alpha_\beta$ the Kronecker delta.
\item $\partial$ will denote a generic spacetime derivative, i.e., 
a generic $\partial_\alpha$ derivative. 
$\bar{\partial}$ will denote a generic spatial derivative, i.e., a generic $\partial_i$ derivative.
They will
be often used when only the derivative counting of some expression is relevant. 
\item $\sim$ will be used to mean equality up to ``unimportant terms.'' Naturally, what classifies as ``unimportant'' depends on the context, which typically will be clear from the surrounding discussion. For example, in a typical PDE fashion, one often is interested only in a derivative counting, and thus would write something like $\nabla_\mu \nabla^\mu \phi \sim \partial \phi$ to mean that in this equation for $\phi$ only derivatives up to first-order\footnote{One often uses a notation like $\partial^{\leq 1} \phi$ to denote an expression depending on up to first-order derivatives of $\phi$, but we do not do so here.} of $\phi$ appear on the RHS, with the precise way this dependence on first-order derivatives (e.g., linearly, nonlinearly, etc.) being of no importance for the argument being discussed. When more than one quantity is involved we  separate them by commas, e.g.,  $\sim \partial \phi, \partial \psi$ denotes equality to a generic expression involving up to first-order derivative of $\phi$ and $\psi$.
\item $A \approx B$ means $A \lesssim B$ and $B \lesssim A$.
\item Different volume form will be used in different parts of the text but for the most part they will always be the ``natural'' one. Thus, when no confusion can arise, we will omit the integration measure, e.g., writing $\int f$ for $\int f \, dx$. 
\item When derivatives and covariant derivatives appear to the left of a term without a parenthesis, that means that it differentiates only the term immediately to its right. E.g., in $\partial_\alpha u v$ only $u$ is differentiated\footnote{This is standard notation, but we remark this here because some authors prefer to emphasize the differentiation of only the term immediately to the right with parenthesis, i.e., writing $(\partial_\alpha u) v$ for $\partial_\alpha u v$.}, with $\partial_\alpha( u v)$ used to indicate a derivative of the product $uv$.
\item New concepts and terminology, whether introduced in Definitions or in the body of the text, will be highlighted with \textdef{bolded text.}
\item We use $:=$ for definitions in mathematical expressions, i.e., $A := B$ means that $A$ is defined by the expression $B$.
\end{itemize}

\section{The relativistic Euler equations}
\label{S:Relativistic_Euler}

The dynamics of a perfect (i.e., no viscous) relativistic fluid is described by the relativistic Euler equations introduced below.

\begin{definition}
\label{D:Energy_momentum_perfect}
The \textdef{energy-momentum tensor of a relativistic perfect fluid} is the symmetric two-tensor on\footnote{Recall our notation and conventions from Sect.~\ref{S:Notation_conventions}.} 	$M$
\begin{align}
\label{E:Energy_momentum_perfect}
\mathcal{T}_{\alpha\beta} := (p+\varrho) u_\alpha u_\beta
+ p g_{\alpha\beta},
\end{align}
where $g$ is a Lorentzian metric on $M$, $p\colon M \rightarrow \mathbb{R}$ and $\varrho\colon M \rightarrow \mathbb{R}$ are real-valued functions representing the \textdef{pressure} and
\textdef{energy density} of the fluid, respectively, $u\colon M \rightarrow TM$
is a vectorfield representing the \textdef{velocity} of the fluid and normalized by
\begin{align}
\label{E:Velocity_normalization}
\abs{u}^2_g = g_{\alpha\beta} u^\alpha u^\beta = u^\alpha u_\alpha = -1. 
\end{align}
\end{definition}

\begin{remark}
$u$ is often referred to as the fluid's \textdef{four-velocity}, emphasizing that it is a vectorfield in spacetime. We will refer to it simply as velocity unless the terminology is ambiguous or we want to emphasize its four-dimensional nature. Similarly for all ``four-'' quantities often used in relativity, e.g., four-acceleration, etc.
\end{remark}

\begin{remark}
Often perfect fluids are also called ideal fluids and both terms are used interchangeably here, although some authors (e.g., \citealt{Rezzolla-Zanotti-Book-2013}) reserve the terminology 
``ideal'' for fluids that obey the equation of state of an ideal gas.
\end{remark}

The assumption $\abs{u}^2_g = -1$ can be understood as follows. Recall that in relativity, observers are defined by their (timelike) world-line up to reparametrizations. More precisely, the norm of a tangent vector to the world-line has no physical meaning if the parameter is not specified. Thus, we can choose to normalize the observer's velocity to $-1$. In the case of a fluid, we can identify the flow lines of $u$ with the world-line of observers traveling with the fluid particles. The normalization $\abs{u}^2_g = -1$ also says that $u$ is timelike, so fluid particles do not travel faster than or at the speed of light.

The normalization $\abs{u}^2_g = -1$ has yet another physical interpretation. The energy density $\varrho$ in $\mathcal{T}$ is the energy density measured by an observer traveling with the fluid (i.e., at rest with respect to the fluid). It is possible to show, using kinetic theory (see \citealt{Denicol-Rischke-Book-2021}), that the energy density measured by an observer with velocity $v$ will be $v^\alpha v^\beta \mathcal{T}_{\alpha\beta}$. Thus, for the fluid velocity itself we need to have $\varrho = u^\alpha u^\beta \mathcal{T}_{\alpha\beta}$, thus $u^\alpha u_\alpha = -1$.

Let us make another remark about kinetic theory. It can be shown that the energy-momentum tensor
\eqref{E:Energy_momentum_perfect} arises as a suitable limit, via a coarse-graining procedure, of a microscopic dynamics described by kinetic theory (under appropriate assumptions and ignoring dissipation), see \cite{Denicol-Rischke-Book-2021,DeGroot:1980dk}. One can also justify definition \eqref{E:Baryon_density_current_perfect} below from kinetic theory.
While kinetic theory provides one of the best justifications for defining $\mathcal{T}$ by \eqref{E:Energy_momentum_perfect}, it is also possible 
to postulate \eqref{E:Energy_momentum_perfect} motivated by certain physical considerations (see \citealt{Weinberg:1972kfs}). Similar remarks apply to \eqref{E:Baryon_density_current_perfect}.

The normalization $\abs{u}^2_g = -1$ implies the following useful identity which will often be silently used in many of our computations,
\begin{align}
\label{E:Consequence_velocity_normalization}
u^\alpha \nabla_\mu u_\alpha = 0.
\end{align}
Contracting \eqref{E:Consequence_velocity_normalization} with $u^\mu$ implies that the \textdef{fluid's acceleration}
given by $a^\alpha := u^\mu \nabla_\mu u^\alpha$ is orthogonal to the fluid's velocity, i.e., $a^\alpha u_\alpha = 0$.

Finally, normalization \eqref{E:Velocity_normalization} allows us to define\footnote{The definition of a
LRF relies only on \eqref{E:Velocity_normalization} and, thus, is not exclusive to the relativistic Euler equations. In particular, we can define a LRF for the viscous fluid theories studied in Sect.~\ref{S:Relativistic_viscous_fluids}.} a fluid's \textdef{local rest frame (LRF)}, which is an orthonormal frame $\{ e_\alpha \}_{\alpha=0}^3$ such that $e_0 = u$.

\begin{remark}
Fluids described by \eqref{E:Energy_momentum_perfect} are sometimes called isotropic as one is assuming that if one is at rest with respect to the fluid then stresses in all directions of the fluid are the same. This means that in a LRF, $T_{ii} = p$. It is possible to consider perfect fluid models without this assumption (see \citealt{Rezzolla-Zanotti-Book-2013}), but we will not do so here. 
For fluids
with viscosity, to be introduced later, such isotropy does not hold.
One should be careful to note that isotropy can be used to mean different things in the literature. In particular, by saying that we will consider only perfect fluids that are isotropic we are not saying that we will impose symmetry conditions in our study (as it would be by assuming, e.g., that the fluid is spherically symmetric).
\end{remark}

\begin{definition}
\label{D:Baryon_density_current_perfect}
The \textdef{baryon density current,} or simply \textbf{baryon current,} of a perfect fluid is the
vectorfield $\mathcal{J}\colon M \rightarrow TM$ given by
\begin{align}
\label{E:Baryon_density_current_perfect}
\mathcal{J}^\alpha := n u^\alpha,
\end{align}
where $n\colon M \rightarrow \mathbb{R}$ is a real-valued function representing the \textdef{baryon number density.} or simply \textbf{baryon density,}  of the fluid and $u$ is the fluid's velocity introduced above.
\end{definition}

Physically, the baryon number density gives the denstiy of matter\footnote{\label{FN:Matter}Due to the ``equivalence of matter and energy'' in relativity, an equivalence that can be made mathematically precise, one needs to be more careful when talking about ``density of matter'' and how it is distinguished from the energy density. But for our purposes, where the focus is mathematical and we treat $\varrho$ and $n$ primarily as scalars entering as variables in the equations, such physical considerations can be neglected. Readers trained in classical fluid dynamics should note that in the non-relativistic limit, $n$ reduces to the classical density variable (matter per volume) that satisfies the continuity equation, while $\varrho$ reduces to the classical notion of energy density.} of the fluid: the rest-mass density (measured by an observer at rest with respect to the fluid) is given by $n m$,
where $m$ is the mass of the baryonic\footnote{Readers unfamiliar with this terminology can take ``baryonic particles'' to mean ``matter particles.'' See also Footnote \ref{FN:Matter}.} particles that constitute the fluid (see \citealt{Rezzolla-Zanotti-Book-2013}), which here we take to be equal to one.

Physically, the quantities $p,\varrho$, and $n$ are not all independent and are related by a relation known as an \textdef{equation of state} (whose choice depends on the nature of the fluid). Under ``normal circumstances'' (e.g., absent phase transitions) this relation is assumed to be invertible, i.e., knowledge of any two quantities, e.g., $\varrho$ and $n$, determines the third, e.g., $p$. In this case, we can choose any two such quantities to be the fundamental/primitive/primary variables/unknowns. We will choose $\varrho$ and $n$ as primary variables, assuming that
$p$ is a given function of them, i.e., $p = p(\varrho,n)$, although later on it will be more convenient to make other choices. It is also possible to use thermodynamic relations (see Sect.~\ref{S:Thermodynamic_properties}) to introduce other scalar quantities of physical interest, such as temperature or entropy, and use these instead as primary variables. We will be more precise about which variables we will take as the unknowns for
the evolution when we discuss the Cauchy problem. 

\begin{definition}[The relativistic Euler equations]
The \textdef{relativistic Euler equations} are defined as
\begin{subequations}{\label{E:Relativistic_Euler_eq_full_system}}
\begin{align}
\nabla_\alpha \mathcal{T}^\alpha_\beta & = 0,
\label{E:Relativistic_Euler_eq_full_system_energy_momentum}
\\
\nabla_\alpha \mathcal{J}^\alpha & = 0, 
\label{E:Relativistic_Euler_eq_full_system_baryon_charge}
\\
g_{\alpha\beta} u^\alpha u^\beta & = -1,
\label{E:Relativistic_Euler_eq_full_system_normalization}
\\
p & = p(\varrho, n).
\label{E:Relativistic_Euler_eq_full_system_equation_of_state}
\end{align}
\end{subequations}
Above, $\mathcal{T}$ and $\mathcal{J}$ are the energy-momentum tensor and baryon current of a relativistic perfect fluid given in Definitions \ref{D:Energy_momentum_perfect}
and \ref{D:Baryon_density_current_perfect}, respectively, $g$ is the spacetime metric figuring in $\mathcal{T}$, and $\nabla$ is its covariant derivative.
\end{definition}

Equation \eqref{E:Relativistic_Euler_eq_full_system_energy_momentum} corresponds to conservation of energy and momentum in the fluid, Eq.~\eqref{E:Relativistic_Euler_eq_full_system_baryon_charge} corresponds to conservation of baryonic charge, \eqref{E:Relativistic_Euler_eq_full_system_normalization} is the normalization condition discussed above (which is a constraint rather than an evolution equation), and \eqref{E:Relativistic_Euler_eq_full_system_equation_of_state} is the equation of state (which defines $p$ in terms of $\varrho$ and $n$).

\begin{remark}
On physical grounds, one often requires $\varrho\geq 0$, 
$n \geq 0$, and $p\geq 0$. From the point of view of the Cauchy problem, these conditions should be assumed for the initial data and showed to be propagated by the flow.
\end{remark}

\begin{remark}
If the metric $g$ is known, then equations
\eqref{E:Relativistic_Euler_eq_full_system} form a closed system of equations. Equations \eqref{E:Relativistic_Euler_eq_full_system_energy_momentum} and \eqref{E:Relativistic_Euler_eq_full_system_baryon_charge} are five equations for the five quantities $\varrho$, $n$ and three components of $u$, with a fourth component of $u$ and $p$ determined by
\eqref{E:Relativistic_Euler_eq_full_system_normalization} and \eqref{E:Relativistic_Euler_eq_full_system_equation_of_state}, respectively. 
\end{remark}

We introduce the symmetric two-tensor
\begin{align}
\label{E:Projection_u_orthogonal}
\proj_{\alpha\beta} := g_{\alpha\beta} + u_\alpha u_\beta,
\end{align}
which corresponds to 
\textdef{projection onto the space orthogonal to} $u$, i.e., 
\begin{align}
\label{E:Projection_u_zero}
\proj_{\alpha\beta} u^\beta = 
u_\alpha + u_\alpha \underbrace{u_\beta u^\beta}_{=-1} = 0,
\end{align}
and if $v$ is orthogonal to $u$,
\begin{align}
\proj_{\alpha\beta} v^\beta = v_\alpha
+ u_\alpha \underbrace{u_\beta v^\beta}_{=0} = v_\alpha.
\nonumber
\end{align}
It is convenient to decompose $\nabla_\alpha \mathcal{T}^\alpha_\beta$ in the directions parallel and orthogonal $u$,
\begin{align}
\begin{split}
\nabla_\alpha \mathcal{T}^\alpha_\beta 
& =
\nabla_\alpha 
	(
	(p + \varrho) u^\alpha u_\beta
	+ p g^\alpha_\beta
	)
\\
& = 
u^\alpha \nabla_\alpha (p+\varrho) u_\beta
+ (p+\varrho) \nabla_\alpha u^\alpha u_\beta
+ (p+\varrho) u^\alpha \nabla_\alpha u_\beta
+ \nabla_\beta p,
\end{split}
\nonumber
\end{align}
so that, using \eqref{E:Relativistic_Euler_eq_full_system_normalization}, 
\eqref{E:Projection_u_orthogonal},
\eqref{E:Consequence_velocity_normalization}, and
\eqref{E:Projection_u_zero}, we obtain
\begin{align}
\begin{split}
u^\beta \nabla_\alpha \mathcal{T}^\alpha_\beta
& = 
-u^\alpha \nabla_\alpha (p+\varrho) - (p+\varrho) \nabla_\alpha u^\alpha + (p+\varrho) u^\alpha
\underbrace{u^\beta \nabla_\alpha u_\beta}_{=0} 
+ u^\beta \nabla_\beta p
\\
& = - u^\alpha\nabla_\alpha\varrho - (p+\varrho)\nabla_\alpha u^\alpha,
\end{split}
\nonumber
\end{align}
and
\begin{align}
\begin{split}
\proj^{\gamma\beta} \nabla_\alpha \mathcal{T}^\alpha_\beta 
& =
u^\alpha \nabla_\alpha (p+\varrho) 
\underbrace{\proj^{\gamma \beta}u_\beta}_{=0}
+(p+\varrho) \nabla_\alpha u^\alpha 
\underbrace{\proj^{\gamma \beta}u_\beta}_{=0}
+ (p+\varrho) \proj^{\gamma \beta} u^\alpha \nabla_\alpha u_\beta
+\proj^{\gamma \beta}\nabla_\beta p
\\
& =
(p+\varrho)u^\alpha
	(
	\underbrace{g^{\gamma\beta} \nabla_\alpha 
	u_\beta}_{=\nabla_\alpha u^\gamma}
	+ u^\gamma  
	\underbrace{u^\beta\nabla_\alpha u_\beta}_{=0}
	)
+\proj^{\gamma \beta}\nabla_\beta p
\\
&=
(p+\varrho) u^\alpha \nabla_\alpha u^\gamma + \proj^{\gamma\beta}\nabla_\beta p.
\end{split}
\nonumber
\end{align}
Expanding \eqref{E:Relativistic_Euler_eq_full_system_baryon_charge}, we have
\begin{align}
\nabla_\alpha \mathcal{J}^\alpha = u^\alpha \nabla_\alpha n
+ n \nabla_\alpha u^\alpha.
\nonumber
\end{align}
Therefore, we can alternatively write equations
\eqref{E:Relativistic_Euler_eq_full_system} as
\begin{subequations}{\label{E:Projected_relativistic_Euler_eq_full_system}}
\begin{align}
u^\alpha \nabla_\alpha \varrho + (p+\varrho) \nabla_\alpha u^\alpha 
&=0,
\label{E:Projected_relativistic_Euler_eq_full_system_energy}
\\
(p+\varrho) u^\alpha \nabla_\alpha u^\beta + \proj^{\beta\alpha}\nabla_\alpha p
&=0,
\label{E:Projected_relativistic_Euler_eq_full_system_momentum}
\\
u^\alpha \nabla_\alpha n + n \nabla_\alpha u^\alpha 
&=0,
\label{E:Projected_relativistic_Euler_eq_full_system_continuity}
\\
g_{\alpha\beta} u^\alpha u^\beta 
&= -1,
\label{E:Projected_relativistic_Euler_eq_full_system_normalization}
\\
p &= p(\varrho,n).
\label{E:Projected_relativistic_Euler_eq_full_system_equation_of_state}
\end{align}
\end{subequations}
Equation \eqref{E:Projected_relativistic_Euler_eq_full_system_energy} is interpreted as conservation of energy, \eqref{E:Projected_relativistic_Euler_eq_full_system_momentum} as conservation of momentum, and \eqref{E:Projected_relativistic_Euler_eq_full_system_continuity}
is the conservation of baryon density and often referred to as
the continuity equation. These equations reduce to the non-relativistic
compressible Euler equations in the non-relativistic limit (see \citealt{Rezzolla-Zanotti-Book-2013}).

Observe that without assuming \eqref{E:Projected_relativistic_Euler_eq_full_system_normalization} but still taking $u$ to be timelike, the projection onto the space orthogonal to $u$ is given by
\begin{align}
\label{E:Projection_u_orthogonal_non_normalized}
\proj_{\alpha\beta} = g_{\alpha\beta} - \frac{u_\alpha u_\beta}{u^\lambda u_\lambda}.
\end{align}

In this case, contracting \eqref{E:Projected_relativistic_Euler_eq_full_system_momentum} with $u_\beta$ gives
\begin{align}
\label{E:Velocity_constraint_propagation}
(p+\varrho) u^\alpha \nabla_\alpha (u_\beta u^\beta) = 0.
\end{align}
Thus, for $p+\varrho > 0$, \emph{the constraint \eqref{E:Projected_relativistic_Euler_eq_full_system_normalization} is propagated by the flow, i.e., it holds at later times provides that it holds initially.} Under these circumstances, one can check
that solutions to \eqref{E:Projected_relativistic_Euler_eq_full_system} yield
solutions to \eqref{E:Relativistic_Euler_eq_full_system}.

\begin{notation}
\label{N:Assuming_normalization_and_propagation}
Henceforth, we will always consider that one of the the equations of motion is the constraint \eqref{E:Projected_relativistic_Euler_eq_full_system_normalization}. This will be the case also for the viscous theories we will discuss later. Thus, \eqref{E:Projected_relativistic_Euler_eq_full_system_normalization} will often be omitted. 
Similarly, an equation of state will always be assumed as given, so that \eqref{E:Projected_relativistic_Euler_eq_full_system_equation_of_state} (or a similar relation when the primary variables are not
$\varrho$ and $n$) will also be omitted.
\end{notation}

\begin{remark}
\label{R:All_velocity_components}
Although  
\eqref{E:Projected_relativistic_Euler_eq_full_system_normalization}
is always a part of the Euler system even if omitted, as  
explained in the Notation \ref{N:Assuming_normalization_and_propagation}, 
from the point of view of the Cauchy problem it will be
more convenient to assume that all components of $u$ are independent, and consider the evolution for all components
on the same footing. After establishing local existence and uniqueness of solutions in such a situation, one recovers the constraint
\eqref{E:Projected_relativistic_Euler_eq_full_system_normalization}
by the above argument showing that it is propagated. 
\end{remark}

While it is not difficult to obtain local existence and uniqueness by writing \eqref{E:Projected_relativistic_Euler_eq_full_system}
as a first-order symmetric hyperbolic system (see \citealt{Anile-Book-1990}), we will use a different approach due to \cite{Lichnerowicz-Book-1967}, generalizing earlier work
by \cite{Choquet-Bruhat-1958}, that makes the role
of the characteristics manifest and connects with what we will discuss later. In fact, as we will see, but also as expected physically, \emph{there are two types of propagation in a fluid described by the relativistic Euler equations: sound waves and transport of vorticity and entropy\footnote{The entropy is introduced in Sect.~\ref{S:Thermodynamic_properties}.}.} These correspond to \emph{different characteristics} and thus we would expect that they should be treated differently.
The first-order symmetric hyperbolic formulation, however, treats both characteristics at the same level.

Before investigating local existence and uniqueness of solutions,
we need to introduce a few more concepts.

\subsection{Thermodynamic properties of relativistic fluids}
\label{S:Thermodynamic_properties}

We begin by introducing the following quantities:
\begin{itemize}
\item The \textdef{internal energy density} $\mathcal{E}$ of a fluid
is defined by
\begin{align}
\varrho = n(1+\mathcal{E}).
\nonumber
\end{align}
Thus, the energy density of the fluid takes into account the energy coming from the fluid's rest mass (strictly speaking the factor $n$ should be replaced by the rest-mass density $m n$, but recall from the previous section that we take $m=1$).
\item The \textdef{specific enthalpy} $h$ of a fluid is defined by
\begin{align}
\label{E:Enthalpy_definition}
h := \frac{p+\varrho}{n}, \, n > 0.
\end{align}
\item We assume the existence of functions $s$ and $\uptheta$, called the
\textdef{entropy density}, a.k.a. \textdef{specific entropy}, and \textdef{temperature} of the fluid, respectively, such that the \textdef{first-law of thermodynamics} holds:
\begin{align}
\label{E:First_law_thermodyanmics_dp_version}
dp = n dh - n\uptheta ds,
\end{align}
where $d$ is the exterior derivative in spacetime.
\end{itemize}

The above definitions and relations can be introduced and justified in a systematic way based on well-known physical and mathematical notions (see 
\citealt{Christodoulou-2007,Landau-Lifshitz-Book-Fluids-1987,Rezzolla-Zanotti-Book-2013}). Here, we will take them as god-given and work out their mathematical consequences.

We note that the first-law of thermodynamics can alternatively be written as
\begin{align}
d\varrho &= h dn + n\uptheta ds,
\nonumber
\\
d\mathcal{E} &= - p d\left(\frac{1}{n}\right) + \uptheta ds.
\nonumber
\end{align}

As before, we can choose which two functions among the so-called 
\textdef{thermodynamic scalars} $\varrho, p, n, h, \mathcal{E}, s, \uptheta$ are independent, with the remaining ones a function of the two primary ones determined by an equation of state and the above relations (called thermodynamics relations). Different choices will be more appropriate for different questions.

With the above definitions, we can write \eqref{E:Energy_momentum_perfect} as
\begin{align}
\mathcal{T}_{\alpha\beta} &= nh u_\alpha u_\beta + p g_{\alpha\beta},
\nonumber
\end{align}
so that
\begin{align}
\nabla_\alpha \mathcal{T}^\alpha_\beta 
& = 
\nabla_\alpha(n h u^\alpha) u_\beta + nh u^\alpha \nabla_\alpha u_\beta + \nabla_\beta p,
\nonumber
\end{align}
and thus, using \eqref{E:Consequence_velocity_normalization}, \eqref{E:Relativistic_Euler_eq_full_system_baryon_charge}, 
and \eqref{E:First_law_thermodyanmics_dp_version}, we find
\begin{align}
u^\beta \nabla_\alpha \mathcal{T}^\alpha_\beta 
& = 
- h \underbrace{\nabla_\alpha(n  u^\alpha)}_{=0} - n u^\alpha \nabla_\alpha h
+ u^\beta \nabla_\beta p
\nonumber
\\
& = u^\alpha \underbrace{(-n\nabla_\alpha h + \nabla_\alpha p )}_{=-n\uptheta \nabla_\alpha s}
\nonumber
\\
&= -n\uptheta u^\alpha \nabla_\alpha s.
\nonumber
\end{align}
Thus, under the physically natural assumptions $\uptheta > 0$ and
$n>0$ (the latter needed for the definition of $h$),
\emph{which we hereafter assume},
 we find
\begin{align}
\label{E:Entropy_transported}
u^\alpha \nabla_\alpha s = 0.
\end{align}
Equation \eqref{E:Entropy_transported} says that the fluid
motion is locally adiabatic, meaning that the entropy is constant
along the fluid's flow lines.

\subsection{The characteristics of the Euler system\label{S:Characteristics_Euler}}
In view of Sect.~\ref{S:Thermodynamic_properties}, using $\varrho$, $s$, and $u$ as primary variables (so in particular $p = p(\varrho,s)$),
the relativistic Euler equations can be written as
\begin{subequations}{\label{E:Projected_relativistic_Euler_density_entropy_full_system}}
\begin{align}
(p+\varrho) u^\alpha \nabla_\alpha u^\beta + 
\frac{\partial p}{\partial \varrho} \proj^{\alpha \beta} \nabla_\alpha \varrho + \frac{\partial p}{\partial s} \proj^{\alpha \beta} \nabla_\alpha s 
& = 0,
\label{E:Projected_relativistic_Euler_density_entropy_full_system_momentum}
\\
u^\alpha \nabla_\alpha \varrho + (p+\varrho) \nabla_\alpha u^\alpha 
& = 0,
\label{E:Projected_relativistic_Euler_density_entropy_full_system_density}
\\
u^\alpha \nabla_\alpha s & = 0,
\label{E:Projected_relativistic_Euler_density_entropy_full_system_entropy}
\end{align}
\end{subequations}
or equivalently
\begin{align}
A^\alpha \nabla_\alpha \Phi &= 0,
\nonumber
\end{align}
where $\Phi$ is the six-component vector $\Phi = (u^\lambda, \varrho, s)$ and the matrices $A^\alpha$ are given by
\begin{align}
\label{E:Matrix_Euler_first-order_formualtion_density_entropy}
A^\alpha = 
\begin{bmatrix}
\left.(p+\varrho) u^\alpha \updelta^\beta_\lambda \right._{
\textcolor{gray}{\, 4\times 4}} 
& 
\left.\proj^{\beta \alpha}\frac{\partial p}{\partial \varrho}\right._{\textcolor{gray}{\, 4\times 1}}
& 
\left.\proj^{\beta\alpha} \frac{\partial p}{\partial s}\right._{\textcolor{gray}{\, 4\times 1}}
\\
\left.(p+\varrho) \updelta^\alpha_\lambda\right._{\textcolor{gray}{\, 1\times 4}}
& 
\left. u^\alpha \right._{\textcolor{gray}{1\times 1}}
& 
\left. 0 \right._{\textcolor{gray}{1\times 1}}
\\
\left. 0 \right._{\textcolor{gray}{1\times 4}}
& 
\left. 0 \right._{\textcolor{gray}{1\times 1}}
& 
\left. u^\alpha \right._{\textcolor{gray}{1\times 1}}
\end{bmatrix},
\end{align}
where we indicated with subscripts and different color the size of each submatrix (observe that the index $\alpha$ labels the matrices and not their entries). For any one-form $\xi \in T^*M$, 
we proceed to compute the characteristic determinant
\begin{align}
\label{E:Computation_characteristic_determinant_Euler}
\begin{split}
\det  (A^\alpha \xi_\alpha)
&= u^\alpha \xi_\alpha 
\det
\begin{bmatrix}
(p+\varrho) u^\alpha\xi_\alpha \delta^\beta_\lambda
&
\proj^{\beta \alpha}\xi_\alpha \frac{\partial p}{\partial \varrho}
\\
(p+\varrho) \xi_\lambda 
&
u^\alpha \xi_\alpha
\end{bmatrix}
\\
&=
\det
\begin{bmatrix}
(p+\varrho) u^\alpha\xi_\alpha \delta^\beta_\lambda
&
\proj^{\beta \alpha}\xi_\alpha \frac{\partial p}{\partial \varrho}
\\
0
&
(u^\alpha \xi_\alpha)^2 - \proj^{\alpha\beta} \xi_\alpha \xi_\beta \frac{\partial p}{\partial \varrho}
\end{bmatrix}
\\
&= 
(p+\varrho)^4 (u^\alpha \xi_\alpha)^4 
	[ 
	(u^\mu \xi_\mu)^2
	- \frac{\partial p}{\partial \varrho} \proj^{\mu\nu} \xi_\mu 	
	\xi_\nu 
	].
\end{split}
\end{align}
To obtain the second line, we multiplied the first row (more precisely, each of the first four rows) by $\xi_\beta$ and subtracted it from the fifth row times $u^\alpha \xi_\alpha$.

As usual, the characteristics are determined by $\det(A^\alpha \xi_\alpha) = 0$, so from \eqref{E:Computation_characteristic_determinant_Euler} we obtain that one set of characteristics are determined by 
$u^\alpha \xi_\alpha = 0$. \emph{These characteristics are
the \textdef{flow lines,}} i.e., the integral curves of $u$.

Another set of characteristics is determined by setting the term in brackets in \eqref{E:Computation_characteristic_determinant_Euler} equal to zero. In order to analyze this term, the invariance of the characteristics allows us to choose a convenient frame
$\{ e_A \}_{A=0}^3$ with $e_0 = u$ and $\{e_1, e_2, e_3\}$ orthonormal and orthogonal to $u$ (i.e., we consider the LRF, but use uppercase Latin letters as indices to avoid confusion). We also introduce the dual frame $\{ e^A \}_{A=0}^3$ given by $(e^A)_\alpha : = m^{AB}(e_B)_\alpha$, where $m$ is $g$ expressed in this frame, so that $m$ takes the form of the Minkowski metric, thus $e^A(e_B) = \delta^A_B$. Decomposing $\xi$ with respect to the dual frame, $\xi_A = e_A^\mu \xi_\mu$, we have
$\xi_{A=0} = -\xi^{A=0} = u^\mu \xi_\mu$ and $\xi_{A=i} = v_{A=i}$, where $v^\mu = \proj^{\mu \alpha} \xi_\alpha$ and $v_A = e_A^\mu \xi_\mu$. With these definitions, the remaining characteristics are determined by
\begin{align}
\label{E:Characteristic_equation_sound_cones_in_frames}
\xi_{A=0}^2 - \frac{\partial p}{\partial \varrho} \sum_{i=1}^3 \xi_{A=i}^2 = 0.
\end{align}

If $\frac{\partial p}{\partial \varrho} < 0$, then there are no real solutions to \eqref{E:Characteristic_equation_sound_cones_in_frames} so that the equations will not be hyperbolic. If 
$\frac{\partial p}{\partial \varrho} > 1$, then $\xi$ must be timelike, so that the corresponding characteristic speeds (in physical space) will be greater than the speed of light (see also Remark \ref{R:Sound_speed_grater_than_one}). The first case leads to an evolution incompatible with relativity. The second case is more subtle. While it is common to restrict attention to theories with sub-luminal propagation speeds, it has been argued that finite super-luminal speeds can in principle be accommodated in a relativistic framework if one restricts attention to field configurations that arise as solutions to the Cauchy problem \citep{Geroch:2010da}. Moreover, one can in principle be interested in studying super-luminal fluids from a purely mathematical point of view (although $\frac{\partial p}{\partial \varrho} > 1$ will then impose restrictions on the initial data, see  Remark \ref{R:Sound_speed_grater_than_one}). 
In any case, most physical studies focus on theories where super-luminal speeds are not allowed, so we restrict out attention to the case
\begin{align}
0 \leq  \frac{\partial p}{\partial \varrho} \leq 1.
\nonumber
\end{align}

The case $\frac{\partial p}{\partial \varrho}=0$ is allowed but has to be treated with additional care as it corresponds to some sort of degeneracy (which will in fact be present in the case of a free-boundary fluid studied in Sect.~\ref{S:Vacuum_bry}), thus for the time being we consider only the case
\begin{align}
\label{E:Sound_speed_squared_greater_zero_less_one}
0 < \frac{\partial p}{\partial \varrho} \leq 1.
\end{align}
When \eqref{E:Sound_speed_squared_greater_zero_less_one} holds, we see from \eqref{E:Characteristic_equation_sound_cones_in_frames} that the corresponding characteristics have the structure of two opposite cones with opening given by $\sqrt{\frac{\partial p}{\partial \varrho}}$. This cone structure is interpreted as corresponding to the propagation of sound waves (see below). It makes sense
to call these cones \textdef{sound cones} or \textdef{acoustic cones} and to define the \textdef{fluid's sound speed} $c_s$ by\footnote{For physically relevant equations of state the pressure is a non-decreasing function of the density. One can check that $c_s$ has units of speed.}
\begin{align}
\label{E:Sound_speed_density_entropy}
c_s^2 := \left. \frac{\partial p}{\partial \varrho}\right|_s,
\end{align}
where we write $\left. \right|_s$ to emphasize that the partial derivative in \eqref{E:Sound_speed_density_entropy} is taken at constant $s$ (thus, taking the partial derivative when writing
$p$ as a function of $\varrho$ and $s$). Under assumption \eqref{E:Sound_speed_squared_greater_zero_less_one} the corresponding picture of characteristics in tangent space is as illustrated in Fig.~\ref{F:Light_cone_and_sound_cones_tangent}.

\begin{figure}[ht]
\centering
  \includegraphics[scale=0.3]{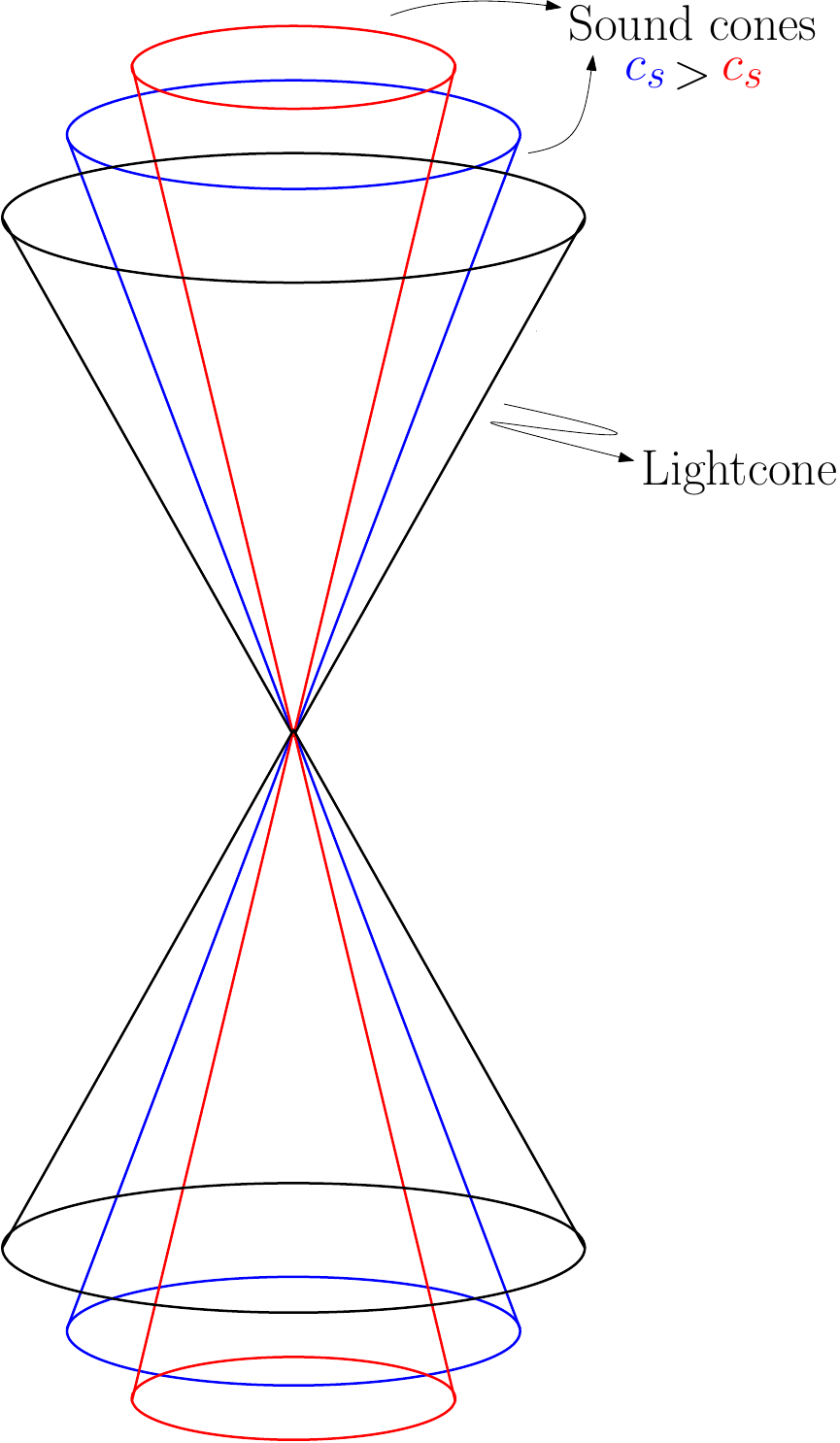}
  \caption{Illustration of sound cones (in tangent space) of different speeds in comparison to the lightcone.}
  \label{F:Light_cone_and_sound_cones_tangent}
\end{figure}

In order to better justify that  the sound cones indeed should be interpreted as corresponding to the propagation of sound waves, take a $u$-derivative of \eqref{E:Projected_relativistic_Euler_eq_full_system_energy} and 
use \eqref{E:Projected_relativistic_Euler_eq_full_system_momentum} 
and \eqref{E:Sound_speed_density_entropy}
to get
\begin{align}
\label{E:Wave_equation_density}
\begin{split}
0 
&= 
u^\mu \nabla_\mu (u^\alpha \nabla_\alpha \varrho + (p+\varrho)
\nabla_\alpha u^\alpha )
\\
&=
u^\mu u^\alpha \nabla_\mu \nabla_\alpha \varrho
+ (p+\varrho) \nabla_\alpha
	(
	\underbrace{u^\mu \nabla_\mu u^\alpha}_{\mathclap{=-\frac{c_s^2}{p+\varrho} 
	\proj^{\alpha \mu} \nabla_\mu \varrho}}
	) 
+ 
\overbrace{\text{L.O.T.}}^{\mathclap{\text{Includes curvature terms}}}
\\
&=
u^\mu u^\alpha \nabla_\mu \nabla_\alpha \varrho - c_s^2 \proj^{\alpha \mu} \nabla_\mu \nabla_\alpha \varrho + \text{L.O.T.}
\end{split}
\end{align}
which is a wave equation for $\varrho$. Compared to the flat wave equation, derivatives on the direction of $u$ play the role of $\partial_t$ and $\proj^{\alpha \mu} \nabla_\mu \nabla_\alpha$ plays the role of the flat Laplacian $\Delta$ since the space orthogonal to $u$ is spacelike. This wave evolution captures the basic idea that sound propagates as disturbances, i.e., expansion and contraction, of density, justifying $c_s$ as sound speed and the characteristics \eqref{E:Characteristic_equation_sound_cones_in_frames} as sound cones. More precisely, the propagation of sound is associated not only with the above wave evolution for the density but also with the evolution of the divergence part of the velocity\footnote{We recall that a vectorfield can be decomposed into and divergence and curl parts.}, as it is the part of the velocity tied with expansions and contractions in the fluid. We will come
back to this point in Sects.~\ref{S:Relativistic_vorticity} and \ref{S:Irrotational_flows}.

The above discussion motivates the following definition.

\begin{definition}[Acoustical metric]
\label{D:Acoustical_metric}
For $0 < c_s \leq 1$ , the \textdef{acoustical metric} is the Lorentzian metric $G$ given by
\begin{align}
\label{E:Acoustical_metric}
G_{\alpha\beta} := c_s^{-2} g_{\alpha\beta} + (c_s^{-2} - 1) u_\alpha u_\beta,
\end{align}
whose inverse is
\begin{align}
\label{E:Acoustical_metric_inverse}
\begin{split}
(G^{-1})^{\alpha \beta} & = c_s^2 \proj^{\alpha\beta} - u^\alpha u^\beta
\\
&= c_s^2 g^{\alpha \beta} + (c_s^2 - 1) u^\alpha u^\beta
\end{split}
\end{align}
\end{definition}

One can immediately verify that for $u$ satisfying \eqref{E:Velocity_normalization}, \eqref{E:Acoustical_metric} indeed defines a Lorentzian metric in spacetime whose inverse is given by \eqref{E:Acoustical_metric_inverse}. Note also that $G_{\alpha\beta} u^\alpha u^\beta = -1$. Observe that the $G$-characteristics given by $(G^{-1})^{\alpha\beta} \xi_\alpha \xi_\beta = 0$ are precisely the sound cones.
In particular, the wave equation \eqref{E:Wave_equation_density} for the density can be written as
\begin{align}
\label{E:Wave_equation_density_acoustical}
(G^{-1})^{\alpha\beta} \nabla_\alpha \nabla_\beta \varrho = 
\text{L.O.T.}
\end{align}

\begin{remark}
\label{R:Inverse_sign_explicit}
We explicitly write ${}^{-1}$ in \eqref{E:Acoustical_metric_inverse} because all indices are raised and lowered with the spacetime metric, but
$g^{\alpha \gamma} g^{\beta \delta} G_{\gamma\delta} \neq (G^{-1})^{\alpha\beta}$.
\end{remark}

The existence of the acoustical metric and its relation to the sound cones
is indicative of the following \emph{key idea} to be exploited later.
\emph{The relevant geometry for the study of a perfect fluid is the acoustic geometry, i.e., the characteristic geometry of the acoustical metric,} and \emph{not} the spacetime geometry. \emph{The acoustic geometry will in general not be flat even if the spacetime is Minkowski.}
This is a consequence of the quasilinear nature of the problem, as the characteristics of the Euler system depend on the solution variables\footnote{In particular, we can see how the case $c_s=0$ is special, as we no longer obtain a Lorentzian metric through \eqref{E:Acoustical_metric}.}.
When coupling to Einstein's equations is considered, then the spacetime and acoustic geometry \emph{interact} with each other, giving rise to a very rich dynamics. We stress this point with the following definition:

\begin{definition}[Acoustic geometry]
\label{D:Acoustic_geometry}
The geometry of the acoustical metric is called the \textdef{acoustic geometry.} We will often refer to ideas like ``controlling the acoustic geometry,'' ``estimates for the acoustic geometry,'' and so on\footnote{Especially in Sects.~\ref{S:Ingredient_one} and \ref{S:Control_acoustic_geometry}.}. The precise meaning of what is being estimated will depend on details of the problem being discussed. But, roughly,
such terminology refers to the fact that in order to close the estimates in some of our problems we need to derive bounds for several geometric quantities associated with the the sound cones (for example, in Sect.~\ref{S:Control_acoustic_geometry} we discuss estimates for the null mean curvature of the sound cones).
\end{definition}

In sum, \emph{the characteristics of the Euler system are the sound cones corresponding to propagation of sound and the flow lines (the integral curves of $u$) which
correspond to the transport of entropy (see \eqref{E:Entropy_transported})
and vorticity (see Sect.~\ref{S:Irrotational_flows}) in the fluid.}

\begin{remark}
\label{R:Sound_speed_grater_than_one}
Above, we excluded $c_s^2 = \frac{\partial p}{\partial \varrho} > 1$ based on the physical requirement that no information propagates faster than the speed of light (often called the principle of \textdef{causality}; we will have more to say about causality when we study viscous fluids in Sect.~\ref{S:Relativistic_viscous_fluids}). One can ask, however, if we could study fluids with $c_s^2 > 1$ from a purely mathematical point of view. Consider the matrix $A^0$ in 
\eqref{E:Matrix_Euler_first-order_formualtion_density_entropy} and for simplicity take $g$ to be the Minkowski metric. Then
\begin{align}
\det A^0 = (p+\varrho)^4 (u^0)^4 (1 + (1-c_s^2) u^i u_i).
\nonumber
\end{align}
We see that while $A^0$ is invertible for any $u$ if $c_s^2 \leq 1$, the invertibility of $A^0$ can fail if $c_s^2 > 1$ (so, e.g., $u$ cannot be prescribed arbitrarily). Since invertibility of $A^0$ is needed for use of many basic PDE tools (e.g., the Cauchy--Kovalevskaya theorem in the simplest case of analytic data; alternatively, we can say that if $c_s^2 > 1$ then there are choices of $u$ that make the ``initial surface'' $\{t=0\}$ characteristic), we see that the assumption $c_s^2 \leq 1$ is also justified mathematically.
\end{remark}

\begin{remark}
The first instance of the acoustical metric that we have been able to locate in the literature is in the work by \cite{Synge-1937-1,Synge-1937-2}, see Eq.~(14.15) in the reprinted edition \citep{Synge-2002}. The acoustical metric also features prominently in the works of
\cite{Lichnerowicz-Book-1967} and \cite{Choquet-Bruhat-1958} that we discuss in detail in Sect.~\ref{S:LWP_Einstein_Euler}. None of these works, however, employs the term acoustical metric, which seems to have been employed by the first time in the work of Bili\'c from 1999, see \cite{Bilic-1999}. (For the classical compressible Euler equations, whose acoustical metric is given by \eqref{E:Acoustical_metric_classical}, the first instance of the term we have been able to locate is in the work by \citealt{Visser-1993-arxiv}.)
\end{remark}

\subsection{Relativistic vorticity\label{S:Relativistic_vorticity}}

A very important quantity in fluids is the vorticity. For classical fluids, it is the curl of the velocity (although one often works with the specific vorticity, the curl of the velocity divided by the density). Since the curl in three dimensions can be identified (using Hodge duality) with the exterior derivative of the velocity (thought of as a one-form) or a suitable multiple of it, it seems natural to define the vorticity of a relativistic fluid (thus in four dimensions) as the (spacetime) exterior derivative of $u$. Up to a multiple, that is what we will do.

\begin{definition}[Relativistic vorticity]
\label{D:Relativistic_vorticity}
The \textdef{enthalpy current} is the vectorfield $w: M \rightarrow TM$ defined by
\begin{align}
w^\alpha := h u^\alpha.
\label{E:Enthalpy_current_definition}
\end{align}
The \textdef{vorticity} is the two-form on $M$ defined\footnote{Our definition differs by a minus sign from the one used in \cite{Rezzolla-Zanotti-Book-2013}. We also recall from Sect.~\ref{S:Notation_conventions}, that we use the same letter for a vector and one-form related by the metric. So, in \eqref{E:Vorticity_definition}, $w$ is viewed as a one-form since it is acted by the exterior derivative. More generally, whenever the exterior derivative acts on a quantity previously defined as a vectorfield, that means that we have used the spacetime metric to identify it as a one-form.} by
\begin{align}
\Omega := dw,
\label{E:Vorticity_definition}
\end{align}
where $d$ is the exterior derivative in spacetime.
\end{definition}

In view of \eqref{E:Velocity_normalization}, $w$ satisfies
\begin{align}
w^\alpha w_\alpha = -h^2.
\label{E:Enthalpy_current_normalization}
\end{align}

In components, $\Omega$ is given by the equivalent expressions
\begin{align}
\Omega_{\alpha\beta} & = \partial_\alpha (hu_\beta) - \partial_\beta (h u_\alpha) 
\nonumber
\\
& = \nabla_\alpha (hu_\beta) - \nabla_\beta (h u_\alpha).
\nonumber
\end{align}

One reason to define the vorticity as above (rather than, say, $du$) is to have a relativistic version of Kelvin's circulation theorem. For a classical fluid with velocity\footnote{Note that the classical velocity has only three components, $v= (v^1, v^2, v^3)$.} $v$, we define its circulation along a closed loop $\gamma$ as
\begin{align}
\mathscr{C}_{\text{classical}} := \oint_\gamma v \cdot d\ell.
\nonumber
\end{align}
Kelvin's theorem states that this quantity is conserved along fluid lines, i.e., 
\begin{align}
(\partial_t + v \cdot \vec{\nabla} )\mathscr{C}_{\text{classical}} = 0,
\end{align}
where $\vec{\nabla}$ is the standard Euclidean three-dimensional gradient.
Fig.~\ref{F:Illustration_Kelvin_theorem} illustrates this situation. 

\begin{figure}[ht]
\centering
  \includegraphics[scale=0.4]{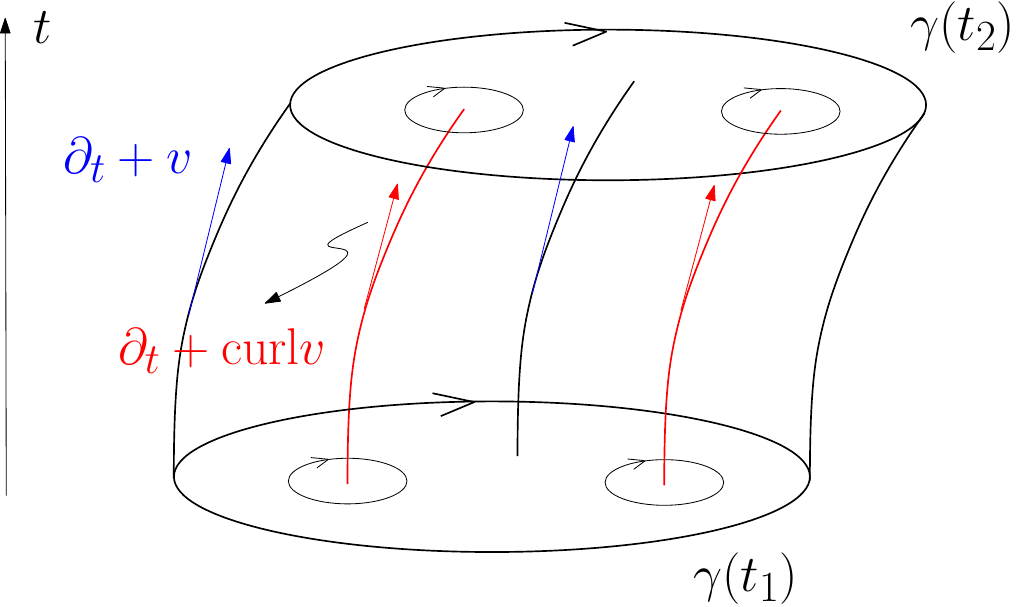}
  \caption{Illustration of Kelvin's circulation theorem in classical physics. The closed curve $\gamma$ at time $t_1$ is advected by the flow of $v$. Thinking of a vortex (illustrated by the small closed loops) as a localized region of nonzero vorticity, the vortex lines are the trajectories of $\curl v$, depicted in red. Kelvin's theorem says that the number of such vortex lines crossing a small element of area moving with the fluid remains constant in time.}
  \label{F:Illustration_Kelvin_theorem}
\end{figure}

Kelvin's circulation theorem has such a clear physical interpretation as ``conservation of vorticies'' that we expect something similar to hold for relativistic fluids (absent dissipation). It indeed holds true but the quantity that is conserved in the relativistic setting is
\begin{align}
\mathscr{C} := \oint_\gamma w_\alpha dx^\alpha =
\oint_\gamma h u_\alpha dx^\alpha.
\nonumber
\end{align}
With this definition one obtains
\begin{align}
u^\mu \nabla_\mu \mathscr{C} = 0.
\nonumber
\end{align}
The same way that the proof of Kelvin's theorem in the classical setting goes through using $dv$, which is the vorticity, the relativistic version involves $d(hu)$, leading to a natural definition of vorticity as in Definition \ref{D:Relativistic_vorticity}. We refer the reader to 
\cite{Rezzolla-Zanotti-Book-2013} for a proof of the relativistic Kelvin theorem and further discussion.

Next, we will derive an important relation between vorticity and entropy. Direct computation and use of \eqref{E:Consequence_velocity_normalization}, 
\eqref{E:Enthalpy_definition},
\eqref{E:First_law_thermodyanmics_dp_version},
\eqref{E:Projected_relativistic_Euler_density_entropy_full_system_momentum},
and
\eqref{E:Projected_relativistic_Euler_density_entropy_full_system_entropy}
gives
\begin{align}
\begin{split}
u^\alpha \Omega_{\alpha\beta} 
&= 
u^\alpha 
	(
	h \nabla_\alpha u_\beta + \nabla_\alpha h u_\beta
	- h \nabla_\beta u_\alpha - \nabla_\beta h u_\alpha
	)
\\
&=
h \underbrace{u^\alpha \nabla_\alpha u_\beta}_{\mathrlap{=-\frac{1}{p+\varrho} \proj^\alpha_\beta \nabla_\alpha p \, = -\frac{1}{nh} \proj^\alpha_\beta \nabla_\alpha p }}
+ u_\beta u^\alpha \nabla_\alpha h 
+\nabla_\beta h
\\
&= -\frac{1}{n} \proj^\alpha_\beta \nabla_\alpha p 
+ u_\beta u^\alpha \nabla_\alpha h 
+\nabla_\beta h
\\
&=
\underbrace{
	-\frac{1}{n} \nabla_\beta p + \nabla_\beta h
}_{\mathclap{= \uptheta \nabla_\beta s}} 
- 
u_\beta
	\underbrace{	
	(
	\frac{1}{n} u^\alpha \nabla_\alpha p - u^\alpha 
	\nabla_\alpha h
	)
	}_{\mathclap{= - u^\alpha \nabla_\alpha s \, = 0}}
\nonumber
\end{split}
\nonumber
\end{align}
Therefore
\begin{align}
\label{E:Lichnerowicz_equation}
u^\alpha \Omega_{\alpha\beta} = \uptheta \nabla_\beta s.
\end{align}

Equation \eqref{E:Lichnerowicz_equation} is known as
\textdef{Lichnerowicz's equation.} It implies that an
\textdef{irrotational fluid,} i.e., a fluid such that 
$\Omega=0$, must be \textdef{isentropic,} i.e., have constant entropy (or have zero temperature), a result with no classical analogue.

We next derive an evolution equation for $\Omega$. Multiplying \eqref{E:Lichnerowicz_equation} by $h$ we have, in compact notation,
\begin{align}
\iota_w \Omega = h \uptheta ds,
\nonumber
\end{align}
where $\iota_X$ is contraction of a form with the vectorfield $X$. Taking the exterior derivative and using that $d^2 s =0$,
\begin{align}
d ( \iota_w \Omega ) = d(h \uptheta ) \wedge ds,
\nonumber
\end{align}
where $\wedge$ is the wedge product of forms.
Recalling Cartan's formula
\begin{align}
\mathcal{L}_X \mu = d(\iota_X \mu) + \iota_X (d \mu),
\label{E:Lie_derivative_iota_d}
\end{align}
for the Lie derivative $\mathcal{L}$ of a form $\mu$ in the direction of a vectorfield $X$, and using that by definition $\Omega$ is an exact form (recall \eqref{E:Vorticity_definition}), we have
\begin{align}
\mathcal{L}_w \Omega = d (h \uptheta) \wedge ds.
\nonumber
\end{align}
Using a well-known formula for the Lie derivative in terms of covariant derivatives, expanding the RHS, and writing the resulting expression in components, we find
\begin{align}
\label{E:Vorticity_evolution}
w^\mu \nabla_\mu \Omega_{\alpha \beta} + \nabla_\alpha 
w^\mu \Omega_{\mu \beta} + \nabla_\beta w^\mu \Omega_{\alpha\mu} 
= \nabla_\alpha (h\uptheta) \nabla_\beta s- \nabla_\beta (h\uptheta) \nabla_\alpha s,
\end{align}
which is the desired evolution equation for the vorticity.

Equation \eqref{E:Vorticity_evolution} is interesting
because of the following. From 
\eqref{E:Projected_relativistic_Euler_density_entropy_full_system_momentum}
we have $u^\alpha \nabla_\alpha u \sim \partial p \sim \partial s$. Commuting with $h$ to get $w$ we have
$u^\alpha \nabla_\alpha w \sim \partial s, \partial h$.
Since $\Omega \sim \partial w$, we would thus naively expect $u^\alpha \nabla_\alpha \Omega \sim \partial^2 s, \partial^2 h$. However, this does not happen; the structure of the Eq.~\eqref{E:Lichnerowicz_equation}, which in particular casts the derivatives of $s$ on the RHS as a perfect derivative, leads to only one derivative on the RHS. This gain of derivative will be important for the local existence and uniqueness theorem we will establish below. In fact, more refined and crucial cancellations that lead to equations with a fewer number of derivatives than expected will play a major role in our discussion of Sect.~\ref{S:New_formulation}. Note also that in deriving \eqref{E:Lichnerowicz_equation} we made use of the first-law of thermodynamics \eqref{E:First_law_thermodyanmics_dp_version}. We did not simply apply $u^\mu \nabla_\mu$ to $\Omega$ and used
$\nabla_\alpha \mathcal{T}^\alpha_\beta = 0$.

We will now provide a simple application of \eqref{E:Vorticity_evolution}, showing that a fluid that is  irrotational and isentropic at a given time (which without loss of generality we can take to be at time zero) remains so through the evolution. 
In other words, \emph{the irrotationality condition, and the corresponding isentropic condition which is a necessary condition for irrotationality, are propagated by the flow.}

\begin{proposition}
\label{P:Irrotational_isentropic_propagated}
Suppose that $s=s_0=\text{constant}$ and $\Omega=0$ on $\{t=0\}$. Then $s=s_0$ and $\Omega=0$ for $t>0$.
\end{proposition}

\begin{proof}
Integrating \eqref{E:Projected_relativistic_Euler_density_entropy_full_system_entropy} along the flow lines gives $s=s_0$ in spacetime. Then, Eq.~\eqref{E:Vorticity_evolution} reduces to
\begin{align}
\mathcal{L}_w \Omega = 0,
\nonumber 
\end{align}
which is a homogeneous transport equation for $\Omega$, implying the result.
\end{proof}

\begin{remark}
From the point of view of the Cauchy problem, in order to know if $\Omega=0$ at $t=0$ we need to know $\partial_t h$ and $\partial_t u$ at $t=0$, but only $u$ and two thermodynamic scalars, say, $s$ and $\varrho$, and \emph{not} their time derivatives, are given as initial data (since the equations are a system of first-order PDEs). The values $\left. \partial_t u \right|_{t=0}$, $\left. \partial_t \varrho \right|_{t=0}$ and $\left. \partial_t s\right|_{t=0}$, from which one readily computes $\left. \partial_t h \right|_{t=0}$, can be written in terms of the initial data by algebraically solving for the time derivatives in terms of spatial derivatives in Eqs.~\eqref{E:Projected_relativistic_Euler_density_entropy_full_system}.
\end{remark}

In Sect.~\ref{S:Irrotational_flows},
we will see that $(G^{-1})^{\alpha\beta} \xi_\alpha \xi_\beta = 0$ are the only characteristics of the Euler system when the fluid is irrotational. In particular, the flow lines $u^\alpha \xi_\alpha = 0$, which are associated with the operator $u^\mu \nabla_\mu$ appearing in \eqref{E:Vorticity_evolution} and
\eqref{E:Projected_relativistic_Euler_enthalpy_entropy_full_system_entropy} below,
 are not characteristics of the irrotational system. \emph{This is why, in Sect.~\ref{S:Characteristics_Euler}, we said that the flow lines are associated with the transport of vorticity (and entropy).}

Equation \eqref{E:Vorticity_evolution} shows that the vorticity is transported by the flow with source terms depending on derivatives of $h$, $\uptheta$, and $s$. 
From a PDE point of view, in order to assert that 
\eqref{E:Vorticity_evolution} is a true evolution equation
for the vorticity rather than an uninteresting identity we need
to verify that the number of derivatives appearing on the RHS is compatible with transport estimates for the vorticity. This will be done in the proof of Theorem \ref{T:LWP_relativistic_Euler}.

\subsection{Local existence and uniqueness via Lichnerowicz's approach\label{S:LWP_relativistic_Euler}}
The proof of local existence and uniqueness for the relativistic Euler equations we present follows that of \cite{Lichnerowicz-Book-1967}, who generalized an earlier proof of \cite{Choquet-Bruhat-1958}.
Lichnerowicz's proof boiled down to recast the equations
as a system of hyperbolic PDEs with diagonal principal 
part for which he could directly apply 
Leray's celebrated theorem
on local existence and uniqueness of solutions for systems of hyperbolic PDEs (see \citealt{Leray-Book-1953}). While we will
use the same system derived by Lichnerowicz, we will do a bit more and show how one can derive the fundamental energy inequality that is the basis of the result. Leray's theorem, naturally, is also based on energy estimates,
and is applicable to very general systems of hyperbolic PDEs. We belive that a direct derivations of energy estimates for the Euler system (recast in diagonal form) is of greater pedagogical (and thus more in line with the goal of these notes) than a direct quotation of Leray's very general but not so easy to prove theorem.

As mentioned earlier, local existence existence and uniqueness for the relativistic Euler equations can be readily obtained by rewriting them as a first-order symmetric hyperbolic system\footnote{We note that in the first-order symmetric hyperbolic systems approach, there is no need to introduce the vorticity in the system as we will do below.} (see \citealt{Anile-Book-1990}). 
In comparison, the approach adopted here seems very cumbersome. We remind the reader that the point of the approach we present is that it makes the role of the characteristics manifest, opening the door for the more refined and powerful approach of Sect.~\ref{S:New_formulation}, which also relies on a careful analysis of the characteristic geometry.

In Lichnerowicz's approach to rewrite the relativistic Euler equations, one works with the enthalpy current 
$w$ as a primary variable instead of the velocity $u$. Hence, we need a good evolution equation for $w$. The derivation of such an equation is the goal of the next Lemma.

It will be convenient to rewrite \eqref{E:Projected_relativistic_Euler_density_entropy_full_system} in terms
of $(s, h, u)$, in which case it reads
\begin{subequations}{\label{E:Projected_relativistic_Euler_enthalpy_entropy_full_system}}
\begin{align}
h u^\alpha \nabla_\alpha u^\beta + 
\proj^{\alpha \beta} \nabla_\alpha h - \uptheta  \nabla^\beta s 
& = 0,
\label{E:Projected_relativistic_Euler_enthalpy_entropy_full_system_momentum}
\\
u^\alpha \nabla_\alpha h + c_s^2  \nabla_\alpha u^\alpha 
& = 0,
\label{E:Projected_relativistic_Euler_enthalpy_entropy_full_system_enthalpy}
\\
u^\alpha \nabla_\alpha s & = 0,
\label{E:Projected_relativistic_Euler_enthalpy_entropy_full_system_entropy}
\end{align}
\end{subequations}
where $c_s = c_s(s,h)$ and $\uptheta = \uptheta(s,h)$.
For the proof of local existence and uniqueness, we will
make one further choice of variables:

\begin{notation}
\label{N:Notation_h_s_w_primary}
For the remaining of this section, we will take $(s,h,w)$ are primary variables for the Euler system, so that all other quantities will be written in terms of them. In particular, the sound speed will be $c_s = c_s(s,h)$
and the velocity $u = h^{-1} w$.
\end{notation}

\begin{notation}
\label{N:Err_terms}
We will write $\Err(\partial^\ell \phi)$ to denote
an expression depending smoothly on $\phi$ and its derivatives of order at most $\ell$, and similarly for expressions involving two or more variables. E.g., 
$\Err(\partial^\ell \phi, \partial^m \psi)$ denotes an
expression depending smoothly on $\phi$ and its derivatives up to order $\ell$ and on $\psi$ and its derivatives up to order $m$. The precise form of $\Err$ can vary from line to line\footnote{As in the case of the notation $\sim$ (see Sect.~\ref{S:Notation_conventions}), $\Err$ will be used mostly in cases where only the derivative count matters, but differently than $\sim$, we use $\Err$ to emphasize the smooth dependence on the variables and their derivatives, so that standard tools,
such such the Sobolev calculus or nonlinear versions of
Gronwall's inequality, can be applied.}. When needed for consistency of free indices in an equation,
we will attached indices to  $\Err$, e.g., $\Err_\alpha$.
When norms are involved, $\Err$ will denote a smooth function of its arguments which, again, can vary from line to line; e.g., $\Err(\norm{\phi}_N, \norm{\psi}_{N^\prime})$ denotes
a smooth function of  $\norm{\phi}_N$ and $\norm{\psi}_{N^\prime}$ (see Sect.~\ref{S:Notation_conventions} for the notation $\norm{\cdot}_N$).
\end{notation}

\begin{lemma}
\label{L:Lemma_enthalpy_current_evolution}
The enthalpy current defined in \eqref{E:Enthalpy_current_definition} satisfies the 
following equation\footnote{Recall Notation \ref{N:Err_terms} for the meaning of $\Err$.}
\begin{align}
(G^{-1}(s,h,w))^{\alpha\beta} w^\mu \nabla_\mu \nabla_\alpha \nabla_\beta w_\gamma = \Err_\gamma(\partial^2 g, \partial^2 w, \partial^2 s, \partial^2 h, \partial \Omega),
\label{E:Enthalpy_current_evolution}
\end{align}
where we write $G^{-1}(s,h,w)$ to emphasize that the acoustical metric is written in terms of $(s,h,w)$
(see Notation \ref{N:Notation_h_s_w_primary}).
\end{lemma}
\begin{proof}
The proof is a long computation that we leave for Appendix \ref{S:Lemma_enthalpy_current_evolution}.
\end{proof}

We are now ready to establish local existence and uniqueness
of solutions to the Cauchy problem for the relativistic Euler equations.

\begin{theorem}[\citealt{Lichnerowicz-Book-1967}, \citealt{Choquet-Bruhat-1958}]
\label{T:LWP_relativistic_Euler}
Let $(\mathbb{R}\times \Sigma, g)$ be a globally hyperbolic smooth Lorentzian manifold, where $\Sigma$ is 
a Cauchy surface that is a compact smooth three-dimensional manifold without boundary.
Consider initial data
$(\mathring{s}, \mathring{h}, \mathring{u}) \in H^{N}(\Sigma) \times H^N(\Sigma) \times H^{N}(\Sigma)$, $N > \frac{3}{2} + 3$, 
for the relativistic Euler equations
\eqref{E:Projected_relativistic_Euler_enthalpy_entropy_full_system}, where
$\mathring{s}, \mathring{h}\colon \Sigma \rightarrow (0,\infty)$ and $\mathring{u}\colon \Sigma \rightarrow T\Sigma$ are initial data for the entropy, enthalpy, and velocity, respectively. Assume that the equation of 
state 
is an analytic function of its arguments, and that the initial data and equation of state are such that $0 < c_s(\mathring{s},\mathring{h}) < 1$ and $\uptheta(\mathring{s},\mathring{h})$,
$n(\mathring{s},\mathring{h})$, $\mathcal{E}(\mathring{s},\mathring{h})$, $p(\mathring{s},\mathring{h}) > 0$.
Then, there exists a $T > 0$ and a unique classical solution
$(s, h, u)$ to the relativistic Euler equations defined on $(-T,T) \times \Sigma$ and taking the given initial data.
Moreover, this solution satisfies
\begin{align}
(s(t, \cdot), h(t,\cdot), u(t, \cdot) )
\in 
H^N(\Sigma_t), \, -T < t < T,
\nonumber
\end{align}
where\footnote{See Sect.~\ref{S:Notation_conventions}.} $\Sigma_t := \{ (\tau,x) \in \mathbb{R} \times \Sigma \, | \, \tau = t, x \in \Sigma \}$.
\end{theorem}

Before giving the proof, we make several remarks.

\begin{itemize}
\item Global hyperbolicity is the natural setting for the study of the Cauchy problem (see \citealt{Bar-Ginoux-Book-2007,Friedlander-Book-1975}). Globally hyperbolic spacetimes have the topology $\mathbb{R} \times \Sigma$, where $\Sigma$ is a three-dimensional manifold (see
\citealt{Wald:1984rg}), thus we are not making any restriction with such assumption.
\item We took $\Sigma$ compact for simplicity in order to avoid discussing conditions at infinity\footnote{Alternatively, one could work with uniformly local Sobolev spaces instead.}. The non-compact case follows essentially along the same lines of the proof we provide since we can localize the estimates via finite-propagation speed, requiring only small changes to the function spaces to accommodate the asymptotic behavior of the variables (see \citealt[Chapter~IX]{Choquet-Bruhat-Book-2009}). In 
particular, compactness and continuity (guaranteed by Sobolev embedding) gives that $\mathring{s}, \mathring{h} \geq \,\text{constant}\, > 0$ and the remaining $>0$ inequalities are in fact
$\geq \,\text{constant}\, > 0$.
\item Only a three-dimensional vectorfield is given as initial condition for $u$. The full initial data for $u$ is then obtained from the normalization \eqref{E:Velocity_normalization} (recall Remark \ref{R:All_velocity_components}). In coordinates, we can think of as prescribing $u^i$ at $t=0$ and determining $u^0$ at $t=0$ from \eqref{E:Velocity_normalization}. In this situation, saying that $u$ takes the initial data means that its restriction to tangent space of $\Sigma$ equals 
$\mathring{u}$.
We could equivalently imagine prescribing the full $u$ at time zero subject to the condition \eqref{E:Velocity_normalization},
but it seems to us more natural to think of initial data as intrinsic to $\Sigma$.
\item We have assumed $c_s$ to be strictly less than one initially because it is easier to deal with open conditions
when doing an iteration for construction of local solutions.
The case $\leq 1$ can be handled with a bit more care (see
\citealt{Choquet-Bruhat-1958}).
\item The regularity we require on the data, $N > \frac{3}{2}+ 3$, is twice degrees higher than the regularity needed for a proof using first-order symmetric hyperbolic systems. 
This is related to the fact that we took two derivatives of the equations, obtaining a third-order equation for $w$ (see Lemma \ref{L:Lemma_enthalpy_current_evolution}), which will be used in the proof.
We stress that our goal with Theorem \ref{T:LWP_relativistic_Euler} is not 
to optimize regularity but rather to present a proof of local existence and uniqueness that highlights the role of the characteristics, opening the door for the discussion of Sect.~\ref{S:New_formulation}. 
Low-regularity solutions below the level 
$N>\frac{3}{2}+1$ given by the theory of first-order symmetric hyperbolic systems will be discussed in Sect.~\ref{S:Rough_solutions}.
In addition, the system we will use in the proof is of independent interest in that it is a system with diagonal principal part, unlike Eqs.~\eqref{E:Relativistic_Euler_eq_full_system} or their formulation as a first-order symmetric hyperbolic system. In Sect.~\ref{S:New_formulation}, we will present another formulation of the relativistic Euler equations which is also diagonal, but at the expense of introducing elliptic equations, so that the resulting system is a hyperbolic-elliptic system. Lichnerowicz's formulation, while lacking
some of the refined structures that the formulation
of Sect.~\ref{S:New_formulation} has, is on the other hand a purely hyperbolic system. We remark that \cite{Lichnerowicz-Book-1967} derived yet another formulation of the relativistic Euler equations with principal diagonal part which 
was useful for his study of the equations of magneto-hydrodynamics.
\end{itemize}
\medskip
\noindent \emph{Proof of Theorem \ref{T:LWP_relativistic_Euler}.} For simplicity, we take $\Sigma$ to be $\mathbb{T}^3$ with coordinates $\{ x^i \}_{i=1}^3$ and let $t = x^0$ be a coordinate on the first component $\mathbb{R}$ of spacetime. The general case can easily be reduced to this one in view of the finite-speed-of-propagation property that follows from our
computation of the characteristics in Sect.~\ref{S:Characteristics_Euler}. 

We will consider $h$ as a function of $w$ given by $h = \sqrt{-w^\alpha w_\alpha}$. We consider Eqs.~\eqref{E:Projected_relativistic_Euler_enthalpy_entropy_full_system_entropy}, \eqref{E:Vorticity_evolution}, and \eqref{E:Enthalpy_current_evolution} for the unknowns $s$, $\Omega$, and $w$. All other functions appearing in the equations are known functions of $s$, $\Omega$, and $w$
(see  Sect.~\ref{S:Thermodynamic_properties}) once
we also take $h$ as a function of $w$ as above (recall also Notation \ref{N:Notation_h_s_w_primary}). Thus, expanding the covariant derivatives in 
Eqs.~\eqref{E:Projected_relativistic_Euler_enthalpy_entropy_full_system_entropy}, \eqref{E:Vorticity_evolution}, and \eqref{E:Enthalpy_current_evolution} we have the the following system\footnote{Since $g$ is given and smooth, for this proof there is no need to keep track of the derivatives of $g$. We will do so, however, in order to readily adapt this proof to the Einstein--Euler system studied in Sect.~\ref{S:LWP_Einstein_Euler}.}
\begin{subequations}{\label{E:LWP_relativistic_Euler_system}}
\begin{align}
w^\alpha \partial_\alpha s & = 0,
\label{E:LWP_relativistic_Euler_system_s}
\\
w^\mu \partial_\mu \Omega_{\alpha \beta} &=
\Err_{\alpha\beta}(\partial g, \partial s, \Omega, \partial w),
\label{E:LWP_relativistic_Euler_system_Omega}
\\
(G^{-1})^{\alpha\beta} w^\gamma \partial_\alpha \partial_\beta \partial_\gamma w_\delta &= \Err_\delta (\partial^3 g, \partial^2 s,\partial \Omega, \partial^2 w). 
\label{E:LWP_relativistic_Euler_system_w}
\end{align}
\end{subequations}
The order of the derivatives on the the RHS of \eqref{E:LWP_relativistic_Euler_system} is compatible with the order of this mixed-order system, so that its charcteristics
are given simply by the characteristics of the operators
on the LHS\footnote{More precisely, we compute the Leray indices of the system to conclude that Eqs.~\eqref{E:LWP_relativistic_Euler_system} form a Leray system
in diagonal form. At this point, we could simply invoke 
Leray's theorem for such systems, as in fact Lichnerowicz did, but as explained, we believe that it is instructive to work out the energy estimates in more detail. See Appendix \ref{S:Leray_systems}
for a discussion of Leray systems, and
\cite{Leray-Book-1953} for the respective proofs.
}. Hence, the characteristics of \eqref{E:LWP_relativistic_Euler_system} are given by
\begin{align}
w^\alpha \xi_\alpha = 0 \, \text{ and } \, (G^{-1})^{\alpha\beta} w^\gamma \xi_\alpha \xi_\beta \xi_\gamma= 
( (G^{-1})^{\alpha\beta} \xi_\alpha \xi_\beta) (w^\gamma  \xi_\gamma)
 = 0.
 \nonumber
\end{align}
  
Consider a sufficiently regular solution to the relativistic Euler equations and the corresponding system \eqref{E:LWP_relativistic_Euler_system} derived from it. In this case, the characteristic of \eqref{E:LWP_relativistic_Euler_system} are indeed  the flow lines  $w^\alpha \xi_\alpha = 0$ and the sound cones $(G^{-1})^{\alpha\beta} \xi_\alpha \xi_\beta = 0$ (in particular, our procedure of rewriting the relativistic Euler equations as \eqref{E:LWP_relativistic_Euler_system} did not introduce spurious characteristics). It then follows that
the operator $(G^{-1})^{\alpha\beta} w^\gamma \partial_\alpha \partial_\beta \partial_\gamma$ is a third-order strictly 
hyperbolic operator for which energy estimates are readily
available, see \cite{Hormander-Book-2003-3,Leray-Book-1953}.
Combining this with standard estimates for transport operators, we have
\begin{align}
\begin{split}
&\norm{s}_{H^N(\Sigma_t)}  \lesssim
\norm{s}_{H^N(\Sigma_0)} 
+ \int_0^t 
\Err(
\norm{s}_{H^N(\Sigma_\tau)},
\norm{w}_{H^N(\Sigma_\tau)}
) \, d\tau,
\\
&\norm{\Omega}_{H^{N-1}(\Sigma_t)}  
\lesssim
\norm{\Omega}_{H^{N-1}(\Sigma_0)}
\\
& \ \ \
+ \int_0^t 
\Err(
\norm{\partial g}_{H^{N-1}(\Sigma_\tau)},
\norm{\partial s}_{H^{N-1}(\Sigma_\tau)}, 
\norm{\Omega}_{H^{N-1}(\Sigma_\tau)},
\norm{\partial w}_{H^{N-1}(\Sigma_\tau)}
) \, d\tau,
\\
& 
\norm{\partial_t^2 w}_{H^{N-2}(\Sigma_t)}
+ 
\norm{\partial_t w}_{H^{N-1}(\Sigma_t)}
+
\norm{w}_{H^N(\Sigma_t)}
\lesssim 
\norm{\partial_t^2 w}_{H^{N-2}(\Sigma_0)}
\\
&
\ \ \
+ 
\norm{\partial_t w}_{H^{N-1}(\Sigma_0)}
+
\norm{w}_{H^N(\Sigma_0)}
+
\\
&
\ \ \
\int_0^t 
\Err(
\norm{\partial^3 g}_{H^{N-2}(\Sigma_\tau)},
\norm{\partial^2 s}_{H^{N-2}(\Sigma_\tau)}, 
 \norm{\partial \Omega}_{H^{N-2}(\Sigma_\tau)},
 \norm{\partial^2 w}_{H^{N-2}(\Sigma_\tau)}
) \, d\tau,
\end{split}
\label{E:LWP_relativistic_Euler_energy_estimates}
\end{align}
where $N>\frac{3}{2} + 3$ has been used together with the
Sobolev calculus (including the Sobolev embedding theorem)
to estimate lower order terms, including those appearing when estimating the coefficients of the principal part of the system \eqref{E:LWP_relativistic_Euler_system}. In particular, we used the Sobolev calculus to control the 
terms on the RHS of \eqref{E:LWP_relativistic_Euler_system}, e.g.
\begin{align}
\begin{split}
&\norm{
\Err (\partial^3 g, \partial^2 s, \partial \Omega,
\partial^2 w)
}_{H^{N-2}(\Sigma_\tau)}
\\
& 
\ \ \ 
\lesssim
\Err(
\norm{\partial^3 g}_{H^{N-2}(\Sigma_\tau)},
\norm{\partial^2 s}_{H^{N-2}(\Sigma_\tau)}, 
 \norm{\partial \Omega}_{H^{N-2}(\Sigma_\tau)},
 \norm{\partial^2 w}_{H^{N-2}(\Sigma_\tau)}
 ).
\end{split}
\nonumber
\end{align}
The $\partial$ derivatives on the RHS of \eqref{E:LWP_relativistic_Euler_energy_estimates} involve $\partial_t$ and $\partial_i$. Using \eqref{E:LWP_relativistic_Euler_system_s} and \eqref{E:LWP_relativistic_Euler_system_Omega} we can solve for $\partial_t (s,\Omega)$ in terms of their spatial derivatives and terms already present in the error terms $\Err$ but 
without $\partial_t (s,\Omega)$, so we can replace
terms involving $\partial s$ and $\partial \Omega$ on the RHS of \eqref{E:LWP_relativistic_Euler_energy_estimates} by corresponding norms of $s$ and $\Omega$ with the same number of derivatives but no time derivatives, e.g., we 
can replace $\norm{\partial s}_{H^{N-1}(\Sigma_\tau)}$
by $\norm{s}_{H^N(\Sigma_\tau)}$ and similarly for the other terms in $\partial(s,\Omega)$. With these substitutions,
if we define
\begin{align}
\label{E:LWP_relativistic_Euler_norm}
\begin{split}
\mathcal{N}_N(t) & := 
\norm{s}_{H^N(\Sigma_t)}
+
\norm{\Omega}_{H^{N-1}(\Sigma_t)}
\\
& 
\ \ \ +
\norm{\partial_t^2 w}_{H^{N-2}(\Sigma_t)}
+ 
\norm{\partial_t w}_{H^{N-1}(\Sigma_t)}
+
\norm{w}_{H^N(\Sigma_t)}
\end{split}
\end{align}
a standard continuation argument\footnote{See
\citealt[Sect.~5.2,]{Disconzi-Kukavica-2019-2} for an example of such continuation argument.} or equivalently a nonlinear Gr\"onwall inequality 
implies that if $T$ is sufficiently small, then 
\eqref{E:LWP_relativistic_Euler_energy_estimates} gives
\begin{align}
\mathcal{N}_N(t) \lesssim \mathcal{P}(\mathcal{N}_N(0)), \, 0\leq t \leq T,
\label{E:LWP_relativistic_Euler_energy_estimates_compact}
\end{align}
for some smooth function $\mathcal{P}$.
Observe that the norms of $g$, in particular, have been absorbed in the implicit constant appearing in $\lesssim$.

As usual, the energy estimate \eqref{E:LWP_relativistic_Euler_energy_estimates_compact} 
is the key for establishing local existence and uniqueness. Expert readers should have no difficulties 
using estimate \eqref{E:LWP_relativistic_Euler_energy_estimates_compact} to obtain the Theorem.
Thus, we relegate the remainder of the proof to Appendix \ref{S:Complementary_LWP_relativistic_Euler}. \hfill \qed

\medskip

Once solutions to \eqref{E:Projected_relativistic_Euler_enthalpy_entropy_full_system} have been obtained, it is straightforward to get solutions to the relativistic Euler equations written in the other formats, e.g.,
\eqref{E:Relativistic_Euler_eq_full_system}, \eqref{E:Projected_relativistic_Euler_eq_full_system}, or \eqref{E:Projected_relativistic_Euler_density_entropy_full_system}.

As in other quasilinear hyperbolic problems, Theorem \ref{T:LWP_relativistic_Euler} can be improved to give continuity in time with respect to the top Sobolev norm $H^N$
and also continuous dependence of solutions on the data. The
proof of these statements is more involved and will not be given here. Continuous dependence on the data, in particular, 
typically requires employing Kato's framework, as in \cite{Kato-1975-1}.

Finally, we observe that a stronger form of uniqueness, based
on the domain of dependence property, can also be obtained (and is in fact needed to localize the problem in the case of general manifolds $\Sigma$, as noted). While we will not present it here, it can be obtained with the tools discussed in Appendix \ref{S:Leray_systems} and our computation of the characteristics.

\subsection{Irrotational flows\label{S:Irrotational_flows}}
Recall from Sect.~\ref{S:Thermodynamic_properties} that a fluid is irrotational if it satisfies $\Omega=0$, a condition that in particular requires the entropy to be constant, i.e., the fluid to be isentropic. We also saw in Proposition \ref{P:Irrotational_isentropic_propagated} of Sect.~\ref{S:Relativistic_vorticity} that the irrotationally and isentropic conditions are propagated by the flow. We recall that Proposition \ref{P:Irrotational_isentropic_propagated} requires knowing $\Omega$ at time zero, and thus $\left. \partial_t (h,u)\right|_{t=0}$. These can be readily computed from the initial data by algebraically solving for the time derivatives in the relativistic Euler equations.

Assume that $\Omega=0$. This means that the form $w$ is closed\footnote{Recall that a differential form $\mu$ is closed if $d\mu = 0$ and that $d^2 = 0$.}
(see Definition \ref{D:Relativistic_vorticity}) and thus, at least locally, can be written as
\begin{align}
w = d\phi,
\nonumber
\end{align}
for some scalar function $\phi$ sometimes called the 
\textdef{fluid potential.} The Hodge-Laplacian of $\phi$ gives
\begin{align}
\square_H \phi = (d d^* + d^* d) \phi = d^* d \phi 
= d^* w = \iota_w dF,
\nonumber
\end{align}
where $d^*$ is the formal adjoint of $d$ (with respect to the spacetime metric), $F = \ln \frac{n}{h}$ and we used \eqref{E:Lemma_enthalpy_current_evolution_divergence_w}.
From \eqref{E:Lemma_enthalpy_current_evolution_dF} with $\nabla_\alpha s=0$ and \eqref{E:Lemma_enthalpy_current_evolution_F_computation} we have
\begin{align}
\iota_w dF = -\frac{c_s^2-1}{c_s^2} \frac{w^\alpha w^\beta}{h^2} \nabla_\alpha w_\beta = 
-\frac{c_s^2-1}{c_s^2} \frac{w^\alpha w^\beta}{h^2} \nabla_\alpha \nabla_\beta \phi.
\nonumber
\end{align}
Multiplying $d^* d\phi - \iota_w dF=0$ by $-c_s^2$ and using that $d^* d \phi 
= - \nabla_\alpha \nabla^\alpha \phi = -g^{\alpha\beta} \nabla_\alpha \nabla_\beta \phi$ we find
\begin{align}
\label{E:Wave_equation_fluid_potential}
(c_s^2 g^{\alpha\beta} - (1-c_s^2) \frac{w^\alpha w^\beta}{h^2} ) \nabla_\alpha \nabla_\beta \phi
= (G^{-1})^{\alpha\beta} \nabla_\alpha\nabla_\beta \phi = 0,
\end{align}
which is a quasilinear wave equation for $\phi$.

Observe that $\phi$ completely determines a solution to the relativistic Euler equations, so in particular \eqref{E:Wave_equation_fluid_potential} is a closed equation. 
Indeed, since $w=d\phi$, the velocity is given by
$u = h^{-1} w = h^{-1} d\phi$ and thus, by 
\eqref{E:Enthalpy_current_normalization}, the enthalpy is determined 
by $h^2 = -\nabla_\mu \phi \nabla^\mu \phi$. Since the entropy is constant, we can determine all thermodynamic quantities (including the sound speed) from $s$ and $h$ (and an equation of state) from the thermodynamic relations of Sect.~\ref{S:Thermodynamic_properties}.

The only characteristics for the system in the irrotational case are the sound cones $(G^{-1})^{\alpha\beta} \xi_\alpha \xi_\beta = 0$; in particular, the flow lines are not characteristics of the irrotational system. Thus, \emph{the flow lines are the characteristics associated with transport of vorticity and entropy,} as mentioned in Sect.~\ref{S:Characteristics_Euler}.

\subsection{The Einstein--Euler system}
\label{S:LWP_Einstein_Euler}

We will now investigate the Cauchy problem for the relativistic Euler equations coupled to Einstein's equations, i.e.,
the \textdef{Einstein--Euler system,} 
\begin{align}
\label{E:Einstein_Euler}
\Ric_{\alpha\beta} - \frac{1}{2}R g_{\alpha\beta} + \Lambda g_{\alpha\beta} = \mathcal{T}_{\alpha\beta},
\end{align}
where $\Ric$ is the Ricci curvature, $R$ is the scalar curvature, $\Lambda$ is a constant (the cosmological constant), and $\mathcal{T}$ is the energy-momentum of a perfect fluid given by \eqref{E:Energy_momentum_perfect}. In order to establish local existence of solutions, the main task is to show that we can obtain energy estimates for the coupled system. The rest of the argument will then be standard and will not be given here (see, e.g., \citealt{Wald:1984rg}).

As usual, we consider the equivalent form of Einstein's equations
\begin{align}
\label{E:Einstein_Euler_trace_reversed}
\Ric_{\alpha\beta} = \mathcal{T}_{\alpha\beta}
- \frac{1}{2} g^{\mu\nu} \mathcal{T}_{\mu\nu} g_{\alpha\beta} + \Lambda g_{\alpha\beta}.
\end{align}
As in Sect.~\ref{S:LWP_relativistic_Euler},
we consider the relativistic Euler equations
written in terms of $(s,\Omega,w)$. For Einstein's equations, we work in wave (or harmonic) coordinates,
so that \eqref{E:Einstein_Euler_trace_reversed} reads (see, e.g., \citealt{Wald:1984rg})
\begin{align}
\label{E:Einstein_Euler_wave_gauge_s_Omega_w}
-\frac{1}{2} g^{\mu\nu} \partial_\mu \partial_\nu g_{\alpha\beta} = \Err_{\alpha\beta}(\partial g, s, w),
\end{align}
where we are using Notation \ref{N:Err_terms}.

Consider the system comprised of Eqs.~\eqref{E:LWP_relativistic_Euler_system} and 
\eqref{E:Einstein_Euler_wave_gauge_s_Omega_w}.
From \eqref{E:Einstein_Euler_wave_gauge_s_Omega_w}, we immediately obtain
\begin{align}
\begin{split}
\norm{\partial_t g}_{H^N(\Sigma_t)}
+ \norm{g}_{H^{N+1}(\Sigma_t)}
& \lesssim 
\norm{\partial_t g}_{H^N(\Sigma_0)}
+ \norm{g}_{H^{N+1}(\Sigma_0)}
\\
& \ \ \ 
+ \int_0^t\Err( \norm{\partial g}_{H^N(\Sigma_\tau)},
\norm{s}_{H^N(\Sigma_\tau)}, \norm{w}_{H^N(\Sigma_\tau)} ) \, d\tau.
\end{split}
\nonumber
\end{align}
Combining  this with\footnote{Recall that in the proof
of Theorem \ref{T:LWP_relativistic_Euler} we purposely kept track of the derivatives of $g$ entering in the estimates even though $g$ was given there, precisely because we wanted to use those estimates here.}
 \eqref{E:LWP_relativistic_Euler_energy_estimates},
we immediately see that the estimates close, so that we obtain a bound for 
\begin{align}
\norm{\partial_t g}_{H^N(\Sigma_t)}
+ \norm{g}_{H^{N+1}(\Sigma_t)}
+ \norm{s}_{H^N(\Sigma_t)}
+ \norm{w}_{H^N(\Sigma_t)}
+ \norm{\Omega}_{H^{N-1}(\Sigma_t)}
\nonumber
\end{align}
in terms of the corresponding norms of the initial data, provided $t$ is taken on a sufficiently small interval. We observe that in deriving this control in terms of the initial data we are using
\eqref{E:Einstein_Euler_wave_gauge_s_Omega_w} to solve for the second and higher time derivatives 
of $g$, so that terms in $\partial^\ell g$ in the error terms $\Err$ can be replaced by terms with the same number of derivatives but no time derivatives and further terms already present in 
$\Err$, precisely as we did in the proof of Theorem \ref{T:LWP_relativistic_Euler}. For example,
the term 
$\norm{\partial^3 g}_{H^{N-2}(\Sigma_\tau)}$ on the RHS of \eqref{E:LWP_relativistic_Euler_energy_estimates} can be replaced by 
$\norm{g}_{H^{N+1}(\Sigma_\tau)}$. 

The surface $\{t=0\}$ is non-characteristic for the 
system \eqref{E:LWP_relativistic_Euler_system}-\eqref{E:Einstein_Euler_wave_gauge_s_Omega_w}.
Thus, using an approximation by analytic functions as in the proof of Theorem \ref{T:LWP_relativistic_Euler} and Appendix \ref{S:Complementary_LWP_relativistic_Euler}, we obtain classical and Sobolev regular solutions to the system 
\eqref{E:LWP_relativistic_Euler_system} and \eqref{E:Einstein_Euler_wave_gauge_s_Omega_w} defined for small time and in a neighborhood of the origin (taking the origin as the point $p$ about which we wave coordinates are constructed). A standard argument gives solutions to the full Einstein equations (for data satisfying the constraints). We summarize the result as follows.

\begin{theorem}[\citealt{Lichnerowicz-Book-1967}, \citealt{Choquet-Bruhat-1958}]
\label{T:LWP_Einstein_Euler}
Let $\mathcal{I} = (\Sigma, \mathring{g}, \mathring{\kappa}, \mathring{s}, \mathring{h}, \mathring{u})$ be an initial data set for the Einstein--Euler system \eqref{E:Einstein_Euler}, with $\Sigma$ compact. Suppose that 
$\mathring{s}$, $\mathring{h}$, $\mathring{u}$, and the equation of state satisfy the same assumptions of Theorem \ref{T:LWP_relativistic_Euler} and that
$\mathring{g} \in H^{N+1}(\Sigma)$, 
$\mathring{\kappa} \in H^N(\Sigma)$. Then,
there exists a globally hyperbolic development of $\mathcal{I}$ satisfying the Einstein--Euler system 
\eqref{E:Einstein_Euler}. This development is unique if taken to be the maximal globally hyperbolic development of the data.
\end{theorem}

\begin{remark}
As in Theorem \ref{T:LWP_relativistic_Euler}, we took $\Sigma$ compact for simplicity. See \cite{Choquet-Bruhat-Book-2009} for the non-compact case.
\end{remark}

\begin{remark}
Theorem \ref{T:LWP_Einstein_Euler} assumes the existence of an initial-data set, i.e., a solution to the constraint equations. Solutions of the constraints with perfect fluid matter sources can be found in \cite{Isenberg-Maxwell-2021-arxiv}.
\end{remark}

\section{New formulation of the relativistic Euler equations: null structures and regularity properties}
\label{S:New_formulation}

Equations \eqref{E:LWP_relativistic_Euler_system} that Lichnerowicz derived to prove local existence and uniqueness of solutions to the relativistic Euler equations involve operators that make the role of the characteristics manifest. They are the operators $u^\mu \partial_\mu$ and
$(G^{-1})^{\alpha\beta} u^\mu \partial_\alpha \partial_\beta \partial_\mu$ that form the principal part of \eqref{E:LWP_relativistic_Euler_system}, the latter operator being a combination of the transport operator $u^\mu \partial_\mu$ and the wave operator $(G^{-1})^{\alpha\beta}\partial_\alpha \partial_\beta $. Nevertheless, such equations are not yet good enough for more refined applications, such as the study of shock formation or the study of low-regularity solutions. Standard
first-order formulations of relativistic Euler (i.e., any of the forms \eqref{E:Relativistic_Euler_eq_full_system}, \eqref{E:Projected_relativistic_Euler_eq_full_system}, 
\eqref{E:Projected_relativistic_Euler_density_entropy_full_system},
\eqref{E:Projected_relativistic_Euler_enthalpy_entropy_full_system} or variants thereof, including formulations as first-order symmetric hyperbolic systems) are not appropriate for such applications either and have (from the point of view of these applications) the further disadvantage that the system's characteristics are hidden in the way the equations are presented. Indeed, while the sound cones are a key dynamical object for the evolution (since they are characteristics corresponding to the propagation of sound), they are hidden in first-order formulations of the equations.

In this section, we will present another way of writing the relativistic Euler equations --- what we call a ``new formulation of relativistic Euler'' or simply ``new formulation'' for short --- which is specially suited for investigating delicate mathematical questions for which very precise information about the behavior of solutions is needed. We will discuss how this new formulation can be used to investigate the problem of shock formation for relativistic perfect fluids, improved regularity of solutions, and solutions of low regularity (the latter topic will be dealt with in detail and will thus be discussed separately in Sect.~\ref{S:Rough_solutions}). None of these applications seem attainable using standard formulations of the relativistic Euler equations. The new formulation is specially tailored to the characteristics of the Euler system (the sound cones and the flow lines).

The \emph{key idea} underlying the new formulation is the following. \emph{The new formulation of the relativistic Euler equations allows us to employ geometric-analytic techniques 
from the theory of quasilinear wave equations (of which many were developed in the context of mathematical general relativity, as mentioned in the Introduction) in the study of relativistic perfect fluids.} This is because the new formulation  recasts the relativistic Euler equations as \emph{perturbations} (in an appropriate sense, see below) of quasilinear wave equations of the form
\begin{align}
\label{E:Prototype_wave_equation}
\square_{\mathscr{G}(\Psi)} \Psi = F(\Psi, \partial \Psi),
\end{align}
where $\square_{\mathscr{G}}$ is a wave operator 
relative to a Lorentzian metric $\mathscr{G}$  that depends on the solution variable $\Psi$ and 
$F$ is a nonlinearity depending on $\Psi$ and $\partial \Psi$ with a definite (typically good) structure.
(In the case of systems, it is understood that 
the LHS of \eqref{E:Prototype_wave_equation} is a diagonal operator with 
the wave operator $\square_{\mathscr{G}}$ acting on each diagonal entry.) 

Systems of the form \eqref{E:Prototype_wave_equation} have been intensively investigated in the mathematical community, with significant contributions coming from the study of mathematical general relativity. Several powerful geometric-analytic techniques have been discovered in these studies, leading to a series of 
breakthroughs and sharp results\footnote{A thorough discussion of these results is beyond our scope here; we refer to 
\cite{Speck-Book-2016,Klainerman-Rodnianski-Szeftel-2012-arxiv} and references therein.}. Thus, by recasting the relativistic Euler equations as a perturbation of \eqref{E:Prototype_wave_equation}, we hope to be able to adapt such powerful techniques to the study of relativistic fluids.

There is, however, a \emph{crucial difference} when it comes to the relativistic Euler equations as compared to
\eqref{E:Prototype_wave_equation}. The aforementioned techniques from quasilinear wave equations are adapted to systems with a single characteristic speed, namely, the (solution dependent) characteristics of $\mathscr{G}$. The relativistic Euler equations, on the other hand, form a system with \emph{multiple characteristic speeds,} the sound cones and the flow lines, as seen in Sect.~\ref{S:Characteristics_Euler}. 
We need to account for the highly nontrivial 
\emph{interaction of sound waves and transport of vorticity and entropy.} 
Thus,
while the goal is to treat contributions coming
from the ``transport-part'' of the system as a perturbation of its ``wave-part,'' \emph{the precise nonlinear structure of the perturbations is crucial.}

The best way to illustrate the above point is 
by comparing with the irrotational case. When the fluid is irrotational, the relativistic Euler equations
take the form \eqref{E:Prototype_wave_equation} with $\mathscr{G}$ being the acoustical metric\footnote{\label{FN:Irrotational_differentiated} Equation \eqref{E:Wave_equation_fluid_potential} is not yet in the form \eqref{E:Prototype_wave_equation} because $G$ depends on one derivative of $\phi$. But upon differentiating \eqref{E:Wave_equation_fluid_potential} and considering $\Psi=\partial \phi$, we obtain a system of equations of the form \eqref{E:Prototype_wave_equation}.}. In particular,
the only characteristics present are the sound cones, see Sect.~\ref{S:Irrotational_flows}.
It turns out that the irrotational system possesses
the required structures for the application of the powerful geometric-analytic techniques we have been referring to\footnote{The presence of such structures was crucial for the major work by \cite{Christodoulou-Book-2007} on formation of shocks
in the irrotational relativistic Euler equations. In fact,
\cite{Christodoulou-Book-2007} 
introduced many of the geometric-analytic techniques referred above and was the first systematic application 
of ideas from mathematical relativity to the study of fluids (although some geometric-analytic tools have been used in earlier work by \citealt{Alinhac-1999-1, Alinhac-1999-2, Alinhac-2001-1, Alinhac-2001-2}).} (we will be more specific about what such structures are when we introduce the new formulation in Sect.~\ref{S:New_formulation_statement}). If we blindly apply to the general (non-irrotational) case the procedure that leads to Eqs.~\eqref{E:Prototype_wave_equation} in the irrotational setting, we will naturally obtain additional terms
in \eqref{E:Prototype_wave_equation}\footnote{\label{FN:Similar_derivation_irrotational}More precisely, we cannot apply to the non-irrotational case the procedure leading to \eqref{E:Prototype_wave_equation}, since we do not have $u=d\phi$ in this case. What we mean is that following similar calculations, applied to $u$ directly, will produce additional terms.}. These additional terms will depend on the entropy, the vorticity, and their derivatives, coupling to the evolution of the the wave-part of the system to that of its transport-part, i.e., the evolution of the entropy and vorticity. \emph{The presence of these extra terms will in general preclude application of techniques proven successful in the study of systems of the form \eqref{E:Prototype_wave_equation}.}

The goal of the new formulation of relativistic Euler is to write the equations as a system of the form \eqref{E:Prototype_wave_equation} containing extra terms involving the entropy and vorticity, as mentioned above, but \emph{in such a way that these extra terms
have a good nonlinear structure that allows us to 
adapt geometric-analytic techniques previously used in the study of quasilinear wave equations with a single characteristic speed.} 
These extra terms involving the entropy and vorticity and possessing good nonlinear structure is what we think of as ``perturbations'' of \eqref{E:Prototype_wave_equation}.
In particular, when we say that we will recast the relativistic Euler equations as a perturbation of \eqref{E:Prototype_wave_equation}, this does not mean that such perturbations can be thought of as negligible 
or small in any sense. We use the term perturbation here in the sense of ``perturbing the underlying structures,'' i.e., obtaining terms whose  \emph{structural properties are compatible} with those of the unperturbed system.

In particular, when applying
to Euler
the geometric-analytic techniques developed in the context of systems with a single characteristic speed,
one cannot treat the perturbations coming from the entropy and vorticity as mere ``passive'' variables
with respect to those techniques. This means
that an application of techniques
from quasilinear wave equations to the non-irrotational Euler system will be successful \emph{only 
if complemented with new geometric-analytic techniques tailored to the dynamics of the entropy and vorticity.}

In order to obtain a system of the form \eqref{E:Prototype_wave_equation} with the desired  perturbations (in the sense just explained), following a derivation similar\footnote{See Footnote \ref{FN:Similar_derivation_irrotational}.} to the one giving wave equations for an irrotational system and carrying the terms in vorticity and entropy along will not  do the trick (i.e., the resulting terms in entropy and vorticity will not have good structure).
Rather, we have to differentiate a first-order formulation of the 
relativistic Euler equations with some carefully chosen geometric vectorfields and observe some remarkable cancellations.

\subsection{Auxiliary quantities}
In order to derive the new formulation of relativistic Euler, we need to introduce some auxiliary variables that will be used.

\begin{notation} 
For simplicity, for the derivation of the new formulation of the relativistic Euler equations, we will take the spacetime metric to be the Minkowski metric. We will work with standard rectangular coordinates. Throughout, we will denote by $\upepsilon$ the totally anti-symmetric symbol in $3+1$ dimensions with the normalization $\upepsilon^{0123} = 1$. In the Definition
\ref{D:Auxiliary_variables_new_formulation},
we will take $s$, $\hat{h}$, and $u$ as primary variables, where $\hat{h}$ is defined in 
\eqref{E:Log_enthalpy}. In particular, $h,n,\uptheta,\varrho,\mathcal{E}$, and $p$ will be functions of $(s,\hat{h})$.
\end{notation}

\begin{assumption}
For the rest of this section, we assume 
$h,n,\uptheta,\varrho,\mathcal{E}, p,s > 0$ and $0 < c_s \leq 1$. The normalization \eqref{E:Velocity_normalization} will also be assumed.
\end{assumption}

\begin{definition}[Auxiliary variables]
\label{D:Auxiliary_variables_new_formulation}
We introduce the following quantities:
\begin{itemize}
\item The (dimensionless) \textdef{logarithmic enthalpy,}
\begin{align}
\hat{h} := \log\frac{h}{\bar{h}},
\label{E:Log_enthalpy}
\end{align}
where $\bar{h}>0$ is some fixed constant reference value.
\item The \textdef{$u$-orthogonal vorticity of a one form $V$,}
\begin{align}
\vort^\alpha (V) := -\upepsilon^{\alpha\beta\gamma\delta} u_\beta \partial_\gamma V_\delta.
\label{E:Vort_operator}
\end{align}
Note that indeed $\vort^\alpha (V) u_\alpha =0$.
\item The \textdef{$u$-orthogonal vorticity,}
\begin{align}
\upomega^\alpha := \vort^\alpha(hu).
\label{E:u_orthogonal_vorticity}
\end{align}
\item The \textdef{entropy gradient,}
\begin{align}
S_\alpha := \partial_\alpha s.
\label{E:Entropy_gradient}
\end{align}
\item The \textdef{modified vorticity of the vorticity,}
\begin{align}
\begin{split}
\C^\alpha :&= \vort^\alpha(\upomega) + c_s^{-2} \upepsilon^{\alpha\beta\gamma\delta} u_\beta \partial_\gamma\hat{h} \upomega_\delta
+(\uptheta-\frac{\partial \uptheta}{\partial\hat{h}} )S^\alpha \partial_\lambda u^\lambda 
\\
& \ \ \ 
+(\uptheta-\frac{\partial \uptheta}{\partial\hat{h}} ) u^\alpha S^\lambda\partial_\lambda \hat{h}
+ (\uptheta-\frac{\partial \uptheta}{\partial\hat{h}} ) S^\lambda g^{\alpha\beta}\partial_\lambda u_\beta.
\end{split}
\label{E:Modified_vorticity_of_vorticity}
\end{align}
\item The \textdef{modified divergence of the entropy gradient,}
\begin{align}
\D := \frac{1}{n} \partial_\lambda S^\lambda + \frac{1}{n} S^\lambda \partial_\lambda\hat{h} - \frac{1}{n}c_s^{-2} S^\lambda \partial_\lambda \hat{h}.
\label{E:Modified_divergence_entropy_gradient}
\end{align}
\end{itemize}
\end{definition}

The modified quantities $\C$ and $\D$ come about because of the following. In the applications we will discuss, we need to estimate $\vort(\upomega)$ and $\partial_\lambda S^\lambda$. These quantities do not seem to satisfy good equations which would produce the desired estimates. However, adding the right combination of variables to 
$\vort(\upomega)$ and $\partial_\lambda S^\lambda$ we obtain quantities ($\C$ and $\D$) that satisfy equations with a good structure, for which estimates can be obtained\footnote{This is illustrated in more detail in Sect.~\ref{S:Rough_solutions}; see in particular Remark \ref{R:Trading_C_for_curl_Omega}.}. The desired estimates for $\vort(\upomega)$ and $\partial_\lambda S^\lambda$ then follow because the terms added to them to define the modified quantities can be estimated separately.

\begin{remark}
The quantities  $\C$ and $\D$ play a crucial role in the proofs of the mathematical results we will present, but we are not aware if they have a physical meaning.
\end{remark}

The $u$-orthogonal vorticity $\upomega$ is related to $\Omega$ by duality\footnote{We do not write the term with $\sqrt{-\det g}$ in dualities involving the anti-symmetric symbol because we are considering the Minkowski metric here.},
\begin{align}
\upomega^\alpha = u^\mu 
(\Omega^*)\indices{_\mu^\alpha},
\nonumber
\end{align}
where $\Omega^*$ is the Hodge dual of $\Omega$,
\begin{align}
(\Omega^*)_{\alpha\beta} = \frac{1}{2} \upepsilon_{\alpha\beta\mu\nu} \Omega^{\mu\nu}.
\nonumber
\end{align}
The role of $\upomega$ is to interpret the vorticity as a vector rather than as a two-form. It turns out that even though $\upomega$ contains less information than $\Omega$, using $\upomega$ suffices for our purposes.

\begin{definition}[Null-forms]
The \textdef{null-forms relative\footnote{Recall that $G$ is the acoustical metric introduced in Definition \ref{D:Acoustical_metric}.} to $G$} are the following quadratic forms
\begin{subequations}{\label{E:Null_forms}}
\begin{align}
\mathcal{Q}^{G}(\partial \phi,\partial \psi) &:=
(G^{-1})^{\alpha\beta} \partial_\alpha \varphi \partial_\beta \psi,
\label{E:Null_forms_symmetric}
\\
\mathcal{Q}_{\alpha\beta} (\partial \varphi, \partial \psi) &:=
\partial_\alpha \varphi \partial_\beta \psi -
\partial_\beta \varphi \partial_\alpha \psi,
\label{E:Null_forms_anti_symmetric}
\end{align}
\end{subequations}
where $\varphi$ and $\psi$ are scalar functions.
\end{definition}

The use of null-forms has a long history in hyperbolic PDEs and we will highlight some of its properties below.

\subsection{The new formulation\label{S:New_formulation_statement}}
We can now state the new formulation of the relativistic Euler equations. As the complete statement of the new formulation is quite long, we will give only a schematic statement. For this, we will write $\simeq$ to denote equality up to ``harmless terms,'' where ``harmless'' here means from the point of view of the applications of the new formulation we will discuss.

\begin{theorem}[\citealt{Disconzi-Speck-2019}] 
\label{T:New_formulation}
Assume that $(\hat{h}, s, u)$ is a $C^3$ solution to relativistic Euler equations\footnote{\label{FN:Solution_quantities_new_formulation_well_defined}It is implicitly assumed that $(\hat{h}, s, u)$ is such that all quantities appearing in Eqs.~\eqref{E:New_formulation_wave}, \eqref{E:New_formulation_transport}, and 
\eqref{E:New_formulation_transport_div_curl}, are well-defined, e.g., $0 < c_s \leq 1$, $n> 0$ so that $h$ is well-defined, etc.}. Then, $(\hat{h},s,u)$ also verifies the following system of equations.

\medskip
\noindent \underline{Wave equations:}
\begin{subequations}{\label{E:New_formulation_wave}}
\begin{align}
\square_G \hat{h} 
& \simeq 
\D + Q(\partial\hat{h},\partial u) + L(\partial \hat{h}),
\label{E:New_formulation_wave_h}
\\
\square_G u^\alpha 
& \simeq
\C^\alpha + Q(\partial\hat{h}, \partial u) +
L(\partial\hat{h},\partial u),
\label{E:New_formulation_wave_u}
\\
\square_G s 
& \simeq 
\D + L(\partial\hat{h}),
\label{E:New_formulation_wave_s}
\end{align}
\end{subequations}

\noindent \underline{Transport equations:}
\begin{subequations}{\label{E:New_formulation_transport}}
\begin{align}
u^\lambda \partial_\lambda s 
& = 0,
\label{E:New_formulation_transport_s}
\\
u^\lambda \partial_\lambda S^\alpha 
& \simeq
L(\partial u),
\label{E:New_formulation_transport_S}
\\
u^\lambda \partial_\lambda \upomega^\alpha
& \simeq
L(\partial\hat{h},\partial u),
\label{E:New_formulation_transport_omega}
\end{align}
\end{subequations}

\noindent \underline{Transport-div-curl equations:}
\begin{subequations}{\label{E:New_formulation_transport_div_curl}}
\begin{align}
u^\lambda \partial_\lambda \D
& \simeq
\C + Q(\partial S, \partial\hat{h}, \partial u)
+ L(\partial\hat{h}, \partial u),
\label{E:New_formulation_transport_div_curl_D}
\\
\vort(S) 
&= 0,
\label{E:New_formulation_transport_div_curl_S}
\\
\partial_\lambda \upomega^\lambda 
& \simeq
L(\partial\hat{h}),
\label{E:New_formulation_transport_div_curl_omega}
\\
u^\lambda \partial_\lambda \C^\alpha
& \simeq
\C + \D + Q(\partial S, \partial \upomega,
\partial\hat{h},\partial u) 
+ L(\partial S, \partial \upomega, \partial\hat{h},\partial u).
\label{E:New_formulation_transport_div_curl_C}
\end{align}
\end{subequations}
Above, $L(\partial f_1,\dots, \partial f_m)$ denotes linear combinations of terms that are at most linear in $\partial f_i$, $i=1,\dots,m$, whereas $Q(\partial f_1,\dots, \partial f_n)$
denotes linear combinations of null-form 
$\mathcal{Q}^G(\partial f_i, \partial f_j)$ and
$\mathcal{Q}(\partial f_i, \partial f_j)$, $i,j=1,\dots, n$. It is understood that the coefficients in these linear combinations are smooth function of 
$(\hat{h},s,u)$ without involving their derivatives. $\square_G$ is the wave operator relative to $G$ and in $\square_G u^\alpha$ the wave operator acts on $u^\alpha$ treated as a scalar function\footnote{For a scalar $\mathsf{f}$, we recall that the wave operator is then $\square_G \mathsf{f} = \frac{1}{|\det G|^\frac{1}{2}} \partial_\alpha( |\det G|^\frac{1}{2} (G^{-1})^{\alpha\beta} \partial_\beta \mathsf{f} )$.}.
\end{theorem}
\begin{proof}
The proof is quite long and we refer to \cite{Disconzi-Speck-2019} for details. The core idea is to differentiate a first-order formulation of the equations with several carefully chosen geometric differential operators and observe remarkable cancellations. However, in order for the reader to get an idea of how
these cancellations work, we will provide a proof of \eqref{E:New_formulation_wave_h}.

Direct computation gives
\begin{align}
\det G & = -c_s^{-6},
\nonumber
\\
|\det G|^\frac{1}{2} (G^{-1})^{\alpha\beta} &= c_s^{-1} g^{\alpha\beta} + (c_s^{-1} - c_s^{-3})
u^\alpha u^\beta,
\nonumber
\end{align}
from which we can compute
\begin{align}
\begin{split}
\square_G \hat{h} & = \frac{1}{|\det G|^\frac{1}{2}} \partial_\alpha( |\det G|^\frac{1}{2} (G^{-1})^{\alpha\beta} \partial_\beta \hat{h} )
\\
&=
c_s^3\partial_\alpha
	( 
		-(c_s^{-3} - c_s^{-1} )u^\alpha u^\beta \partial_\beta \hat{h}
		+c_s^{-1} g^{\alpha\beta} \partial_\beta \hat{h}
	)
\\
&=
\textcolor{blue}{-(1-c_s^2) u^\alpha \partial_\alpha (u^\beta \partial_\beta \hat{h} )
-(1-c_s^2) \partial_\alpha u^\alpha u^\beta \partial_\beta \hat{h}}
+ (3c_s^{-1} - c_s) u^\alpha \partial_\alpha c_s u^\beta \partial_\beta \hat{h}
\\
& \ \ \
- c_s g^{\alpha\beta} \partial_\alpha c_s \partial_\beta \hat{h}
\textcolor{blue}{+ c_s^2 g^{\alpha\beta} \partial_\alpha\partial_\beta \hat{h}}
\\
&=
\textcolor{blue}{
(c_s^2-1) u^\alpha \partial_\alpha(u^\beta \partial_\beta\hat{h}) 
+ c_s^2 g^{\alpha\beta}\partial_\alpha\partial_\beta \hat{h}
+(c_s^2-1) \partial_\alpha u^\alpha u^\beta \partial_\beta \hat{h}
}
\nonumber 
\\
& \ \ \ 
\textcolor{green}{+ (3c_s^{-1} -c_s) \frac{\partial c_s}{\partial \hat{h}} u^\alpha \partial_\alpha \hat{h} u^\beta \partial_\beta \hat{h}}
+ (3c_s^{-1} -c_s) \frac{\partial c_s}{\partial s} 
\underbrace{u^\alpha \partial_\alpha s}_{=0} u^\beta \partial_\beta \hat{h}
\\
& \ \ \ 
\textcolor{green}{-c_s g^{\alpha\beta} \frac{\partial c_s}{\partial\hat{h}} \partial_\alpha\hat{h}\partial_\beta \hat{h}}
-c_s g^{\alpha\beta} \frac{\partial c_s}{\partial s} \underbrace{\partial_\alpha s}_{=S_\alpha}\partial_\beta \hat{h}
\\
& =
(c_s^2-1) u^\alpha \partial_\alpha(u^\beta \partial_\beta\hat{h}) 
+ c_s^2 g^{\alpha\beta}\partial_\alpha\partial_\beta \hat{h}
+(c_s^2-1) \partial_\alpha u^\alpha u^\beta \partial_\beta \hat{h}
\nonumber 
\\
& \ \ \ 
\textcolor{green}{
	- \frac{\partial c_s}{\partial \hat{h}} c_s^{-1} 
	\big( 
	\underbrace{c_s^2 g^{\alpha\beta}  \partial_\alpha\hat{h}
	\partial_\beta \hat{h}
	- u^\alpha \partial_\alpha \hat{h} u^\beta \partial_\beta \hat{h}
	+ c_s^2 u^\alpha \partial_\alpha \hat{h} u^\beta \partial_\beta \hat{h}
	}_{=(G^{-1})^{\alpha\beta} \partial_\alpha\hat{h}\partial_\beta \hat{h}}
	- 2 u^\alpha \partial_\alpha\hat{h}u^\beta \partial_\beta\hat{h}
	\big)
	}
\\
& \ \ \ 
-c_s g^{\alpha\beta} \frac{\partial c_s}{\partial s} S_\alpha \partial_\beta \hat{h},
\end{split}
\nonumber
\end{align}
where the terms in \textcolor{blue}{blue} in the third equality have been combined into the terms in 
\textcolor{blue}{blue} in the fourth equality, the terms in \textcolor{green}{green} in the fourth equality
have been combined into the terms in \textcolor{green}{green} in the fifth equality and in doing so we split the term with a coefficient $3$ into a sum of terms with coefficients $2$ and coefficient $1$, and we also used
\eqref{E:Projected_relativistic_Euler_enthalpy_entropy_full_system_entropy}
and \eqref{E:Entropy_gradient}. Thus, so far we have
\begin{align}
\begin{split}
\square_G \hat{h} &=
(c_s^2-1) u^\alpha \partial_\alpha(u^\beta \partial_\beta\hat{h}) 
+ c_s^2 g^{\alpha\beta}\partial_\alpha\partial_\beta \hat{h}
+(c_s^2-1) \partial_\alpha u^\alpha u^\beta \partial_\beta \hat{h}
\\
& \ \ \ 
	+2 \frac{\partial c_s}{\partial \hat{h}} c_s^{-1} u^\alpha \partial_\alpha\hat{h}u^\beta \partial_\beta\hat{h}
		- \frac{\partial c_s}{\partial \hat{h}} c_s^{-1} (G^{-1})^{\alpha\beta} \partial_\alpha\hat{h}\partial_\beta \hat{h}
-c_s g^{\alpha\beta} \frac{\partial c_s}{\partial s} S_\alpha \partial_\beta \hat{h},
\end{split}
\label{E:Proof_wave_equation_enthalpy_1}
\end{align}
In terms of the variables of Definition \ref{D:Auxiliary_variables_new_formulation}, equation
\eqref{E:Projected_relativistic_Euler_enthalpy_entropy_full_system_momentum} reads
\begin{align}
u^\alpha \partial_\alpha u_\beta + \partial_\beta \hat{h} + u_\beta u^\alpha\partial_\alpha \hat{h}
- \mathsf{q} S_\beta = 0,
\label{E:Projected_relativistic_Euler_enthalpy_entropy_full_system_momentum_new_variables}
\end{align}
where\footnote{Here the denominator is $h$ and not $\hat{h}$, but recall that $h$ is viewed as a function of $\hat{h}$.} $\mathsf{q}:= \uptheta/h$, and \eqref{E:Projected_relativistic_Euler_enthalpy_entropy_full_system_enthalpy} reads
\begin{align}
u^\alpha \partial_\alpha \hat{h} + c_s^2 \partial_\alpha u^\alpha = 0.
\label{E:Projected_relativistic_Euler_enthalpy_entropy_full_system_enthalpy_new_variables}
\end{align}
Contracting \eqref{E:Projected_relativistic_Euler_enthalpy_entropy_full_system_momentum_new_variables} with $c_s^2 g^{\alpha\beta} \partial_\alpha$ and using
\eqref{E:Projected_relativistic_Euler_enthalpy_entropy_full_system_enthalpy_new_variables} to replace
for the term in divergence of $u$,
\begin{align}
\begin{split}
c_s^2 g^{\alpha\beta} \partial_\alpha \partial_\beta \hat{h} 
&=
\textcolor{blue}{
-c_s^2 u^\alpha\partial_\alpha \overbrace{\partial_\beta u^\beta}^{\mathclap{= -c_s^{-2} u^\beta \partial_\beta \hat{h}}}
}
- c_s^2 \partial_\beta u^\alpha \partial_\alpha u^\beta
-c_s^2 u^\beta\partial_\beta(u^\alpha\partial_\alpha \hat{h})
\\
& \ \ \ 
-c_s^2 \partial_\beta u^\beta u^\alpha \partial_\alpha \hat{h}
+ c_s^2 \mathsf{q} \partial_\beta S^\beta
+c_s^2 \frac{\partial \mathsf{q}}{\partial\hat{h}} S^\beta \partial_\beta \hat{h}
+c_s^2 \frac{\partial \mathsf{q}}{\partial s} S^\beta S_\beta
\\
&=
\textcolor{blue}{
u^\alpha\partial_\alpha (u^\beta \partial_\beta \hat{h}) 
- 2c_s^{-1} \frac{\partial c_s}{\partial \hat{h}}u^\alpha \partial_\alpha\hat{h}u^\beta \partial_\beta \hat{h}
- 2c_s^{-1} \frac{\partial c_s}{\partial s}\underbrace{u^\alpha \partial_\alpha s}_{=0} u^\beta \partial_\beta \hat{h}
}
\\
& \ \ \
- c_s^2 \partial_\beta u^\alpha \partial_\alpha u^\beta
-c_s^2 u^\beta\partial_\beta(u^\alpha\partial_\alpha \hat{h})
-c_s^2 \partial_\beta u^\beta u^\alpha \partial_\alpha \hat{h}
\\
& \ \ \ 
+ c_s^2 \mathsf{q} \partial_\beta S^\beta
+c_s^2 \frac{\partial \mathsf{q}}{\partial\hat{h}} S^\beta \partial_\beta \hat{h}
+c_s^2 \frac{\partial \mathsf{q}}{\partial s} S^\beta S_\beta,
\end{split}
\label{E:Proof_wave_equation_enthalpy_2}
\end{align}
where the terms in \textcolor{blue}{blue} in the first equality produce the terms in 
\textcolor{blue}{blue} in the second equality and we used \eqref{E:Projected_relativistic_Euler_enthalpy_entropy_full_system_entropy}
and \eqref{E:Entropy_gradient}.

Using \eqref{E:Proof_wave_equation_enthalpy_2} to substitute for the term
$c_s^2 g^{\alpha\beta} \partial_\alpha \partial_\beta \hat{h}$
on the RHS of \eqref{E:Proof_wave_equation_enthalpy_1} and highlighting these terms in \textcolor{blue}{blue}, we find
\begin{align}
\begin{split}
\square_G \hat{h} &=
\textcolor{blue}{
u^\alpha\partial_\alpha (u^\beta \partial_\beta \hat{h}) 
- 2c_s^{-1} \frac{\partial c_s}{\partial \hat{h}}u^\alpha \partial_\alpha\hat{h}u^\beta \partial_\beta \hat{h}
}
\\
& \ \ \
\textcolor{blue}{- c_s^2 \partial_\beta u^\alpha \partial_\alpha u^\beta
-c_s^2 u^\beta\partial_\beta(u^\alpha\partial_\alpha \hat{h})
-c_s^2 \partial_\beta u^\beta u^\alpha \partial_\alpha \hat{h}
}
\\
& \ \ \ 
\textcolor{blue}{
+ c_s^2 \mathsf{q} \partial_\beta S^\beta
+c_s^2 \frac{\partial \mathsf{q}}{\partial\hat{h}} S^\beta \partial_\beta \hat{h}
+c_s^2 \frac{\partial \mathsf{q}}{\partial s} S^\beta S_\beta
}
\\
& \ \ \
+ (c_s^2-1) u^\alpha \partial_\alpha(u^\beta \partial_\beta\hat{h}) 
+(c_s^2-1) \partial_\alpha u^\alpha u^\beta \partial_\beta \hat{h}
\\
& \ \ \ 
	+2 \frac{\partial c_s}{\partial \hat{h}} c_s^{-1} u^\alpha \partial_\alpha\hat{h}u^\beta \partial_\beta\hat{h}
		- \frac{\partial c_s}{\partial \hat{h}} c_s^{-1} (G^{-1})^{\alpha\beta} \partial_\alpha\hat{h}\partial_\beta \hat{h}
-c_s g^{\alpha\beta} \frac{\partial c_s}{\partial s} S_\alpha \partial_\beta \hat{h},
\end{split}
\nonumber
\end{align}
We will observe some cancellations. Thus, we rewrite 
in exactly the same form and color code some terms in a way that is explained below:
\begin{align}
\begin{split}
\square_G \hat{h} &=
\textcolor{red}{u^\alpha\partial_\alpha (u^\beta \partial_\beta \hat{h})} 
\textcolor{blue}{- 2c_s^{-1} \frac{\partial c_s}{\partial \hat{h}}u^\alpha \partial_\alpha\hat{h}u^\beta \partial_\beta \hat{h}}
\\
& \ \ \
- c_s^2 \partial_\beta u^\alpha \partial_\alpha u^\beta
\textcolor{red}{-c_s^2 u^\beta\partial_\beta(u^\alpha\partial_\alpha \hat{h})}
\textcolor{green}{-c_s^2 \partial_\beta u^\beta u^\alpha \partial_\alpha \hat{h}}
\\
& \ \ \ 
+ c_s^2 \mathsf{q} \partial_\beta S^\beta
+c_s^2 \frac{\partial \mathsf{q}}{\partial\hat{h}} S^\beta \partial_\beta \hat{h}
+c_s^2 \frac{\partial \mathsf{q}}{\partial s} S^\beta S_\beta
\\
& \ \ \
\textcolor{red}{+ (c_s^2-1) u^\alpha \partial_\alpha(u^\beta \partial_\beta\hat{h}) }
+(\textcolor{green}{c_s^2}-1) \partial_\alpha u^\alpha u^\beta \partial_\beta \hat{h}
\\
& \ \ \ 
	\textcolor{blue}{+2 \frac{\partial c_s}{\partial \hat{h}} c_s^{-1} u^\alpha \partial_\alpha\hat{h}u^\beta \partial_\beta\hat{h}}
		- \frac{\partial c_s}{\partial \hat{h}} c_s^{-1} (G^{-1})^{\alpha\beta} \partial_\alpha\hat{h}\partial_\beta \hat{h}
-c_s g^{\alpha\beta} \frac{\partial c_s}{\partial s} S_\alpha \partial_\beta \hat{h}.
\end{split}
\label{E:Proof_wave_equaton_enthalpy_3}
\end{align}
The terms in same colors in \eqref{E:Proof_wave_equaton_enthalpy_3} cancel each other. We note that for the second term in \textcolor{green}{green},
it is that term multiplied by what is outside the parenthesis that cancels with the first term in \textcolor{green}{green}, but we did not color the term outside the parenthesis because it also multiplies 
another term that does not cancel out. Erasing the terms that cancel, rearranging the order of the remaining terms, and using 
\eqref{E:Projected_relativistic_Euler_enthalpy_entropy_full_system_enthalpy_new_variables} to substitute
for the term $u^\beta \partial_\beta \hat{h}$ we find
\begin{align}
\begin{split}
\square_G \hat{h} &=
\textcolor{blue}{- c_s^2 \partial_\beta u^\alpha \partial_\alpha u^\beta
- \partial_\alpha u^\alpha \overbrace{u^\beta \partial_\beta \hat{h}}^{\mathclap{-c_s^2 \partial_\beta u^\beta}}}
- \frac{\partial c_s}{\partial \hat{h}} c_s^{-1} (G^{-1})^{\alpha\beta} \partial_\alpha\hat{h}\partial_\beta \hat{h}
\\
& \ \ \ 
-c_s g^{\alpha\beta} \frac{\partial c_s}{\partial s} S_\alpha \partial_\beta \hat{h}
\textcolor{red}{+ c_s^2 \mathsf{q} \partial_\beta S^\beta}
+c_s^2 \frac{\partial \mathsf{q}}{\partial\hat{h}} S^\beta \partial_\beta \hat{h}
\\
& \ \ \ 
+c_s^2 \frac{\partial \mathsf{q}}{\partial s} S^\beta S_\beta
\\
&=
\textcolor{red}{
nc_s^2 \mathsf{q}\D + (1-c_s^2) q S^\beta \partial_\beta \hat{q}
}
 - \frac{\partial c_s}{\partial \hat{h}} c_s^{-1} (G^{-1})^{\alpha\beta} \partial_\alpha\hat{h}\partial_\beta \hat{h}
\\
& \ \ \ 
\textcolor{blue}{+c_s^2 (\partial_\alpha u^\alpha \partial_\beta u^\beta - \partial_\beta u^\alpha \partial_\alpha u^\beta)}
-c_s \frac{\partial c_s}{\partial s} S^\alpha \partial_\alpha \hat{h}
\\
& \ \ \
+c_s^2 \frac{\partial \mathsf{q}}{\partial\hat{h}} S^\beta \partial_\beta \hat{h}
+c_s^2 \frac{\partial \mathsf{q}}{\partial s} S^\beta S_\beta,
\end{split}
\label{E:Proof_wave_equaton_enthalpy_4}
\end{align}
where the terms in \textcolor{blue}{blue} in the first equality combine into the terms in
\textcolor{blue}{blue} in the second equality and the term in \textcolor{red}{red}
in the first equality produces the terms in \textcolor{red}{red} in the second equality in view of
\eqref{E:Modified_divergence_entropy_gradient} (we also raised an index in one of the terms).

We claim that \eqref{E:Proof_wave_equaton_enthalpy_4} is the desired expression \eqref{E:New_formulation_wave_h}. Recall that in \eqref{E:New_formulation_wave_h}, each term is allowed 
to have a coefficient that is a smooth function of $(\hat{h},s,u)$. The first term on the RHS of 
\eqref{E:Proof_wave_equaton_enthalpy_4} is the term linear in $\D$ in \eqref{E:New_formulation_wave_h}.
The term $(G^{-1})^{\alpha\beta} \partial_\alpha\hat{h}\partial_\beta \hat{h}$ is a null-form of the form
\eqref{E:Null_forms_symmetric} with $\varphi = \psi = \hat{h}$. For each fixed indices $\alpha$ and $\beta$,
the term $\partial_\alpha u^\alpha \partial_\beta u^\beta - \partial_\beta u^\alpha \partial_\alpha u^\beta$
is a null-form of the form \eqref{E:Null_forms_anti_symmetric} with $\varphi = u^\alpha$ and $\psi = u^\beta$;
thus, $\partial_\alpha u^\alpha \partial_\beta u^\beta - \partial_\beta u^\alpha \partial_\alpha u^\beta$ in 
\eqref{E:Proof_wave_equaton_enthalpy_4} is a sum of such null-forms.
The term $S^\beta S_\beta$ involves no derivatives (recall that $S$ is treated as a variable on its own footing) and the remaining terms are linear in derivatives of $\hat{h}$.
\end{proof}

The derivation of the remaining equations in Theorem \ref{T:New_formulation} is significantly more involved and we refer the reader to \cite{Disconzi-Speck-2019}. The above proof, however, highlights two important aspects. First, one cannot blindly differentiate a first-order formulation to obtain the second-order formulation of Theorem \ref{T:New_formulation}. In order to obtain the desired cancellations it is important to selectively choose when to substitute, expand, or differentiate an expression. Second, the form of the equations of Theorem \ref{T:New_formulation} is unstable under perturbations. \emph{A minor difference, e.g., in a numerical coefficient, would prevent some of the key cancellations or combinations that are needed.}

\begin{remark}
\label{R:Wave_equations_general_case_reduce_prototype}
Equations \eqref{E:New_formulation_wave} are wave equations sourced by terms depending on vorticity and entropy which reduce to a system of the form \eqref{E:Prototype_wave_equation} in the irrotational case.
\end{remark}

When the fluid is irrotational, this new formulation reduces to the equations found by \cite{Christodoulou-Book-2007} in his work on shock formation for the relativistic Euler equations. In this case, the equations are essentially the equations for the fluid potential\footnote{More precisely, a once-differentiated
form of  \eqref{E:Wave_equation_fluid_potential}, see Footnote \ref{FN:Irrotational_differentiated}.}
derived in Sect.~\ref{S:Irrotational_flows}. 
Theorem \ref{T:New_formulation} is a generalization
of a similar new formulation of the classical compressible Euler equations (with entropy and vorticity) found by 
\cite{Luk-Speck-2020,Luk-Speck-2018,Luk-Speck-2024,Speck-2019}.

It is important to stress that \emph{Theorem \ref{T:New_formulation} should not be taken for granted, i.e., as a simple addition on top of the formulations found in the simpler settings of irrotational or classical fluids.} This is because the structures uncovered by Christodoulou and Luk and Speck in those cases are \emph{unstable under perturbations,} in the sense
that even the smallest change in a numerical factor or coefficient would prevent the exact cancellations needed for the equations
of Theorem \ref{T:New_formulation} to hold.

We will next discuss three applications of the new formulation presented in Theorem \ref{T:New_formulation}: improved regularity for the entropy and vorticity, the study of shock formation for the relativistic Euler equations, and low-regularity solutions.
The latter will be carried out in more detail so we leave it for Sect.~\ref{S:Rough_solutions}.

\emph{None of these applications seems attainable using standard formulations of the equations.} The latter observation, in particular, highlights the following important point. Despite looking a monstrosity, the equations of Theorem \ref{T:New_formulation} are very nice in that they have very good structures that allow us to prove several results seemingly inaccessible with first-order formulations of the relativistic Euler equations. In other words, what makes a set of equations ``good'' or ``bad'' is not how long or ugly-looking they are, but rather the structures they possess, a point certainly known to researchers but that might escape some students. The relativistic Euler equations, written as in Theorem \ref{T:New_formulation}, are very good equations (as we will show through the applications we will discuss), despite being scary-looking, because they enjoy several good properties. First-order formulations of relativistic Euler, e.g., \eqref{E:Relativistic_Euler_eq_full_system} or \eqref{E:Projected_relativistic_Euler_density_entropy_full_system}, on the other hand, look relatively benign, but are in fact inappropriate for studying many sophisticated mathematical questions in that they do not have many good mathematical properties.

\subsection{Improved regularity}
Here, we will show that the new formulation of the relativistic Euler equations can be used to establish that the entropy and the $u$-orthogonal vorticity can be one degree more regular than what is given by standard theory.

\begin{theorem}[\citealt{Disconzi-Speck-2019}]
\label{T:Improved_regularity}
The relativistic Euler equations are locally well-posed (i.e., 
existence, uniqueness, and continuous dependence on the data) with\footnote{Footnote \ref{FN:Solution_quantities_new_formulation_well_defined} applies here for the initial data with uniform lower bounds, so for simplicity one can assume to be working in $[0,T]\times \mathbb{T}^3$.}
\begin{align}
(h,s,u,\upomega) \in H^N \times H^{N+1} \times H^N \times H^N,
\nonumber
\end{align}
for $N > \frac{3}{2} + 1$.
\end{theorem}

In other words, if 
\begin{align}
\left. (h,s,u,\upomega)\right|_{t=0} \in H^N \times H^{N+1} \times H^N \times H^N,
\nonumber
\end{align}
then this regularity is propagated by the flow. To understand the point of Theorem \ref{T:Improved_regularity}, 
the important observation is that standard theory (e.g., symmetric 
hyperbolic systems or the mixed order formulation we derived in Sect.~\ref{S:LWP_relativistic_Euler}) gives only 
$(h,s,u,\upomega) \in H^N \times H^N \times H^N \times H^{N-1}$,
even if $\left. (h,s,u,\upomega)\right|_{t=0} \in H^N \times H^{N+1} \times H^N \times H^N$.

\begin{proof}
We refer to \cite{Disconzi-Speck-2019} for details. Here, we only 
highlight the main ingredients.

A smooth solution with quantitative control in the space 
$(h,s,u,\upomega) \in H^N \times H^N \times H^N \times H^{N-1}$ can be obtained by standard methods (e.g., writing the equations as a first-order symmetric hyperbolic system). Thus, using a standard approximation of the data by smooth functions,
the problem can be reduced to deriving a priori estimates
for solutions that give control in the space
$(h,s,u,\upomega) \in H^N \times H^{N+1} \times H^N \times H^N$.

First, it is not difficult to see that direct energy estimates on the evolution equations of the new formulation lose derivatives. For example, we want to control $u$ in $H^N$ and $\upomega$ also in $H^N$ (recall that $\upomega \sim \partial u$). The evolution for $\upomega$, Eq.~\eqref{E:New_formulation_transport_omega}, gives (writing schematically and omitting the initial data and other terms that will not be important for the discussion) $u^\alpha \partial_\alpha \upomega \sim \partial u$, so 
standard transport estimates give
\begin{align}
\norm{\upomega}_{H^N(\Sigma_t)} \lesssim \int_0^t \norm{\partial u}_{H^N(\Sigma_\tau)} \, d\tau.
\nonumber
\end{align}
Equation \eqref{E:New_formulation_wave_u} gives
$\square_G u \sim \C \sim \partial \upomega$ and then, using estimates for wave equations,
\begin{align}
\norm{\partial_t u}_{H^{N-1}(\Sigma_t)}
+
\norm{u}_{H^N(\Sigma_t)} 
& \lesssim
\int_0^t \norm{\partial \upomega}_{H^{N-1}(\Sigma_\tau)} \, d\tau
\sim 
\int_0^t \norm{ \upomega}_{H^{N}(\Sigma_\tau)} \, d\tau
\nonumber
\\
& \lesssim
\int_0^t \norm{\partial u}_{H^{N}(\Sigma_\tau)} \, d\tau
\sim
\int_0^t \norm{ u}_{H^{N+1}(\Sigma_\tau)} \, d\tau,
\nonumber
\end{align}
where on the second line we used the above estimate\footnote{We observe that $\partial$ can involve time derivative, but they play the same role as $\bar{\partial}$ for the derivative counting since we can algebraically solve for time derivatives in terms of spatial derivatives using the equations.} for 
the $H^N$-norm of $\upomega$. We see thus that there is a loss
of derivatives.

The way around this problem is to use the fact that $\upomega$
satisfies not only a transport equation but also (taking into account the evolution for $\C \sim \curl \upomega$,
Eq.~\eqref{E:New_formulation_transport_div_curl_C}),
a div-curl-transport system. Thus, \emph{we can use elliptic
regularity through the div-curl part of the system to gain 
a derivative.}

It is not, however, so simple. The div and curl operators
in the new formulation are \emph{spacetime} div and curl
(recall, in particular, \eqref{E:Vort_operator} and \eqref{E:u_orthogonal_vorticity}). We need to extract regularity
across $\Sigma_t$ and for this we need spatial div-curl operators.
To do so, we use that \eqref{E:Vort_operator} and \eqref{E:u_orthogonal_vorticity} give
\begin{align}
u_\alpha \upomega^\alpha = 0 \Rightarrow u_\alpha \partial \upomega^\alpha = - \partial u_\alpha \upomega^\alpha,
\nonumber
\end{align}
which allows us to independently control the ``timelike part"
of $\partial \upomega$. We can then remove this timelike part
from the div-curl system, treating it as a souce, obtaining in this way a purely spatial div-curl system for which
elliptic estimates on $\Sigma_t$ can be applied. (Similar remarks apply to $S$ and its corresponding div-curl system.)
\end{proof}

\begin{remark}
The above div-curl elliptic estimates work because these are estimates across $\Sigma_t$, without boundary terms. For other applications, including the study of the maximal development of the initial data (see Sect.~\ref{S:Study_of_shocks}), one needs to localize such estimates and elliptic theory demands treatment of boundary terms. This has been done in \cite{Abbrescia-Speck-2022-arxiv} for the classical compressible Euler equations (see also \citealt{Shkoller-Vicol-2024}), but is open in the relativistic case (see \citealt{Abbrescia-Speck-2023}).
\end{remark}

\begin{remark}
The above procedure of excising the timelike part of $\partial \upomega$ can be done while preserving the null structure of the equations. While the null structure is not important for 
Theorem \ref{T:Improved_regularity}, it is important for the study of shock formation discussed in Sect.~\ref{S:Study_of_shocks}. In fact, in the shock problem we do need the extra regularity of $s$ and $\upomega$ given by Theorem \ref{T:Improved_regularity}.
\end{remark}

\begin{remark}
Improved regularity for the vorticity and entropy had been proven in the classical, compressible setting, by Luk and Speck using a similar new formulation of the classical compressible Euler equations (see \citealt{Luk-Speck-2020,Luk-Speck-2018,Luk-Speck-2024,Speck-2019}). One key difference to keep in mind is that in the classical setting the div and curl operators are honest, spatial operators, unlike the relativistic case where we have to deal with spacetime operators as mentioned above. See also the discussion after the statement of Theorem \ref{T:New_formulation}.
\end{remark}

\subsection{The study of shock formation\label{S:Study_of_shocks}}
We will next discuss the problem of singularity formation in solutions to the relativistic Euler equations and how it relates to the study of shock waves. Formation of singularities and shock waves in the Euler flow, and more generally in quasilinear hyperbolic systems, are very broad and active topics of study and it would be unfeasible to present here an account that does justice to topic. Our goal in what follows is very focused, intended primarily to highlight how the new formulation of Theorem \ref{T:New_formulation}, and the geometric-analytic framework that it entails, can be used to study these subjects. Even such a limited scope, however, already reveals a great deal of rich mathematics. \emph{The main takeaway is that, while singularity formation and shock waves are distinct topics that can each be studied on its own right, for the relativistic and the classical compressible Euler equations, there is a deep connection between these topics\footnote{Although there is no general simple picture, see Remark \ref{R:Connections_collapse_etc.}.}.} Such a connection is \emph{visible in great detail through the geometric-analytic framework we have been discussing.}
Before presenting details, we need some preliminary discussion. For the next definition, we remind the reader that we denote by $M$ the spacetime.

\begin{definition}
\label{D:Shock_wave}
A \textbf{shock wave, or simply a shock, for the relativistic Euler equations} \eqref{E:Relativistic_Euler_eq_full_system_energy_momentum}--\eqref{E:Relativistic_Euler_eq_full_system_baryon_charge}
is an oriented regular\footnote{Regular means that $\mathcal{M}$ can be locally written as the level $\{ \varphi = 0 \}$ of a $C^1$ function $\varphi$ satisfying $\left.d\varphi\right|_{\mathcal{M}} \neq 0$.} co-dimension one surface\footnote{Typically, this will be a surface with boundary, see Remark \ref{R:Past_boundary_shock}.} 
$\mathcal{M}$ in spacetime satisfying the following properties:\footnote{Here, we follow closely the definition in Chapter~8 of \cite{Anile-Book-1990}.}
\begin{enumerate}
\item For any $p \in \mathcal{M}$, there exists an open neighborhood $V \subset M$ of $p$ such that $V = V^+ \cup (V\cap \mathcal{M}) \cup V^-$, where $V^\pm$ are open sets satisfying $V^\pm \cap \mathcal{M} = \varnothing$, $V^+ \cap V^- = \varnothing$, and the limits
\begin{align}
\mathcal{T}^\pm(p) := \lim_{\substack{x \rightarrow p \\ x \in V^\pm}} \mathcal{T}(x),
\,\,
\mathcal{J}^\pm(p) := \lim_{\substack{x \rightarrow p \\ x \in V^\pm}} \mathcal{J}(x),
\nonumber
\end{align}
exist and satisfy $\mathcal{T}^+(p) \neq \mathcal{T}^-(p)$ or $\mathcal{J}^+(p) \neq \mathcal{J}^-(p)$.
\item Equations \eqref{E:Relativistic_Euler_eq_full_system_energy_momentum}-\eqref{E:Relativistic_Euler_eq_full_system_baryon_charge} hold in the sense of distributions in a neighborhood of $\mathcal{M}$, and in the classical sense in each of the sets $V^\pm$ introduced above.
\end{enumerate}
\end{definition}

The first condition in Definition \ref{D:Shock_wave} states that the energy momentum tensor $\mathcal{T}$ or the baryon density current $\mathcal{J}$ have a jump discontinuity across $\mathcal{M}$, with their values from each ``side'' of $\mathcal{M}$ being different but remaining finite. Under this situation we cannot talk about their derivatives in the classical sense, so the second condition requires the equations of motion to be satisfied in a distributional sense. The second condition also states that away from $\mathcal{M}$ the relativistic Euler equations are satisfied in the usual, classical sense. Thus, shock waves are essentially discontinuous solutions to the equations of motion. Observe that the above conditions are imposed on $\mathcal{T}$ and $\mathcal{J}$, from which we can then derive conditions on $\varrho$, $p$, $u$, etc. 

Physically, 
shock waves play a prominent role in the study of the relativistic (and classical compressible) Euler equations, with many important physical phenomena adequately descried by shocks. The basic idea is that shocks can be used to model abrupt, nearly instantaneous, variations in the fluid variables, such as those happening in a medium in the neighborhood a region where an object transitions from subsonic to supersonic speeds. 
See Chapter~8 of \cite{Anile-Book-1990} and Chapter~4 of \cite{Rezzolla-Zanotti-Book-2013} for a discussion of applications of shock waves in relativistic fluids. 

Given the importance of shocks in physics, the development of a mathematical theory of shock waves has received a great deal of attention through the years; a precise account would not be appropriate here, see the above monographs and the ones cited further below. As mentioned above, we will focus instead on a particular aspect of the problem related to the formation of singularities. In order to explain this in more detail, it is useful to start with some landmark results in $1+1$ dimensions.

\subsubsection{The 1+1-dimensional case}
\label{S:Shocks_1d}

In $1+1$ dimensions, \cite{John-1974} proved that for strictly hyperbolic systems in one spatial dimension (which include the relativistic and classical compressible Euler equations) satisfying a mild condition called the genuinely nonlinear condition, arbitrarily small perturbations of constant states develop a singularity in finite time, with the singularity given by a gradient blow-up.
John's result generalized earlier findings in this direction by  
\cite{Lax-1964} and \cite{Glimm-Lax-1970} (see also the related subsequent work of \citealt{Liu-1979}). Thus, in particular John's result shows that 
for large classes of smooth initial data, solutions to the relativistic Euler equations break down in finite time due to a gradient blow-up.

Once the gradient of the fluid variables blows up, we can no longer talk about them satisfying the equations of motion in the classical sense. But one can ask whether it is possible to continue solutions past the singularity in a distributional sense; in particular, whether solutions admitting shock waves exist. It turns out that distributional solutions to 
\eqref{E:Relativistic_Euler_eq_full_system_energy_momentum}-\eqref{E:Relativistic_Euler_eq_full_system_baryon_charge} are not unique, thus simply determining their existence is not enough. In $1+1$ dimensions, however, uniqueness can be enforced if one restricts attention the to class of \emph{entropy solutions.} An entropy solution is a distributional solution for which the entropy increases discontinuously, in a controlled fashion, across a hypersurface of discontinuity\footnote{See references below for the precise definition of an entropy solution, which will not be needed here. We warn the reader that the terminology is not uniform in the literature, and what we called distributional vs. entropy solutions is sometimes called weak vs. entropy solutions, weak vs. admissible solutions, etc. Moreover, the expression entropy solution itself is used to mean different things in the literature. Because we are introducing this concept only as a motivation for subsequent problems, it suffices for our discussion to know that such a notion of entropy solution exists, without getting into the details of its definition and differences in the literature.}, i.e., it is a shock wherein the entropy\footnote{The history of how this condition came about is quite interesting, going back to the works or Riemann, Helmholtz, and Clausius, see the introduction of \cite{Christodoulou-2007}.} jumps across $\mathcal{M}$. 

Going back to \cite{Glimm-1965}, through the work of many authors including \cite{Bressan-Book-2000}, \cite{Dafermos-Book-2005}, \cite{LeFloch-Book-2002}, \cite{Liu-Book-2021}, and \cite{Serre-Book-1999-1, Serre-Book-2000-2}, a well-developed theory of entropy solutions for the relativistic and classical compressible Euler equations in $1+1$ dimensions is available. The main result can be summarized as follows. 
Given initial data with sufficiently small total variation, there exists a unique entropy solution which is defined globally in time. Because entropy solutions in particular accommodate shocks, this gives a very satisfactory answer to the question of existence of shock waves for the Euler equations in $1+1$ dimensions (entropy solutions also include other types of singularities, such as contact discontinuities or rarefaction waves, see the previous references). We refer the reader to the above references for the precise statements and proofs. The result holds under certain further natural assumptions on the equation of state and on the initial data, e.g., the assumption that the density is bounded away from zero from below (thus, no vacuum regions are present). See also the related works by \cite{Groah-Smoller-Temple-Book-2007,Smoller-Temple-1993}. 

\begin{remark}
Here, we are focusing on the relativistic Euler equations, but the above references establish the aforementioned result on entropy solutions for a large class of (hyperbolic) systems of conservation laws, of which the relativistic and classical compressible Euler equations are simple particular cases. We remark that for such systems a notion of entropy can be introduced in order to define entropy solutions, even if the equations have no direct physical meaning.
\end{remark}

\subsubsection{The higher-dimensional case}
\label{S:Shocks_higher_d}

Turning now to the $2+1$ and $3+1$ cases, it is natural to ask whether the aforementioned theory of entropy solutions that is so successful in the case of one spatial dimension can be generalized to higher dimensions.
The assumption of data with small variation is key for the existence and uniqueness of entropy solutions in $1+1$ dimensions. The proofs rely crucially on estimates in the space of functions of bounded variation (BV). Unfortunately, it was showed by \cite{Rauch-1986} that BV estimates are not true for most quasilinear systems of interest in more than one spatial dimension. 

Establishing existence and uniqueness of entropy solutions for the relativistic and classical compressible Euler equations in $2+1$ and $3+1$ dimensions is currently a major open problem, whereas developing a general theory of entropy solutions for quasilinear equations in divergence form in $2+1$ or $3+1$ dimensions, in the spirit of the general theory of systems of conservation laws in $1+1$, seems out of reach. See the discussion in \cite{Holzegel-Klainerman-Speck-Wong-2016} and also Sect.~\ref{S:Some_context_shocks}. Nevertheless, there has been important progress within some regimes of interest.

On important case is that of what is sometimes called the \textbf{shock front problem.} In this case, one starts with initial data that on $\Sigma_0$ that suffers a jump discontinuity and asks about the existence and uniqueness of entropy solutions containing shocks that emanate from the surface of discontinuity contained in $\Sigma_0$. See Fig.~\ref{F:Shock_front} for an illustration. In the case of the classical compressible Euler equations, this problem was solved by  \cite{Majda-Book-1984} for initial data where a jump discontinuity happens across a smooth hypersurface, with the data being smooth on both sides of the jump.

\begin{figure}[ht]
\centering
  \includegraphics[scale=0.5]{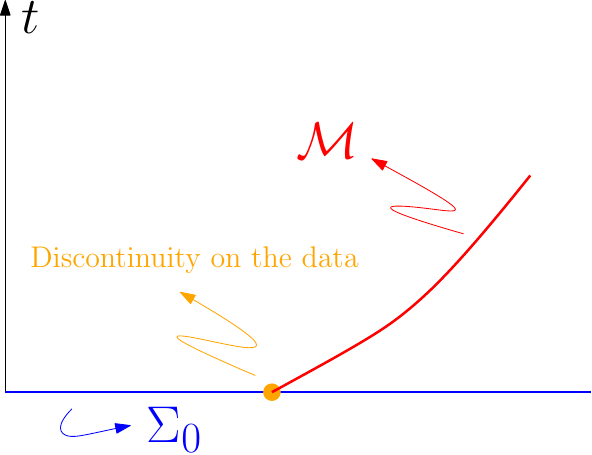}
  \caption{Illustration of the shock front problem.}
  \label{F:Shock_front}
\end{figure}

Important as the shock front problem is, we will focus on a different question, namely, that of the emergence of shocks through a dynamical mechanism, as follows.
Consider sufficiently regular initial data for the relativistic Euler equations prescribed on $\Sigma_0$, where for simplicity we consider the Minkowski spacetime. We would like to determine a time $t_* > 0$, called the \textbf{time of first blow-up,} for which a solution to the Cauchy problem exists and remains regular for $t<t_*$, $\mathcal{T}(t,\cdot)$ and $\mathcal{J}(t,\cdot)$ remain bounded but their gradient blows up as one approach $\Sigma_{t^*}$, with the blow-up occurring along a set $\mathcal{V} \subset \Sigma_{t^*}$, which we call the \textbf{blow-up set\footnote{There is no standard terminology for the set $\mathcal{V}$.}.} Then, we would like to continue the solution past the blow-up in a distributional sense. More precisely, we would like to find an entropy solution that agrees with the classical solution for $t<t_*$ and that contains shock waves emanating from $\mathcal{V}$. This problem is known as the \textbf{shock development problem.} See Fig.~\ref{F:Shock_development} for an illustration.

\begin{figure}[ht]
\centering
  \includegraphics[scale=0.5]{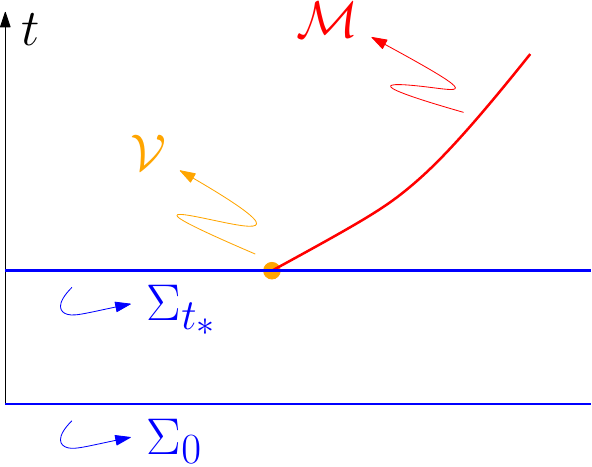}
  \caption{Illustration of the shock development problem.}
  \label{F:Shock_development}
\end{figure}

\begin{remark}
\label{R:Past_boundary_shock}
When we say that a shock wave ``emanates'' from a set $\mathcal{S}$, it roughly means that $\mathcal{M}$ has $\mathcal{S}$ as its past boundary. For our present high-level discussion, a precise definition will not be needed. The interested reader can consult \cite{Abbrescia-Speck-2022-arxiv,Abbrescia-Speck-2023}. But we remark that in general the past boundary of $\mathcal{M}$ will not be contained into a single time slice, thus in the case of the set $\mathcal{V} \subset \Sigma_{t^*}$ above, a more precise description would be to say that $\mathcal{M}$ emanates from a set containing $\mathcal{V}$.
\end{remark}

Observe that the shock development problem contains two different aspects. One is about the formation of singularities and the other is about the continuation of solutions past the singularity\footnote{It is more common in the literature to refer to the shock development problem as only the second aspect, i.e., the continuation past the singularity. Here, however, it is convenient to think of formation of singularities and the subsequent continuation of solutions as two different aspects of a single problem.}. In $1+1$ dimensions, both aspects of the problem, as well as the shock front problem, are all treated together. In higher dimensions, formation of singularities for the relativistic Euler equations was first treated by 
\cite{Guo-Tahvildar-Zadeh-1999} and, subsequently, by \cite{Pan-Smoller-2006}. These results showed that for large classes of smooth initial data which include open sets, classical solutions to the relativistic Euler equations in $3+1$ dimensions break down in finite time\footnote{Naturally, one can conclude breakdown of solutions in $3+1$ dimensions from the $1+1$ case by simply viewing the $1+1$ equations as equations in $3+1$ dimensions under plane symmetry. The corresponding $1+1$ initial data, however, will not form an open set among $3+1$ data under any useful topology.}, generalizing to the relativistic setting a similar earlier result by \cite{Sideris-1985} for the classical compressible Euler equations. The nature of the singularity, however, is not revealed in these works (see more below). While, the full shock development problem remains open in more than one spatial dimension, a very important step toward its resolution was accomplished by \cite{Christodoulou-Book-2019}, who solved a simplified version called the restricted shock development problem. 

Without going into details, the current state-of-affairs in the shock development problem can be summarized as follows. In regimes accessible with current techniques, solving the shock development problem in $2+1$ or $3+1$ dimensions requires understanding the \textbf{maximal globally hyperbolic development of the initial data,} or \textbf{maximal development} for short, which is roughly speaking the largest spacetime domain in which there exists a unique classical solution determined entirely by the initial data\footnote{We say ``the'' maximal development under the implicit assumption that one considers sets of initial data for which there exists a unique maximal development. It is known that, for general quasilinear systems, such a maximal development need not to be unique, see \cite{Eperon-Reall-Sbierski-2019}.}. We refer the reader to \cite{Abbrescia-Speck-2022-arxiv,Shkoller-Vicol-2024,Holzegel-Klainerman-Speck-Wong-2016} for background and a more in-depth discussion of such current state-of-affairs.

For the classical compressible Euler equations in $3+1$ dimensions, \cite{Abbrescia-Speck-2022-arxiv} have recently constructed a portion of the maximal development for specific open sets of data consisting of perturbations of plane-symmetric solutions. We note that the perturbations are not restricted to a symmetry class, thus this is a full $3+1$-dimensional result. They have constructed a singular-boundary portion of the maximal development wherein the solution experiences a gradient blow-up. For a full description of the maximal development, one also has to construct the boundary part that is not singular but determined by domain-of-dependence considerations, i.e., the so-called Cauchy horizon\footnote{See Sect.~8.3 of \cite{Wald:1984rg} for a definition of Cauchy horizons. There, the focus is on the spacetime metric, but the same concept applies to the case of an acoustical metric.}. A portion of the maximal development containing also a Cauchy horizon has been constructed by \cite{Shkoller-Vicol-2024}, albeit in two spatial dimensions and under some convexity assumption on the initial data. 

A key feature in constructing the maximal development is that it requires describing the nature of the singularity, i.e., a refined description of the behavior of the classical solution prior to the breakdown
is needed (recall that the results by \citealt{Guo-Tahvildar-Zadeh-1999,Pan-Smoller-2006,Sideris-1985} do not reveal this information). For the relativistic Euler equations in $3+1$ dimensions, such a precise description was first accomplished in the celebrated work by \cite{Christodoulou-Book-2007} in the case of an irrotational fluid. There, it was shown that there exists an open set of initial data consisting of perturbations of constant states for which the corresponding solutions break down in finite time; the breakdown is caused by a gradient blow-up, whereas the fluid variables themselves remain bounded; and, finally, the breakdown mechanism is the ``collapse of the characteristics,'' where the characteristics in this case are the sound cones, similarly to what happens to the characteristics in the case of the one-dimensional Burgers equation; see below for a definition and more details. We note that prior to \cite{Christodoulou-Book-2007}, \cite{Alinhac-1999-1, Alinhac-1999-2, Alinhac-2001-1, Alinhac-2001-2} obtained results pointing toward the general picture described by Christodoulou. Christodoulou's result was generalized by \cite{Speck-Book-2016}, who considered more general classes of quasilinear wave equations (recall from Sect.~\ref{S:Irrotational_flows} that the irrotational Euler equations can be written as a system of quasilinear wave equations). Finally, a similar result for the classical compressible isentropic irrotational Euler equations was obtained by \cite{Christodoulou-Miao-Book-2014}. 

In the case of the classical compressible Euler equations in $2+1$ and $3+1$ dimensions, 
prior to the works by \cite{Abbrescia-Speck-2022-arxiv,Shkoller-Vicol-2024} on the maximal development, the results mentioned in the previous paragraph had been generalized to include vorticity by 
\cite{Luk-Speck-2020,Luk-Speck-2018,Luk-Speck-2024}. See also the related works by \cite{Buckmaster-Shkoller-Vicol-2022,
Buckmaster-Shkoller-Vicol-2023-1,
Buckmaster-Shkoller-Vicol-2023-2,Buckmaster-Drivas-Shkoller-Vicol-2022}.

The upshot of the previous discussion is the following. \emph{Tackling the shock development problem requires understanding the maximal development of the initial data. The latter, in turn, requires a precise description of the nature of singularities, which entails in particular an understanding of the blow-up set $\mathcal{V}$.} 

With this background in mind, we can now turn attention to the problem we want to discuss in connection with the new formulation of Theorem \ref{T:New_formulation}. Before doing so, however, we need to 
make an important remark on terminology. When one has a gradient blow-up driven by the collapse of the characteristics, with the fluid variables remaining bounded, as described above, it has become 
common practice in the literature (see, e.g., 
\citealt{Abbrescia-Speck-2022-arxiv,Shkoller-Vicol-2024,Abbrescia-Speck-2023,Christodoulou-Book-2007,Luk-Speck-2020,Luk-Speck-2018,Luk-Speck-2024,Speck-2019,Speck-Book-2016}) to refer to the blow-up set $\mathcal{V}$ itself as a shock wave or simply as a shock, and to the process leading to blow-up as a shock formation. This abuse of terminology is understandable if one keeps in mind the connections among gradient blow-up, maximal development, and the shock development problem described above. We will follow this practice, henceforth referring to shocks and the formation of shocks in this sense. 

We can now finally state the problem we want to investigate. It is the problem of \textdef{constructive proofs of stable shock formation without symmetry assumptions in more than one spatial dimension for the relativistic Euler equations,} henceforth referred to simply as the \textdef{problem of shock formation.} By this we mean the following:

\begin{itemize}
\item Shocks form for an open set $\mathcal{B}$ of (small, usually perturbations of constant states) smooth initial data. This means that shocks are stable. More precisely, we want to determine open sets of data for which a time of first blow-up occurs along with a precise description of the nature of the singularity. 
In more detail, we want to show that the gradient of the fluid variables blows up, while the fluid variables themselves remain bounded. Moreover, the goal is to show that blow-up occurs through a specific mechanism of ``collapse of the characteristics,'' where the characteristics in this case are the sound cones. 
The precise definition is given in Sect.~\ref{S:Ingredient_one} (see in particular Eq.~\eqref{E:Inverse_foliation_density} and the surrounding discussion), but the basic idea is that this mechanism is similar to the blow-up mechanism for the Burgers equation in $1+1$ dimensions\footnote{See Footnote \ref{FN:Intersection_characteristics}.}. 
In view of the foregoing discussion, the motivation for considering only singularities where this collapse happens is that these are the types of singularities believed to be connected with the existence of shock waves, as opposed to other types of singularities such as the so-called implosion singularities recently discovered for the classical compressible Euler equations \citep{Merle-Raphael-Rodnianski-Szeftel-2022-1,Merle-Raphael-Rodnianski-Szeftel-2022-2,Merle-Raphael-Rodnianski-Szeftel-2022-3} (see also the related work by \citealt{Chen-Cialdea-Shkoller-Vicol-2024-arxiv}), wherein the fluid variables themselves blow-up, as opposed to only their derivatives.
\item Proofs are constructive, in the sense that they provide information about the nature of the singularity. We do not want, as mentioned above, to describe all sorts of singularities, but only gradient blow-up driven by the collapse of the characteristics. Moreover, as we will explain in Sect.~\ref{S:Ingredient_two}, we want a more refined description intrinsically tied to the collapse of the characteristics. Instead of general gradient blow-up, we want to single out directions transverse to the sound cones as the direction along which derivatives blow up, with derivatives tangent to the sound cones remaining bounded.
\item $\mathcal{B}$ contains ``arbitrary'' initial data, in the sense that it is not restricted to symmetry classes.
\item The problem is considered in $2+1$ or $3+1$ dimensions.
\end{itemize}

In order to carry out a proof with the above elements, there are three key ingredients that are crucial to apply current geometric-analytic techniques. We will discuss each one of them in what follows. Our discussion will elucidate why the formulation
of Theorem \ref{T:New_formulation} is well suited
for the problem of shock formation. After discussing these
three key ingredients, we will provide some further comments on studies of shocks. We stress that all the discussion of these three ingredients focuses on the
case of more than one spatial dimension. We point out that the work by \cite{Abbrescia-Speck-2023} provides a useful complement to the discussion that follows.

\begin{remark}
\label{R:Connections_collapse_etc.}
The connections among collapse of the characteristics, a gradient blow-up, and the shock development problem, are rather intricate and have been understood within specific regimes and for certain sets of initial data. Although these are deep connections that underscore a great deal of beautiful mathematics, we do not intend to suggest that a general theory is available.
\end{remark}

\subsubsection{Ingredient one: nonlinear geometric optics\label{S:Ingredient_one}}
This is done by introducing an \textdef{eikonal function} $\mathcal{U}$, which is a solution to the \textdef{eikonal equation}
\begin{align}
(G^{-1})^{\alpha\beta} \partial_\alpha \mathcal{U} \partial_\beta \mathcal{U} = 0,
\label{E:Eikonal_equation_shocks}
\end{align}
with appropriate initial conditions. The eikonal function plays two crucial roles.

First, the level sets of $\mathcal{U}$ are the characteristics associated with acoustical the metric $G$, which
are the sound cones. In this regard, we note $\mathcal{U}$
is adapted to the wave-part of the system and not the transport-part\footnote{In view of the discussion 
at the beginning of Sect.~\ref{S:New_formulation},
wherein we see transport of entropy and vorticity as a suitable
perturbation of the underlying wave equation \eqref{E:Prototype_wave_equation} (see Remark \ref{R:Wave_equations_general_case_reduce_prototype}), 
we will often refer to the wave-part and the transport-part of the relativistic Euler system. This is intended as a useful conceptual
scheme to help us untangle the behavior of the different characteristics of the system, but should not be thought of as suggesting that we can consider each part separately, as the evolution of the wave-part and the transport-part are coupled.}. This
choice is based on the fact that the transport part corresponds to the evolution of entropy and vorticity, and there are no
known shock results for these quantities. On the other hand, the only known mechanism\footnote{For classical compressible Euler, other types of singularities have been recently constructed in \cite{Merle-Raphael-Rodnianski-Szeftel-2022-1,Merle-Raphael-Rodnianski-Szeftel-2022-2,Merle-Raphael-Rodnianski-Szeftel-2022-3}, but their stability is unknown.} of shock formation for relativistic Euler is the \emph{collapse of the sound cones,} very much like the collapse of the characteristics is the blow-up mechanism for the $1+1$-dimensional Burgers' equation (see Sect.~\ref{S:Some_context_shocks}). Here, ``collapse'' of the sound cones or characteristics refers to the phenomenon of the characteristics accumulating at an ``infinite density."
The collapse of the sound cones is measured by the
\textdef{inverse foliation density} $\upmu$ defined as
\begin{align}
\upmu := -\frac{1}{(G^{-1})^{\alpha\beta} \partial_\alpha t \partial_\beta \mathcal{U} },
\label{E:Inverse_foliation_density}
\end{align}
which has the property that $\upmu \rightarrow 0$ corresponds to the infinite density of the characteristics. In particular, this notion of infinite density\footnote{One would informally talk about the intersection of the characteristics, but this is not very precise because, for generic data, the characteristics do not intersect. Rather, they accumulate in a way that is measured by the inverse foliation density $\upmu$ (see \citealt{Abbrescia-Speck-2022-arxiv} for more discussion).\label{FN:Intersection_characteristics}} is made precise by quantifying it as $\upmu \rightarrow 0$. We note that these remarks highlight the importance, in the context of shock formation, of not treating the transport and wave parts of the system together, as it is done in first-order formulations of the equations.

Second, in order to detect the blow-up, we need to identify precisely in which directions the solution blows up, and in which directions it remains bounded. This is done with the help of a \textdef{null-frame}
\begin{align}
\{ e_1, e_2, \underline{L}, L \},
\nonumber
\end{align}
adapted to the sound cones. Here, $L$ and $\underline{L}$ are null vectors with respect to $G$, satisfying $G(\underline{L},L)=-2$, and $\{e_1,e_2\}$ is an orthonormal (with respect to $G$) frame on the (topological) spheres given by\footnote{We use a harmless abuse of notation to simplify the presentation. Instead of writing $S_{t_0,\mathcal{U}_0} := \{ t = t_0 \} \cap \{ \mathcal{U} = \mathcal{U}_0 \}$
for given $t_0$ and $\mathcal{U}_0$, we simply write $S_{t,\mathcal{U}} :=
\{ t = \, \text{constant} \} \cap \{ \mathcal{U} = \, \text{constant} \}$.}
\begin{align}
S_{t,\mathcal{U}} :=
\{ t = \, \text{constant} \} \cap \{ \mathcal{U} = \, \text{constant} \},
\nonumber
\end{align}
see Fig.~\ref{F:Null_frame}. 
$\underline{L}$ is ``incoming'' and $L$ ``outgoing,'' so that
$\underline{L}$ is transverse to the the characteristics and $L$ tangent to the characteristics.
We also have $G(e_A,L)=0=G(e_A, \underline{L})$, $A=1,2$.

\begin{figure}[ht]
\centering
  \includegraphics[scale=0.4]{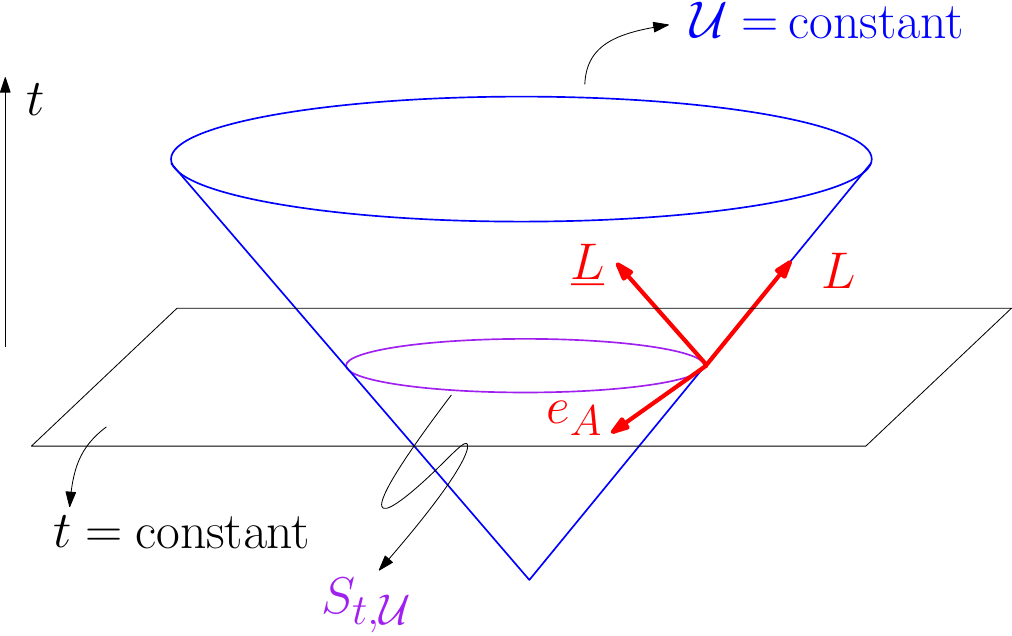}
  \caption{Illustration of a null-frame}
  \label{F:Null_frame}
\end{figure}

We can decompose quantities with respect to this null-frame, and identify that blow-up occurs in the $\underline{L}$ direction, while derivatives of the fluid variables in the other directions remain bounded (see Sects.~\ref{S:Ingredient_two} and \ref{S:Ingredient_three}). To carry out the analysis, we also introduce a geometric system of coordinates adapted to the sound characteristics,
\begin{align}
\{ t, \mathcal{U}, \vartheta^1, \vartheta^2 \},
\nonumber
\end{align}
where $\vartheta^A$, $A=1,2$, are coordinates on the spheres
$S_{t,\mathcal{U}}$, constructed by solving $(G^{-1})^{\alpha\beta}\partial_\alpha \mathcal{U} \partial_\beta \vartheta^A = 0$ with suitable initial conditions.

\subsubsection{Ingredient two: nonlinear null-structure}
\label{S:Ingredient_two}

The basic philosophy for the proof of shock formation is to show that, \emph{relative to the geometric coordinates} $
\{ t, \mathcal{U}, \vartheta^1, \vartheta^2 \}$, \emph{the solution remains bounded all the way to the shock.} In this way we transform the problem of shock formation into a more traditional one, where the goal is to derive long-time estimates for the solution (relative to geometric coordinates). The blow-up of the solution with respect to the original coordinates is recovered by showing that the geometric coordinate system \emph{degenerates} (in a precise fashion, see Remark \ref{R:Genuinely_nonlinear}) relative to the original coordinates\footnote{This is akin to the behavior of Burgers' equation, where in characteristic coordinates solutions are global, although the change of coordinates from standard to characteristic coordinates is no longer valid when a shock forms. An even simpler example is given by the Riccati ODE $\dot{y} = y^2$, which blows up in finite time. If we introduce the 
``new (solution-dependent) coordinate'' $z$ by solving $\dot{z} =- y z$, $z(0)=1$, and define $\tilde{y} :=  z y$, then $\dot{\tilde{y}} = 0$, so $\tilde{y}(t)= \tilde{y}(0)=y(0)$ for all $t$. However, the ``coordinate'' $z$ has to approach zero when $y$ approaches $\pm \infty$.} (since the characteristics are intersecting at the shock, we expect the geometric coordinates to degenerate there).

A crucial aspect of these constructions is that the \emph{null-frames and the geometric coordinates depend on the fluid's solution variables,} since they are constructed out of $\mathcal{U}$, which depends on $G$. (In broad philosophical terms, this resembles the approach to Einstein's equations, where wave coordinates depend on the solution, i.e., on the spacetime metric.) Therefore, in order to implement these ideas, we have to show that the geometric coordinates remain regular all the way up to the shock. And to do so, we need to obtain precise estimates for the fluid's variables, showing, in particular, that the derivatives tangent to the sound cones do not produce singularities, the latter coming from derivatives in the $\underline{L}$ direction.
One important \emph{key idea} here is the following.

One shows that \emph{the evolution can be decomposed into a Riccati-type\footnote{Recall that the Riccati ODE is $\dot{y}=\pm y^2$, which blows up in finite time.} term that drives the blow-up and error terms that do not significantly alter the high-frequency behavior of the Riccati term.} Such terms appear as follows (we will illustrate with $\hat{h}$; similar statements hold for $u$). Expanding the covariant wave operator $\square_G$ relative to the null-frame, Eq.~\eqref{E:New_formulation_wave_h} reads, schematically,
\begin{align}
L (\underline{L} \hat{h} ) \simeq - (\underline{L} \hat{h} )^2 + Q,
\label{E:Decomposition_enthalpy_equation}
\end{align}
where $Q$ denotes linear combinations of null-forms relative to $G$ and we have omitted harmless terms, e.g., terms linear in derivatives. The equation $L (\underline{L} \hat{h} ) \simeq - (\underline{L} \hat{h} )^2$ is the Riccati equation for the variable $\underline{L} \hat{h}$, since $L$ is differentiation along the sound cones so that $L = \frac{d}{d \tau}$ for a suitable parametrization of the flow lines of $L$. Thus, in order to derive blow-up, we need to show that $Q$ is a perturbation that does not significantly alter the Riccati behavior. This is difficult because Riccati terms are generally unstable under perturbations (see Remark \ref{R:ODE_null_forms_analogy}). However, and here is where the role of null-forms becomes important, \emph{Riccati terms are stable under perturbation by null-forms.} Relative to the null-frame, we have,
for a null-form $\mathcal{Q}$ relative to $G$,
\begin{align}
\mathcal{Q}(\partial \phi, \partial \psi) = \mathsf{T}(\phi) \partial \psi + \mathsf{T}(\psi) \partial \phi,
\nonumber
\end{align}
where $\mathsf{T}$ is differentiation tangent to the sound cones. This implies that even though $\mathcal{Q}$ is quadratic, it \emph{never involves terms quadratic in the direction the system ``wants to blow up."} Specifically, in our case then \eqref{E:Decomposition_enthalpy_equation} reads
\begin{align}
\nonumber
L (\underline{L} \hat{h} ) \simeq - (\underline{L} \hat{h} )^2 
+ \mathsf{T}(\hat{h}) \partial \hat{h},
\end{align}
so that the first term on the RHS is the only term quadratic in $\underline{L} \hat{h}$. If instead of $\mathsf{T}(\hat{h})$ we had a generic derivative $\partial \hat{h}$, then we would get a $(\partial\hat{h})^2$ term. After decomposing in a null-frame, this 
$(\partial\hat{h})^2$ term could produce a $(\underline{L}\hat{h})^2$ that cancels or nearly cancels the $-(\underline{L}\hat{h})^2$ from the Riccati part, thus working against blow-up and preventing us from proving that shocks form. The term $\mathsf{T}(\hat{h}) \partial \hat{h}$, however, is at most \emph{linear} in $\underline{L}\hat{h}$, so that in the end we have
\begin{align}
L (\underline{L} \hat{h} ) \simeq - (\underline{L} \hat{h} )^2 
+ \mathsf{T}(\hat{h}) ( \underline{L} \hat{h} ).
\nonumber
\end{align}
Since tangential (to the sound cones) derivatives remain bounded (see Sect.~\ref{S:Ingredient_three} below), the first term on the RHS dominates over the last term, leading to the blow-up of $\underline{L}\hat{h}$, as desired. Note, in particular, that it is precisely a transverse (to the sound cones) derivative that blows up.

\begin{remark}
\label{R:ODE_null_forms_analogy}
A straw man ODE analogy of the above discussion is the following.
Consider the two following perturbations of the Riccati ODE
$\dot{y} = y^2$: $\dot{y} = y^2 + \varepsilon y$ and
$\dot{y} = y^2 \pm \varepsilon y^3$, $y(0) > 0$, $\varepsilon > 0$ small. The first equation still blows up and it does it at the same rate as the original Riccati. For the second perturbation, depending on the sign $\pm$, the solution will either exists for all time or it will blow up at an entirely different rate (thus, effectively, altering the blow-up). The null-forms are the PDE analogue of the $\varepsilon y$ perturbation.
\end{remark}

\begin{remark}
\label{R:Genuinely_nonlinear}
Being a bit less schematic, the evolution for $\underline{L} \hat{h}$ relative the null frame reads
\begin{align}
L (\underline{L} \hat{h} ) \simeq \mathsf{f}(\hat{h}) (\underline{L} \hat{h} )^2 
+ \mathsf{T}(\hat{h}) ( \underline{L} \hat{h} ),
\nonumber
\end{align}
where $\mathsf{f}$ is a smooth function.
John's \emph{genuinely nonlinear condition} (see \citealt{John-1974} and also the overview in \citealt{Speck-Book-2016}), which holds in the present discussion, gives $\mathsf{f}(\hat{h}) < 0$, which gives the negative sign in the above schematic form $L (\underline{L} \hat{h} ) \simeq - (\underline{L} \hat{h} )^2 
+ \mathsf{T}(\hat{h}) ( \underline{L} \hat{h} )$. In the irrotational case, $\mathsf{f}(\hat{h}) =0$ corresponds to an equation of state where the fluid's potential solves the timelike minimal surface equation, where shocks do not occur (see \citealt{Christodoulou-Book-2007}). 
Plane symmetric perturbations of this situations do not form shocks either (see \citealt{Abbrescia-Wong-2020}). We also note that $\upmu$ satisfies an evolution equation of the form $L \upmu \simeq \mathsf{f}(\hat{h}) \upmu$, so if $\left. \upmu \right|_{t=0} > 0$ and $\mathsf{f}(\hat{h}) < 0$, the dynamics favors driving $\upmu$ towards zero (corresponding to the collapse of the sound cones and thus a shock). Finally, recall that we mentioned that blow-up  of the solution with respect to the original coordinates is recovered by showing that the geometric coordinate system degenerates, although the solution remains regular with respect to the geometric coordinates. This can be roughly viewed by noting that 
$\underline{L} \simeq \frac{\partial}{\partial t} + \frac{1}{\upmu} \frac{\partial}{\partial \mathcal{U}} + \sum_{i=1,2} \frac{\partial}{\partial \vartheta^i}$, so $\upmu \rightarrow 0$ gives the shock. (We recall that in all this discussion we are focusing on $\hat{h}$ and omitting $u$ under the $\simeq$ signs.)
\end{remark}

\subsubsection{Ingredient three: energy estimates and regularity\label{S:Ingredient_three}}
The previous discussion assumes that we can in fact close estimates establishing several elements needed for proving that transverse derivatives blow up -- e.g., that tangential derivatives in fact remain bounded or that the geometric coordinate system remains regular prior to the shock. Thus, we need to derive estimates establishing such claims. 
This means estimates not only for the fluid variables but also 
for the eikonal function (note that the regularity of the null-frame and of the geometric coordinate system is tied to that of 
$\mathcal{U}$). It should be pointed out that the regularity theory of the eikonal function is very challenging (see Sect.~\ref{S:Rough_solutions} for more on this).

Energy estimates for the fluid variables are obtained by commuting the equations with derivatives, but in order to avoid generating uncontrollable source terms, we need to commute the equation with certain geometric vectorfields that are adapted to the sound characteristics. This leads to vectorfields of the form
$\mathcal{Z} \sim \partial \mathcal{U} \cdot \partial$.
Commuting through, e.g., Eq.~\eqref{E:New_formulation_wave_h}, we find
\begin{align}
\begin{split}
\mathcal{Z} (\square_G \hat{h} ) & \sim \square_G ( \mathcal{Z} \hat{h} ) + 
(\square_G \partial \mathcal{U} ) \partial \hat{h}
\\
& 
\sim \square_G ( \mathcal{Z} \hat{h} ) + 
\partial^3 \mathcal{U}  \partial \hat{h},
\end{split}
\nonumber
\end{align}
so that the equation for $\mathcal{Z} \hat{h}$ is of the form

\begin{align}
\begin{split}
 \square_G ( \mathcal{Z} \hat{h} ) \sim   
\partial^3 \mathcal{U}  \partial \hat{h} + \dots
\end{split}
\nonumber
\end{align}
Since $\mathcal{U}$ solves a (fully nonlinear) transport equation, standard regularity theory for transport equations gives that $\mathcal{U}$ is only as regular as the coefficients of the equation, which in this case is $G$; and since $G=G(\hat{h}, s, u)$, we find $\partial^3 \mathcal{U} \sim \partial^3 G \sim \partial^3 \hat{h} + \dots$, so that
\begin{align}
\begin{split}
 \square_G ( \mathcal{Z} \hat{h} ) \sim   
\partial^3 \hat{h}+ \dots
\end{split}
\label{E:Illustration_derivative_loss_h_hat}
\end{align}
Standard energy estimates for $\square_G$ give that from 
$ \square_G ( \mathcal{Z} \hat{h} ) $ we obtain an energy-norm
for $\partial \mathcal{Z} \hat{h} \sim \partial^2\hat{h}$. Thus, from
\eqref{E:Illustration_derivative_loss_h_hat} we see that we are trying to bound $\partial^2 \hat{h}$ in terms of $\partial^3\hat{h}$, i.e., there is a derivative loss.

It turns out that we can overcome the above regularity loss by exploiting some delicate tensorial properties of the eikonal function and of the wave equation relative to geometric coordinates. Together, these properties can be used to show that certain geometric tensors constructed out of $\mathcal{U}$ enjoy \emph{extra regularity properties in directions tangent to the sound cones.} Carefully accounting for the precise structure of the aforementioned $\partial^3 \mathcal{U} \partial \hat{h}$ term, we can show that it is precisely one of such terms with extra regularity. It turns out that all terms that seem to exhibit loss of regularity are of this form and can thus be controlled. In particular, this better behavior in directions tangent to the characteristics\footnote{Better behavior in directions tangent to the characteristics is a recurrent theme in the theory of quasilinear wave equations and is already visible in the case of the flat standard wave equation, see Footnote \ref{FN:Better_behavior_flat_wave}.} is what ultimately allows us to show that tangential derivatives remain bounded, as mentioned in Sect.~\ref{S:Ingredient_two}.

\begin{remark}
The special structures mentioned above that are used to prevent loss of regularity of the eikonal function are tied to the geometry of the sound cones. The improved estimates, without derivative loss, are not based directly on the eikonal equation
Eq.~\eqref{E:Eikonal_equation_shocks}, bur rather on an intricate system for several geometric quantities related to the sound cones, the so-called null-structure equations. We will
have more to say about these equations in Sect.~\ref{S:Rough_solutions}.
\end{remark}

To close the estimates we also need to use the extra regularity for $s$ and $\upomega$ obtained in Theorem \ref{T:Improved_regularity}. To see this, let us do a simple derivative counting. From \eqref{E:New_formulation_wave_u} we get $\square_G u \sim \C$, so we can control $\partial u \lesssim \C$. But $\C \sim \vort(\upomega) \sim \partial \upomega$. From
\eqref{E:New_formulation_transport_omega}, we can control
$\upomega \lesssim \partial u$, so in the end are controlling 
$\partial u \lesssim \partial^2 u$, which has a loss of a derivative. This loss of regularity can be avoided, however, by using the extra regularity for $\upomega$ obtained in Theorem \ref{T:Improved_regularity}. Something similar applies to $s$.

Finally, we point out that the energy estimates that are needed
are in fact \emph{weighted estimates,} where the weight is given by the inverse foliation density $\upmu$. Such weights are needed because $\upmu \rightarrow 0$ at the shock. In particular, the energies are degenerate at top order. This is a major technical point that requires a complicated bootstrap argument to close the estimates (see \citealt{Speck-Book-2016}).

\subsection{Some further context for the work on shocks\label{S:Some_context_shocks}}
The ideas discussed in \ref{S:Study_of_shocks} have not all
been introduced in \cite{Disconzi-Speck-2019}. They are the culmination of a series of beautiful ideas developed by several authors (see the introduction of \cite{Disconzi-Speck-2019} and
\cite{Holzegel-Klainerman-Speck-Wong-2016} for a historical account and further discussion).

As mentioned earlier, if the fluid is irrotational, the new formulation of Theorem \ref{T:New_formulation} simplifies considerably and the equations reduce to those found by \cite{Christodoulou-Book-2007}. The inclusion of vorticity and entropy causes several new difficulties and it is quite remarkable that the general case presents many of the good structures found in the irrotational case. Similar good structures
are also present in the classical compressible Euler equations with vorticity and entropy (see \citealt{Luk-Speck-2020,Luk-Speck-2018,Luk-Speck-2024,Speck-2019}), but once more we remark that one should not expect such delicate structures to be present in the relativistic Euler equations just because they
can be found in their classical analogue or in the irrotational case.

Connecting back to Sect.~\ref{S:Shocks_higher_d}, in the irrotational case, the problem of shock formation for the relativistic Euler equations has been solved in the landmark monograph by \cite{Christodoulou-Book-2007}, whereas 
\cite{Luk-Speck-2020,Luk-Speck-2018,Luk-Speck-2024,Speck-2019} solved it for the classical compressible Euler equations in the presence of entropy and vorticity (in both cases,
a shock is signaled by $\upmu \rightarrow 0$, i.e., by the collapse of the sound characteristics). A similar result, for special classes of quasilinear wave equations (those not satisfying Klainerman's null condition) has been obtained by \cite{Speck-Book-2016}. All these works rely
on a geometric-analytic framework along the lines of what was discussed in Sects.~\ref{S:Ingredient_one}--\ref{S:Ingredient_three} and further discussed in Sect.~\ref{S:Rough_solutions}. 
More recently, \cite{Buckmaster-Shkoller-Vicol-2022,
Buckmaster-Shkoller-Vicol-2023-1,
Buckmaster-Shkoller-Vicol-2023-2},
and \cite{Buckmaster-Drivas-Shkoller-Vicol-2022}, have applied a different formalism
to solve the problem of shock formation for the classical compressible Euler equations with entropy and vorticity.

At this point, readers will probably have noticed that despite all our discussion about the problem of shocks, we have not stated a theorem about it. This is because the problem of shock formation
for the relativistic Euler equations with entropy and vorticity,
as described at the beginning of Sect.~\ref{S:Study_of_shocks},
\emph{has not yet been solved.} More precisely, Theorem \ref{T:New_formulation} provides the first but crucial step
for the resolution of this problem, as it casts the relativistic
Euler equations as a system possessing all the good nonlinear structures that were crucial in Christodoulou's work on shock formation for the irrotational relativistic Euler system and
in Luk and Speck's work on shock formation for the full classical compressible Euler equations, 
These structures have been discussed in Sects.~\ref{S:Ingredient_one}--\ref{S:Ingredient_three}
and Theorem \ref{T:New_formulation} shows how we can uncover them in the relativistic setting with vorticity and entropy. 
Thus, one has
good reason to believe that the geometric-analytic techniques of works\footnote{It is not known if the different approach used in 
\cite{Buckmaster-Shkoller-Vicol-2022,
Buckmaster-Shkoller-Vicol-2023-1,
Buckmaster-Shkoller-Vicol-2023-2,
Buckmaster-Drivas-Shkoller-Vicol-2022} can be adapted to the relativistic setting.} 
by \cite{Christodoulou-Book-2007, Luk-Speck-2020,Luk-Speck-2018,Luk-Speck-2024,Speck-2019} can be generalized to the relativistic Euler equations with entropy and vorticity, even if such a generalization is expected to be daunting in that the equations of Theorem \ref{T:New_formulation} are significantly more intricate than their irrotational or classical counterparts.
We return to this point in Sect.~\ref{S:Open_problems}.

\begin{remark}
We remind the reader that all our discussion of the problem of shocks focused on the case of more than one spatial dimension. In $1+1$ dimensions, the problem of shocks for the relativistic Euler equations
and many other systems is relatively well understood with the theory of systems of conservation laws, see the Sect.~\ref{S:Shocks_1d}.
\end{remark}

\section{Rough solutions to the relativistic Euler equations}
\label{S:Rough_solutions}

The standard existence theory for the relativistic Euler equations in three spatial dimensions gives local well-posedness in $H^N$ for $N > \frac{3}{2} + 1$ (taking, say, $(h,s,u)$ as primary variables). A natural question that drives a great deal of research in PDEs is that of the minimum value $N_0$ such that a given system of PDEs is locally well-posed in $H^{N_0}$. A less ambitious but related question is whether we can establish local well-posedness in $H^N$ for $N$ below the threshold given by standard theory\footnote{Where what is considered ``standard'' naturally depends on the equation; for Euler, standard theory refers to the methods yielding $N > \frac{3}{2} + 1$ such as, for example, first-order symmetric hyperbolic systems.}. Questions of this nature are commonly referred to as \textdef{low-regularity questions/problems} and solutions with regularity below such a standard $N$ are known as \textdef{rough or low-regularity solutions.}

In the irrotational case, the relativistic Euler equations can be written as a system of the form \eqref{E:Prototype_wave_equation}, with $\mathscr{G}$ being the acoustical metric, $F$ a quadratic nonlinearity,  
and\footnote{Recall Footnote \ref{FN:Irrotational_differentiated}.} $\Psi = d\phi$, where $\phi$ is the fluid potential (see Sect.~\ref{S:Irrotational_flows}). The study of rough solutions for equations of the form \eqref{E:Prototype_wave_equation} has a long history and remarkable results have been achieved. Some key results, which we state  in terms of the corresponding fluid variables which are of interest here, are the following. The irrotational relativistic Euler equations in $3+1$ dimensions are locally well-posed for\footnote{Recall from Sect.~\ref{S:Irrotational_flows} that the fluid potential $\phi$ completely determines a solution to the relativistic Euler equations, and from Sect.~\ref{S:New_formulation} that the irrotational equations take the form \eqref{E:Prototype_wave_equation} with $\Psi = d\phi$. Thus, it is more natural to state the regularity for $d\phi$. Moreover, while all fluid variables are determined by $\phi$, it is more natural in a fluid context to state the required regularity for both $h$ and $u$, recalling that in fact $u\sim d\phi$ (there is no need to include the entropy because it is necessarily constant in the irrotational case). This also makes comparisons with the non-irrotational case easier. Finally, it is more convenient to consider $\hat{h}$ rather than $h$ because, in this case, conditions at infinity are immediately accommodated by stating that $\hat{h}$ belongs to some Sobolev space, whereas if we used $h$ we would have to make statements for $h-\bar{h}$ for some constant reference value $\bar{h} >0$ (and similarly if we used, say, the density as a variable).
}
\begin{align}
(\hat{h}, u = d\phi) \in H^N,
\nonumber
\end{align}
where $\hat{h}$ is the logarithm of the density, see below, with (we state the results in chronological order and also in order of improvements over the previous ones):

\begin{itemize}
\item $N > \frac{9}{4} = 2.25$, by
a result of \cite{Bahouri-Chemin-1999}.
\item $N > \frac{13}{6} = 2.16\dots$, by a result of \cite{Tataru-2002-1}.
\item $N > 2 + \frac{2-\sqrt{3}}{2} = 2.13\dots$, by a result of \cite{Klainerman-Rodnianski-2003}.
\item $N > 2$, by a result of \cite{Smith-Tataru-2005} (later, an alternative proof was given by \citealt{Wang-2017}\footnote{Roughly speaking, Smith and Tataru's proof uses a frequency-space approach whereas Wang's proof uses a physical-space approach. The low-regularity result for Euler we will discuss, Theorem \ref{T:Rough_classical_Euler}, is also based on a physical-space approach and borrows substantially from Wang for the treatment of the wave-part of the system.}).
\end{itemize}

We remark the following:

\begin{itemize}
\item Within the context of ``linear theory,'' i.e., assuming a pre-specified regularity for the coefficients but no additional structure (in particular, one does not assume that $\square_G G$ satisfies an equation, as it will be the case for the Euler flow), Tataru's $13/6$ result is optimal in view of a result by \cite{Smith-Tataru-2002}.
\item Smith and Tataru's $N>2$ result is optimal under the stated assumptions, as \cite{Lindblad-1998}
proved ill-posedness in $H^2$ (the breakdown mechanism is the instantaneous formation of shocks).
\end{itemize}

We can now ask whether similar low-regularity results hold in the case $\Omega\neq 0$. The irrotational and rotational case are qualitatively different, with the transport-part deeply coupled to the wave-part, a manifestation of the already alluded fact that when $\Omega\neq 0$ the Euler flow is a system with multiple characteristic speeds. Nevertheless, one hopes that the philosophy of Sect.~\ref{S:New_formulation} can be applied here as well, namely, that we can establish existence of rough solutions to the relativistic Euler equations
with vorticity and entropy by treating it as a perturbation of the irrotational case modeled by 
a system of the form \eqref{E:Prototype_wave_equation} (keeping in mind that genuinely new ideas will be needed to treat the transport-part of the system, as we will discuss below).

Before discussing rough solutions for the relativistic Euler equations, we will sketch the proof in the case of the classical compressible Euler equations. We do so \emph{not} because the result for the relativistic Euler equations follows form easy modifications of its classical counterpart one; as we stressed a few times, results in the relativistic setting in general do not follow by tweaking with
techniques applied to the classical equations. Rather, considering the classical case first will provide greater conceptual clarity, as it will allow
us to avoid extra unpleasant technicalities present in the relativistic setting and focus on the some key ideas that are shared by the proofs in both the classical and relativistic cases. Additionally, 
there has been substantially more work on rough solutions to the classical compressible Euler equations than its relativistic counterpart, so that a more comprehensive theory of rough solutions is available in the classical case (see discussion after Theorem \ref{T:Rough_classical_Euler}). Finally, as mentioned in the Introduction, \emph{we want to provide one application showing how the geometric-analytic formalism can be used in the study of classical fluids.}

Therefore, we introduce the following notation representing quantities of a classical fluid governed by the classical compressible Euler equations (Eqs.~\eqref{E:Classical_Euler} below). \emph{We purposely use the same notation as in the relativistic case in order to keep the analogy with the relativistic case, discussed later, easier} (e.g., we call $u$ the classical velocity of the fluid).

\begin{notation}
We henceforth adopt the following notation for fluid variables of a classical fluid described by the classical compressible Euler equations (see equations \ref{E:Classical_Euler}). This notation will be used until Sect.~\ref{S:Back_to_relativistic_Euler}, when we return to the relativistic Euler equations.

All quantities will be defined in $[0,T]\times\mathbb{R}^3$ for some $T>0$. We adopt standard rectangular coordinates $\{ x^\alpha \}_{\alpha=0}^3$ in $[0,T]\times\mathbb{R}^3$, with $t:= x^0$ denoting a time coordinate and $\{x^i\}_{i=1}^3$ denoting spatial coordinates and refer to 
$[0,T]\times\mathbb{R}^3$ as the spacetime. Our indices convention of Sect.~\ref{S:Notation_conventions} still applied here, but spatial vectorfields, i.e., maps $X\colon [0,T]\times\mathbb{R}^3 \rightarrow \mathbb{R}^3$ will have their indices raised and lowered with the 
three-dimensional Euclidean metric denoted $\updelta$. We introduce the following quantities
\begin{itemize}
\item The \textdef{logarithmic density}
$\hat{h}\colon [0,T]\times\mathbb{R}^3 \rightarrow \mathbb{R}$,
\begin{align}
\hat{h} := \log \frac{\varrho}{\bar{\varrho}},
\nonumber
\end{align}
where $\varrho \colon [0,T]\times\mathbb{R}^3 \rightarrow \mathbb{R}$ is the fluid's density and $\bar{\varrho} > 0$ is a constant reference density.
\item $s\colon [0,T]\times\mathbb{R}^3 \rightarrow \mathbb{R}$ is the \textdef{entropy} and we introduce the \textdef{entropy gradient} by
\begin{align}
S := \bar{\nabla} s,
\nonumber
\end{align}
where $\bar{\nabla}$ is the standard three-dimensional gradient in $\mathbb{R}^3$.
\item The \textdef{fluid's velocity} is a spatial
vectorfield
\begin{align}
u = (u^1,u^2,u^3) \colon [0,T]\times\mathbb{R}^3 \rightarrow \mathbb{R}^3
\nonumber
\end{align}
representing the velocity of fluid particles.
\item $\dive$ and $\curl$ are the standard three-dimensional divergence and curl operators 
in $\mathbb{R}^3$.
\item The \textdef{specific vorticity} is the 
spatial
vectorfield
\begin{align}
\Omega := \frac{\curl u}{\varrho/\bar{\varrho}}
= \frac{\curl u}{e^{\hat{h}}}.
\nonumber
\end{align}
\item The \textdef{material vectorfield,}
\begin{align}
\B := \partial_t + u^i\partial_i.
\label{E:Material_derivative_classical}
\end{align}
In the classical setting, $\B$ plays a role similar to differentiation in the direction of the four-velocity in the relativistic case.
\item For an equation of state $p = p(\varrho,s)$, where $p$ represents the \textdef{fluid's pressure,} the \textdef{fluid's sound speed} is defined as\footnote{It is assumed that the equation of state is such that $\frac{\partial p}{\partial \varrho} \geq 0$, which is the case for most equations of state of interest.}
\begin{align}
c_s^2 := \frac{\partial p}{\partial \varrho}.
\nonumber
\end{align}
\item For $c_s > 0$, we define the \textdef{acoustical metric} by
\begin{align}
G := - dt \otimes dt + c_s^{-2} \sum_{i=1}^3 (dx^i - u^i dt) \otimes (dx^i - u^i dt),
\label{E:Acoustical_metric_classical}
\end{align}
whose inverse is\footnote{See Remark \ref{R:Inverse_sign_explicit} for the explicit use of $^{-1}$, where in the classical context indices are raised and lowered with the Euclidean metric.} 
\begin{align}
G^{-1} := - \B \otimes \B + c_s^2 \sum_{i=1}^3 \partial_i \otimes \partial_i.
\label{E:Acoustical_metric_inverse_classical}
\end{align}
\end{itemize}
\end{notation}

It can be verified that
\eqref{E:Acoustical_metric_classical}
defines a Lorentzian metric in $[0,T]\times \mathbb{R}^3$ whose inverse is given by 
\eqref{E:Acoustical_metric_inverse_classical}. 
Observe also that $G(\B,\B) = -1$.
The introduction of a Lorentzian metric preludes the idea that techniques of quasilinear wave equations, many of which were developed in the context of mathematical 
general relativity, will be applied to the study of classical fluids.

We now assume that the pressure is a function\footnote{For the classical compressible Euler equations, similar remarks as Sect.~\ref{S:Thermodynamic_properties} hold regarding the choice of independent thermodynamics variables.} of $\hat{h}$ and $s$.
The \textdef{classical compressible Euler equations}
in $[0,T]\times \mathbb{R}^3$ are given by
\begin{subequations}{\label{E:Classical_Euler}}
\begin{align}
\B\hat{h} & = -\dive u,
\\
\B u^i & = - e^{-\hat{h}} \updelta^{ij} \partial_j p,
\\
\B s & = 0.
\end{align}
\end{subequations}

A computation similar to that of Sect.~\ref{S:Characteristics_Euler} shows that the 
characteristics of \eqref{E:Classical_Euler} are the flow lines of $\B$, $\B^\alpha \xi_\alpha = 0$, and the sound cones, i.e., the characteristic surfaces of the acoustical metric \eqref{E:Acoustical_metric_classical}, $(G^{-1})^{\alpha\beta} \xi_\alpha \xi_\beta = 0$.

As discussed for the relativistic Euler equations in Sect.~\ref{S:New_formulation}, equations
\eqref{E:Classical_Euler}, or equivalent first-order formulations of the equations, do not seem amenable to the types of applications we want to consider.
Thus, we will employ a new formulation of \eqref{E:Classical_Euler} that plays in the classical setting a role similar to the equations of Theorem \ref{T:New_formulation} for the relativistic Euler equations (see Eqs.~\eqref{E:New_formulation_classical} below). We will need the classical analogue 
of quantities \eqref{E:Modified_vorticity_of_vorticity} and \eqref{E:Modified_divergence_entropy_gradient}, thus we introduce
the \textdef{modified vorticity of the vorticity} and the \textdef{modified divergence of the entropy gradient}, respectively, by
\begin{align}
\C := e^{-\hat{h}} \curl \Omega + \dots \sim \curl \Omega 
\label{E:Modified_vorticity_of_vorticity_classical}
\end{align}
and
\begin{align}
\D := e^{-2\hat{h}} \dive S + \cdots \sim \dive S.
\label{E:Modified_divergence_entropy_gradient_classical}
\end{align}
As in the case of the relativistic Euler equations, 
these modified quantities are introduced because 
one needs to obtain good estimates for $\curl \Omega$ and $\dive S$, but these quantities by themselves to not satisfy equations with good structure. However, by adding an appropriate combination of lower-order terms we obtain quantities, $\C$ and $\D$, that satisfy good equations from which estimates can be derived, eventually yielding the desired control of $\curl \Omega$ and $\dive S$. The precise form of the lower-order terms added to $\curl \Omega$ and $\dive S$ will not be important in our discussion because we will mostly write things in schematic form, thus we omit them, only indicating their presence by $\dots$ in \eqref{E:Modified_vorticity_of_vorticity_classical} and \eqref{E:Modified_divergence_entropy_gradient_classical}. We also indicate with $\sim$ the main structural form that is capture by $\C$ and $\D$ and that will be important in our discussion. See \cite{Speck-2019} for the exact definition of $\C$ and $\D$ for the classical compressible Euler equations.

We are now ready to state the main result concerning 
rough solutions of the classical compressible Euler equations. We will denote by $C^{0,\upalpha}$ the standard H\"older spaces and recall that $\Sigma_0 = \{ t = 0 \}$ represents the initial hypersurface where data is prescribed.

\begin{theorem}[\citealt{Disconzi-Luo-Mazzone-Speck-2022}]
\label{T:Rough_classical_Euler}
Consider a smooth\footnote{For convenience, it is  assumed that the solutions are as many times differentiable as necessary. Thus, ``smooth'' means ``as smooth as necessary for qualitative arguments (such as integration by parts) to go through.'' However, all of our quantitative estimates depend only on the Sobolev and H\"older norms mentioned in Theorem \ref{T:Rough_classical_Euler}.} solution to the compressible Euler equations \eqref{E:Classical_Euler} whose initial data obey the following assumptions for some real numbers 
$0<\varepsilon \leq \frac{1}{2}$, a small $\upalpha > 0$, 
$0 < D = D(\varepsilon,\upalpha) < \infty$, $0 < c_1 < c_2 < \infty$, $0 < c_3$:

\begin{enumerate}
\item $\norm{( \hat{h},u,\curl u )}_{H^{2+\varepsilon}(\Sigma_0)} + \norm{s}_{H^{3+\varepsilon}(\Sigma_0)} \leq D$.
\item $\norm{(\C,\D)}_{C^{0,\upalpha}(\Sigma_0)} \leq D$.
\item Along $\Sigma_0$, the data are contained in the interior of a compact subset $K$ of state space in which $\varrho \geq c_3$ and $c_1 \leq c_s \leq c_2$.
\end{enumerate}
Then, the solution's time of classical existence $T$ depends only on $D$ and $K$, $T = T(D, K)$, and the Sobolev and H\"older regularity of the data are propagated by the flow\footnote{I.e., the norms we can control are uniformly bounded functions of $(D,K)$ for $t \in [0,T]$.}.
\end{theorem}

We now make several remarks regarding the assumptions and conclusions of Theorem \ref{T:Rough_classical_Euler}. It will be useful in such remarks and in what follows to 
introduce $\Psi = (\hat{h}, u, s)$ and call 
$\Psi$ the \textdef{wave variables} because they satisfy wave equations of the form $\square_G \Psi = \dots$, while the variables $\{ \Omega, S, \C, \D \}$ will be called the \textdef{transport variables}
because they satisfy transport equations\footnote{In view of the decomposition of a vectorfield into its divergence and curl parts, it would be more precise to include only $\dive u$ in $\Psi$ since $\curl u$ has been grouped with the transport variables. In practice, however, we will employ a wave equation for $u$ itself analogue to \eqref{E:New_formulation_wave_u} in the relativistic case.\label{FN:Wave_part_full_velocity}} (see Eqs.~\eqref{E:New_formulation_classical} below). We intuitively think of $\Psi$ as corresponding to the wave-part of the system,
i.e., the part whose dynamics is tied to the sound cones, and of $\{ \Omega, S, \C, \D \}$ as the transport-part of the system, i.e., the part whose dynamics is tied to the flow lines of the material derivative, with the understanding that this is a useful way of organizing the proof but \emph{one should not think that we can treat the wave and transport parts separately,} as their evolution is deeply coupled.

\begin{itemize}
\item The proof of Theorem \ref{T:Rough_classical_Euler} involves several ideas of independent interest: sharp estimates for the acoustic geometry (i.e., the characteristic geometry of the acoustical metric, see Sect.~\ref{S:Control_acoustic_geometry} and recall Definition \ref{D:Acoustic_geometry}), Strichartz estimates for waves coupled to vorticity, Schauder estimates for the div-curl part.
\item Compared to the aforementioned low-regularity results for systems of the form \eqref{E:Prototype_wave_equation}, the main challenge is that \eqref{E:Classical_Euler} is a system with multiple characteristic speeds. Low-regularity techniques for systems of the form \eqref{E:Prototype_wave_equation} are based on Strichartz estimates, which are well-adapted to the wave part of the system since they are based on dispersion. \emph{There are no Strichartz estimates} for transport equations because solutions to such equations do not disperse. In addition, one has to handle the interaction of the wave and transport parts; as we will see, the transport variables enter as source terms in the estimates for the acoustic geometry  (see Sect.~\ref{S:Control_acoustic_geometry}). This highlights the fact that the rotational and irrotational problems are qualitatively different; \emph{even the tiniest amount of vorticity is a game changer.}
\item Aside from $(\hat{h}, u) \in H^{2+\varepsilon}(\Sigma_0)$, which are in line with Smith and Tataru's optimal result mentioned above for systems of the form \eqref{E:Prototype_wave_equation}, we have the ``extra'' regularity assumptions $\curl u \in H^{2+\varepsilon}(\Sigma_0)$, $s \in H^{2+\varepsilon}(\Sigma_0)$, and $\C \sim \curl \curl u, D \sim \partial^2 s \in C^{0,\upalpha}(\Sigma_0)$.
We will explain how this extra regularity is needed in the context of the techniques we will present. However, it should be remarked that such additional regularity of the transport-part of the data is \emph{propagated} by the flow\footnote{Observe that $\curl u \in H^{2+\varepsilon}(\Sigma_0)$ and $s \in H^{3+\varepsilon}(\Sigma_0)$ are in line with the extra regularity of vorticity and entropy established in the relativistic setting in Theorem \ref{T:Improved_regularity}.}, even though
the transport-part is deeply coupled to the rougher wave-part (through source terms in estimates for the acoustic geometry, in particular). 
In particular, note that propagation of H\"older regularity is non-standard for hyperbolic systems and is, thus, of independent interest.
Ultimately, our regularity assumptions are tied to the regularity of the characteristics.
\item Assumption (3) is a type of non-degeneracy condition (the degenerate case when the density and sound speed vanish will be investigated in Sect.~\ref{S:Vacuum_bry}).
\item Theorem \ref{T:Rough_classical_Euler} is optimal with respect to the wave-part of the system, i.e., 
$(\hat{h}, u) \in H^{2+\varepsilon}(\Sigma_0)$.
Indeed, without assuming zero vorticity or entropy,
\cite{An-Chen-Yin-2021-arxiv} obtained that equations
\eqref{E:Classical_Euler} are ill-posded for data in $H^2$ due to instantaneous formation of shocks, thus extending result by \cite{Lindblad-1998}
to the rotational and non-isentropic case\footnote{Unlike the case of a relativistic fluid, 
a classical compressible can be irrotational and have dynamic entropy.}.
\item Theorem \ref{T:Rough_classical_Euler} was the first low-regularity result for a system with multiple characteristics in three spatial dimensions. After its appearance, 
\cite{Wang-2022}, and independently \cite{Zhang-Andersson-arxiv-2022},  
improved Theorem \ref{T:Rough_classical_Euler}, particularly by removing the H\"older assumption on the data\footnote{Although, as mentioned, the propagation of H\"older regularity in Theorem \ref{T:Rough_classical_Euler} is of independent interest.} and lowering the required regularity of the vorticity to $\curl u \in H^{2+{\varepsilon^\prime}}(\Sigma_0)$, $0< \varepsilon^\prime < \varepsilon$. 
We discuss these improvements in Sect.~\ref{S:Lowering_regularity_classical_Euler}.
From these works and \cite{
An-Chen-Yin-2021-arxiv} mentioned above, Wang conjectured an optimal-regularity result for the compressible Euler equations (see \citealt{Wang-2022} and Sect.~\ref{S:Lowering_regularity_classical_Euler}).
\end{itemize}

\medskip
\noindent \emph{Proof of Theorem \ref{T:Rough_classical_Euler}.} The proof is quite long. Here, we will only outline the main ideas, expanding on them the ensuing sections and referring to \cite{Disconzi-Luo-Mazzone-Speck-2022} for more details. The logic of the argument is presented in Fig.~\ref{F:Flowchart}.

The strategy for proving Theorem \ref{T:Rough_classical_Euler} is the following:

\begin{enumerate}
\item We will use known techniques from quasilinear wave equations (energy and Strichartz estimates) to control the wave-part of the system. This requires, in particular, controlling the acoustic geometry (the regularity of the characteristic surfaces of the acoustical metric, i.e., the sound cones). For this, one needs to derive complementary estimates for several geometric quantities associated with the sound cones.
\item We need to control the transport variables at a consistent amount of regularity as in the previous item. Energy estimates for transport equations are not enough and there are no Strichartz estimates for transport equations. We will instead combine the transport-type energy estimates with elliptic estimates and exploit structural features of the equations.
\item Transport variables appear as source terms in the acoustic geometry, thus we need to handle their interaction\footnote{Note that this is a feature of a problem with multiple characteristic speeds.}.
\end{enumerate}
\hfill \qed

\begin{figure}[htbp]
\centering
  \includegraphics[scale=0.5]{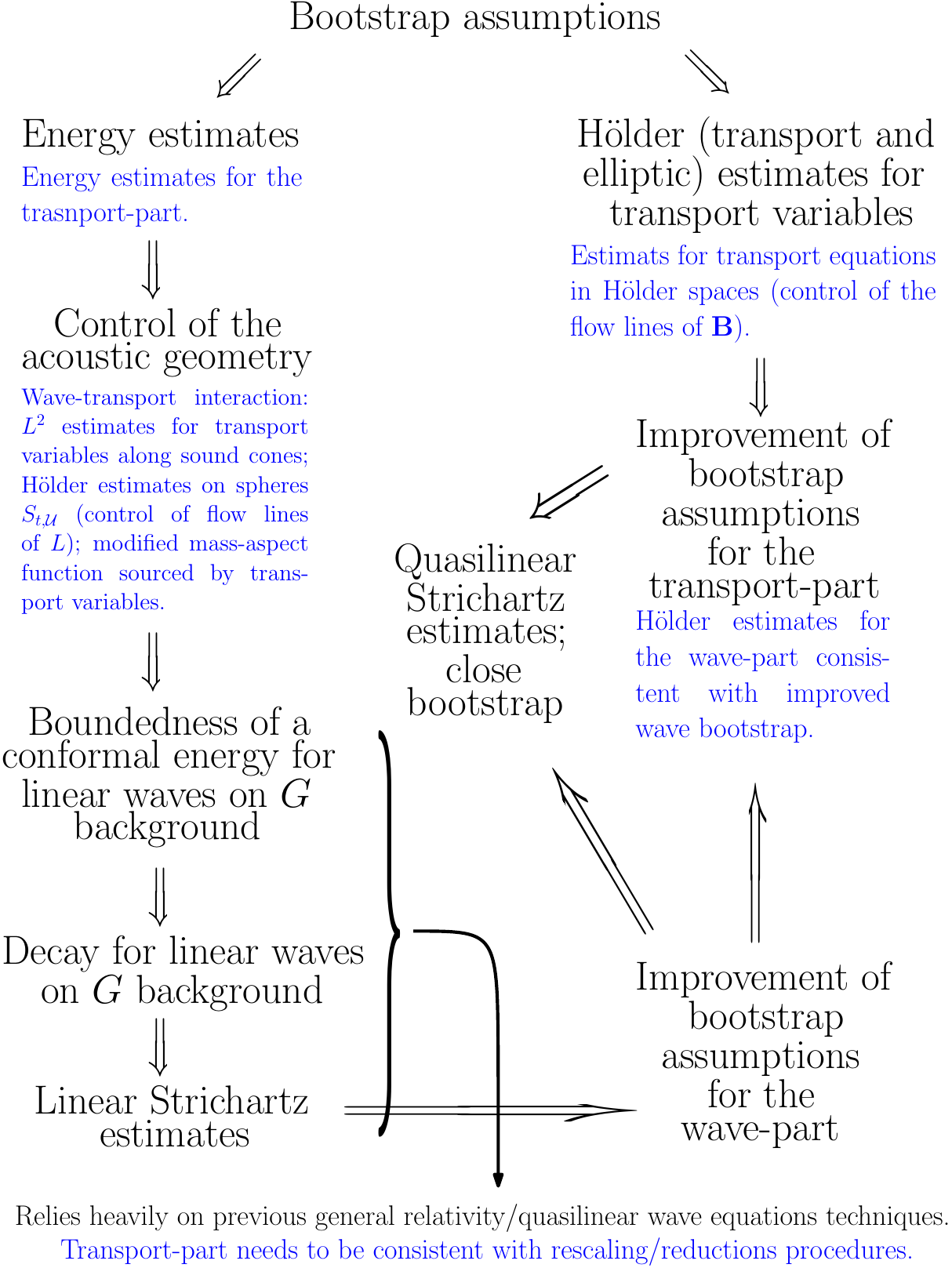}
  \caption{Sketch of the logic of the proof of Theorem \ref{T:Rough_classical_Euler}. New elements, as compared to the irrotational case, are highlighted in \textcolor{blue}{blue.} Terms indicated in the figure but not yet defined are discussed later in the text. The idea of this flowchart was borrowed from \cite{Yu-2024}.}
  \label{F:Flowchart}
\end{figure}

We will discuss the above strategy in more detail in Sects.~\ref{S:New_formulation_classical_schematic}--\ref{S:Closing_bootstrap}. Since Theorem \ref{T:Rough_classical_Euler} can be reduced to the derivation of a priori estimates for smooth solutions in the relevant norms of the problem,
we will assume throughout a given smooth solution
to Eqs.~\eqref{E:Classical_Euler} defined for some time interval $[0,T]$, $T>0$, whose size will ultimately be determined by the estimates we will derive and in accordance with the conclusion of Theorem \ref{T:Rough_classical_Euler}.

\subsection{The new formulation in schematic form\label{S:New_formulation_classical_schematic}}
In order to provide the main ideas in the proof of Theorem \ref{T:Rough_classical_Euler} while keeping the discussion accessible and at a high level, we will need to make some simplifying assumptions, carry out some convenient abuses of notation, and downplay certain distinctions among variables and equations. This will allow us to summarize the nearly hundred fifty pages needed for the proof of Theorem \ref{T:Rough_classical_Euler} (which in turn
relies on hundreds of pages of previous works) in a way that will still be useful to readers.

Thus, we begin by assuming an isentropic fluid, i.e., $s=\text{constant}$, so that $\D=0$. 
In this case $\C$ is given exactly by $\C = e^{-\hat{h}} \curl \Omega$, i.e., the $\dots$ terms in \eqref{E:Modified_vorticity_of_vorticity_classical} vanish (see equation (1.3.13a) in \cite{Speck-2019}). We will further ignore the 
$e^{-\hat{h}}$ term, which is lower order, and write
$\C \sim \curl \Omega$ (see Remark \ref{R:Trading_C_for_curl_Omega} below). 

For the classical compressible Euler system, equations analogue to those of Theorem \ref{T:New_formulation} were found by \cite{Speck-2019,Luk-Speck-2020}. In our simplified setting, and writing schematically, these equations read (recall Sect.~\ref{S:Notation_conventions} for the difference between $\partial$ and $\bar{\partial}$):
\begin{subequations}{\label{E:New_formulation_classical}}
\begin{align}
\square_G \Psi & \simeq \curl \Omega + \partial \Psi \cdot \partial \Psi,
\label{E:New_formulation_classical_wave}
\\
\B \Omega & \simeq \partial \Psi,
\label{E:New_formulation_classical_transport}
\\
\B \curl \Omega & \simeq \partial \Psi \cdot 
\bar{\partial} \Omega,
\label{E:New_formulation_classical_transport_curl_Omega}
\\
\dive \Omega & \simeq \partial \Psi.
\label{E:New_formulation_classical_transport_div_Omega}
\end{align}
\end{subequations}

In Eqs.~\eqref{E:New_formulation_classical}, $\simeq$ is employed as in the statement of Theorem \ref{T:New_formulation}, so in particular 
harmless linear terms, as well as coefficients that are smooth functions of $(\Psi,\Omega)$, are omitted.
Recall that $\Psi= (\hat{h}, s, u)$, or, in our case, $\Psi= (\hat{h}, u)$ since $s$ is constant. Equation \eqref{E:New_formulation_classical_wave} thus represents the evolution equations for both $u$ and $\hat{h}$, which we group together as a single vector equation as the wave equations for $u$ and $\hat{h}$ have similar structures\footnote{In reality, $\C$ is a source term only in the equation for $u$, see equations (3.1.1a,b) in \cite{Speck-2019}. However, since we will treat $\Psi$ as a single variable in what follows, without distinguishing between $\hat{h}$ and $u$, we write it as in \eqref{E:New_formulation_classical_wave}. Moreover, the equation for $\hat{h}$ is in fact sourced by $\D$, which vanishes here but in the general case is the modified variable playing a role for the $\hat{h}$ equation similar to the role played by $\C \sim \curl \Omega$ for the $u$ equation.}. The term $\partial \Psi \cdot \partial \Psi$ on the RHS of \eqref{E:New_formulation_classical_wave} represent terms quadratic in $\Psi$, whereas 
$\partial \Psi \cdot \bar{\partial} \Omega$ on the RHS of \eqref{E:New_formulation_classical_transport_curl_Omega} represents quadratic\footnote{The quadratic terms on the RHS of Eqs.~\eqref{E:New_formulation_classical_wave} and \eqref{E:New_formulation_classical_transport_curl_Omega} are null-forms relative to the acoustical metric. For Theorem \ref{T:Rough_classical_Euler} it is only important that such terms are quadratic and the null-form structure is not important. However, in their work on shock formation \cite{Luk-Speck-2020,Luk-Speck-2018,Luk-Speck-2024,Speck-2019}, when Luk and Speck first introduced the new formulation of the classical compressible Euler equations, the null-forms are crucial, for reasons similar to our discussion in Sect.~\ref{S:Ingredient_two}.}
 terms involving the product of $\partial \Psi$ and $\bar{\partial} \Omega$ but \emph{not} terms in
$(\partial \Psi)^2$ or $(\bar{\partial} \Omega)^2$.
Equation \eqref{E:New_formulation_classical_wave} should be viewed as the classical counterpart to
Eqs.~\eqref{E:New_formulation_wave},
equation \eqref{E:New_formulation_classical_transport} as the classical counterpart to \eqref{E:New_formulation_transport}, and
Eqs.~\eqref{E:New_formulation_classical_transport_curl_Omega}-\eqref{E:New_formulation_classical_transport_div_Omega} as the classical counterparts to \eqref{E:New_formulation_transport_div_curl}, with the further simplification $s=\text{constant}$ in the classical case. 

\begin{remark}
In our exposition, the distinction between
spacetime derivatives $\partial$ and spatial
derivatives $\bar{\partial}$ will not be important,
as we will be interested only in a derivative counting for such terms. The terms for which the specific form of the derivatives \emph{do matter} have already been distinguished by writing explicitly $\B$, $\curl$, and $\dive$ in Eqs.~\eqref{E:New_formulation_classical}. 
\end{remark}

Finally, we note the following important point.

\begin{remark}
\label{R:Trading_C_for_curl_Omega}
In an attempt to simplify the exposition, as explained at the beginning of this section, it is easier to think of the new formulation of classical Euler as a system for
$\Psi$ and $\Omega$ only, as written in \eqref{E:New_formulation_classical}. So, in particular, we used $\C \sim \curl \Omega$ to replace $\C$ by $\curl \Omega$ in the equations derived by \cite{Speck-2019,Luk-Speck-2020}. Nevertheless, we remind that it is $\C$, and not $\curl \Omega$, that satisfies good equations, and estimates for $\curl \Omega$ can be derived only through estimates for $\C$. This can be viewed, for example, from comparing Eqs.~\eqref{E:New_formulation_classical_transport} and \eqref{E:New_formulation_classical_transport_curl_Omega}. Taking the $\curl$ of \eqref{E:New_formulation_classical_transport} would produce a term in $\partial^2 \Psi$, and estimates would not close were such term present. What happens is that in reality it is $\B \C$, and not $\B \curl \Omega$, that we need to consider. The ``extra terms'' in $\C$, implicit under $\dots$ in 
\eqref{E:Modified_vorticity_of_vorticity_classical} (but written explicitly in the relativistic case in \eqref{E:Modified_vorticity_of_vorticity}) that are added to $\curl \Omega$ work precisely to produce a cancellation, so that no $\partial^2 \Psi$ term is present in the evolution for $\C$. However, without getting into the details of the calculations, details that are beyond our goal, one cannot see the usefulness of considering $\C$, so that writing
$\B \C$ instead of $\B \curl \Omega$ and having to switch back and forth between $\C$ and $\curl \Omega$ would be a distraction that we would rather avoid.
In other words, we write 
 $\curl \Omega$ with the understanding that it should be placeholder for what in reality is $\C$.
\end{remark}

With the above set-up in place, we are finally ready to start discussing the proof of Theorem \ref{T:Rough_classical_Euler}.

\subsection{Energy estimates\label{S:Energy_estimates_low_regularity}} We want to control $\Psi$ in $H^{2+\varepsilon}$, so at top order we need $\bar{\partial}^{2+\varepsilon} \Psi \in L^2$. 
Taking $\bar{\partial}^{1+\varepsilon}$ derivatives\footnote{Fractional derivatives are computed with help of the Littlewood--Paley projections.} of \eqref{E:New_formulation_classical_wave} we obtain
\begin{align}
\square_G \bar{\partial}^{1+\varepsilon} \Psi 
\simeq \bar{\partial}^{1+\varepsilon} \curl \Omega,
\nonumber
\end{align}
so that that standard wave equation estimates give
\begin{align}
\norm{\partial\bar{\partial}^{1+\varepsilon} \Psi}_{L^2(\Sigma_t)}
\lesssim  \norm{\partial\bar{\partial}^{1+\varepsilon} \Psi}_{L^2(\Sigma_0)} + \int_0^t 
\norm{\bar{\partial}^{1+\varepsilon} \curl \Omega}_{L^2(\Sigma_\tau)} \, d\tau,
\label{E:Energy_estimates_2_plus_epsilon_part_1}
\end{align}
where we are making use of the following:

\begin{notation}
\label{N:L_infty_terms}
Above and in what follows, we ignore terms that can be Gr\"onwalled in a standard fashion, keeping track only of ``dangerous terms'' whose derivatives have to be controlled at top order. Observe, however, that some of these Gr\"onwall terms come with coefficients in $L^\infty$ which also need to be controlled. We subsume these $L^\infty$ terms
into $\lesssim$ for now, while discussing top-order terms in $L^2$, returning to the $L^\infty$ terms in the sequel.
\end{notation}

From \eqref{E:Energy_estimates_2_plus_epsilon_part_1} we see that we need to control $\bar{\partial}^{1+\varepsilon} \curl \Omega$ in $L^2$. We cannot obtain this control from \eqref{E:New_formulation_classical_transport} because it gives 
\begin{align}
\B \bar{\partial}^{1+\varepsilon}\curl  \Omega 
& \simeq \bar{\partial}^{1+\varepsilon}\curl  \partial \Psi 
\nonumber
\\
& \simeq \partial \bar{\partial}^{2+\varepsilon} \Psi.
\nonumber
\end{align}
However, using \eqref{E:New_formulation_classical_transport_curl_Omega} we get
\begin{align}
\B \bar{\partial}^{1+\varepsilon} \curl \Omega
\simeq \partial \Psi \cdot \bar{\partial}^{2+\varepsilon}\Omega + \partial\bar{\partial}^{1+\varepsilon} \Psi \cdot \bar{\partial} \Omega.
\nonumber
\end{align}
All terms on the RHS are controlled at our regularity level except possibly $\bar{\partial}^{2+\varepsilon} \Omega$ in 
$L^2$, hence:
\begin{align}
\norm{\bar{\partial}^{1+\varepsilon} \curl \Omega}_{L^2(\Sigma_t)}
& \lesssim
\norm{\bar{\partial}^{1+\varepsilon} \curl \Omega}_{L^2(\Sigma_0)} + 
\int_0^t \norm{\bar{\partial}^{2+\varepsilon} \Omega}_{L^2(\Sigma_\tau)} \, d\tau.
\label{E:Energy_estimates_2_plus_epsilon_part_2}
\end{align}

To control $\bar{\partial}^{2+\varepsilon} \Omega$ in $L^2$,
we use the following standard Hodge estimate for any vectorfield $V$ in $\mathbb{R}^3$,
\begin{align}
\norm{\bar{\partial} V}_{L^2(\Sigma_t)}
\lesssim 
\norm{\dive V}_{L^2(\Sigma_t)} +
\norm{\curl V}_{L^2(\Sigma_t)},
\nonumber
\end{align}
which gives
\begin{align}
\begin{split}
\norm{\bar{\partial}^{2+\varepsilon} \Omega}_{L^2(\Sigma_t)}
& \lesssim 
\norm{\dive \bar{\partial}^{1+\varepsilon} \Omega}_{L^2(\Sigma_t)} +
\norm{\curl \bar{\partial}^{1+\varepsilon}\Omega}_{L^2(\Sigma_t)}.
\end{split}
\nonumber
\end{align}
Using \eqref{E:New_formulation_classical_transport_div_Omega}
to estimate the divergence term,
\begin{align}
\begin{split}
\norm{\bar{\partial}^{2+\varepsilon} \Omega}_{L^2(\Sigma_t)}
& \lesssim 
\norm{\bar{\partial}^{2+\varepsilon} \Psi}_{L^2(\Sigma_t)} +
\norm{\curl \bar{\partial}^{1+\varepsilon}\Omega}_{L^2(\Sigma_t)}.
\end{split}
\label{E:Energy_estimates_2_plus_epsilon_part_3}
\end{align}
Using \eqref{E:Energy_estimates_2_plus_epsilon_part_3} into
\eqref{E:Energy_estimates_2_plus_epsilon_part_2} produces
\begin{align}
\begin{split}
\norm{\bar{\partial}^{1+\varepsilon} \curl \Omega}_{L^2(\Sigma_t)}
& \lesssim
\norm{\bar{\partial}^{1+\varepsilon} \curl \Omega}_{L^2(\Sigma_0)} + 
\int_0^t \norm{\bar{\partial}^{2+\varepsilon} \Omega}_{L^2(\Sigma_\tau)} \, d\tau
\\
& \lesssim
\norm{\bar{\partial}^{2+\varepsilon} \Omega}_{L^2(\Sigma_0)} + 
\int_0^t \norm{\curl \bar{\partial}^{1+\varepsilon}\Omega}_{L^2(\Sigma_\tau)} \, d\tau
\end{split}
\label{E:Energy_estimates_2_plus_epsilon_part_4}
\end{align}
Combining 
\eqref{E:Energy_estimates_2_plus_epsilon_part_1},
\eqref{E:Energy_estimates_2_plus_epsilon_part_2},
\eqref{E:Energy_estimates_2_plus_epsilon_part_3},
and
\eqref{E:Energy_estimates_2_plus_epsilon_part_4} 
and Gr\"onwalling gives
\begin{align}
\norm{\partial\bar{\partial}^{1+\varepsilon} \Psi}_{L^2(\Sigma_t)}
+
\norm{\bar{\partial}^{2+\varepsilon} \Omega}_{L^2(\Sigma_t)}
\lesssim  
\norm{\partial\bar{\partial}^{1+\varepsilon} \Psi}_{L^2(\Sigma_0)}
+ 
\norm{\bar{\partial}^{2+\varepsilon} \Omega}_{L^2(\Sigma_0)}.
\label{E:Energy_estimates_2_plus_epsilon_part_5} 
\end{align}

Thus, we obtain the desired control of $\bar{\partial}^{2+\varepsilon} \Psi$ in $L^2$ with the additional control
of $\bar{\partial}^{2+\varepsilon} \Omega$ in $L^2$
\emph{provided} 
the data terms on the RHS of \eqref{E:Energy_estimates_2_plus_epsilon_part_5} are indeed finite. 
This follows from our assumptions 
$\left. \Psi \right|_{t=0} \in H^{2+\varepsilon}$
and
$\left. \curl u \right|_{t=0} \in H^{2+\varepsilon}$.
In particular, \emph{this explains our extra regularity 
assumption on} $\left. \curl u \right|_{t=0}$.

As mentioned in Notation \ref{N:L_infty_terms},
terms in $L^\infty$ have been subsumed into $\lesssim$.
But such terms are not a priori bounded, so, when we Gronwall, we get in the usual fashion an exponential term
containing the $L^\infty$ norms. Thus a more precise
form of \eqref{E:Energy_estimates_2_plus_epsilon_part_5} is
\begin{align}
\norm{\partial\bar{\partial}^{1+\varepsilon} \Psi}_{L^2(\Sigma_t)}
+
\norm{\bar{\partial}^{2+\varepsilon}  \Omega}_{L^2(\Sigma_t)}
\lesssim  
\text{data} \, \times \, 
e^{ \int_0^t \norm{\partial \Psi}_{L^\infty(\Sigma_\tau)}
+
\norm{\bar{\partial} \Omega}_{L^\infty(\Sigma_\tau)} \,
d\tau}.
\label{E:Energy_estimates_2_plus_epsilon_part_6} 
\end{align}
The term $\norm{\partial \Psi}_{L^\infty(\Sigma_\tau)}$ in \eqref{E:Energy_estimates_2_plus_epsilon_part_6}  comes from commuting $\square_G$ with derivatives, whereas
$\norm{\bar{\partial} \Omega}_{L^\infty(\Sigma_\tau)}$ comes from $\bar{\partial} \Omega$ in \eqref{E:New_formulation_classical_transport_curl_Omega}.
The focus now becomes bounding the time integral of these terms.

\subsection{Key element: control of mixed spacetime norms and bootstrap assumptions\label{S:Control_spacetime_norms}}
In view of \eqref{E:Energy_estimates_2_plus_epsilon_part_6},
we can close estimates for $\Psi$ and $\Omega$ in $H^{2+\varepsilon}$ if we can bound
\begin{align}
\norm{\partial \Psi}_{L^1_tL^\infty_x} :=
\int_0^t \norm{ \partial \Psi }_{L^\infty(\Sigma_\tau)} \,
d\tau
\nonumber
\end{align}
and
\begin{align}
\norm{\bar{\partial} \Omega}_{L^1_tL^\infty_x} :=
\int_0^t \norm{ \bar{\partial}\Omega }_{L^\infty(\Sigma_\tau)} \,
d\tau
\nonumber
\end{align}
in terms of the initial data.

If we had that $2+\varepsilon > \frac{3}{2}+1$ then we would immediately obtain the desired bound from Sobolev embedding\footnote{We recall that in the case $2+\varepsilon > \frac{3}{2}+1$ we can use a first-order formulation of the equations and no additional assumptions on the vorticity are needed.}. In the irrotational and
isentropic case we can control $\norm{\partial \Psi}_{L^1_tL^\infty_x}$ using Strichartz estimates and no control of $\Omega$ is needed (recall the introductory
discussion at the beginning of Sect.~\ref{S:Rough_solutions}).

In order to control $\norm{\partial \Psi}_{L^1_tL^\infty_x}$, the goal will be to use Strichartz estimates suitably
adapted to the presence of vorticity, since
Strichartz estimates are designed for controlling mixed spacetime norms for wave-like systems and are
based on dispersion (recall that $\Psi$ is a wave variable).
 For $\norm{\bar{\partial} \Omega}_{L^1_tL^\infty_x}$, there are no Strichartz estimates (no dispersion for transport equations; recall that $\Omega$ is a transport variable). Since $\Omega$ satisfies a div-curl-transport system, we would like to estimate the spatial norm $\norm{\bar{\partial} \Omega}_{L^\infty_x}$ with help of elliptic estimates. As stated, however, this is not possible because Calder\'on-Zygmund operators are not bounded in $L^\infty$. We can, however, bound $\norm{\bar{\partial} \Omega}_{L^1_tL^\infty_x}$ by the stronger norm $\norm{\bar{\partial} \Omega}_{L^1_t C^{0,\upalpha}_x}$ and for $C^{0,\upalpha}$ elliptic estimates are available. \emph{This explains our H\"older assumptions on the data.}

Using Cauchy--Schwarz in the time integral, it suffices to bound
$\norm{(\partial \Psi, \bar{\partial} \Omega)}_{L^2_t L^\infty_x}$. (In view of the energy formalism, it is more convenient to work in $L^2$ spaces.) This will be obtained by improving, for small time, the following bootstrap assumptions:
\begin{subequations}{\label{E:Bootstrap_assumption}}
\begin{align}
\norm{\partial \Psi }^2_{L^2_t L^\infty_x}
+ \sum_{\upnu \geq 2} \upnu^{2\updelta_0} \norm{ P_\upnu \partial \Psi}^2_{L^2_t L^\infty_x} 
& \leq 1,
\label{E:Bootstrap_assumption_Psi}
\\
\norm{\bar{\partial} \Omega }^2_{L^2_t L^\infty_x}
+ \sum_{\upnu \geq 2} \upnu^{2\updelta_0} \norm{ P_\upnu \bar{\partial} \Omega}^2_{L^2_t L^\infty_x} 
& \leq 1,
\label{E:Bootstrap_assumption_Omega}
\end{align}
\end{subequations}
where $P_\upnu$ is the Littlewood--Paley projection onto dyadic frequencies and $\updelta_0$ is a small number depending on $\varepsilon$.  We can view 
\eqref{E:Bootstrap_assumption_Psi} as a bootstrap assumption on the wave-part and \eqref{E:Bootstrap_assumption_Omega} as a bootstrap assumption on the transport-part.

\begin{remark}
\label{R:Bootstrap_assumptions_energy_estimates}
Only the bootstrap assumptions on 
$\norm{\partial \Psi }_{L^2_t L^\infty_x}$ and 
$\norm{\bar{\partial} \Omega }_{L^2_t L^\infty_x}$
are needed for the energy estimates. The bootstrap assumptions involving the sums in \eqref{E:Bootstrap_assumption} are needed for control of the acoustic geometry in Sect.~\ref{S:Control_acoustic_geometry}.
\end{remark}

The $L_t^2$ norms correspond to the $L^2$ norm on a time-interval $[0,T_{\text{bootstrap}}]$ with
$T_{\text{bootstrap}}$ being a bootstrap time that is eventually taken very small. Our goal is to improve \eqref{E:Bootstrap_assumption} to

\begin{subequations}{\label{E:Bootstrap_assumption_improved}}
\begin{align}
\norm{\partial \Psi }^2_{L^2_t L^\infty_x}
+ \sum_{\upnu \geq 2} \upnu^{2\updelta_0} \norm{ P_\upnu \partial \Psi}^2_{L^2_t L^\infty_x} 
& \lesssim T_{\text{bootstrap}}^\updelta,
\label{E:Bootstrap_assumption_Psi_improved}
\\
\norm{\bar{\partial} \Omega }^2_{L^2_t L^\infty_x}
+ \sum_{\upnu \geq 2} \upnu^{2\updelta_0} \norm{ P_\upnu \bar{\partial} \Omega}^2_{L^2_t L^\infty_x} 
& \lesssim T_{\text{bootstrap}}^\updelta,
\label{E:Bootstrap_assumption_Omega_improved}
\end{align}
\end{subequations}
for some small $\updelta > 0$.

The bootstrap assumptions are improved by establishing
Strichartz estimates for the wave-part of the system complemented by estimates controlling the $L^2_t L^\infty_x$ norm of the transport-part. 
We stress that this distinction and the ensuing sections, where we will discuss the wave and transport parts of the system somewhat separately, should not suggest that the estimates for $\norm{\partial \Psi }_{L^2_t L^\infty_x}$ and 
$\norm{\bar{\partial} \Omega }_{L^2_t L^\infty_x}$ are decoupled. In particular, we need to handle the interactions between the transport and wave parts (see, in particular, Sect.~\ref{S:Control_acoustic_geometry}).

The logic of the complete argument is illustrated in Fig.~\ref{F:Flowchart}, where we highlight some of the new (in comparison to the irrotational case) ideas that are needed. In particular, in the irrotational case, Strichartz estimates require control of the rough geometry of the sound cones, i.e.,
closing the estimates for $\Psi$ requires deriving complementary, highly tensorial estimates for several  quantities associated to the geometry of the sound cones. In our case, this geometry interacts non-trivially with the transport-part, leading to a several intersting new aspects, some of which we discuss next.
We will discuss the steps in Fig.~\ref{F:Flowchart} in a ``constructive'' way, i.e., more or less in reverse order to the logic in Fig.~\ref{F:Flowchart}, starting with what we want to establish and identifying what we need to be true for that to hold.

\subsection{The Strichartz estimate and reductions\label{S:Strichartz_and_reductions}}
In view of the discussion of Sect.~\ref{S:Control_spacetime_norms}, we have to establish the Strichartz estimate
\begin{align}
\norm{\partial \Psi}_{L^t_tL^\infty_x} \lesssim \text{data}.
\label{E:Basic_Strichartz_wave_part}
\end{align}

Through a series of \emph{technical reductions} that involve rescaling, energy estimates, and the use of Duhamel's principle, it is possible to show that \eqref{E:Basic_Strichartz_wave_part} follows from the following 
Frequency-localized Strichartz estimate
\begin{align}
\norm{ P_\uplambda \partial \varphi }_{L^q_t L^\infty_x} \lesssim \uplambda^{\frac{3}{2}-\frac{1}{q}} \norm{\partial \varphi}_{L^2(\Sigma_0)},
\label{E:Frequency_localized_Strichartz}
\end{align}
where $P_\uplambda$ is the Littlewood--Paley projection onto dyadic frequency $\uplambda$, $q$ is a number slightly greater than $2$, and $\varphi$ is a solution to the following \emph{linear} wave equation
\begin{align}
\square_G \varphi = 0.
\label{E:Linear_in_varphi_wave_Duhamel}
\end{align}
(Recall that $\square_G$ is the wave operator for the acoustical metric.)
The fact that we are dealing ultimately with a linear equation comes from Duhamel's principle. Note, however, that the coefficients in \eqref{E:Linear_in_varphi_wave_Duhamel} depend on $G=G(\Psi)$, thus on the rough variables
$\Psi$. It is because of the low regularity of these coefficients that the analysis is delicate even at the linear level.

With a further reduction, estimate \eqref{E:Frequency_localized_Strichartz} can be reduced to the following
fixed-frequency Strichartz estimate,
\begin{align}
\norm{ P \partial \varphi }_{L^q_t L^\infty_x} \lesssim  \norm{\partial \varphi}_{L^2(\Sigma_0)},
\label{E:Fixed_frequency_Strichartz}
\end{align}
where $P$ is the Littlewood--Paley projection onto unit frequencies $\{ \frac{1}{2} \leq |\xi| \leq 2\}$.

Finally, an abstract duality argument known as the $TT^*$ argument, can be used to show that 
\eqref{E:Fixed_frequency_Strichartz} follows from the dispersive estimate\footnote{The $TT^*$ argument involves in particular duality in the $t$ variable. In order to obtain the desired boundedness in this duality, one needs certain integrability in time. Roughly speaking, this is the reason why the duality can be obtained from a dispersive, i.e., decay-in-time, estimate.}
\eqref{E:Dispersive_estimate_L_infty}, or equivalently \eqref{E:Dispersive_estimate_L_2}, below.

We remark that while the series of reductions mentioned above is technical and long, it follows from known steps used in the works on rough solutions to quasilinear wave equations discussed at the beginning of Sect.~\ref{S:Rough_solutions}. For our purposes, these reductions can be taken as a black box. We refer to 
\cite{Disconzi-Luo-Mazzone-Speck-2022} for more details on the reduction procedure. We also observe the following  point:

\begin{remark}
\label{R:Omitting_rescaling}
In estimate \eqref{E:Fixed_frequency_Strichartz}, $\Psi$ in $G(\Psi)$ is not the ``original'' wave-variable $\Psi$ but a suitably rescaled version of it. Suitable rescalings of the transport variables and a subdivision of the time interval $[0,T_\text{bootstrap}]$
are also considered. These rescalings involve the dyadic frequencies $\uplambda$. For simplicity of notation, we will not
relabel the rescaled quantities, continuing to call them $\Psi$, $\Omega$, etc. Nor will we discuss these rescalings, as, while important for the complete argument, would be a distraction for the high-level presentation we are providing. Readers should keep in mind, however, that the remaining of the discussion
is applied to the rescaled variables and not the original ones. Similarly, the estimates we will discuss
are often localized in some specific regions of spacetime due to technical reasons and some furhter reductions (e.g., reductions to compactly supported data or to data given at a different time than $t=0$) are also carried out. Once again, in order to focus on the main ideas of the estimates and avoid technicalities, we will not introduce such localizations and reductions here.
\end{remark}

\subsection{The dispersive estimate\label{S:Dispersive_estimate}} As mentioned above, from a $TT^*$ argument, we have now reduced estimate \eqref{E:Fixed_frequency_Strichartz}
to the following dispersive estimate (see Sect.~\ref{S:Strichartz_and_reductions}):
\begin{align}
\norm{ P \B \varphi }_{L^\infty(\Sigma_t)} \lesssim
\left( \frac{1}{(1+t)^\frac{2}{q}} + \mathsf{d}(t) \right) 
( \norm{\partial \varphi}_{L^2(\Sigma_0)} + \norm{\varphi}_{L^2(\Sigma_0)} ),
\label{E:Dispersive_estimate_L_infty}
\end{align}
where the notation is the same as in Sect.~\ref{S:Strichartz_and_reductions}. Here, $\mathsf{d}$ is a function that satisfies
\begin{align}
\norm{\mathsf{d}}_{L^\frac{q}{2}_t} \lesssim 1,
\nonumber
\end{align}
i.e., it has the same integrability as $(1+t)^{-\frac{2}{q}}$. The term $\mathsf{d}$ is ``quasilinear in nature,'' i.e., even though we seek an estimate for the solution to the linear wave equation \eqref{E:Linear_in_varphi_wave_Duhamel}, the coefficients of $\square_G$ depend on the solution $G=G(\Psi)$, and hence need to be suitably controlled. This control leads to $\mathsf{d}$.

We have reduced estimate \eqref{E:Fixed_frequency_Strichartz} to \eqref{E:Dispersive_estimate_L_infty}. Note that \eqref{E:Dispersive_estimate_L_infty} was an estimate for $\partial \varphi$ whereas \eqref{E:Dispersive_estimate_L_infty} is an estimate for $\B \varphi$. This is because in the duality $TT^*$ argument referred to in Sect.~\ref{S:Strichartz_and_reductions}, spatial derivatives can be handled with an integration by parts argument. We are then left with \emph{a} time direction to handle. The argument is geometric in nature so that the time direction we need to consider is the true normal, with respect to $G$, to the constant-in-time surfaces $\Sigma_t$, which is precisely $\B$.

We finally note one further reduction. Since we now want an estimate at unity frequency, we can, via Bernstein's inequality, replace $\norm{P\B \varphi}_{L^\infty(\Sigma_t)}$ by $\norm{P\B \varphi}_{L^2(\Sigma_t)}$, so that
instead of \eqref{E:Dispersive_estimate_L_infty} we seek to establish:

\begin{align}
\norm{ P \B \varphi }_{L^2(\Sigma_t)} \lesssim
\left( \frac{1}{(1+t)^\frac{2}{q}} + \mathsf{d}(t) \right) 
( \norm{\partial \varphi}_{L^2(\Sigma_0)} + \norm{\varphi}_{L^2(\Sigma_0)} ),
\label{E:Dispersive_estimate_L_2}
\end{align}
The use of $L^2$ instead of $L^\infty$ is convenient as it allows us to rely on energy-type estimates for wave
equations.

\subsection{A detour with the Minkowski metric and a prelude of the acoustic geometry\label{S:Detour_Minkowski}}

Before discussing the strategy for proving \eqref{E:Dispersive_estimate_L_2}, we will take a brief detour to discuss
how an estimate like \eqref{E:Dispersive_estimate_L_2} could be proven in the case of the flat linear wave equation.
This will help newcomers to the filed get a taste for the need of geometric methods discussed in 
Sects.~\ref{S:Decay_acoustic_geometry} and \ref{S:Control_acoustic_geometry}. We expect that many readers
will be acquainted with the ideas presented in this section and such readers can probably skip it. 
But given the importance of the acoustic geometry
for the proof of \eqref{E:Dispersive_estimate_L_2}, we want to make sure all readers understand how it comes about\footnote{See the introduction of \cite{Disconzi-Luo-Mazzone-Speck-arxiv-v1}, which is an earlier arXiv version of \cite{Disconzi-Luo-Mazzone-Speck-2022}, for further discussion.}.

Decay-in-time estimates such as \eqref{E:Dispersive_estimate_L_2} are typically obtained by controlling
a suitably weighted energy.  To explain how one derives such weighted energy estimates, let us 
illustrate the argument in a simplified setting, namely
\begin{subequations}{\label{E:Standard_wave_sketch_argument}}
\begin{align}
\square_\mathsf{m} \varphi 
& = 0,
\label{E:Standard_wave_sketch_argument_eq}
\\
\left(\varphi,\partial_t \varphi \right)|_{\Sigma_0}
& =
\left(\mathring{\varphi},\mathring{\varphi}_1 \right),
\label{E:Standard_wave_sketch_argument_IC}
\end{align}
\end{subequations}
where $\mathring{\varphi},\mathring{\varphi}_1$ are given functions on $\mathbb{R}^3$
and $\square_\mathsf{m} := - \partial_t^2 + \Delta$ is the standard wave operator of the Minkowski 
metric $\mathsf{m}$ in $\mathbb{R}^{1+3}$, given, relative to standard coordinates,
by $\mathsf{m} = \operatorname{diag}(-1,1,1,1)$.

Given a vectorfield
$X := X^0 \partial_t + X^a \partial_a$, known as a multiplier,
we multiply
\eqref{E:Standard_wave_sketch_argument_eq}
by $X\varphi$ to obtain the identity
\begin{align}
\label{E:Multiplier_identity_sketch_argument}
(\square_\mathsf{m} \varphi) X \varphi  
& = \partial_t e_{X}^{\text{(Time)}} + \partial_a e_{X}^{\text{(Space)},a} + Q_X,
\end{align}
where 
\begin{align}
\nonumber
e_{X}^{\text{(Time)}} 
&:=
 -\frac{1}{2}  X^0 ( (\partial_t \varphi)^2 + |\bar{\nabla}^{\text{(Flat)}} \varphi|^2 )
 -
 (\partial_t \varphi) X^a \partial_a \varphi,
 \\
\nonumber
e_{X}^{\text{(Space)},i} 
& :=
\frac{1}{2}   X^i ( (\partial_t \varphi)^2 - |\bar{\nabla}^{\text{(Flat)}} \varphi|^2 )
+ \updelta^{ia} (\partial_a \varphi) X^b \partial_b \varphi 
+  X^0 (\partial_t \varphi ) \updelta^{ia} \partial_a \varphi,
\end{align}
and
\begin{align}
\nonumber
Q_X & := \frac{1}{2}( \partial_t X^0 - \partial_a X^a ) (\partial_t \varphi)^2 
+ 
\updelta^{ac} (-\partial_a X^0 + \partial_t X^b \delta_{ab} ) (\partial_c \varphi) \partial_t \varphi 
\\
\nonumber
& \ \
+ 
\frac{1}{2}(\partial_t X^0 + \partial_a X^a ) |\bar{\nabla}^{\text{(Flat)}} \varphi|^2 
- 
\updelta^{ac} (\partial_a X^b) (\partial_c \varphi) \partial_b \varphi.
\end{align}

Above, $\bar{\nabla}^{\text{(Flat)}}$ is the gradient with respect
to the standard spatial coordinates
and we recall that $\updelta$ is the Kronecker delta.
Thus, $(\square_\mathsf{m} \varphi) X \varphi$ can be written as a sum of 
two types of terms: perfect derivatives
of quadratic expressions in derivatives of $\varphi$ with coefficients 
depending on $X$ (but not on derivatives of $X$), 
i.e., $\partial_t e_{X}^{\text{(Time)}} + \partial_i e_{X}^{\text{(Space)},i}$;
and a quadratic form in derivatives of $\varphi$ with coefficients
depending on first derivatives of $X$, i.e., $Q_X$. 
Integrating \eqref{E:Multiplier_identity_sketch_argument} over
a spacetime region $D$, using that the LHS of  \eqref{E:Multiplier_identity_sketch_argument} vanishes
 for solutions to \eqref{E:Standard_wave_sketch_argument},
and integrating by parts, one can convert the perfect-derivative terms on RHS~\eqref{E:Multiplier_identity_sketch_argument}
into boundary terms that can be interpreted as suitable ``energies"
for the solution.\footnote{One can show that if the boundary of $D$ is causal with respect to $\mathsf{m}$ and $X$ is timelike
with respect to $\mathsf{m}$, then the boundary integrals are coercive, allowing us to interpret them as ``energies.''} 
For example, choosing $X :=\partial_t$ 
and $D :=[0,t]\times \mathbb{R}^3$ produces the standard conserved energy
$E(t):= \int_{\mathbb{R}^3}( (\partial_t\varphi(t,x))^2 + |\nabla^{\text{(Flat)}} \varphi(t,x)|^2 ) \, dx$ 
for solutions to \eqref{E:Standard_wave_sketch_argument}. 

We now turn to the issue of deriving $L^2$-type decay estimate for solutions to the standard linear wave equation \eqref{E:Standard_wave_sketch_argument}.
More precisely, we seek to derive suitable bounds for a weighted energy. 
In order to construct suitable weighted energies\footnote{We will be more specific about the form of such energies in Sect.~\ref{S:Conformal_energy}, when considering the case of Eq.~\eqref{E:Linear_in_varphi_wave_Duhamel} which we are investigating.},
one can\footnote{There are different
ways of obtaining decay for solutions to \eqref{E:Standard_wave_sketch_argument}. Here,
we only mention the method whose generalization to the quasilinear setting we will be interested.} 
combine the standard energy $E(t)$ from the previous paragraph 
with other coercive integrals obtained via ``weighted'' multipliers
of type $X = f(r) \partial_r$ and $X= r^m L^{\text{(Flat)}}$,
supplemented by suitable lower-order terms.
Here, $r$ is the standard radial coordinate 
in $\mathbb{R}^3$ (so that $\partial_r$ is the Euclidean-unit outer normal to the Euclidean spheres $\{r = \, \text{constant} \}$), 
$L^{\text{(Flat)}} := \partial_t + \partial_r$ is outgoing and null with respect to the Minkowski metric and is tangent 
to the Minkowski lightcones $\{ t = r + \text{constant} \}$,
$f(r) > 0$ and $m \geq 0$ have to be suitably chosen.
We clarify that the multiplier $f(r) \partial_r$ allows one to control a spacetime integral that
plays a role in bounding the weighted energy in an ``interior'' region (say, $|r|\leq 1$), 
while $r^m L^{\text{(Flat)}}$ yields coercive hypersurface integrals that allow one to control the weighted
energy in an ``exterior'' region (say, $|r| \geq 1$).
The idea of using multipliers of type $f(r) \partial_r$ to generate coercive spacetime integrals goes back to \cite{Morawetz-1968},
while the idea of combining Morawetz multipliers with multipliers of type 
$r^m L^{\text{(Flat)}}$ to prove various types of decay estimates for wave equations 
originated in \cite{Dafermos-Rodnianski-2010}, in what is now known as the $r^p$-method of Dafermos and Rodnianski.

We would like to adapt the above procedure to quasilinear problems, where in particular the Minkowski metric $\mathsf{m}$ is replaced by the acoustical metric $G$. The vectorfield $L^\text{(Flat)}$
is tangent and orthogonal to the lightcones, which are null hypersurfaces for the Minkowski metric.
The corresponding analogue for the acoustical metric are the sound cones. The lightcones
are level sets of the function $\mathcal{U}^{\text{(Flat)}} := t-r$. In the quasilinear setting, we will define the analogue of $r$ by $\tilde{r} := t - \mathcal{U}$, where $\mathcal{U}$ is an eikonal function
whose level sets are sound cones. Similarly, the analogue of $L^\text{(Flat)}$ will be a $G$-null vectorfield
$L$ tangent to the sound cones. Vectors $N$ that are tangent to $\Sigma_t$ and $G$-orthogonal to the (topological)
spheres $\tilde{r} = \text{constant}$ will be the analogue of $\partial_r$ in the flat case. 

There will be, however, a crucial difference. While in the Minkowski case the vectorfields $L$ and 
$\partial_r$ are given, smooth, and independent of solutions $\varphi$ to 
\eqref{E:Standard_wave_sketch_argument}, in the quasilinear setting, $\mathcal{U}$ has to be constructed out of $G$. Thus, it will depend on the solution $\Psi$ since $G = G(\Psi)$, and so will $L$ and $N$. When we carry an integration by parts similar to the one discussed for the flat case, derivatives of $N$ and $L$ will appear. 
At first sight, one can expect to control derivatives of $N$ and $L$ directly from $\Psi$, since $N$ and $L$ are in the end determined by $\Psi$. With current techniques, however, this direct control does not seem to be possible: our goal is to obtain decay for $\Psi$ by controlling a suitable weighted energy, but the 
control one obtains for $N$ and $L$ by simply using the equation satisfied by $\Psi$ is too weak for that\footnote{Indeed, one should not expect control of $N$ and $L$ ``for free.'' If bounds 
for $N$ and $L$ compatible with the desired decay could be obtained by directly estimating $\Psi$, then presumably we could prove decay for $\Psi$ independently, i.e., without having to introduce $N$ and $L$ in the first place.}.

In the end, good bounds for derivatives of $N$ and $L$ are obtained by exploiting some delicate geometric structures present in the system. The integration by parts is not carried with respect to standard coordinate derivatives but rather with respect to certain geometric vectorfields. One thus obtain geometric derivatives of the vectorfields $N$ and $L$ which are themselves geometric quantities (being orthogonal to spheres and sound cones, respectively). Thus, \emph{the derivatives we need to control have geometric meaning and can be tied to certain geometric objects associated with the geometry of the sound cones, i.e., the acoustic geometry. This suggests the use of geometric techniques and, in a nutshell, is the reason why the acoustic geometry
plays a prominent role in the proof of \eqref{E:Dispersive_estimate_L_2}.} In more specific terms,
derivatives of $N$ and $L$ can be decomposed in terms of some objects tied to the geometry of the sound cones;
these are the so-called connection coefficients of the acoustic geometry. These connection coefficients
satisfy a delicate evolution-elliptic system of PDEs, the so-called null-structure equations. By exploiting some highly delicate and tensorial properties of the null-structure equations, one can obtain non-trivial control of the connection coefficients, leading to the desired bounds for $N$ and $L$ and then to the sought-after control of a weighted energy which, in turn, gives the bound \eqref{E:Dispersive_estimate_L_2}.

\subsection{Decay properties of the wave-part and the acoustic geometry\label{S:Decay_acoustic_geometry}}
We have reduced the desired Strichartz estimate for the wave-part, estimate \eqref{E:Basic_Strichartz_wave_part}, to a decay estimate for solutions to \eqref{E:Linear_in_varphi_wave_Duhamel}, namely, estimate \eqref{E:Dispersive_estimate_L_2}.

While \eqref{E:Dispersive_estimate_L_2} is an estimate for the linear wave equation \eqref{E:Linear_in_varphi_wave_Duhamel}, we recall from Sect.~\ref{S:Dispersive_estimate} that 
the quasilinear nature of the problem is hidden in the acoustical metric $G$. A key feature of solutions to quasilinear wave equations is that their regularity and decay properties are
directional dependent, with derivatives in directions of the characteristics behaving better than 
derivatives in directions transverse to the characteristics\footnote{\label{FN:Better_behavior_flat_wave}This can be illustrated with the standard linear wave equation \eqref{E:Standard_wave_sketch_argument_eq} 
as follows. Using the fundamental solution, 
one can show that if the initial data is supported on $\{r \leq 2\}$, then the solution is supported 
in the region $\{ |t-r| \leq 2 \}$ (i.e., near the forward Minkowski lightcone $\{ t = r \}$ emanating from the origin)
and can be written as $r^{-1} H(r-t, \upomega^{\text{(Round)}}, r^{-1})$ for some function $H$ (which is $C^{\infty}$ if the data are), 
where $\upomega^{\text{(Round)}} = r^{-1} x \in \mathbb{S}^2$. 
From direct computation, we see that 
$L^{\text{(Flat)}} H$ and $\slashed{\partial}^{\text{(round)}} H$ 
(where $\slashed{\partial}^{\text{(Round)}}$ is the gradient with respect to round metric on the Euclidean spheres $\{r = \text{ constant } \}$), 
which are derivatives tangent to the Minkowski lightcones, 
are of order $r^{-2}$, whereas the transverse derivative $(\partial_t - \partial_r) H$
is of order $r^{-1}$. Since $t \sim r$ in the region $\{ |t-r| \leq 2 \}$, 
we conclude that the tangential derivatives decay better with respect to $t$ than transverse derivatives.}.
Indeed, to prove \eqref{E:Dispersive_estimate_L_2}, we need to exploit that derivatives of solutions to \eqref{E:Dispersive_estimate_L_2}
in directions tangent to the characteristics of the acoustical metric decay faster than derivatives transverse to it.

In order to distinguish between directions tangent and transverse to the sound cones, we introduce an \textdef{eikonal function}\footnote{The eikonal function, the construction of an associated null-frame, and some of the related discussion in this section are the same as in Sect.~\ref{S:Ingredient_one}, but we repeat them here for convenience and also because here we will be more specific about the null-frame.} $\mathcal{U}$, which is a solution to the \textdef{eikonal equation}
\begin{align}
(G^{-1})^{\alpha\beta} \partial_\alpha \mathcal{U} \partial_\beta \mathcal{U} = 0,
\label{E:Eikonal_equation_rough}
\end{align}
with appropriate initial conditions. The level sets of $\mathcal{U}$ are the characteristics associated with acoustical the metric $G$, i.e., the sound cones. In this regard, we note $\mathcal{U}$ is tied to the wave-part of the Euler system, and not to its transport-part. 
(We recall that we are presenting the discussion of estimates for the wave and transport parts separately for
convenience, but that the problem is coupled and both parts interact non-trivially.)
Next, to identify directions transverse and tangent to the sound cones, we introduce a \textdef{null-frame}
adapted to $\mathcal{U}$, as follows. Denote the level sets\footnote{\label{FN:Abuse_notation_sound_cone_rough}We use a harmless abuse of notation to simplify the presentation. Instead of writing $\SoundCone_{\mathcal{U}_0} := \{ \mathcal{U} = \mathcal{U}_0 \}$
for a given $\mathcal{U}_0$, we simply indicate that $\SoundCone_\mathcal{U}$ is a level set of $\mathcal{U}$ by  $\SoundCone_\mathcal{U} := \{ \mathcal{U} = \text{constant} \}$.} of $\mathcal{U}$ by $\SoundCone_\mathcal{U}$, i.e., 
\begin{align}
\SoundCone_\mathcal{U} := \{ \mathcal{U} = \text{constant} \}.
\nonumber
\end{align}
We set up the initial conditions for \eqref{E:Eikonal_equation_rough} in such a way that the sound cones
$\SoundCone_\mathcal{U}$ have their tip at a fixed\footnote{By a localization argument, it suffices to consider the construction for a generic, fixed, flow line of $\B$, see \cite{Disconzi-Luo-Mazzone-Speck-2022}.} flow line of $\B$, see Fig.~\ref{F:Sound_cone_and_null_frame_rough}. Consider the (topological) spheres\footnote{We again use an abuse of notation, see Footnote \ref{FN:Abuse_notation_sound_cone_rough}.}
\begin{align}
\begin{split}
S_{t,\mathcal{U}} & :=
\{ t = \, \text{constant} \} \cap \{ \mathcal{U} = \, \text{constant} \}
\\
& \, \, = \Sigma_t \cap \SoundCone_\mathcal{U}.
\end{split}
\nonumber
\end{align}
Let $N$ be the unit outer normal (with respect to $G$) to $S_{t,\mathcal{U}}$ along $\Sigma_t$ (see Fig.~\ref{F:Sound_cone_and_null_frame_rough}).
Set $L := \B + N$, $\underline{L} := \B-N$, and let $\{e_1,e_2\}$ be an orthonormal (with respect to $G$) frame on $S_{t,\mathcal{U}}$. The set
\begin{align}
\{ e_1, e_2, \underline{L}, L \}
\label{E:Null_frame_rough}
\end{align}
is the desired null-frame, which forms a basis at each point. We immediately check that $\{e_1,e_2, L\}$ are tangent to $\SoundCone_\mathcal{U}$,
$\underline{L}$ is transverse to $\SoundCone_\mathcal{U}$, $L$ and $\underline{L}$ are null (with respect to $G$), i.e., $G(L,L)=0=G(\underline{L},\underline{L})$, that $G(L,e_A) = 0 = G(\underline{L},e_A)$, $A=1,2$, and
$G(L,\underline{L})=-2$.

\begin{figure}[ht]
\centering
  \includegraphics[scale=0.4]{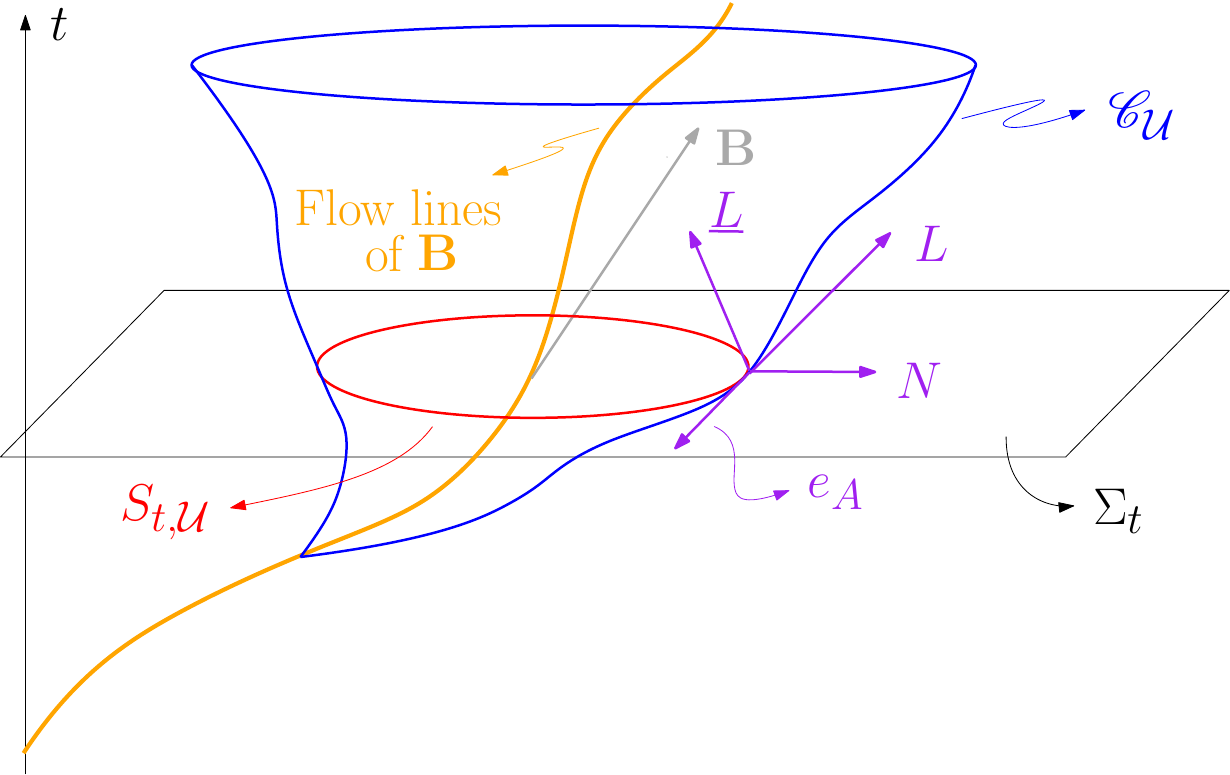}
  \caption{Illustration of a sound cone and corresponding geometric constructions (see the text for definitions).}
  \label{F:Sound_cone_and_null_frame_rough}
\end{figure}

In order to establish the desire decay \eqref{E:Dispersive_estimate_L_2}, the strategy is to construct a suitable \emph{weighted} energy and control its growth using certain multipliers with suitable vectorfields. 
The precise form of the energy is stated in Sect.~\ref{S:Conformal_energy}. For now, in order to follow the logic of the argument, it suffices to know that (a) the relevant energy is going to be weighted because are interested in establishing decay, and (b) bounding the weighted energy involves, as usual, an integration by parts.

It turns out that we need to use two different vectorfields, one whose weights are good in an ``interior'' region but become weak in the ``exterior'' part, and one whose weights behave the opposite way\footnote{See the discussion of Sect.~\ref{S:Decay_acoustic_geometry}.}. For the interior region, we take $f(\tilde{r}) N$ for suitable $f$, and for the exterior region we use $\tilde{r}^m L$ for suitable\footnote{There are lower-order terms added to $f(\tilde{r}) N$ and $\tilde{r}^m L$ which we omit because they are not important for our high-level discussion.} $m$. Here,
\begin{align}
\tilde{r} := t - \mathcal{U}
\label{E:r_tilde}
\end{align}
should be thought of as the quasilinear analogue of the radial coordinate\footnote{Note that for the flat eikonal function $\mathcal{U}^{\text{(Flat)}}=t-r$, $\tilde{r} = r$.} in $\mathbb{R}^3$. The interior estimates, using $f(\tilde{r}) N$, is like a Morawetz estimate (see \citealt{Morawetz-1968}), adapted to the acoustic geometry and produces integrated energy-decay estimates, whereas the exterior estimate using 
$\tilde{r}^m L$ is related to the so-called $r^p$-method of \cite{Dafermos-Rodnianski-2010}.

After testing \eqref{E:Linear_in_varphi_wave_Duhamel} with the multipliers $f(\tilde{r}) N$ and $\tilde{r}^m L$
and integrating by parts, we are left with error terms that involve $\partial N$ and $\partial L$. 
Because $N$ and $L$ are constructed out of $\mathcal{U}$ and the latter depends on $G= G(\Psi)$ by \eqref{E:Eikonal_equation_rough}, we see that controlling these error terms is not independent of controlling the fluid variables which we ultimately want to bound. Said differently, while Eq.~\eqref{E:Linear_in_varphi_wave_Duhamel} is linear in $\varphi$, the coefficients of the equation depend on $\Psi$. The quasilinear nature of the problem is still with us, even after all the reductions that lead to estimates for a linear equation, as mentioned in Sect.~\ref{S:Dispersive_estimate}.

The weighted energy that will be controlled is called a \emph{conformal energy} because, for reasons discussed in Sect.~\ref{S:Conformal_energy}, in the end one considers not \eqref{E:Linear_in_varphi_wave_Duhamel} but instead
the equation
\begin{align}
\square_{\tilde{G}} \varphi = 0,
\nonumber
\end{align}
where $\tilde{G}$ is a metric conformal to $G$.

\subsection{Control of the acoustic geometry\label{S:Control_acoustic_geometry}} 
From the discussion of Sect.~\ref{S:Decay_acoustic_geometry}, in order to establish the decay
estimate \eqref{E:Dispersive_estimate_L_2}, it is crucial to control the error terms in $\partial N$ and $\partial L$ that will appear in estimates for the conformal energy (see Sect.~\ref{S:Conformal_energy}). Tracking the dependence
of $N$ and $L$ on $\mathcal{U}$, we see that $N,L \sim \partial \mathcal{U}$. Since the error terms mentioned in Sect.~\ref{S:Decay_acoustic_geometry} involve $\partial (N,L)$, we see that we need to control $\partial^2 \mathcal{U}$. In order to see how this is not straightforward, let us do a derivative counting. From 
\eqref{E:Eikonal_equation_rough}, $\partial^2 \mathcal{U}$ satisfies
\begin{align}
(G^{-1})^{\alpha\beta} \partial_\alpha \mathcal{U} \partial_\beta (\partial^2 \mathcal{U} )
+ \partial^2 \mathcal{U} \cdot \partial^2 \mathcal{U} \sim \partial^2 \Psi \partial \mathcal{U},
\label{E:Transport_partial_2_eikonal_schematic}
\end{align}
where we used that $\partial^2 G \sim \partial^2 \Psi$ since $G = G(\Psi)$. Equation \eqref{E:Transport_partial_2_eikonal_schematic} is a transport equation for $\partial^2 \mathcal{U}$ (note
that \eqref{E:Eikonal_equation_rough} is a fully nonlinear transport equation for $\mathcal{U}$; see also 
\eqref{E:Transport_partial_2_eikonal} below). For the derivation of standard $L^2$-type transport estimates, 
we seek to control $\norm{\partial^2 \mathcal{U}}^2_{L^2(\Sigma_t)}$ by multiplying \eqref{E:Transport_partial_2_eikonal_schematic} by $\partial^2 \mathcal{U}$ 
and integrating by parts. This, however, produces the a term cubic in $\partial^2\mathcal{U}$ which cannot be controlled by the natural energy $\norm{\partial^2 \mathcal{U}}^2_{L^2(\Sigma_t)}$, since the cubic structure
forces us to put one factor of $\partial^2 \mathcal{U}$ in $L^\infty$, and  
$\norm{\partial^2 \mathcal{U}}_{L^\infty(\Sigma_t)}$ is not controlled by 
$\norm{\partial^2 \mathcal{U}}_{L^2(\Sigma_t)}$ in any direct way (e.g., by Sobolev embedding). Indeed,
because \eqref{E:Eikonal_equation_rough} is a (fully nonlinear) transport equation for $\mathcal{U}$, transport estimates will give that $\mathcal{U}$ is only as regular as the coefficients $G=G(\Psi)$, so morally
$\partial^2 \mathcal{U} \sim \partial^2 \Psi$, and we do not
have control over $\partial^2 \Psi \in L^\infty(\Sigma_t)$ since we only have $\Psi \in H^{2+\varepsilon}(\Sigma_t)$.

The above derivative counting can be made more precise by deriving the exact equation satisfied by $\partial^2\mathcal{U}$. If we write $\vec{\mathcal{U}}$ for the vector of second-order derivatives 
of $\mathcal{U}$,  $\vec{\mathcal{U}} = (\partial_\alpha\partial_\beta \mathcal{U})_{\alpha,\beta=0,\dots,3}$, then $\vec{\mathcal{U}}$ satisfies an evolution equation of the form
\begin{align}
L \, \vec{\mathcal{U}} + \vec{\mathcal{U}} \cdot \vec{\mathcal{U}} \sim \operatorname{Riem}(G)(\cdot,L,\cdot,L),
\label{E:Transport_partial_2_eikonal}
\end{align}
where $\operatorname{Riem}(G)(\cdot,L,\cdot,L)$ is the Riemann curvature of the acoustical metric $G$ contracted with $L$ in the second and fourth components (the free indices in $\operatorname{Riem}_G(\cdot,L,\cdot,L)$ correspond to the entries of the vector $\vec{\mathcal{U}}$, which are indexed by $\alpha,\beta=0,\dots,3$).
We refer to Chapter~7 of \cite{Alinhac-Book-2010} for the precise form of \eqref{E:Transport_partial_2_eikonal} as well as its derivation. For our purposes, \eqref{E:Transport_partial_2_eikonal} reveals the following important points:
\begin{itemize}
\item $\partial^2 \mathcal{U}$ indeed satisfies a transport equation (along the integral curves of $L$). 
\item Equation \eqref{E:Transport_partial_2_eikonal} reflects the derivative counting in \eqref{E:Transport_partial_2_eikonal_schematic} since $\operatorname{Riem}_G \sim \partial^2 G \sim \partial^2 \Psi$.
\item Equation \eqref{E:Transport_partial_2_eikonal} makes abundantly clear how the geometry of the acoustical
metric, i.e., the acoustic geometry, plays a role in our problem, since it is precisely the Riemann curvature of $G$ that sources the evolution equation for $\partial^2 \mathcal{U}$.
\item Not all components of the Riemann curvature of $G$ enter as a source in the evolution for $\partial^2 \mathcal{U}$, but only those that have two $L$ components. This hints at the idea that there are special directions in the problem, which is consistent with the aforementioned discussion that the properties of solutions to quasilinear wave equations are directional dependent. 
\end{itemize}

While Eq.~\eqref{E:Transport_partial_2_eikonal} is very revealing from a conceptual point of view, explicitly showing links between the behavior of $\mathcal{U}$, $\Psi$, and the acoustic geometry, it is does not present any good structure that would allow us to control $\partial^2 \mathcal{U}$. 
The above derivative counting shows the difficulties in carrying out direct $L^2$-type transport estimates and integration along the integral curves of $L$ will not produce better control. A limitation of Eq.~\eqref{E:Transport_partial_2_eikonal} is that it treats all derivatives of $\mathcal{U}$ at the same level,
whereas we know from the above discussion that properites of solutions to quasilinear wave equations
are directional dependent. Such direction dependence should be reflected in the behavior of $\mathcal{U}$, as 
$\mathcal{U}$ is essentially a geometric representation of the propagation of sound waves.

One can show that derivatives of $\mathcal{U}$ in fact exhibit better behavior in directions tangent to the sound cones, but this is far from straightforward. The desired estimates
are obtained by studying a delicate evolution-elliptic system called the null-structure equations, which goes back to
Christodoulou and Klainerman's proof of stability of Minkowski space (see \citealt{Christodoulou-Klainerman-Book-1993}).
See Sects.~9.8, 9.9 of \cite{Disconzi-Luo-Mazzone-Speck-2022} for the null-structure equations for the sound cones in the case of of the Euler system. Here, we will illustrate the argument using only one of these equations
(see \eqref{E:Raychaudhuri_modified_source} below).

Recall that we need to control error terms in $\partial(N,L) \sim \partial^2 \mathcal{U}$, but we saw that direct estimates for $\partial^2 \mathcal{U}$ are not available. We focus instead on equations satisfied by the error terms $\partial(N,L) $. Anticipating the fact that derivatives in directions tangent to the sound cones behave better, and in order to isolate the bad derivatives (transversal to the sound cones), we decompose 
$\partial(N,L)$ relative to the null-frame \eqref{E:Null_frame_rough}. For example, the derivative of $N$ in the direction of $e_A$ projected (with respect to $G$) onto $e_B$ defines a tensor on $S_{t,\mathcal{U}}$, the second fundamental form of $S_{t,\mathcal{U}}$:
\begin{align}
\uptheta_{AB} := G(\mathbf{D}_{e_A} N, e_B),
\nonumber
\end{align}
where $\mathbf{D}$ is the covariant derivative with respect to $G$. The derivative of $L$ in the direction of $e_A$ projected (with respect to $G$) onto $e_B$ is the \textdef{null second fundamental form} of $S_{t,\mathcal{U}}$:
\begin{align}
\upchi_{AB} := G(\mathbf{D}_{e_A} L, e_B).
\nonumber
\end{align}
Another example of such decompositions is the torsion $\upzeta$ given by
\begin{align}
\upzeta_A := \frac{1}{2} G(\mathbf{D}_{\underline{L}} L, e_A).
\nonumber
\end{align}

Considering all derivatives of $e_1,e_2,\underline{L}, L$ in the directions $e_1,e_2,\underline{L}, L$
decomposed relative the the null-frame \eqref{E:Null_frame_rough} gives a collection of tensors known as 
\textdef{connection coefficients.} Each connection coefficient is 
a special combination of up to second-order derivatives of $\mathcal{U}$ with coefficients depending on up to first-order derivatives of $G$.
Only the null second fundamental form $\upchi_{AB}$ or, more precisely, its trace, the null mean curvature given by \eqref{E:Null_mean_curvature},
will be discussed here, but we introduced $\uptheta_{AB}$ and $\upzeta_A$ as further examples of connection coefficients.

At the most basic level, the null-structure equations are equations satisfied by the connection coefficients. More precisely, at a geometric level, the null-structure equations are simply identities relating the connection coefficients. Such identities are geometric in nature and hold for an arbitrary Lorentzian metric
and corresponding eikonal function\footnote{Note that \eqref{E:Eikonal_equation_rough} can be considered for an arbitrary Lorentzian metric.}. However, in our case, we eventually use the fact that the acoustical metric depends on the fluid variables and Eqs.~\eqref{E:New_formulation_classical} to obtain a system of PDEs relating the conection coefficients to the fluid variables. This is illustrated with the passage of \eqref{E:Raychaudhuri_modified} to \eqref{E:Raychaudhuri_modified_source} below.

The desired bounds for $\partial N$ and $\partial L$ are then derived somewhat indirectly, as we employ
the null-structure equations to obtain estimates for the connection coefficients. Such estimates take into account the fact that the connection coefficients capture the directional behavior of derivatives of $N$ and $L$, distinguishing between derivatives tangent versus transverse to the sound cones. 

A detailed presentation of estimates for the acoustic geometry, i.e., all the relevant estimates for the connection coefficients, is significantly beyond the goal of these notes. Having highlighted how the acoustic geometry plays a role in the problem\footnote{It is worth to take a step back to recall how we got to this understanding. We were not trying to force a geometric approach into the problem. The goal is to derive weighted estimates for $\varphi$ to prove \eqref{E:Dispersive_estimate_L_2}. This requires the use of suitable multipliers; these multipliers need to take into account special directions in the behavior of solutions to quasilinear wave equations. This leads us to error terms in $\partial N$ and $\partial L$ which are of the form $\partial^2 \mathcal{U}$, but direct estimates for $\mathcal{U}$ based on \eqref{E:Eikonal_equation_rough} are not available.  Compare with Sect.~\ref{S:Detour_Minkowski}.}, we will restrict ourselves to 
illustrating one of the key estimates, focusing on the role played by the vorticity, which is what distinguishes Theorem
\ref{T:Rough_classical_Euler} from low-regularity results for wave equations
of the form \eqref{E:Prototype_wave_equation}, of which the irrotational Euler system is a particular case.
For this illustration, we will consider the 
\textdef{null mean curvature,} which is the trace of $\upchi_{AB}$, given by
\begin{align}
\tr_{\slashed{G}} \upchi := \sum_{A=1,2} G( \mathbf{D}_{e_A} L, e_A ).
\label{E:Null_mean_curvature}
\end{align}
where $\slashed{G}$ is the metric induced by $G$ on the spheres $S_{t,\mathcal{U}}$. The idea of introducing 
$\tr_{\slashed{G}} \upchi$ is that we can decompose $\upchi_{AB}$ into its trace and trace-free parts. Bounding $\tr_{\slashed{G}} \upchi$ turns out to be one of the more delicate
parts of the argument, both in \cite{Disconzi-Luo-Mazzone-Speck-2022} as well as in estimates
for systems with a single characteristic speed of the form \eqref{E:Prototype_wave_equation}.

One of the relevant null-structure equations is the equation satisfied by 
$\tr_{\slashed{G}} \upchi$, which is the Raychaudhuri equation
\begin{align}
L \tr_{\slashed{G}} \upchi = -\Ric(G)_{LL} + \dots,
\nonumber
\end{align}
where $\Ric(G)$ is the Ricci curvature
of the metric $G$ and $\Ric(G)_{LL}$ is its $LL$ component relative to the null-frame. 

After a careful decomposition of the Ricci tensor first introduced by Klainerman and Rodnianski in \cite{Klainerman-Rodnianski-2003}, we find
\begin{align}
L( \tr_{\slashed{G}} \upchi + \Gamma_L ) 
= \frac{1}{2} L^\alpha L^\beta (G^{-1})^{\mu\nu} \partial_\mu \partial_\nu G_{\alpha\beta} + \dots,
\label{E:Raychaudhuri_modified}
\end{align}
where $\Gamma_L := L^\alpha \Gamma_\alpha$, with
$\Gamma^\alpha \sim \partial G \sim \partial \Psi$ being a contracted Cartesian Christoffel symbol of $G$.

To explain why we have grouped $\Gamma_L$ with $\tr_{\slashed{G}} \upchi$ on the RHS of \eqref{E:Raychaudhuri_modified}, let us first introduce the following notation:

\begin{notation}
Denote by $\slashed{\partial}$ derivatives tangent to the sound cones\footnote{One usually 
uses $\slashed{\partial}$ to denote derivatives tangent to $S_{t,\mathcal{U}}$. Derivatives tangent
to the sound cones are then $\slashed{\partial}$ and $L$. Here we write
$\slashed{\partial}$ only since the difference will not be important for the points we want to stress.
}.
\end{notation}

For reasons explained in Sect.~\ref{S:Conformal_energy}, we need to bound 
$\slashed{\partial} \tr_{\slashed{G}} \upchi$ along the sound cones, , i.e., we need to control $\norm{\slashed{\partial} \tr_{\slashed{G}} \upchi}_{L^2(\SoundCone_\mathcal{U})}$.
But if we treat $L \Gamma_L$ as a source term in the evolution for $\tr_{\slashed{G}} \upchi$ then, upon differentiation, we would obtain 
\begin{align}
L \slashed{\partial} \tr_{\slashed{G}} \upchi
= \slashed{\partial} L \Gamma_L + \dots
= \partial^3 \Psi + \dots,
\nonumber
\end{align}
producing a term we cannot control. Thus, 
$\tr_{\slashed{G}} \upchi + \Gamma_L$ is the good variable to consider.

Returning to \eqref{E:Raychaudhuri_modified}, observe that on the RHS, we have essentially $\square_G G$. Using that $G = G(\Psi)$, so that $\square_G G \sim \square_G \Psi$, and Eq.~\eqref{E:New_formulation_classical_wave} we find
\begin{align}
L( \tr_{\slashed{G}} \upchi + \Gamma_L ) 
\simeq \curl \Omega.
\label{E:Raychaudhuri_modified_source}
\end{align}

Equation \eqref{E:Raychaudhuri_modified_source} is one of the key equations in the proof of Theorem \ref{T:Rough_classical_Euler}. It shows that in order to control $\tr_{\slashed{G}} \upchi + \Gamma_L$, we need to control
$\curl \Omega$ at a \emph{consistent regularity level.} This is challenging because we want to control $\tr_{\slashed{G}} \upchi + \Gamma_L$ at a low regularity level, thus we need to control $\curl\Omega$ with very few derivatives. But $\curl\Omega$ satisfies a transport-div-curl system for which low-regularity techniques, which are based on dispersion, are not available. In particular, since we  need to control
$\slashed{\partial}(\tr_{\slashed{G}} \upchi + \Gamma_L)$ in $L^2(\SoundCone_\mathcal{U})$, we have to bound $\norm{\slashed{\partial}\curl \Omega}_{L^2(\SoundCone_\mathcal{U})}$. However, because $\Omega$ satisfies a transport equation, we expect that $L^2$-type estimates are available only along $\Sigma_t$, and not along sound cones $\SoundCone_\mathcal{U}$. Indeed, recall that the eikonal function is adapted to the wave-part of the system and not to the transport-part. Thus, while estimates along sound cones for the wave variables are consistent with the dynamics of the wave-part, estimates along sound cones for the trasport variables are not\footnote{To see this from another point of view that illustrates difficulties in dealing with multi-characteristics systems, imagine trying the opposite, i.e., to derive $L^2$ estimates for the wave variables along the characteristics associated with the transport-part, i.e., along the flow lines of $\B$.}. 

The presence of $\curl \Omega$ on the RHS of \eqref{E:Raychaudhuri_modified_source} and the difficulties it causes is an example of the previously mentioned \emph{interaction} between the wave and transport parts, i.e., the transport variables entering as source terms in the estimates for the acoustic geometry.
This is a manifestation of the presence of multiple characteristic speeds. In particular, we stress the following two important points:
\begin{itemize}
\item The problem does not become much easier by assuming
that $\curl \Omega$ has more regularity initially, as we have done 
(through $\left. \curl u \right|_{t=0}$), because this extra regularity needs to be \emph{propagated by the flow.} Since $\Psi$ enters non-trivially as a source for the evolution of $\curl\Omega$ in \eqref{E:New_formulation_classical_transport_curl_Omega},
there is little room for how many derivatives we can take of the evolution for $\curl \Omega$. In other words, even though we assumed extra regularity for $\curl \Omega$ at $t=0$, this regularity needs to be propagated through a rough geometry that interacts with the vorticity.
\item Taking $\curl\Omega$ to be small rather than identically zero does not provide any immediate advantage 
over the case of $\curl\Omega$ large because we would still be faced with the problem of controlling it at a low regularity level. This shows that the case with vorticity is qualitatively different than the case without vorticity (and entropy).
\end{itemize}

Ultimately, control of the acoustic geometry is obtained
by exploiting some very special structures of the new formulation of classical Euler illustrated through system
\eqref{E:New_formulation_classical}. In particular, it is \emph{crucial} that the source term appearing in \eqref{E:New_formulation_classical_wave} is a curl (recall Remark \ref{R:Trading_C_for_curl_Omega} here but note that $\C$ does have a curl structure). \emph{Had a generic derivative of $\Omega$ appeared on the RHS of \eqref{E:New_formulation_classical_wave}, the argument would not work} (see also bullet points in Sect.~\ref{S:Estimates_curl_Omega_along_sound_cones}). In order to give an idea of how 
special features of \eqref{E:New_formulation_classical} are used, and keeping the focus on the role of vorticity, in the next section we discuss how $\slashed{\partial} \curl \Omega$ is controlled along sound cones.

\subsection{Estimate for $\curl\Omega$ along sound cones}
\label{S:Estimates_curl_Omega_along_sound_cones}

Here we explain how to bound $\slashed{\partial} \curl \Omega$ along sound cones. Define
\begin{align}
\mathsf{J} := | \slashed{\partial} \curl \Omega|^2 \B.
\label{E:Current_curl_Omega_sound_cones}
\end{align} 
A computation gives
\begin{align}
\begin{split}
\mathbf{D}_\alpha \mathsf{J}^\alpha 
& = \slashed{\partial} \curl \Omega \, \B \slashed{\partial} \curl \Omega
+ \dots
\\
& \simeq 
\slashed{\partial} \curl \Omega \, \B \slashed{\partial} \curl \Omega,
\end{split}
\label{E:Divergence_current_L2_cones}
\end{align}
where $\mathbf{D}$ is the covariant derivative associated with the acoustical metric.

\begin{figure}[ht]
\centering
  \includegraphics[scale=0.4]{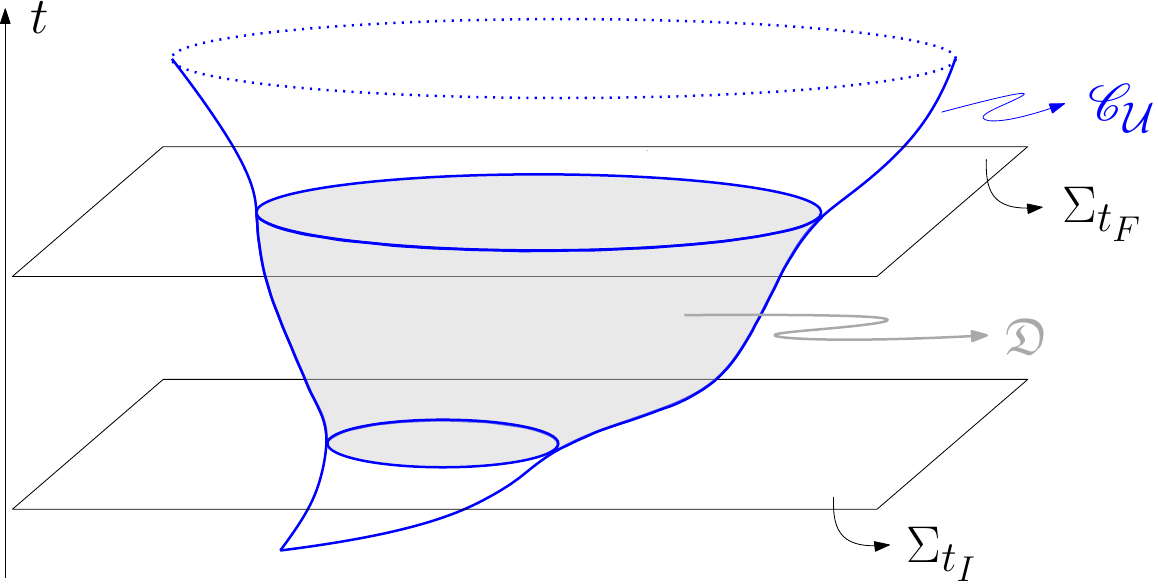}
  \caption{Illustration of the integration region $\mathfrak{D}$ in \eqref{E:Divergence_theorem_L2_cones}.}
  \label{F:Interior_region_sound_cone}
\end{figure}

Consider a spacetime region $\mathfrak{D}$ interior to a sound cone $\SoundCone_\mathcal{U}$ and between
two times\footnote{The times $t_I$ and $t_F$ are chosen in relation to the bootstrap time
$T_{\text{bootstrap}}$.} $t_I$ and $t_F$, as shown in Fig.~\ref{F:Interior_region_sound_cone}. Integrating \eqref{E:Divergence_current_L2_cones} over $\mathfrak{D}$ and invoking the divergence theorem,
\begin{align}
\begin{split}
\int_\mathfrak{D} \slashed{\partial} \curl \Omega \, \B \slashed{\partial} \curl \Omega
&= \int_\mathfrak{D} \mathbf{D}_\alpha \mathsf{J}^\alpha 
\\
& = - \int_{\SoundCone_\mathcal{U}} G(\mathsf{J},\mathsf{V}) + 
\int_{\Sigma_{t_F}} G(\mathsf{J},\mathsf{V}) - \int_{\Sigma_{t_I}} G(\mathsf{J},\mathsf{V})
+ \dots,
\end{split}
\label{E:Divergence_theorem_L2_cones}
\end{align}
where $\mathsf{V}$ is a null vector (with respect to $G$) normal to $\SoundCone_\mathcal{U}$ that allows us to apply the divergence theorem with a null boundary; see Sect.~6.1 of \cite{Disconzi-Luo-Mazzone-Speck-2022} for the construction of $\mathsf{V}$. 
More precisely, $\SoundCone_\mathcal{U}$ is null with respect to $G$ and thus it does not carry a natural volume element induced from $G$, but it does have a natural volume element for which Stokes theorem (which does not depend on the metric) applies. Aside from the integral over the sound cones, all remaining integrals in \eqref{E:Divergence_theorem_L2_cones} are with respect to appropriately geometrically induced, from $G$, volume elements (we are applying the divergence theorem for the $G$-covariant derivative 
$\mathbf{D}$ and thus the volume elements in the integrals are $G$-related). From the construction of $\mathsf{V}$ and $G(\B,\B) = -1$, it follows that $G(\B,\mathsf{V})=-1$ so that
\begin{align}
\begin{split}
G(\mathsf{J},\mathsf{V}) & = | \slashed{\partial} \curl \Omega|^2 G(\B,\mathsf{V})
\\
&=  -| \slashed{\partial} \curl \Omega|^2
\end{split}
\label{E:Product_V_current_L2_cones}
\end{align}
and \eqref{E:Divergence_theorem_L2_cones} gives
\begin{align}
\begin{split}
\int_{\SoundCone_\mathcal{U}} | \slashed{\partial} \curl \Omega|^2
&\lesssim  
\int_\mathfrak{D} | \slashed{\partial} \curl \Omega|  |\B \slashed{\partial} \curl \Omega| 
+ 
\int_{\Sigma_{t_F}} |G(\mathsf{J},\mathsf{V})| + \int_{\Sigma_{t_I}} |G(\mathsf{J},\mathsf{V})|
\\
&
\lesssim
\int_\mathfrak{D} | \slashed{\partial} \curl \Omega|  |\partial^2 \Psi| 
+ \int_{\Sigma_{t_F}}  | \slashed{\partial} \curl \Omega|^2 + \int_{\Sigma_{t_I}}  | \slashed{\partial} \curl \Omega|^2 ,
\end{split}
\label{E:Estimate_curl_Omega_along_cones_intermediate}
\end{align}
where we used \eqref{E:Product_V_current_L2_cones} 
and \eqref{E:New_formulation_classical_transport_curl_Omega}
to obtain 
$\B \slashed{\partial} \curl \Omega \simeq \partial^2 \Psi$. The last two integrals on the RHS are bounded
by the energy estimates of Sect.~\ref{S:Energy_estimates_low_regularity}; see
\eqref{E:Energy_estimates_2_plus_epsilon_part_6} (note Remark \ref{R:Use_energy_estimates_L2_cones} below). For the first integral on the right hand side, we use Fubini's theorem to write
\begin{align}
\int_\mathfrak{D} | \slashed{\partial} \curl \Omega|  |\partial^2 \Psi| 
\lesssim 
\int_{t_I}^{t_F}  \int_{\Sigma_\tau} | \slashed{\partial} \curl \Omega|^2 \, d\tau
+ 
\int_{t_I}^{t_F} \int_{\Sigma_\tau} |\partial^2 \Psi|^2 \,d\tau.
\nonumber  
\end{align}
The integrals over $\Sigma_\tau$, $t_I \leq \tau \leq t_F$, are bounded by \eqref{E:Energy_estimates_2_plus_epsilon_part_6}. Thus, we have obtained the desired control
of $\slashed{\partial} \curl \Omega$ in $L^2(\SoundCone_\mathcal{U})$.

\begin{remark}
\label{R:Use_energy_estimates_L2_cones}
The energy estimate \eqref{E:Energy_estimates_2_plus_epsilon_part_6} involves the 
$L^1_t L^\infty_x$ norms we ultimately want to control (see Sect.~\ref{S:Control_spacetime_norms}).
We recall that the argument is organized as a consistent bootstrap, where the
can use the bootstrap assumption \eqref{E:Bootstrap_assumption}, which provide the desired bound
for the spacetime norms, provided we can improve it to \eqref{E:Bootstrap_assumption_improved}.
\end{remark}

We point out the following important observations:
\begin{itemize}
\item The estimate for $\slashed{\partial} \curl \Omega$ in $L^2(\SoundCone_\mathcal{U})$ relies fundamentally on $G(\mathsf{J},\mathsf{V})=-1$, which is \emph{true because $\B$ is everywhere transverse to 
$\SoundCone_\mathcal{U}$ in view of $G(\B,\B)=-1$.} Absent such transversality, 
$G(\B,\mathsf{V})$, and thus $G(\mathsf{J},\mathsf{V})$, could change sign or be zero, so that
the boundary term $- \int_{\SoundCone_\mathcal{U}} G(\mathsf{J},\mathsf{V}) $ would not correspond
to the square of the $L^2$ norm along $\SoundCone_\mathcal{U}$.
\item Control of the integral over $\mathfrak{D}$ works because $\curl \Omega$ (or, rather, 
$\C$, recall Remark \ref{R:Trading_C_for_curl_Omega}) has \emph{improved regularity
properties} as compared to a generic derivative of $\Omega$. Had a generic derivative $\partial \Omega$ been present instead of $\curl \Omega$, then we would not have been able to use
\eqref{E:New_formulation_classical_transport_curl_Omega} in \eqref{E:Estimate_curl_Omega_along_cones_intermediate}, having to rely on \eqref{E:New_formulation_classical_transport} instead, leading to an uncontrollable $\partial^3 \Psi$ term.
\item In particular, the above highlights that if we had a generic derivative of $\Omega$ as source term in \eqref{E:Raychaudhuri_modified_source}, the argument would not close. Generic derivatives of $\Omega$ cannot be estimated along sound cones because they require Hodge estimates (see Sect.~\ref{S:Energy_estimates_low_regularity}) which cannot be implemented along cones. This is a feature that repeats throughout the estimates for the acoustic geometry: \emph{it is a remarkable fact that the transport variables that appear as source terms for the acoustic geometry appear only in certain special combinations of derivatives for which improved estimates are available.} Had generic derivatives been present, the argument would break down.
\end{itemize}

\subsection{The conformal energy}
\label{S:Conformal_energy}

In Sect.~\ref{S:Decay_acoustic_geometry} we highlighted that to prove the dispersive
estimate \eqref{E:Dispersive_estimate_L_2} we need to control a weighted energy, and that
control of this weighted energy requires control of the acoustic geometry discussed in 
Sect.~\ref{S:Control_acoustic_geometry}. We can now state the energy to be controlled.

We let $\mathcal{W}$ and $\underline{\mathcal{W}}$ be two smooth cut-off functions satisfying 
$0\leq \mathcal{W},\underline{\mathcal{W}}\leq 1$ and such that, for $t>0$,
\begin{align}
\begin{split}
&\mathcal{W}(t,\mathcal{U}) = 
\begin{cases}
1, & \frac{\mathcal{U}}{t} \in [0,\frac{1}{2}]
\\
0, & \frac{\mathcal{U}}{t} \in (-\infty,-\frac{1}{4}] \cup [\frac{3}{4},1],
\end{cases}
\\
&\underline{\mathcal{W}}(t,\mathcal{U}) = 
\begin{cases}
1, & \frac{\mathcal{U}}{t} \in [0,1]
\\
0, & \frac{\mathcal{U}}{t} \in (-\infty,-\frac{1}{4}],
\end{cases}
\\
&\mathcal{W}(t,\mathcal{U}) = \underline{\mathcal{W}}(t,\mathcal{U}) \text{ for } \, t \geq 1 \text{ and }
\frac{\mathcal{U}}{t} \in (-\infty,-\frac{1}{4}].
\end{split}
\nonumber
\end{align}
The regions\footnote{Recall from Remark \ref{R:Omitting_rescaling} that we are dealing with rescaled versions of the original variables and the original interval $[0,T_{\text{bootstrap}}]$. Thus, even though $T_{\text{bootstrap}}$ is small, its rescaled version will be large so that it makes sense to consider the region $t\geq 1$ in the definition of $\mathcal{W}$ and $\underline{\mathcal{W}}$.}
where $\mathcal{W}$ and $\underline{\mathcal{W}}$ change definitions are illustrated in Fig.~\ref{F:Regions_conformal_energy}.

\begin{figure}[ht]
\centering
  \includegraphics[scale=0.4]{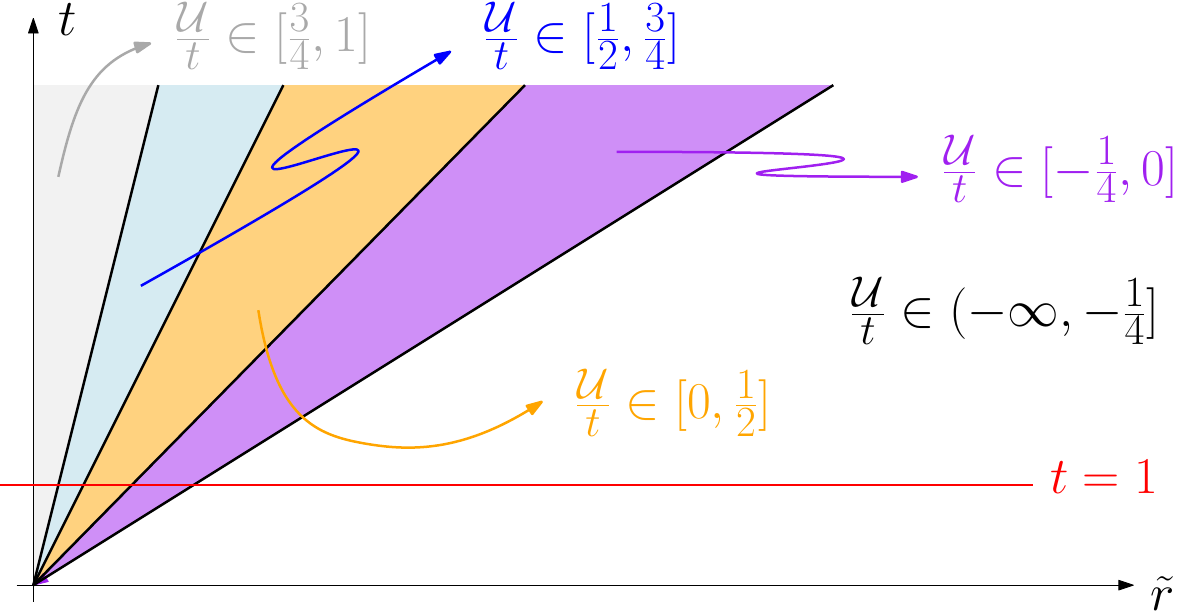}
  \caption{Illustration of the regions of definition of $\mathcal{W}$ and $\underline{\mathcal{W}}$. In this schematic representation, we have omitted the quasilinear nature of the problem and depicted the geometry as flat.}
  \label{F:Regions_conformal_energy}
\end{figure}

The desired energy whose bound implies \eqref{E:Dispersive_estimate_L_2} is given by
\begin{align}
\label{E:Conformal_energy}
\begin{split}
E_{\text{conformal}} & := \int_{\Sigma_t} (\underline{\mathcal{W}} - \mathcal{W} )t^2 
(|\mathbf{D} \varphi|^2 + |\tilde{r}^{-1} \varphi|^2 ) 
\\
& \ \ \
+ \int_{\Sigma_t} \mathcal{W} (|\tilde{r} \mathbf{D}_L \varphi|^2 + |\tilde{r} \slashed{\nabla} \varphi|^2_{\slashed{G}} + |\varphi|^2 ),
\end{split}
\end{align}
where $\slashed{\nabla}$ is the gradient on $S_{t,\mathcal{U}}$ with respect to the metric 
$\slashed{G}$ induced on $S_{t,\mathcal{U}}$, the integrals are with respect to the volume
form of the metric induced on\footnote{Due to a localization argument, the integral is only over a localized region of $\Sigma_t$. We downplay this aspect as explained in Remark \ref{R:Omitting_rescaling}.} $\Sigma_t$ by $G$, and we recall that $\tilde{r}$ is given by
\eqref{E:r_tilde}. $E_{\text{conformal}}$ is called the \textdef{conformal energy.}

The use of cut-off functions is to separate the conformal energy according to regions where
different multiplies are needed, as explained in Sect.~\ref{S:Decay_acoustic_geometry}. The goal is to prove a bound establishing a mild growth of the conformal energy, namely,
\begin{align}
\label{E:Mild_growth_conformal_energby}
E_{\text{conformal}} \lesssim C_\varepsilon (1+t)^{2\varepsilon} 
( \norm{\partial \varphi}_{L^2(\Sigma_1)} + \norm{\varphi}_{L^2(\Sigma_1)} ),
\end{align}
for some small $\varepsilon > 0$ (the constant $C_\varepsilon$ blows up when $\epsilon \rightarrow 0$). Once \eqref{E:Mild_growth_conformal_energby} has been established, one can then extract decay for $\varphi$ on each $\Sigma_t$, thus establishing \eqref{E:Dispersive_estimate_L_2}, via an argument by \cite{Dafermos-Rodnianski-2010}. The argument is technical and will not be given here (see, aside from \citealt{Dafermos-Rodnianski-2010}, \citealt[Sect.~7]{Wang-2017}). We notice, however, that decay for $\varphi$ is consistent with \eqref{E:Mild_growth_conformal_energby}. Indeed, the first integral in \eqref{E:Conformal_energy}
has an explicit weight in $t$; the second integral has $\tilde{r}$ weights, but these are comparable to $t$ in the region where $\mathcal{W}$ does not vanish. (Although for the proof of \eqref{E:Mild_growth_conformal_energby} it is important that the weights in the second integral
of \eqref{E:Conformal_energy} are precisely $\tilde{r}$ and not $t$; it is only after 
\eqref{E:Mild_growth_conformal_energby} has been established that we can invoke 
$\tilde{r} \approx t$ in the region $\mathcal{W} \neq 0$.) Since \eqref{E:Mild_growth_conformal_energby} states that $E_{\text{conformal}}$ grows very slowly, the presence of $t$ weights in $E_{\text{conformal}}$ requires that $\varphi$ itself must have some time decay.

\begin{remark}
Due to technical reasons related to localizations that are employed in the proof, in the end we consider problem
\eqref{E:Linear_in_varphi_wave_Duhamel} with data given at $\{t=1\}$ instead of $\{ t = 0\}$, which 
explains the right-hand side of \eqref{E:Mild_growth_conformal_energby} being evaluated at $\Sigma_1$.
(In fact, \eqref{E:Dispersive_estimate_L_infty} and \eqref{E:Dispersive_estimate_L_2} should be adjusted
to have the right-hand side evaluated at $\Sigma_1$ as well.)
Readers should not let this minor technicality distract from the main points of the argument.
\end{remark}

To prove \eqref{E:Mild_growth_conformal_energby}, one proceeds as usual, computing $\frac{d}{dt} 
E_{\text{conformal}}$, carrying out an integration by parts and then an integration in time. Schematically, then
\begin{align}
E_{\text{conformal}}(t) \lesssim E_{\text{conformal}}(1) + \int_\mathfrak{D} \mathfrak{E},
\nonumber
\end{align}
where $\mathfrak{D}$ is a spacetime region (say, $\mathfrak{D} = [0,T_{\text{bootstrap}}] \times \mathbb{R}^3$ for concreteness) and $\mathfrak{E}$ an error term coming from the integration by parts\footnote{Using the geometric formalism for quasilinear wave equations, we can identify 
$\mathfrak{E}$ as a contraction of the energy-momentum tensor for $\square_G$ with the deformation tensor of $G$ for suitably chosen vectorfields, see \cite{Alinhac-Book-2010}. This precise form will not be needed for our high-level discussion.}. As mentioned in Sect.~\ref{S:Decay_acoustic_geometry}, this integration by parts produces, in particular, error terms in $\partial (N,L)$ that are integrated over spacetime (observe that $L$ appears explicitly in $E_{\text{conformal}}$ whereas $N$ is hidden in the spacetime gradient $\mathbf{D} \varphi$). Such $\partial (N,L)$ terms are cast as connection coefficients which then need to be controlled. Note that the two integrals in $E_{\text{conformal}}$ represent the interior and exterior regions alluded to in Sect.~\ref{S:Decay_acoustic_geometry}.

We stress that the reduction of proving \eqref{E:Mild_growth_conformal_energby} to bounding
the error terms involving the connection coefficients, (i.e., to the control of the acoustic
geometry) is not trivial; see \citealt{Disconzi-Luo-Mazzone-Speck-2022} and references therein for a full discussion. Here, we restrict ourselves to connect the need for controlling the connection
coefficients with the discussion of Sect.~\ref{S:Control_acoustic_geometry}, since
the exposition in that section highlighted the multiple-speeds behavior of the Euler system
and the role played by the vorticity in comparison to pure wave problems. Moreover, as in Sect.~\ref{S:Control_acoustic_geometry} we focused primarily on $\tr_{\slashed{G}} \upchi$, we do so here in connection to the proof of \eqref{E:Mild_growth_conformal_energby}.

With the above set up in mind, we can now point out some of the main ingredients in the proof of \eqref{E:Mild_growth_conformal_energby}:

\begin{itemize}
\item Observe that $E_{\text{conformal}}(1) \lesssim \norm{\partial \varphi}_{L^2(\Sigma_1)} + \norm{\varphi}_{L^2(\Sigma_1)}$, so we need to show that
\begin{align}
\int_\mathfrak{D} \mathfrak{E} \lesssim C_\varepsilon (1+t)^{2\varepsilon}  E_{\text{conformal}}(1).
\nonumber
\end{align}
\item As mentioned, connection coefficients such as $\upchi$ are produced by an integration by parts. Decomposing $\upchi$ into its trace and trace-free part leads to $tr_{\slashed{G}} \upchi$ (see Sect.~\ref{S:Control_acoustic_geometry}), which needs to be bounded.
\item Recall from Sect.~\ref{S:Control_acoustic_geometry} that we do not expect to be able to control $tr_{\slashed{G}} \upchi$ but rather $tr_{\slashed{G}} \upchi + \Gamma_L$. 
However, the null-structure equations and other tools we employ are geometric in nature, and thus give equations and estimates for $tr_{\slashed{G}} \upchi$.
To overcome this discrepancy, the trick, that goes back to the work by \cite{Wang-2017},
is to make a \emph{conformal change} of the metric $G$,
\begin{align}
\tilde{G} := e^{2 \upsigma} G,
\nonumber
\end{align}
for a suitably chosen function $\upsigma$ such that the null mean curvature of $\tilde{G}$ is 
$tr_{\slashed{G}} \upchi + \Gamma_L$, i.e.,
\begin{align}
tr_{\slashed{\tilde{G}}} \tilde{\upchi} = tr_{\slashed{G}} \upchi + \Gamma_L.
\label{E:conformal_null_mean_curvature}
\end{align}
In this way, we deal with a geometry whose null mean curvature is the quantity we need to control, i.e., $tr_{\slashed{G}} \upchi + \Gamma_L$.
We can then carry out an energy estimate using the conformal metric $\tilde{G}$, in which case
the corresponding connection coefficient appearing upon integration by parts is given by
\eqref{E:conformal_null_mean_curvature}, which, from Sect.~\ref{S:Control_acoustic_geometry}, we expect to be able to control. See further comments below.
\item In the end, however, we want control related to $tr_{\slashed{G}}$,
and not $tr_{\slashed{\tilde{G}}} \tilde{\upchi}$, because it is estimate \eqref{E:Mild_growth_conformal_energby}, based on $G$ and not $\tilde{G}$, that we need to establish. For this, we need to derive several comparison results between quantities constructed out of $G$ and of out $\tilde{G}$. These comparisons require good estimates for the conformal factor $\upsigma$.
\item At some point of the argument we need to control not only 
$tr_{\slashed{\tilde{G}}} \tilde{\upchi}$, or equivalently $tr_{\slashed{G}} \upchi + \Gamma_L$, but also its $\SoundCone_\mathcal{U}$-tangential derivatives, 
$\slashed{\partial}(tr_{\slashed{G}} \upchi + \Gamma_L)$. Roughly, control of this term is needed
in estimates related to $\upsigma$; it is also needed in the aforementioned energy argument
involving the conformal metric $\tilde{G}$: as in the case of $E_{\text{conformal}}$, the 
null mean curvature (of the conformal metric) $tr_{\slashed{\tilde{G}}} \tilde{\upchi}$ appears after an integration by parts, but a further integration by parts is needed in order to correctly balance terms with different powers of $\tilde{r}$ (recall that the energy we ultimately want to control,  $E_{\text{conformal}}$, is weighted in $\tilde{r}$, see Remark \ref{R:Downplaying_technicalities_conformal_energy}).
\item To bound the term $\slashed{\partial}(tr_{\slashed{G}} \upchi + \Gamma_L)$, the idea is to use the fact that derivatives of the connection coefficients in directions tangent to the sound cones exhibit better behavior than derivatives in directions transvere to the cones (see Sects.~\ref{S:Decay_acoustic_geometry} and \ref{S:Control_acoustic_geometry}). In order to exploit this fact, in the spacetime integral over $\mathfrak{D}$, we want to isolate the integration over sound cones. I.e., instead of using Fubini's theorem to integrate first in space and then in time, 
\begin{align}
\int_\mathfrak{D} (\cdots) \, d\vol= \int_{t_I}^{t_F} \int_{\Sigma_t} (\cdots) \, dx dt
\label{E:Spacetime_interal_error_term_space_time}
\end{align}
we integrate first along each sound cone $\SoundCone_\mathcal{U}$ and then along the $\mathcal{U}$ direction (recall that each $\SoundCone_\mathcal{U}$ is a level set of $\mathcal{U}$):
\begin{align}
\int_\mathfrak{D} (\cdots)  d\vol = \int_{\mathcal{U}_I}^{\mathcal{U}_F} \int_{\SoundCone_\mathcal{U}} (\cdots) \, d \slashed{\upomega} d\mathcal{U},
\label{E:Spacetime_interal_error_term_sound_cone_eikonal}
\end{align}
where $d \slashed{\upomega}$ is a suitable integration measure\footnote{Recall that $\SoundCone_\mathcal{U}$ is null with respect to the acoustical metric; see the discussion in the paragraph after \eqref{E:Divergence_theorem_L2_cones}.} on $\SoundCone_\mathcal{U}$ and
for certain initial and final times $t_I, t_F$ and values of the eikonal function $\mathcal{U}_I,\mathcal{U}_F$. \emph{We can estimate the $\SoundCone_\mathcal{U}$ integral involving 
$\slashed{\partial}(tr_{\slashed{G}} \upchi + \Gamma_L)$ in view of the special structures 
of \eqref{E:New_formulation_classical} which allow us us to bound  $\slashed{\partial} \curl \Omega$ in $L^2(\SoundCone_\mathcal{U})$,} as explained in Sect.~\ref{S:Control_acoustic_geometry}. Observe that we would not be able to take advantage of this special control along $\SoundCone_\mathcal{U}$ if we wrote the spacetime integral as 
\eqref{E:Spacetime_interal_error_term_space_time} instead of \eqref{E:Spacetime_interal_error_term_sound_cone_eikonal}. The idea of considering a decomposition
in the form  \eqref{E:Spacetime_interal_error_term_sound_cone_eikonal} 
in place of \eqref{E:Spacetime_interal_error_term_space_time}, so that one can take advantage of the better behavior of the acoustic geometry in directions tangent to the sound cones, goes back to Christodoulou and Klainerman's proof of stability of Minkowski space (see \citealt{Christodoulou-Klainerman-Book-1993}).
\item In order to bound the conformal factor $\upsigma$, a modified (depending on $\upsigma$) version of the so-called mass-aspect function is introduced. The mass-aspect function is a key quantity depending on $\underline{L} tr_{\slashed{G}}\upchi$ used in Christodoulou and Klainerman's proof of stability of Minkowski space (see \citealt{Christodoulou-Klainerman-Book-1993}). Such quantity is important because $\underline{L} tr_{\slashed{G}}$ is a ``bad'' derivative of $tr_{\slashed{G}}$, i.e., a derivative transverse to the sound cones, and thus difficult to control. The mass-aspect function is a suitable combination of $\underline{L} tr_{\slashed{G}}$ and other variables that satisfies a good equation, for which estimates can be derived. In our context, this quantity is appended by a dependence on derivatives of $\upsigma$ (see \citealt[Equation (209)]{Disconzi-Luo-Mazzone-Speck-2022}. We remark that control of the modified mass-aspect function is quite delicate and it constitutes one of the main technical difficulties of the proof.

\end{itemize}

\begin{remark}
\label{R:Downplaying_technicalities_conformal_energy}
Above, we focused on only one aspect of the proof of \eqref{E:Mild_growth_conformal_energby}, that related to estimating $\slashed{\partial}(tr_{\slashed{G}} \upchi + \Gamma_L)$, which is one of the most delicate but also interesting aspects of the proof. We illustrated the role of the vorticity in the proof and the need to employ some very special structures of Eqs.~\eqref{E:New_formulation_classical}. We are leaving out many other relevant and interesting ideas as well as several technical aspects. In particular, the above outline does not take into account the following important points:
\begin{itemize}
\item For technical reasons, we in fact have to control a slightly higher Lebesgue exponent than $2$ along sound cones, i.e., we need to bound $\slashed{\partial} \curl \Omega$ in $L^p(\SoundCone_\mathcal{U})$, $p \gtrsim 2$. The key point for such estimate is the same as in Sect.~\ref{S:Control_acoustic_geometry}, namely, the transversality of $\B$ and $\SoundCone_\mathcal{U}$.
\item As mentioned, one also needs control the connection coefficients associated with $\partial N$. This is the part of the argument where we rely on a Morawetz-type estimate, as alluded to in Sects.~\ref{S:Decay_acoustic_geometry} and \ref{S:Control_acoustic_geometry}.
\item Recall that $E_{\text{conformal}}$ is a \emph{weighted} energy. The role played by the $t$ and $\tilde{r}$ weights is crucial for closing the estimates. Obtaining estimates with correct powers of the weights is a delicate and crucial aspect of the proof, but one which we do not discuss here because it lies heavily on the technical part of the work. The \emph{conceptual} importance of deriving weighted estimates, namely, the need to obtain decay estimates, has been discussed above and in Sects.~\ref{S:Decay_acoustic_geometry} and \ref{S:Control_acoustic_geometry}
\item We remark that the control needed along sound cones is at top order. We also need energy estimates for Littlewood--Paley projections of $\curl \Omega$ along sound cones and their sum over frequencies. These estimates are obtained by commuting the equations with the projections and arguing as above. In particular, once again the key feature is the transversality of $\B$ and the sound cones. 
\end{itemize}
\end{remark}

\subsection{H\"older control of the transport-part}
\label{S:Holder_control_transport}

We have already discussed one important aspect of controlling the transport-part of the system, namely, control along sound cones, see Sects.~\ref{S:Estimates_curl_Omega_along_sound_cones} and \ref{S:Conformal_energy}. 
In propagating the regularity of $\curl \Omega$, we have in particular to propagate its H\"older regularity, since this is needed for the elliptic estimates, as mentioned. Another interesting feature of the proof of Theorem \ref{T:Rough_classical_Euler} is that we can derive non-trivial H\"older estimates for a hyperbolic system,
a feat that relies crucially on the good structures
of the the new formulation of the classical Euler and which does not seem attainable for generic hyperbolic systems.

Bounds for $\bar{\partial} \Omega$ in $C^{0,\upalpha}$ can be obtained as follows. First, we establish
a div-curl estimate in H\"older spaces of the form
\begin{align}
\norm{\bar{\partial} \Omega}_{C^{0,\upalpha}(\Sigma_t)} \lesssim
\norm{
\dive \Omega}_{C^{0,\upalpha}(\Sigma_t)} + \norm{ \curl \Omega}_{C^{0,\upalpha}(\Sigma_t)} + 
\text{L.O.T.}
\label{E:Div_curl_Holder}
\end{align}
Using \eqref{E:New_formulation_classical_transport_div_Omega}, \eqref{E:Div_curl_Holder} gives
\begin{align}
\norm{\bar{\partial} \Omega}_{C^{0,\upalpha}(\Sigma_t)} \lesssim
\norm{\partial \Psi}_{C^{0,\upalpha}(\Sigma_t)} + \norm{ \curl \Omega}_{C^{0,\upalpha}(\Sigma_t)},
\label{E:Div_curl_Holder_Psi}
\end{align}
where as usual we ignore the lower-order terms. For $\curl \Omega$,  we derive an energy estimate
for transport equations in H\"older spaces that reads
\begin{align}
\norm{\curl \Omega}_{C^{0,\upalpha}(\Sigma_t)}
\lesssim
\norm{\curl \Omega}_{C^{0,\upalpha}(\Sigma_0)}
+ \int_0^t \norm{\B \curl \Omega}_{C^{0,\upalpha}(\Sigma_\tau)} \, d\tau.
\label{E:Energy_estimate_Holder_spaces_transport}
\end{align}

Estimate \eqref{E:Energy_estimate_Holder_spaces_transport} is proven by integrating along the characteristics of the transport-part, i.e., 
along the flow lines of $\B$, and comparing ratios of nearby point. This comparison is needed
because the H\"older norms involve comparing $\curl \Omega$ at different (nearby) points along $\Sigma_t$. On the other hand, integrating along the flow lines of $\B$ only allows us to compare points along the integral curves of $\B$, i.e., compare the value of $\curl \Omega$ at a point
$x_t \in \Sigma_t$ with the value of $\curl \Omega$ at a point
$x_0 \in \Sigma_0$, where $x_0$ and $x_t$ are connected by a flow line of $\B$ (see Fig.~\ref{F:Control_flow_lines_B_Holder}). Along $\Sigma_0$ we have control for  $\norm{\curl \Omega}_{C^{0,\upalpha}(\Sigma_0)}$ by assumption. Thus, in order to propagate the H\"older regularity of $\curl \Omega$ from $\Sigma_0$ to $\Sigma_t$ we need 
show that the distance between two nearby points $x_t, y_t$ along $\Sigma_t$, which enters in the definition of the norm  $\norm{\curl \Omega}_{C^{0,\upalpha}(\Sigma_t)}$ we want to control, is comparable
to the distance between $x_0, y_0$ along $\Sigma_0$, which enters in the definition of the  norm  $\norm{\curl \Omega}_{C^{0,\upalpha}(\Sigma_0)}$, which is bounded by assumption, where
$x_t$ and $y_t$ are connected to $x_0$ and $y_0$, respectively, along flow lines of $\B$ (see Fig.~\ref{F:Control_flow_lines_B_Holder}). In order words, we need to establish that under the bootstrap assumptions \ref{E:Bootstrap_assumption},
\begin{align}
| x_t - y_t | \approx |x_0 - y_0|.
\nonumber
\end{align}
This will be the case as long \emph{as we have uniform control over the integral curves of $\B$, which can be shown to be the case using the bootstrap assumptions.}

\begin{figure}[ht]
\centering
  \includegraphics[scale=0.4]{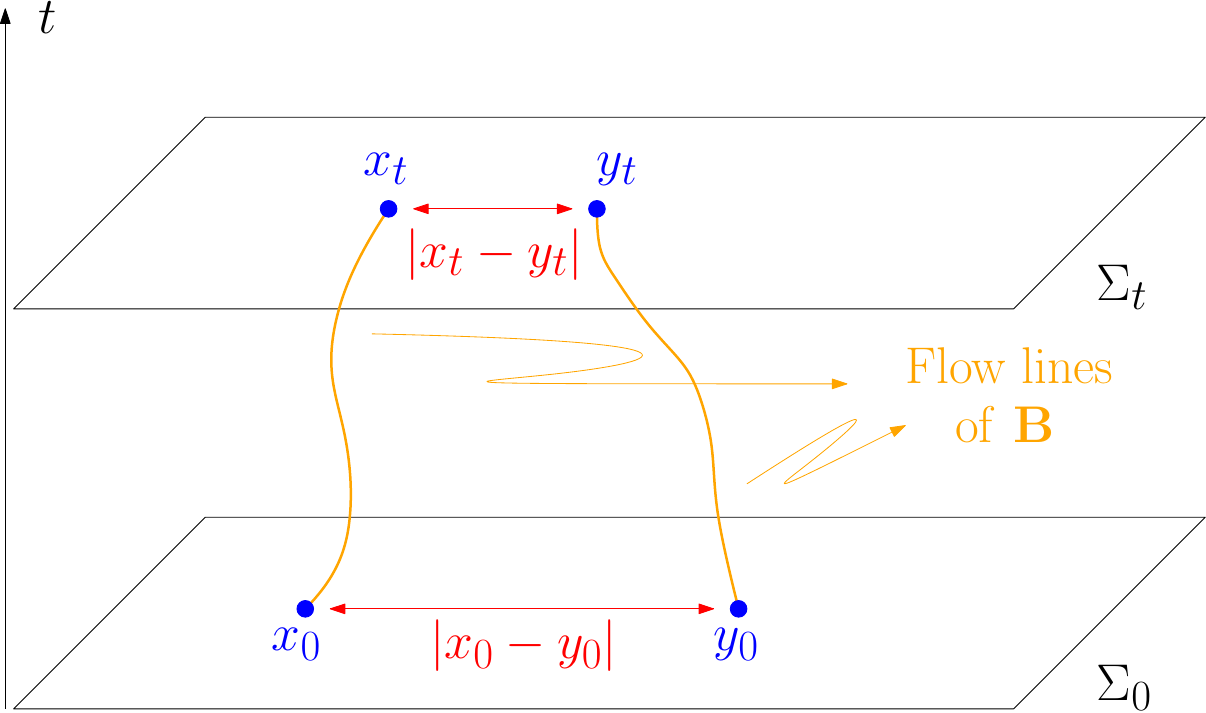}
  \caption{Illustration of distance between points along the flow lines of $\B$.}
  \label{F:Control_flow_lines_B_Holder}
\end{figure}

Combining \eqref{E:Energy_estimate_Holder_spaces_transport} and  
\eqref{E:New_formulation_classical_transport_curl_Omega} gives
\begin{align}
\norm{\curl \Omega}_{C^{0,\upalpha}(\Sigma_t)}
\lesssim
\norm{\curl \Omega}_{C^{0,\upalpha}(\Sigma_0)}
+ \int_0^t \norm{\partial \Psi}_{C^{0,\upalpha}(\Sigma_\tau)} \, d\tau.
\label{E:Energy_estimate_Holder_spaces_transport_Psi}
\end{align}
Then \eqref{E:Div_curl_Holder_Psi} and  \eqref{E:Energy_estimate_Holder_spaces_transport_Psi}
yield
\begin{align}
\norm{\bar{\partial} \Omega}_{C^{0,\upalpha}(\Sigma_t)} \lesssim
\norm{\partial \Psi}_{C^{0,\upalpha}(\Sigma_t)} +
\norm{\curl \Omega}_{C^{0,\upalpha}(\Sigma_0)}
+ \int_0^t \norm{\partial \Psi}_{C^{0,\upalpha}(\Sigma_\tau)} \, d\tau
\label{E:Div_curl_Holder_final}
\end{align}

Estimate \eqref{E:Div_curl_Holder_final} will provide the desired bound for $\Omega$ in H\"older spaces provided we can control $\norm{\partial \Psi}_{C^{0,\upalpha}(\Sigma_\tau)}$. For the latter,
we rely on the Littlewood-Paley characterization of H\"older spaces (see \citealt[Appendix A]{Taylor-Book-1991}) to show that the bootstrap assumption \eqref{E:Bootstrap_assumption_Psi}
controls the $L^2_t C_x^{0,\upalpha}$-norm of $\partial \Psi$. 

It is a \emph{non-trivial} fact that the bootstrap assumptions for the the wave-part allow us to 
squeeze a bit of H\"older control of $\partial \Psi$, as the H\"older norms are associated with control of the transport-part. This is another example of the \emph{special features of Eqs.~\eqref{E:New_formulation_classical} and the corresponding interaction between the wave and transport parts.} 

\begin{remark}
As in the last bullet point of Remark \ref{R:Downplaying_technicalities_conformal_energy}, 
control of $\partial \Psi$ in H\"older spaces is at top order, so we also need H\"older estimates
for Littlewood-Paley projections of $\partial \Psi$ and their sum over frequencies in H\"older spaces. These are controlled by the second sum in the bootstrap assumption \eqref{E:Bootstrap_assumption_Psi}.
\end{remark}

\begin{remark}
Control of the acoustic geometry involves control over spheres $S_{t,\mathcal{U}}$, which are sections of $\SoundCone_\mathcal{U}$. Because of our functional framework, this eventually leads to H\"older estiamtes for connection coefficients on $S_{t,\mathcal{U}}$. Such estimates are obtained by argument similar to the above, but where now we transport along the integral curves of $L$ instead of $\B$. As in the case of $\B$, this requires uniform control of the flow lines of $L$, which can be established under the bootstrap assumptions.
\end{remark}

\subsection{Closing the bootstrap\label{S:Closing_bootstrap}} As it is often the case on arguments based on a bootstrap,
once estimates have been shown to close at a consistent regularity level under bootstrap assumptions, one seeks to improve the bootstrap (in our case, \eqref{E:Bootstrap_assumption_improved}) upon appealing to some smallness in the problem.
In our case, smallness is obtained by taking the bootstrap time $T_{\text{bootstrap}}$ sufficiently small. We refer to \cite{Disconzi-Luo-Mazzone-Speck-2022} for details.

\subsection{Back to the relativistic case\label{S:Back_to_relativistic_Euler}}
As it has been stressed many times in these notes,
one should not expect delicate results for the classical compressible Euler equations to be easily generalizeable to or even true in the relativistic setting. Thus, a legitimate question is whether a version of Theorem \ref{T:Rough_classical_Euler} holds for the relativistic Euler equations. An answer was provided by Yu:

\begin{theorem}[\cite{Yu-2024}] A similar low-regularity result as Theorem \ref{T:Rough_classical_Euler} holds for the relativistic Euler equations.
\label{T:Rough_relativistic_Euler}
\end{theorem}
\begin{proof}
We refer to \cite{Yu-2024} for a precise statement of Theorem \ref{T:Rough_relativistic_Euler}, as well as for its complete proof.
 As the proof of Theorem 
\ref{T:Rough_classical_Euler} employed the new formulation
of classical Euler introduced by 
\cite{Speck-2019,Luk-Speck-2020}, the proof of Theorem \ref{T:Rough_relativistic_Euler} employs the new formulation of the relativistic Euler equations 
given in Theorem \ref{T:New_formulation}.

It would not be feasible for us to present a detailed discussion of the proof of Theorem 
\ref{T:Rough_relativistic_Euler}. Rather, we will restrict ourselves to 
list some of the key points in which the proof of Theorem \ref{T:Rough_relativistic_Euler} is genuinely different than that of Theorem \ref{T:Rough_classical_Euler}. More precisely,
while the proof of Theorem \ref{T:Rough_relativistic_Euler} follows the general structure
of the proof of Theorem \ref{T:Rough_classical_Euler}, there are truly new aspects that Yu 
had to deal with, making the
proof of Theorem \ref{T:Rough_relativistic_Euler} significantly more challenging than that of Theorem \ref{T:Rough_classical_Euler}. Some of these aspects are
\begin{itemize}
\item The elliptic estimates in Theorem \ref{T:Rough_classical_Euler} are for constant-coefficient elliptic operators, whereas those of Theorem \ref{T:Rough_relativistic_Euler} are for quasilinear elliptic operators. This is related to the fact that in the relativistic case the div-curl operators
are spacetime div-curl. One needs to extract spacial regularity along $\Sigma_t$ from these spacetime div-curl operators. The basic idea is reminiscent of excising the timelike part of the vorticity as in Theorem \ref{T:Improved_regularity}. However, at a low-regularity level such excision process is much more delicate.
\item In the geometric framework employed to estimate the acoustic geometry, one has to deal with several decompositions of tensors in directions parallel and perpendicular to the four-velocity\footnote{We use four-velocity for what in other sections we called simply velocity to emphasize here the difference between the relativistic and the classical cases.}, a feature with no analogue in the classical case.
\item In the classical case, the material derivative vectorfield $\B$ is orthogonal, with respect to $G$, to the constant-time hypersurfaces $\Sigma_t$. While we have not discussed the implications of this orthogonality, it is routinely used in the proof of Theorem \ref{T:Rough_classical_Euler} as a  consequence of the geometric formalism. In the relativistic case, derivative in the direction of the four-velocity plays the role of the material derivative, but the four-velocity is \emph{not} orthogonal to $\Sigma_t$ with respect to the acoustical metric. This produces error terms depending on $\Psi$ that need to be controlled through further technical arguments.
\item In general, the new formulation of the relativistic Euler equations is significantly more intricate than the new formulation of the classical Euler equations (compare Theorem \ref{T:New_formulation} with Theorem 3.1 in \citealt{Speck-2019}). In a low-regularity setting, this leads to quite a few 
potentially dangerous error terms that need to be carefully analyzed.
\end{itemize}
\end{proof}

\subsection{Further lowering of regularity for the classical Euler equations}
\label{S:Lowering_regularity_classical_Euler}

We mentioned at the beginning of Sect.~\ref{S:Rough_solutions} that Theorem \ref{T:Rough_classical_Euler} has been improved by \cite{Wang-2022} and independently by \cite{Zhang-Andersson-arxiv-2022}, who established Theorem \ref{T:Rough_classical_Euler} without any H\"older regularity on the data and lowered the regularity on the curl part of the data to $\curl u \in H^{2+{\varepsilon^\prime}}(\Sigma_0)$, $0< \varepsilon^\prime < \varepsilon$. Andersson and Zhang's result also allows for $\varepsilon^\prime=0$.

We will now highlight some of the essential ideas used in the proof in \cite{Wang-2022}. This proof is more in line with the geometric-analytic techniques that we want to emphasize in our presentation. While the work by \cite{Zhang-Andersson-arxiv-2022} involves several new important contributions to the field, its result involves ideas of a different nature, more in line with the approach of \cite{Smith-Tataru-2005} and a detailed exposition would be beyond the scope of this review.  

Recall from Sect.~\ref{S:Energy_estimates_low_regularity} (see the discussion surrounding Eq.~\eqref{E:Energy_estimates_2_plus_epsilon_part_5}) that our assumption $\left. \curl u \right|_{t=0} \in H^{2+\varepsilon}$ in Theorem \ref{T:Rough_classical_Euler} was used in deriving energy estimates for $\Psi$.
From Sect.~\ref{S:Control_spacetime_norms}, we have that our H\"older assumptions in Theorem \ref{T:Rough_classical_Euler} were used to bound 
$\norm{\bar{\partial} \Omega}_{L^1_tL^\infty_x}$. The H\"older assumption is also used to control the acoustic geometry, see Sects.~\ref{S:Control_acoustic_geometry}, \ref{S:Estimates_curl_Omega_along_sound_cones}, and \ref{S:Holder_control_transport}. Without the H\"older assumptions and under the weaker hypothesis $\curl u \in H^{2+{\varepsilon^\prime}}(\Sigma_0)$, $0< \varepsilon^\prime < \varepsilon$, on the curl part, one needs to employ a different argument than the one presented above for closing the energy estimates, bounding the spacetime norms, and controlling the acoustic geometry.

We have considered $u$ as one of the wave variables in the analysis, i.e., $\Psi = (\hat{h}, u, s)$ (see the discussion after Theorem \ref{T:Rough_classical_Euler}). It would have been more precise to include only $\dive u$ in $\Psi$ since $\curl u$ has been grouped with the transport variables\footnote{See Footnote \ref{FN:Wave_part_full_velocity}.}. But since $u$ itself satisfies a wave equation, we have not taken full account of the decomposition of a vectorfield into its divergence and curl parts. Wang's approach takes full advantage of this decomposition by splitting the velocity into its divergence and curl parts and deriving a better wave equation for the divergence part than that satisfied by $u$. The price one pays is that the problem becomes very non-local and some new, delicate, cancellations have to be uncovered to close the estimates.

In more detail, consider 
\begin{align}
u = u_d + u_c,
\label{E:Decomposition_u_div_curl_Wang}
\end{align}
where $u_d$ is associated with the divergence part of the velocity and $u_c$ with its curl part. This is accomplished by defining $u_c$ as a solution to 
\begin{align}
(id- \Delta_\updelta)u_c = \curl \curl u
\label{E:Decomposition_curl_Wang}
\end{align}
where $\Delta_\updelta$ is the Laplace-Beltrami operator relative to the Euclidean metric $\updelta$. Recalling the definition of the specific vorticity $\Omega = e^{-\hat{h}}\curl u$,
and the modified vorticity of the vorticity \eqref{E:Modified_vorticity_of_vorticity_classical}, we see that the RHS of \eqref{E:Decomposition_curl_Wang} is related to $\C$, which, as discussed, plays a key role in the analysis. Thus, \eqref{E:Decomposition_curl_Wang} reads as
\begin{align}
(id- \Delta_\updelta)u_c \simeq \curl \Omega,
\label{E:Decomposition_curl_Wang_schematic}
\end{align}
where, in \eqref{E:Decomposition_curl_Wang_schematic}, we continue to employ Remark \ref{R:Trading_C_for_curl_Omega} to write $\curl \Omega$ as a proxy for $\C$.

The term $\curl \Omega$ on the RHS of \eqref{E:Decomposition_curl_Wang_schematic} is precisely the difficult source term appearing in \eqref{E:New_formulation_classical_wave}. \emph{A key idea in \cite{Wang-2022} is that upon computing $\square_G u$ and using \eqref{E:Decomposition_u_div_curl_Wang} and \eqref{E:Decomposition_curl_Wang_schematic}, the source term $\curl \Omega$ in \eqref{E:New_formulation_classical_wave} cancels} (up to lower-order terms). 

More precisely, writing\footnote{The cross terms, i.e., those mixing $\B$ and derivatives along $\Sigma_t$, which are omitted in $\Square_G \simeq - \B\B + \Delta_{\bar{G}}$, are also important in the analysis, but we will not discuss them here.} 
$\Square_G \simeq  -\B\B + \Delta_{\bar{G}}$, where $\Delta_{\bar{G}}$ is the Laplace-Beltrami operator relatitve to the metric $\bar{G}$ induced on $\Sigma_t$ by the acoustical metric $G$, from \eqref{E:Decomposition_curl_Wang}, \eqref{E:Decomposition_curl_Wang_schematic}, and \eqref{E:New_formulation_classical_wave}, we have
\begin{align}
-\curl \Omega + \partial \Psi \cdot \partial \Psi \simeq \Square_G u & = \Square_G u_d + \Square_G u_c
\nonumber
\\
& = \Square_G u_d - \B \B u_c + \Delta_{\bar{G}} u_c
\nonumber
\\
& = \Square_G u_d - \B \B u_c + \Delta_{\updelta} u_c
\nonumber
\\
&= \Square_G u_d - \B \B u_c - \curl \Omega + u_c,
\nonumber
\end{align}
where we also used that $\Delta_{\bar{G}} \simeq \Delta_{\updelta}$. This follows from the fact that $\bar{G} = c_s^{-2} \updelta$, which allows us to trade $\Delta_{\bar{G}}$ for $\Delta_{\updelta}$ up to terms with good structure that can be subsumed into $\simeq$. We also note the presence of a minus sign on the $\curl\Omega$ term on the LHS. This minus sign is present in the non-schematic form of Eq.~\eqref{E:New_formulation_classical_wave}, but so far we could ignore it due to our schematic presentation. However, in order to observe exact cancellations, such signs matter. In fact, the exact form of the canceled terms, including its coefficients, is important, but for our purposes of illustrating the cancellation it suffices to keep track of the signs. Thus, canceling $\curl \Omega$ on both sides gives the following evolution equation for $u_d$
\begin{align}
\Square_G u_d \simeq  \B\B u_c + \partial \Psi \cdot \partial \Psi .
\label{E:Wave_equation_divergence_part_Wang}
\end{align}

Equation \eqref{E:Wave_equation_divergence_part_Wang} still contains the high-order source term $\B\B u_c$. The next step is to show that in all estimates involving \eqref{E:Wave_equation_divergence_part_Wang}, this term can be canceled. There are two important applications of \eqref{E:Wave_equation_divergence_part_Wang} where this cancellation is crucial. The first is in the energy estimates, in particular when estimates along sound cones are involved. Recall from Sect.~\ref{S:Estimates_curl_Omega_along_sound_cones} that in the proof of Theorem \ref{T:Rough_classical_Euler}, a delicate point was to control $\curl \Omega$ along sound cones. These estimates for $\curl \Omega$ along the sound cones were one of the important new elements in Theorem \ref{T:Rough_classical_Euler} as compared to the irrotational case, and they required the regularity assumptions stated in Theorem \ref{T:Rough_classical_Euler}. Under the weaker regularity assumptions of \cite{Wang-2022}, the approach of Sect.~\ref{S:Estimates_curl_Omega_along_sound_cones} fails. Thus, one first eliminates this term as in \eqref{E:Wave_equation_divergence_part_Wang}. In order to show that $\B\B u_c$ can be canceled as well, Wang rewrites \eqref{E:Wave_equation_divergence_part_Wang} as a first-order system and introduces a vectorfield related to $\B u_c$ and tailored to an application of the divergence theorem (this should be compared to \eqref{E:Current_curl_Omega_sound_cones}, which is used to control $\curl \Omega$ along sound cones in Sect.~\ref{S:Estimates_curl_Omega_along_sound_cones}). The second application of \eqref{E:Wave_equation_divergence_part_Wang} where one has to cancel the $\B\B u_c$ term is in the Strichartz estimates. Recall from Sect.~\ref{S:Strichartz_and_reductions} that to obtain the Strichartz estimates, we employ Duhamel's principle to reduce the problem to controlling solutions to the linear wave equation in a $G$ background (see Eq.~\eqref{E:Linear_in_varphi_wave_Duhamel} and the surrounding discussion). Here, unlike what was done for $\curl \Omega$ in the proof of Theorem \ref{T:Rough_classical_Euler}, the term $\B\B u_c$ is not treated as a source from the point of view of the estimates for the wave equation. The term $\B\B u_c$ is canceled by an intricate modification of the application of Duhamel's principle that goes back to Wang's previous work \cite{Wang-2014}.

In the proof of Theorem \ref{T:Rough_classical_Euler}, we introduced the bootstrap assumption 
\eqref{E:Bootstrap_assumption_Omega} which is consistent with the H\"older control we obtained in view the H\"older assumptions on the data. Without the H\"older assumptions, Wang has to close energy estimates 
for $\Omega$ in $H^{2+\varepsilon^\prime}$ without appealing to a bound on\footnote{See Remark \ref{R:Bootstrap_assumptions_energy_estimates}.} $\norm{\bar{\partial} \Omega}_{L^1_tL^\infty_x}$ (and further without using  $\curl u \in H^{2+{\varepsilon}}(\Sigma_0)$; compare with \eqref{E:Energy_estimates_2_plus_epsilon_part_5} and the surrounding discussion). In view of the Hodge decomposition and \eqref{E:New_formulation_classical_transport_div_Omega}, this boils down to bounding
$\curl \curl \Omega \sim \curl \C$ in $H^{\varepsilon^\prime}$. Wang uncovers some special cancellations in the energy estimate for $\curl \curl \Omega$ in order to close the estimates without appealing to 
a bound on $\norm{\bar{\partial} \Omega}_{L^1_tL^\infty_x}$ (or $\Omega$ in $H^{2+\varepsilon}$). In order to illustrate this, consider an energy estimate for $\norm{\curl \curl \Omega}_{L^2(\Sigma_t)}$. We have
\begin{align}
\frac{1}{2} \partial_t \int_{\Sigma_t} |\curl \curl \Omega|^2 & \simeq
\int_{\Sigma_t} \B \curl \curl \Omega \cdot \curl \curl \Omega
\nonumber
\\
& \simeq \int_{\Sigma_t} \curl \B \curl \Omega \cdot \curl \curl \Omega.
\nonumber
\end{align}
At this point, it seems natural to invoke \eqref{E:New_formulation_classical_transport_curl_Omega} to replace
$\B \curl \Omega$. Instead, Wang observes the following. Direct computation gives
\begin{align}
\curl \B \curl \Omega \cdot \curl \curl \Omega \simeq 
(\updelta^{ij} \partial_j \partial_l u^k \partial_i \Omega_k + \updelta^{ij} \partial_i \B \hat{h} \partial_l \Omega_j)
\updelta^{lm} (\curl\curl \Omega)_m.
\label{E:Trilinear_structure_Wang}
\end{align}
(The term $\B \hat{h}$ comes from the specific form of $\C$; recall Remark \ref{R:Trading_C_for_curl_Omega}.)
Using $\dive \curl \curl \Omega=0$, 
\begin{align}
\updelta^{ij} \partial_j \partial_l u^k \partial_i \Omega_k \updelta^{lm} (\curl \curl\Omega)_m
\simeq \partial_l( \updelta^{ij} \partial_j  u^k \partial_i \Omega_k \updelta^{lm} (\curl \curl\Omega)_m),
\nonumber
\end{align}
which can be integrated by parts. Another factorization allows one to handle the second term in \eqref{E:Trilinear_structure_Wang}; schematically
\begin{align}
\bar{\partial}\B \hat{h} \bar{\partial} \Omega \cdot \curl \curl \Omega \simeq
\B( \bar{\partial} \hat{h} \bar{\partial} \Omega \cdot \curl \curl \Omega )
- \bar{\partial} \hat{h}  \B ( \bar{\partial} \Omega \cdot \curl \curl \Omega )
\nonumber
\end{align}
The first terms can be integrated by parts in time (since the initial energy identity is integrated over time).  This will produce only sufficiently lower-order terms that can thus be estimated without a time integral. The last term can be controlled using Eqs.~\eqref{E:New_formulation_classical_transport} and \eqref{E:New_formulation_classical_transport_curl_Omega}.

In order to obtain a top-order estimate, i.e., for $\norm{\curl \curl \Omega}_{H^{\varepsilon^\prime}(\Sigma_t)}$, one needs to apply the above argument to the once $\bar{\partial}^{\varepsilon^\prime}$ differentiated equations. This causes additional technical difficulties which in particular require enlarging the adopted functional framework via estimates in Besov spaces. On curious aspect of the difficulties in carrying out this top-order estimate is that the condition that $\varepsilon^\prime$ is strictly less than $\varepsilon$ is needed\footnote{One can, of course, consider data $\curl u \in H^{2+{\varepsilon}}(\Sigma_0)$ which will then automatically be in $H^{2+{\varepsilon^\prime}}(\Sigma_0)$. But only $\Omega \in H^{2+
\varepsilon^\prime}(\Sigma_t)$ will be controlled.}.

A substantial account of the ideas and techniques employed in 
\cite{Wang-2022} would require a much more in depth discussion that is beyond the goal of this review. But given our goal of explaining some key geometric-analytic techniques and our exposition on how they are used in the proof of Theorem \ref{T:Rough_classical_Euler}, it would be remiss not to discuss the important improvement of this Theorem obtained in \cite{Wang-2022} and point out some of its important ideas, as done above. In particular, one of the most challenging aspects in \cite{Wang-2022}, namely, control of the acoustic geometry under such limited regularity assumptions, involves several innovations whose discussion would be beyond the scope of this review.

Recall from the beginning of Sect.~\ref{S:Rough_solutions} that the irrotational classical compressible Euler equations are locally ill-posed for $(\hat{h},u) \in H^N$ with $N\leq 2$. On the other extreme from compressibility, the incompressible Euler equations are ill-posed for $\curl u \in H^N$ with $N \leq 1.5$ (in three spatial dimensions) due to a result of \cite{Bourgain-Li-2015}. Based on these considerations, \cite{Wang-2022} proposed the following conjecture: The classical compressible Euler equations in three spatial dimensions are locally well-posed for data $(\hat{h}, u, \curl u) \in H^N(\Sigma_0) \times H^N(\Sigma_0)\times H^{N-\frac{1}{2}}(\Sigma_0)$, with $N> 2$ (under natural non-degeneracy conditions like in assumption 3 of Theorem \ref{T:Rough_classical_Euler}).

\section{The relativistic Euler equations with a physical vacuum boundary}
\label{S:Vacuum_bry}

Consider a fluid within a domain that is not fixed but moves with the fluid motion, see Fig.~\ref{F:Domains_0_t} for an illustration.
Fluids of this type are called \textdef{free-boundary fluids.} Examples include a liquid drop or, more relevant for the relativistic case, a star. We will consider the case where the fluid body is in vacuum, like an isolated star.

\begin{figure}[ht]
\centering
  \includegraphics[scale=0.4]{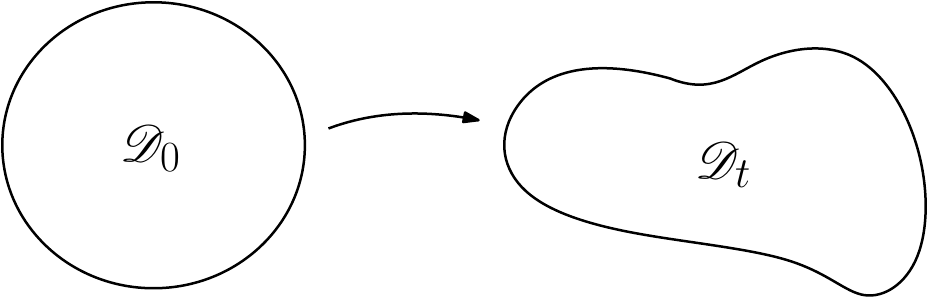}
  \caption{Illustration of moving domain changing its shape
  from $t=0$ to $t>0$ due to the fluid motion.}
  \label{F:Domains_0_t}
\end{figure}

Denoting by $\Md_t$ the region occupied by the fluid at time $t$, the dynamics of the fluid is defined in the spacetime region
\begin{align}
\Md := \bigcup_{0 \leq t < T} \{ t \} \times \Md_t,
\label{E:Moving_domain_spacetime}
\end{align}
for some $T>0$, known as the \textdef{moving domain} (see Fig.~\ref{F:Moving_domain}). The fluid's \textdef{free boundary} (also known as the \textdef{moving boundary} or the \textdef{free interface}) is
\begin{align}
\Fb := \bigcup_{0 \leq t < T} \{ t \} \times \Fb_t, 
\, \Fb_t := \partial \Md_t.
\nonumber
\end{align}
Note that $\Md$ has to be determined alongside the fluid motion, i.e., it cannot be freely prescribed.

\begin{figure}[ht]
\centering
  \includegraphics[scale=0.4]{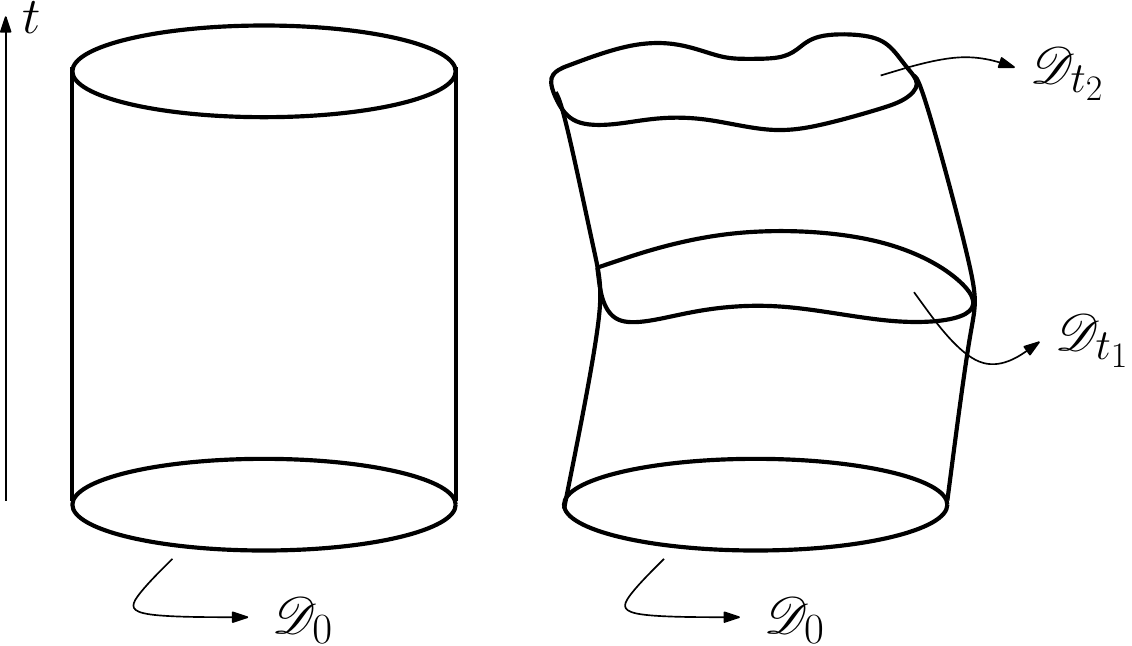}
  \caption{Illustration of the moving domain $\Md$ on the right. The picture on the left shows the domain $[0,T] \times \mathscr{D}_0$ used in standard (i.e., non-free-boundary) initial-boundary value problems.}
  \label{F:Moving_domain}
\end{figure}

The \textdef{free-boundary relativistic Euler equations}
are the relativistic Euler equations defined on a moving domain $\Md$. In this case, we have to impose the following boundary conditions that come from physical considerations\footnote{We continue to use the notation and conventions of Sect.~\ref{S:Relativistic_Euler}.} (see
\citealt{Oppenheimer:1939ne,Tolman:1934za,Tolman:1939jz}):
\begin{align}
\left. p \right|_\Fb = 0,
\label{E:Bry_condition_pressure}
\end{align}
and
\begin{align}
u \in T \Fb,
\label{E:Bry_condition_velocity}
\end{align}
where $T \Fb$ is the tangent bundle of $\Fb$. 
Condition \eqref{E:Bry_condition_pressure} says that the pressure has to vanish at the fluid-vacuum interface. We observe that \eqref{E:Bry_condition_pressure} can be obtained upon imposing that $\nabla_\alpha \mathcal{T}^{\alpha\beta}=0$, where $\mathcal{T}$ is given by \eqref{E:Energy_momentum_perfect}, holds across the free boundary in a distributional sense. Thus, \eqref{E:Bry_condition_pressure} is necessary if the fluid is coupled to Einstein's equations.
Condition \eqref{E:Bry_condition_velocity} says that $\Fb_t$ is advected by the fluid, i.e., $\Fb_t$ moves with speed equal to the normal component of the fluid velocity on the boundary.

Let us assume from now on that we have a \textdef{barotropic} equation of state, meaning that the pressure is a function of the density only, $p = p(\varrho)$. Then \eqref{E:Bry_condition_pressure} gives a condition for $\varrho$ on the boundary. There are two distinct cases to consider, depending on the behavior of $\left.\varrho\right|_\Fb$:
\begin{itemize}
\item $\left.\varrho\right|_\Fb \geq \, \text{constant} \, > 0$, in which case the fluid is called a \textdef{liquid.}
\item $\left.\varrho\right|_\Fb = 0$, in which case the fluid is called a \textdef{gas.}
\end{itemize}
Note that \eqref{E:Bry_condition_pressure} holds in both cases. The liquid and gas cases (whose names are more or less self-explanatory) are very different problems. A key difference is that the relativistic Euler equations degenerate on the boundary in the case of a gas (since 
$\left. (p+\varrho)\right|_\Fb = 0$, see equations
\eqref{E:Projected_relativistic_Euler_eq_full_system}) but not in the case of a liquid 
(since 
$\left. (p+\varrho)\right|_\Fb > 0$). Here, we will consider the case of a gas. In this case, $\Fb$ is also known as a \textdef{vacuum boundary.}

In the gas case, $\Md_t$ is given by
\begin{align}
\Md_t = \{ (\tau,x) \, | \, \tau = t, \varrho(t,x) > 0 \}.
\end{align}
Moreover, in the gas case, we also impose
\begin{align}
\left. c_s^2 \right|_\Fb = 0,
\label{E:Bry_condition_sound_speed}
\end{align}
which is related to the fact that sound waves cannot propagate in vacuum.

\begin{remark}
The condition \eqref{E:Bry_condition_sound_speed} implies that the sound cones degenerate to the flow lines on the boundary. Thus, the free-boundary relativistic Euler equations in the case of a gas is not only a system with multiple characteristic speeds; it also has degenerate characteristics.
\end{remark}

It turns out that the \emph{decay rate} of $c_s^2$ near $\Fb_t$ plays a crucial role in the study of the free-boundary relativistic Euler equations for a gas (see Remark \ref{R:Decay_explanation}). There is essentially only one decay rate that is consistent with the evolution of a free-boundary relativistic perfect fluid in the case of a gas, namely, that $c_s^2$ be comparable to the distance to the boundary (for points near the boundary), i.e.,
\begin{align}
c_s^2(t,x) \approx \operatorname{dist}(x,\Fb_t),
\label{E:Physical_vacuum_bry_condition}
\end{align}
for $x \in \Md_t$, $x$ near $\Fb_t$, where $\operatorname{dist}$ is the distance relative to the spacetime metric restricted to $\Sigma_t$. A decay faster than \eqref{E:Physical_vacuum_bry_condition} leads to a boundary that does not accelerate. There are physically relevant cases where the boundary acceleration vanishes, with the fluid particles at the boundary essentially free falling. But in most situations of interest, including the case of rotating stars, the boundary acceleration is non-zero. A slower decay rate than \eqref{E:Physical_vacuum_bry_condition} leads to a very singular problem, potentially causing an infinite acceleration of the boundary.

Condition \eqref{E:Physical_vacuum_bry_condition} is known as the \textdef{physical vacuum boundary condition.} 
When satisfying \eqref{E:Physical_vacuum_bry_condition},
the free-boundary relativistic Euler equations 
are known as the \textdef{relativistic Euler equations with a physical vacuum boundary.} From the point of view of the Cauchy problem that we will be investigating, condition \eqref{E:Physical_vacuum_bry_condition} should be viewed as a condition on the initial data that is then propagated by the flow.

A standard equation of state used in the study of a gas with a free-boundary is\footnote{See references in \cite{Disconzi-Ifrim-Tataru-2022}.}
\begin{align}
p(\varrho) = \varrho^{\kappa+1}, \, \kappa > 1,
\label{E:Equation_state_free_boundary}
\end{align}
which we henceforth adopt. From now on, we will focus on the class of solutions satisfying the physical vacuum boundary condition \eqref{E:Physical_vacuum_bry_condition}.

\begin{remark}
\label{R:Decay_explanation}
In order to understand the origins of \eqref{E:Physical_vacuum_bry_condition}, let us proceed formally and carry out the following heuristic argument. Assume that $c_s^2$ decays, near the free boundary, as a power of the distance to the boundary, 
\begin{align}
c_s^2(t,x) \approx (\operatorname{dist}(x,\Fb_t))^\beta,
\label{E:Decay_sound_speed_distance_power}
\end{align}
for some $\beta \in \mathbb{R}$. This assumption is natural because $\operatorname{dist}$ is a natural scale to consider since away from the boundary we can use finite speed of propagation to localize the problem and treat it as the standard (non-free boundary) relativistic Euler equations. Alternatively, we can imagine a Taylor expansion for $c_s^2$ near $\Fb_t$ with coordinates such that $x^3 = \operatorname{dist}$. Then, 
using \eqref{E:Projected_relativistic_Euler_eq_full_system_momentum}, we find the acceleration to be
\begin{align}
\begin{split}
a_\alpha & = u^\mu \nabla_\mu u_\alpha = -\frac{ \proj^\mu_\alpha \nabla_\mu p}{p+\varrho}
\sim c_s^2 \frac{\partial \varrho}{\varrho + \varrho^{\kappa+1}} 
\sim \frac{c_s^2 \partial \varrho}{\varrho} 
\sim \frac{ \operatorname{dist}^\beta \operatorname{dist}^{\frac{\beta}{\kappa} -1 } \partial \operatorname{dist} }{\operatorname{dist}^\frac{\beta}{\kappa}}
\\
& \sim   \operatorname{dist}^{\beta - 1},
\end{split}
\nonumber
\end{align}
where we used that $\varrho^\kappa \sim c_s^2 \sim \operatorname{dist}^\beta$ and $\partial \operatorname{dist} = O(1)$. Thus, we find that the acceleration satisfies
\begin{align}
\left. a \right|_{\Fb} 
= 
\begin{cases}
0 & \text{ if } \beta > 1,
\\
\text{finite} \neq 0 & \text{ if } \beta = 1,
\\
\infty & \text{ if } \beta < 1.
\end{cases}
\nonumber
\end{align}
The third and first conditions are not very physical (zero boundary acceleration would not allow the fluid to rotate, as stars do). Thus we should have $\beta =1$, which is condition \eqref{E:Physical_vacuum_bry_condition}.
\end{remark}

\begin{remark}
\label{R:Self_contained_bry_evolution}
The physical vacuum boundary condition implies that linear waves with speed $c_s$ reach the boundary in finite time. Indeed, if we imagine for simplicity the boundary located at $x=0$ and the fluid region as $x>0$, in one dimension, we see the speed
\begin{align}
\frac{dx}{dt} = c_s \sim \sqrt{x},
\nonumber
\end{align}
is integrable\footnote{If we write $\frac{dx}{dt} = c_s \sim x^\beta$, then the acceleration is
$\frac{d^2 x}{dt^2} \sim x^{\beta-1} \frac{dx}{dt} \sim x^{2\beta -1}$, which will be finite and non-zero at $x=0$ precisely for $\beta = \frac{1}{2}$, which is another way of recovering the physical vacuum boundary condition.}.
Thus, the motion of the boundary is strongly coupled with the bulk evolution and cannot be viewed as a self-containing evolution at leading order\footnote{We mention this because it is not uncommon in free-boundary problems for the boundary dynamics to decouple from the interior evolution, as it happens, for example, in the water-waves problem, see \cite{Lannes-2005,Lannes-Book-2013,Wu-2009}.}.
\end{remark}

Our goal is to establish local well-posedness for the 
Cauchy problem for the relativistic Euler equations with a physical vacuum boundary. We need to introduce some further notions before we can state the main theorem.

\begin{assumption}
For simplicity, we will henceforth take the spacetime metric to be the Minkowski metric. All the main difficulties of the problem are already present in this case (see Remark \ref{R:Minkowski_metric_diagonal}).
\end{assumption}

\subsection{Diagonalization}
In order to carry out the analysis, our first step is to re-express the relativistic Euler equations as equations
with good structures, where ``good'' means with respect to the behavior of the fluid near the vacuum boundary. This will be accomplished by finding a set of \emph{good nonlinear variables} that \emph{diagonalize} the system with respect to the \textdef{relativistic material derivative,} given by
(note that $u^0 \neq 0$ in view of $u^\mu u_\mu = -1$; compare also with \eqref{E:v_0_explicit_v_i} below)
\begin{align}
D_t := \partial_t + \frac{u^i}{u^0} \partial_i.
\label{E:Material_derivative_relativistic}
\end{align}
$D_t$ plays a role similar to the vectorfield $\B$ in the classical compressible Euler equations (compare \eqref{E:Material_derivative_relativistic} and \eqref{E:Material_derivative_classical})

We want to have a system diagonal with respect to $D_t$ because our method of constructing solutions will involve a type of Euler's method, for which we require the system to have roughtly the form\footnote{We can always diagonalize the system by algebraically solving for the time derivatives in terms of spatial derivatives  in Eqs.~\eqref{E:Projected_relativistic_Euler_eq_full_system}, but such a procedure will not produce equations with good structure.}
$\partial_t (LHS) = \partial_x (RHS)$, except that it is $D_t$, and not $\partial_t$, that 
should be thought of as the right time derivative in this problem, as it intrinsically tied to the evolution in view of characteristics of the Euler system\footnote{This is readily seen from $D_t = \frac{1}{u^0} u^\mu \partial_\mu$.}, especially
to the characteristics on the boundary.

\begin{remark}
In many situations one can work with $D_t$ or $u^\mu \partial_\mu$ (with the latter also called a relativistic material derivative), the difference in estimates between both cases being lower order. In our approach, however, there are some very delicate estimates that require very precise structures and which might not be true if we used $u^\mu \partial_\mu$ instead.
\end{remark}

Let us introduce the desired new variables:
\begin{subequations}{\label{E:Diagonal_variables}}
\begin{align}
v & := (1+\varrho^\kappa)^{1+\frac{1}{\kappa}} u,
\label{E:Velocity_diagonal}
\\
r & := \frac{\kappa +1}{\kappa} \varrho^\kappa.
\label{E:Sound_speed_diagonal}
\end{align}
\end{subequations}
The quantities $v$ and $r$ are the desired \emph{good nonlinear variables,} and we will take them as our primary variables for the problem. Observe that $r$ is, up to a constant factor $\frac{1}{\kappa}$ that is introduced for convenience, the sound speed squared. Given the key role played by $c_s^2$ in view of the physical vacuum boundary condition \eqref{E:Physical_vacuum_bry_condition}, it is not surprising that we take it to be one of the primary variables. After deriving some consequences of these new variables,
we will motivate their definition.

The physical vacuum boundary condition \eqref{E:Physical_vacuum_bry_condition} can be equivalently stated in terms of $r$,
\begin{align}
r(t,x) \approx \operatorname{dist}(x,\Fb_t),
\label{E:Physical_vacuum_bry_condition_diagonal}
\end{align}
Also, $\Md_t$ can equivalently be characterized by
\begin{align}
\Md_t =  \{ (\tau,x) \, | \, \tau = t, r(t,x) > 0 \}.
\label{E:Moving_domain_diagonal}
\end{align}

In view of \eqref{E:Velocity_normalization} and \eqref{E:Sound_speed_diagonal}, observe that $v$ satisfies the constraint
\begin{align}
v^\alpha v_\alpha = - \left(1+\frac{\kappa}{\kappa+1} r\right)^{2+\frac{2}{\kappa}}.
\label{E:Velocity_normalization_diagonal}
\end{align}

In light of \eqref{E:Velocity_normalization_diagonal}, it suffices to consider an evolution for the spatial components of $v$, i.e., $v^i$, $i=1,2,3$. In particular, we will adopt the following notation.
\begin{notation}
\label{N:v_0}
From now on, references to $v$ will mean only to its spatial part
$(v^1,v^2,v^3)$, with the implicit understanding that $v^0$ is written in terms of $v^i$ and $r$ via \eqref{E:Velocity_normalization_diagonal}, i.e., 
\begin{align}
v^0 = \sqrt{ \left(1+\frac{\kappa}{\kappa+1} r\right)^{2+\frac{2}{\kappa}} + v^i v_i }.
\label{E:v_0_explicit_v_i}
\end{align}
Thus, when we say, e.g., that ``$v$ satisfies P'' we mean that ``$(v^1,v^2,v^3)$ satisfies P."
\end{notation}

Using definitions \eqref{E:Diagonal_variables}, Eqs.~\eqref{E:Projected_relativistic_Euler_eq_full_system_energy}--\eqref{E:Projected_relativistic_Euler_eq_full_system_momentum}, and recalling that the pressure is given by \eqref{E:Equation_state_free_boundary}, we obtain that $(r,v)$ satisfy
\begin{subequations}{\label{E:Diagonal_Euler_system}}
\begin{align}
D_t r + r (\bar{H}^{-1})^{ij} \partial_i v_j + r a_1 v^i \partial_i r & = 0,
\label{E:Diagonal_Euler_system_sound_speed}
\\
D_t v_i + a_2 \partial_i r & =0,
\label{E:Diagonal_Euler_system_velocity}
\end{align}
\end{subequations}
where $\bar{H}^{-1}$ is the following\footnote{See Remark \ref{R:Inverse_sign_explicit} for the explicit use of $^{-1}$.} inverse Riemannian metric\footnote{One can verify that $\bar{H}$ indeed defines a Riemannian metric with the help of \eqref{E:Velocity_normalization_diagonal}.}
on $\Md_t$,
\begin{align}
(\bar{H}^{-1})^{ij} = \frac{\kappa}{a_0 v^0 }(1 + \frac{\kappa}{\kappa+1} r)\left(\updelta^{ij} - \frac{v^i v^j}{(v^0)^2}\right),
\label{E:H_bar_inverse}
\end{align}
$a_0$, $a_1$, and $a_2$ are smooth functions of $(r,v)$ that are $O(1)$ near $\Fb_t$ and $v^0$ is computed from \eqref{E:v_0_explicit_v_i} (recall that $\updelta$ denotes the Euclidean metric).
For the most part, the precise form of these coefficients $a_0$, $a_1$, and $a_2$ will not be important. Only the fact that 
\begin{align}
\label{E:a_2_bounded_from_below}
a_2 \geq \, \text{constant} > 0 
\end{align}
is of relevance, and the specific form of $a_0$, given by
\begin{align}
a_0 := 1 - c_s^2 \frac{v^i v_i}{(v^0)^2} = 1 - \kappa r \frac{v^i v_i}{(v^0)^2},
\label{E:a_0}
\end{align}
will be used in Sect.~\ref{S:Role_weights}.
 Observe that 
\eqref{E:Material_derivative_relativistic} can be written as
\begin{align}
D_t = \partial_t + \frac{v^i}{v^0} \partial_i,
\label{E:Material_derivative_relativistic_v}
\end{align}
so that all explicit dependence on $(\varrho,u)$ has been eliminated from \eqref{E:Diagonal_Euler_system}.
Equations
\eqref{E:Diagonal_Euler_system} are simply the relativistic Euler equations written in terms of $r$ and $v$, and they are the main equations we will investigate.

\begin{remark}
Since $\bar{H}^{-1}$ is an inverse Riemannian metric, $(\bar{H}^{-1})^{ij} \partial_i v_j$ is a divergence of  $v$. This is consistent with the appearance of the spacetime divergence of $u$ in \eqref{E:Projected_relativistic_Euler_eq_full_system_energy}.
$\bar{H}^{-1}$ can be related to a rescaled version of the acoustical metric but this will not be needed here. We also note that $\bar{H}^{-1}$ is pointwise equivalent to the (inverse) Euclidean metric,
\begin{align}
\label{E:H_bar_equivalent_Euclidean}
\bar{H}^{-1} \approx \updelta^{-1}.
\end{align}
\end{remark}

Let us now explain the choices \eqref{E:Diagonal_variables} made to diagonalize the system. It is natural to consider rescalings of the
velocity in the form 
\begin{align}
v = f(\varrho) u,
\label{E:v_definition_general}
\end{align}
where $f$ is to be determined. Using \eqref{E:Projected_relativistic_Euler_eq_full_system}, we can derive the following evolution equation for $v$,
\begin{align}
\frac{p+\varrho}{f} u^\mu \partial_\mu v^i + c_s^2 g^{i \mu} \partial_\mu \varrho 
+ (-\frac{f^\prime}{f}(\varrho + p) + c_s^2 ) u^i u^\mu \partial_\mu \varrho = 0.
\label{E:Attempt_diagonalize_velocity}
\end{align}

We want \eqref{E:Attempt_diagonalize_velocity} to be an evolution equation for $v$ diagonal with respect to $D_t$. In particular, this means that only $v$ should be differentiated with respect to $\partial_t$.
We note that the term $\frac{p+\varrho}{f} u^\mu \partial_\mu v^i$ is a multiple of $D_t v^i$ and thus this term has the desired property. The term $g^{i \mu} \partial_\mu \varrho$ seems to contain a $\partial_t \varrho$; however, we have $g^{i0} = 0$ since $g$ is the Minkowski metric and thus
$g^{i \mu} \partial_\mu \varrho $ contains only spatial derivatives (see Remark \ref{R:Minkowski_metric_diagonal}). The last term in \eqref{E:Attempt_diagonalize_velocity}, however,
contains a $\partial_t \varrho$ or, said differently, a multiple of $D_t \varrho$. Thus, we chose $f$ such that the term in parenthesis vanishes,
\begin{align}
-\frac{f^\prime}{f}(\varrho + p) + c_s^2 = 0.
\nonumber
\end{align}
Solving this ODE for $f$ produces 
\begin{align}
f(\varrho) = e^{\int \frac{c_s^2(\varrho)}{p(\varrho) + \varrho}\, d\varrho }.
\label{E:f_diagonalization_general}
\end{align}

The choice \eqref{E:f_diagonalization_general} also diagonalizes the evolution of $\varrho$, i.e., using 
\eqref{E:v_definition_general} and \eqref{E:f_diagonalization_general}, equations 
\eqref{E:Projected_relativistic_Euler_eq_full_system} give
\begin{align}
u^\mu \partial_\mu \varrho + \frac{p+\varrho}{f a_0 } (\updelta^{ij} - \frac{v^i v^j}{(v^0)^2})
\partial_i v_j - \frac{c_s^2}{a_0 (v^0)^2} f v^i \partial_i \varrho = 0.
 \label{E:Attempt_diagonalize_density}
\end{align}

Equations \eqref{E:Attempt_diagonalize_velocity} and \eqref{E:Attempt_diagonalize_density} hold for a general barotropic equation of state. With the choice \eqref{E:Equation_state_free_boundary},
\eqref{E:f_diagonalization_general} becomes
\begin{align}
f(\varrho) = (1+\varrho^\kappa)^{1+\frac{1}{\kappa}},
\nonumber
\end{align}
which gives \eqref{E:Velocity_diagonal}. With \eqref{E:Sound_speed_diagonal} and \eqref{E:Material_derivative_relativistic_v},
Eqs.~\eqref{E:Attempt_diagonalize_velocity} and \eqref{E:Attempt_diagonalize_density} 
then give \eqref{E:Diagonal_Euler_system}.

\begin{remark}
\label{R:Minkowski_metric_diagonal}
It seems that our argument relies critically on the fact that the spacetime metric is the Minkowski metric since we used that $g^{i0}=0$. However, this is not the case. Given any smooth Lorentzian metric $g$,
we can always choose coordinates such that $g^{i0} = 0$. In this case, the above argument still goes through. For a general metric, there will be extra terms in the estimates we will discuss coming from differentiation of $g$, but these will be lower terms order terms can be handled with similar ideas.
\end{remark}

\subsection{Vorticity} In the case of a barotropic fluid, $p = p(\varrho)$, Eqs.~\eqref{E:Relativistic_Euler_eq_full_system_energy_momentum} do not involve the baryon density $n$ and thus
\eqref{E:Relativistic_Euler_eq_full_system_baryon_charge} decouples, i.e., it can be solved
after a solution to \eqref{E:Relativistic_Euler_eq_full_system_energy_momentum} has been determined.
Defining the vorticity as in \eqref{E:Vorticity_definition}, however, will introduce $n$ back into the problem. It is convenient, therefore, to define the vorticity of a barotropic fluid as
\begin{align}
\tilde{\Omega} := d( f u ) = dv,
\label{E:Vorticity_definition_barotropic}
\end{align}
where $f$ is given by \eqref{E:f_diagonalization_general} and $v$ by \eqref{E:v_definition_general} ($d$ is the exterior derivative in spacetime).
We can think of $\tilde{\Omega}$ as a vorticity variable because it does satisfy an analogue 
of Lichnerowicz's equation \eqref{E:Lichnerowicz_equation}, namely,
\begin{align}
v^\mu \tilde{\Omega}_{\mu \alpha} = 0,
\label{E:Lichnerowicz_equation_barotropic}
\end{align}
and an analogue of the evolution equation \eqref{E:Vorticity_evolution}, namely
\begin{align}
v^\mu \partial_\mu \tilde{\Omega}_{\alpha\beta} + \partial_\alpha v^\mu \tilde{\Omega}_{\mu \beta}
+ \partial_\beta v^\mu \tilde{\Omega}_{\alpha \mu} = 0,
\label{E:Vorticity_evolution_barotropic}
\end{align}
where in contrast to \eqref{E:Vorticity_evolution} there are no terms on the RHS in view of the decoupling
between $n$ and the rest of the system (or, equivalently, decoupling of $s$). In particular, \eqref{E:Vorticity_evolution_barotropic} implies that $\tilde{\Omega}$ remains zero if zero initially.

In our case, since we are considering only the evolution for the spatial components $v^i$ (see \eqref{E:Diagonal_Euler_system_velocity}), it is natural to consider only the spatial components of 
$\tilde{\Omega}$ as well. Using \eqref{E:Lichnerowicz_equation_barotropic}, we can algebraically solve for the time components of $\tilde{\Omega}$ in terms of its spatial components,
\begin{align}
\tilde{\Omega}_{0j} = -\frac{v^i}{v^0} \Omega_{ij}.
\label{E:Vorticity_time_space_algebraically}
\end{align}
(Note that $\tilde{\Omega}_{00}=0$.) Using \eqref{E:Vorticity_time_space_algebraically} into \eqref{E:Vorticity_evolution_barotropic} yields
\begin{align}
D_t \tilde{\Omega}_{ij} 
+ \frac{1}{v^0} \partial_i v^k \tilde{\Omega}_{k j}
+ \frac{1}{v^0} \partial_j v^k \tilde{\Omega}_{ik} 
- \frac{1}{(v^0)^2} v^k \partial_i v^0 \tilde{\Omega}_{k j}
+ \frac{1}{(v^0)^2} v^k \partial_j v^0 \tilde{\Omega}_{ki} = 0, 
\label{E:Vorticity_evolution_barotropic_spatial}
\end{align}
where we recall Notation \ref{N:v_0}. We make the following important remark:

\begin{remark}
\label{R:Transport_estimates_implicit}
The estimates will discuss will focus on the wave-part of the system. Thus, all our estimates need
to be complemented by suitable estimates for the vorticity. Such vorticity estimates are not the main challenge in the argument and are obtained by direct (weighted) transport estimates using Eq.~\eqref{E:Vorticity_evolution_barotropic_spatial}, and thus they will be omitted. Alternatively, readers
can consider the case of an irrotational fluid $\tilde{\Omega}=0$ for simplicity\footnote{The condition
$\tilde{\Omega}=0$ is propagated by the flow in view of \eqref{E:Vorticity_evolution_barotropic_spatial}.}.
\end{remark}

Observe that our choice \eqref{E:f_diagonalization_general} was not motivated by the vorticity but 
by the entirely different question of canceling the term in parenthesis in \eqref{E:Attempt_diagonalize_velocity}. The fact that such choice gives precisely the function used to define $\tilde{\Omega}$ is a strong suggestion that \eqref{E:v_definition_general}, with $f$ given by
\eqref{E:f_diagonalization_general}, is a good variable to consider.

\subsection{Function spaces and control norms}
Observe that, as the original equations, Eqs.~\eqref{E:Diagonal_Euler_system} also degenerate on the boundary in view of the factors in $r$ and \eqref{E:Physical_vacuum_bry_condition_diagonal}. The degenerate nature of these equations suggests the use of weighted Sobolev spaces, with powers of the distance to the boundary as weights. The functional framework which we will use is introduced in this section, after which we will be able to state the main results in Sect.~\ref{S:LWP_vacuum_bry}.
We will provide more intuition for our choice of function spaces in Sect.~\ref{S:Energy_estimates_linearized}, but the basic idea is that they are constructed out of a natural energy satisfied by the linearized equations.

\subsubsection{Weighted norms} The basic weighted Sobolev spaces which we will employ are defined as follows.

\begin{definition}
\label{D:Weighted_Sobolev}
Let $r$ be a defining function for the domain $\Md_t$, i.e., 
$\Md_t = \{ r > 0 \}$ and the gradient of $r$ is non-zero on $\partial \Md_t$.
Given an integer $N\geq 0$ and a real number $\sigma>-\frac{1}{2}$,
we define $H^{N,\sigma}(\Md_t)$ as the space of distributions
on $\Md_t$ such that the norm
\begin{align}
\norm{ \mathsf{f} }^2_{H^{N,\sigma}(\Md_t)} :=
\sum_{|\alpha| \leq N} \norm{ r^\sigma \bar{\partial}^\alpha \mathsf{f}}_{L^2(\Md_t)}^2
\label{E:Weighted_Sobolev}
\end{align}
is finite. In \eqref{E:Weighted_Sobolev}, $\alpha$ are multi-indices and we recall that $\bar{\partial}$ denotes spatial derivatives (see Sect.~\ref{S:Notation_conventions}).
Using interpolation, we can define $H^{N,\sigma}(\Md_t)$ for
non-integer $N$.
\end{definition}

\begin{remark}
We are ultimately interested in the case when the weight $r$ in \eqref{E:Weighted_Sobolev} is a solution to \eqref{E:Diagonal_Euler_system}, which motivates the abuse of notation of calling it $r$ as well. 
However, it is important to define $H^{N,\sigma}$ without making explicit reference to such solution since, in the proof of local well-posedness, a solution to \eqref{E:Diagonal_Euler_system} is not yet available.
\end{remark}

\begin{definition}
\label{D:Weighted_space_diagonal}
We define $\Hspace^{2N}(\Md_t)$ as the product space of real valued maps of class $H^{2N,\frac{1-\kappa}{2\kappa}+N}$
and vectorfields of class
$H^{2N,\frac{1-\kappa}{2\kappa}+N+\frac{1}{2}}$, i.e.,
\begin{align}
\Hspace^{2N}(\Md_t) : = 
H^{2N,\frac{1-\kappa}{2\kappa}+N}(\Md_t) \times
H^{2N,\frac{1-\kappa}{2\kappa}+N +\frac{1}{2}}(\Md_t),
\label{E:Weighted_space_diagonal}
\end{align}
with the understanding that elements in the first component are real valued and elements in the second component are vectorfields.
\end{definition}

\begin{notation}
\label{N:Weighted_space_components}
We will often omit $\Md_t$ and write simply $H^{N,\sigma}$ and 
$\Hspace^{2N}$ when there is no risk of confusion. When $\Hspace^{2N}$ is written on a single component 
of $(s,w)$ it means $H^{2N,\frac{1-\kappa}{2\kappa}+N}(\Md_t)$ on the first component and $H^{2N,\frac{1-\kappa}{2\kappa}+N +\frac{1}{2}}(\Md_t)$ on the second one, i.e. 
$\norm{s}_{\Hspace^{2N}} = \norm{s}_{H^{2N,\frac{1-\kappa}{2\kappa}+N}(\Md_t)}$ on the first component and $\norm{w}_{\Hspace^{2N}} = \norm{w}_{H^{2N,\frac{1-\kappa}{2\kappa}+N +\frac{1}{2}}(\Md_t)}$.
\end{notation}
 
\begin{remark}
In Definition \ref{D:Weighted_space_diagonal}, we wrote the number of derivatives as $2N$ for convenience, but the total number of derivatives need not to be even as $N$ can be any non-negative number.
\end{remark}

We will be seeking solutions with regularity $(r,v) \in \Hspace^{2N}$ for appropriate $N$.
The relation between weights and derivatives in the definition of 
$\Hspace^{2N}$ is motivated by the following factors. The underlying wave evolution is at leading order  governed by a wave operator roughly of the form (see, however, Remark 
\ref{R:Elliptic_estimates_vacuum_bry}; compare also with \eqref{E:Wave_equation_density})
\begin{align}
D_t^2 - r \Delta,
\label{E:Wave_operator_vacuum_bry_leading_order}
\end{align}
so that each two derivatives are paired with one power of $r$. This relates the number of derivatives $2N$ with  weights with power  $N$ in
$\Hspace^{2N}$. The powers in $\frac{1-\kappa}{2\kappa}$ come from the structure of the linearized equations 
(which play an important role in our approach, see Sects.~\ref{S:LWP_vacuum_bry} and \ref{S:Energy_estimates_linearized}). The ``extra'' power of $\frac{1}{2}$ in the weights for $v$ as compared to the weights for $r$ come from the fact that the
derivatives of $v$ in \eqref{E:Diagonal_Euler_system_sound_speed} are weighted by $r$, whereas the derivatives of $r$ in \eqref{E:Diagonal_Euler_system_velocity} are not (see Sect.~\ref{S:Energy_estimates_linearized} for more details).

\subsubsection{Scaling analysis\label{S:Scaling}}
Equations \eqref{E:Diagonal_Euler_system} do not admit a scaling law.
However, it is possible to define a scaling law for the leading-order dynamics near the boundary. 
To see this, we ignore all terms of $O(1)$ in \eqref{E:Diagonal_Euler_system}, in which case the system  \eqref{E:Diagonal_Euler_system}  simplifies to
\begin{subequations}{\label{E:Scaling_analysis_simplified_eq_order_1}}
\begin{align}
(\partial_t + v^i \partial_i) r + r \updelta^{ij} \partial_i v_j + r v^i \partial_i r & \simeq 0,
\label{E:Scaling_analysis_simplified_eq_order_1_r}
\\
(\partial_t + v^j \partial_j) v_i + \partial_i r & \simeq 0.
\label{E:Scaling_analysis_simplified_eq_order_1_v}
\end{align}
\end{subequations}
As we will see in Sect.~\ref{S:Energy_estimates_linearized}, the term $r v^i \partial_i r$ in \eqref{E:Scaling_analysis_simplified_eq_order_1_r} can be handled as a type of perturbation from the point of view of our energies. This will be the case because in the derivation of the energy estimates, we multiply \eqref{E:Scaling_analysis_simplified_eq_order_1_r} by a power of $r$, producing a weighted-type $L^2$ control of $r$, see Definition \ref{D:Weighted_space_diagonal}; but the term 
$r v^i \partial_i r$ has an extra power of $r$ to spare. Taking this into consideration, we also ignore the last term on LHS of \eqref{E:Scaling_analysis_simplified_eq_order_1_r}, in which case the equations further simplify to
\begin{subequations}{\label{E:Scaling_analysis_simplified_eq}}
\begin{align}
(\partial_t + v^i \partial_i) r + r \updelta^{ij} \partial_i v_j & \simeq 0,
\label{E:Scaling_analysis_simplified_eq_r}
\\
(\partial_t + v^j \partial_j) v_i + \partial_i r & \simeq 0.
\label{E:Scaling_analysis_simplified_eq_order}
\end{align}
\end{subequations}
By the above considerations, Eqs.~\eqref{E:Scaling_analysis_simplified_eq} should captuer the leading-order dynamics near the free boundary. Equations \eqref{E:Scaling_analysis_simplified_eq}
admit the following scaling law
\begin{align}
(r(t,x), v(t,x))  \mapsto (r_\uplambda(t,x), v_\uplambda(t,x)) :=
 (\uplambda^{-2}(\uplambda t, \uplambda^2 x), \uplambda^{-1} v(\uplambda t, \uplambda^2 x) ),
\label{E:Scaling_law}
\end{align}
for any $\uplambda > 0$; i.e., if $(r,v)$ solves \eqref{E:Scaling_analysis_simplified_eq} so does
$(r_\uplambda, v_\uplambda)$. With this scaling law, we can in the usual fashion define
 \textdef{critical spaces} $\Hspace^{2N_0}$ where $N_0$ is given by\footnote{The $3$ in $2N_0$ corresponds to the dimension of space. In general $d$ dimensions, we have 
$2 N_0 := d + 1 + \frac{1}{\kappa}$.}
\begin{align}
2 N_0 := 3 + 1 + \frac{1}{\kappa}.
\nonumber
\end{align}
Observe that $2N_0$ is in general not an integer, with the spaces $\Hspace^{2N}$ defined for non-integer $2N$ via interpolation.

We recall that critical homogeneous spaces are defined as invariant under the underlying scaling law. In our case, then, $\Hspace^{2N_0}$ is such that the homogeneous part of the $\Hspace^{2N_0}$-norm
(the terms with $|\alpha|=2N_0$ in \eqref{E:Weighted_Sobolev} for integer $2N_0$) is invariant under
\eqref{E:Scaling_law}. The basic heuristic principle in the study of quasilinear problems is that 
local well-posedness should not hold for data with regularity below the critical scaling value, $2N_0$ in our case. This is because if such result were to hold, then by a simple rescaling of solutions one could transform a small data, small time, result into a large data, large time one, thus deriving a stronger result from a seemingly weaker one. Whether or not this is in fact the case depends on the particular problem being studied. But this heuristic principle is very useful in determining regularity values for which one should expect to be able to determine local well-posedness. Moreover, in quasilinear problems experience suggests that we should not expect local well-posedness at or right above scaling either, except perhaps for very specific equations that have very special structures. In our case, however, we will be able to get close enough to scaling so that we our solutions can legitimately be classified as rough solutions to the 
problem.

\subsubsection{Control norms}
In order to establish a continuation criterion for solutions, we next introduce some time-dependent control norms.
For a solution $(r,v)$ to \eqref{E:Diagonal_Euler_system}, define
\begin{subequations}{\label{E:Control_norm}}
\begin{align}
A &:= \norm{ \bar{\nabla} r - \mathsf{N}}_{L^\infty(\Md_t)}
+ \norm{v}_{\dot{C}^\frac{1}{2}(\Md_t)},
\label{E:Control_norm_A}
\\
B &:= A + \norm{\bar{\nabla} r}_{\tilde{C}^\frac{1}{2}(\Md_t)} + 
\norm{\bar{\nabla} v}_{L^\infty(\Md_t)}.
\label{E:Control_norm_B}
\end{align}
\end{subequations}
Above, the notation is as follows. $\bar{\nabla}$ is the spatial gradient, $\dot{C}^\frac{1}{2}$ is the H\"older semi-norm of exponent $\frac{1}{2}$, and $\tilde{C}^{\frac{1}{2}}$ is the following semi-norm
\begin{align}
\label{E:C_tilde_half_norm}
\norm{\mathsf{f}}_{\tilde{C}^\frac{1}{2}(\Md_t)}
:=
\sup_{\stackrel{x,y \in \Md_t}{x\neq y}}  \frac{|\mathsf{f}(x)-\mathsf{f}(y)|}{(r(x))^\frac{1}{2}+(r(y))^\frac{1}{2} + |x-y|^\frac{1}{2}}.
\end{align}
Thus, $\norm{\bar{\nabla} r}_{\tilde{C}^\frac{1}{2}(\Md_t)}$ behaves roughly as the $\dot{C}^\frac{3}{2}$ semi-norm, but it is weaker as it uses only one derivative away from the boundary. Finally, $\mathsf{N}$ is a vectorfield constructed as follows. 
Since we can localize the problem using the finite-propagation-speed property, it suffices to work in small neighborhoods of the boundary (away from boundary, the system is not degenerate so standard estimates apply). Fix a point $x_0 \in \Gamma_t$. In a small neighborhood of $x_0$, we can construct a vectorfield $\mathsf{N}$ such that $\mathsf{N}(x_0) = \bar{\nabla} r (x_0)$, where $x_0 \in \Fb_t$, see Fig.~\ref{F:N}. The reason
for introducing $\mathsf{N}$ is that we can then make $A$ small by working in sufficiently small neighborhoods, whereas we cannot make
$\norm{\bar{\nabla} r}_{L^\infty(\Md_t)}$ small by localization or scaling arguments. We will need to make $A$ small to close the estimates (see Remark \ref{R:Elliptic_estimates_vacuum_bry} and Sect.~\ref{S:Estimates_solutions_diagonal}).

\begin{figure}[ht]
\centering
  \includegraphics[scale=1]{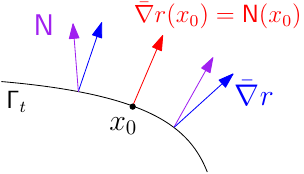}
  \caption{Illustration of construction of the vectorfield $\mathsf{N}$.}
  \label{F:N}
\end{figure}

\begin{remark}
\label{R:Working_neighborhood_boundary}
We will make use, often silently, of the above mentioned fact that we can localize the problem in a small neighborhood of a boundary point. Thus, we will assume that $r$ is sufficiently small whenever needed.
In addition, when needed, we can assume to be working in the neighborhood of a point $x_0$ such that
$\bar{\nabla} r(x_0) = \mathsf{N}(x_0)$, as just explained. 
\end{remark}

\begin{remark}
Observe that the construction of $\mathsf{N}$ is not unique, but different constructions will lead to equivalent norms $A$. We also note that $\mathsf{N}$ is a purely spatial vectorfield, i.e., $\mathsf{N} = (\mathsf{N}^1, \mathsf{N}^2, \mathsf{N}^3)$.
\end{remark}

The norms $A$ and $B$ are associated with the spaces $\Hspace^{2N_0}$ and $\Hspace^{2N_0 + 1}$ in view of the following embeddings:
\begin{align}
A & \lesssim \norm{ (r,v) }_{\Hspace^{2N}}, \, 2N > 2N_0,
\nonumber 
\\
B & \lesssim \norm{ (r,v) }_{\Hspace^{2N}}, \, 2N > 2N_0 + 1.
\nonumber
\end{align}

\begin{remark}
Definitions \eqref{E:Control_norm} can also be applied for a general defining function $r$ and vectorfield $v$ that are not necessarily solutions to \eqref{E:Diagonal_Euler_system}. This is important for the iteration leading to existence of solutions.
\end{remark}

\subsection{Local well-posedness and continuation criterion}
\label{S:LWP_vacuum_bry}

We are now ready to state the main results.

\begin{theorem}[\citealt{Disconzi-Ifrim-Tataru-2022}]
\label{T:LWP_vacuum_bry}
Consider Eqs.~\eqref{E:Diagonal_Euler_system} in $\Md$,
where $\Md$ is given by \eqref{E:Moving_domain_spacetime} with $\Md_t$ given by \eqref{E:Moving_domain_diagonal}. Define the state space
\begin{align}
\mathbf{H}^{2N} := \{ (r,v) \, | \, (r,v) \in \Hspace^{2N} \}.
\nonumber
\end{align}
Then, the Cauchy problem for equations 
\eqref{E:Diagonal_Euler_system} is locally well-posed in $
\mathbf{H}^{2N}$ for data $(\mathring{r},\mathring{v}) \in 
\mathbf{H}^{2N}$ provided that the physical vacuum boundary condition
\begin{align}
\mathring{r}(x) \approx \operatorname{dist}(x,\Fb_0), \, \text{ where } \, \Md_0 := \{ \mathring{r} > 0 \}, \, \Fb_0 = \partial \Md_0,
\nonumber
\end{align}
is satisfied\footnote{The physical vacuum boundary condition is already incorporated into the assumption $\mathring{r} \in \mathbf{H}^{2N}$, but we state it explicitly because of its importance.}
and
\begin{align}
2N > 2N_0 + 1, \, 2N_0 = 3 + 1 +\frac{1}{\kappa}.
\nonumber
\end{align}
\end{theorem}

We make the following remarks about Theorem \ref{T:LWP_vacuum_bry}:
\begin{itemize}
\item Local well-posedness above is meant in the usual Hadamard sense: existence and uniqueness of solutions, with $(r,v) \in 
C^0([0,T], \mathbf{H}^{2N})$ for some $T>0$ and continuous dependence of solutions on the data in this topology. For continuous dependence on the data, we need to define an appropriate topology on $\mathbf{H}^{2N}$ to compare solutions, since  different solutions are defined in different domains. We refer to \cite{Disconzi-Ifrim-Tataru-2022} for the definition.
\item Observe that the regularity of solutions obtained in Theorem \ref{T:LWP_vacuum_bry} is for data very close to scaling, thus qualifying for what we would call \emph{rough solutions} to the problem. These
rough solutions are obtained as unique limits of smooth solutions.
\end{itemize}

Theorem \ref{T:LWP_vacuum_bry} established the first local existence and uniqueness result for the relativistic Euler equations with a physical vacuum boundary. More precisely, 
local existence and uniqueness has been established in $1+1$ dimensions in \cite{Oliynyk-2012-1}. In this case, the boundary are just disconnected points and the main difficulties of the problem are absent. A priori estimates for the relativistic Euler equation 
with a physical vacuum boundary in $3+1$ dimensions have been obtained by \cite{Jang-LeFloch-Masmoudi-2016} and \cite{Hadzic-Shkoller-Speck-2019}. 

\begin{remark}
\cite{Jang-LeFloch-Masmoudi-2016} and \cite{Hadzic-Shkoller-Speck-2019} use the formalism of Lagrangian coordinates, whereas Theorem \ref{T:LWP_vacuum_bry} is established entirely in Eulerian coordinates. 
\end{remark}

\begin{remark}
If one considers the case of a gas not satisfying the physical vacuum boundary condition, when $c_s^2$ decays faster than the distance to the boundary, the motion of the boundary decouples from the bulk evolution (see Remark \ref{R:Self_contained_bry_evolution}) and several difficulties are not present. The problem in this case has been investigated by several authors, including situations with coupling to Einstein's equations (see \citealt{Rendall-1992,Brauer-Karp-2011,Brauer-Karp-2014,LeFloch-Ukai-2009} and reference therein). Perturbations of spherically symmetric static Einstein--Euler system with a physical vacuum boundary have been studied by \cite{Makino-1998,Makino-2016,Makino-2017,Makino-2018}.
In the case of a liquid, the problem has been recently addressed in the works by
\cite{Oliynyk-2017,Oliynyk-2019-arxiv,
Miao-Shahshahani-Wu-2021,
Miao-Shahshahani-2024,
Ginsberg-2019,Ginsberg-Lindblad-2023}. We refer to Sect.~1.6 of \cite{Disconzi-Ifrim-Tataru-2022} for a detailed discussion of previous works.
\end{remark}

\medskip

\noindent \emph{Proof of Theorem \ref{T:LWP_vacuum_bry}.} 
We will provide an outline of the proof. Then, in Sects.~\ref{S:Energy_estimates_linearized}, \ref{S:Estimates_solutions_diagonal}, and \ref{S:Remaining_LWP}, we will provide arguments in more details, focusing mostly in the energy estimates that need to be derived. At a high level, the main ingredients
of the proof are:
\begin{itemize}
\item The analysis of the linearized equations plays a key role in the proof. The linearized system enjoys good structures that come from some special cancellations (see Sect.~\ref{S:Energy_estimates_linearized}).
\item We would like to derive energy estimates for the nonlinear problem by differentiating Eqs.~\eqref{E:Diagonal_Euler_system} and treating the resulting system as a linear system for the derivatives to which we could apply the cancelations mentioned above. However, this does not seem possible. With exception of the one-time differentiated system (which solves the linearized problem), higher-order regular derivatives do not satisfy a linearized system with good structure. This comes from the fact that our problem has weights, and differentiating the system with respect to regular derivatives destroys its weighted structure. When generic derivatives fall on the weights $r$ in \eqref{E:Diagonal_Euler_system_sound_speed} we can get $O(1)$ terms that are no longer comparable to the distance to the boundary.
\item The weighted structure of Eqs.~\eqref{E:Diagonal_Euler_system} is preserved, however, if we differentiate the system with respect to $D_t$ because, each time $D_t$ falls on the weight $r$ we \emph{recover} $r$ using Eq.~\eqref{E:Diagonal_Euler_system_sound_speed}, which gives
$D_t r \simeq r (\bar{\partial} r, \bar{\partial} v)$. This suggests deriving energy estimates for $D_t^{2N}(r,v)$ (see Remark \ref{R:Elliptic_estimates_vacuum_bry}).
\item To control $D_t^{2N}(r,v)$, we would like to show that $D_t^{2N}(r,v)$ satisfies the linearized equations with good perturbative terms. This is, however, not the case. The cancelations that are needed in the linearized system rely in particular on the linearized variables coming from linearizing $D_t$. Differentiating \eqref{E:Diagonal_Euler_system} with respect to $D_t$, however, does not produce such terms because $D_t$ commutes with itself. We therefore introduce some \emph{good linear variables} which are suitable modifications of 
$D_t^{2N} (r,v)$. These good linear variables are then showed to satisfy the linearized system with good perturbative terms (see Sect.~\ref{S:Estimates_solutions_diagonal}).
\item The previous argument gives estimates for $D_t^{2N} (r,v)$ (to be more precise, it is not
$D_t^{2N} (r,v)$ that we control but a modified version of it; we downplay this aspect for now, see Remark \ref{R:Elliptic_estimates_vacuum_bry}). In order to obtain full control of $(r,v)$, we employ elliptic estimates. The elliptic part of the system comes from the underlying wave evolution, whose leading order operator is of the form \eqref{E:Wave_operator_vacuum_bry_leading_order}, so that $D_t^{2N} (r,v)$ satisfies an equation that reads, schematically, as (see Remark \ref{R:Elliptic_estimates_vacuum_bry} below)
\begin{align}
D_t^{2 N} (r,v) - r \Delta D_t^{2N-2} (r,v) = \text{L.O.T.}
\label{E:Schematic_wave_equation_vacuum_boundary}
\end{align}
Given the energy estimates for $D_t^{2N} (r,v)$, we can move this term to the RHS in \eqref{E:Schematic_wave_equation_vacuum_boundary} and apply elliptic estimates for the operator $r\Delta$ to control 
$D_t^{2N-2} (r,v)$. Proceeding inductively (i.e., controlling 
$D_t^{2N-4} (r,v)$ in terms of $D_t^{2N-2} (r,v)$, etc.), we finally obtain control of $(r,v)$ (see Sect.~\ref{S:Estimates_solutions_diagonal}).
\item The above discussion focused on a priori estimates, which are the basis for the argument. The full Theorem is proved with help of a regularization and a time-discretization akin to an Euler's method.
In particular, the passage from the linear to the nonlinear problem requires the nonlinear equations to have good structures, which is another reason why we rewrote the original Euler system in terms of the $(r,v)$ variables (see Sect.~\ref{S:Remaining_LWP}).
\end{itemize}
\hfill \qed

\begin{remark}
\label{R:Elliptic_estimates_vacuum_bry}
Here we point out the following further remarks at a high-level. They are discussed in more detail in Sects.~\ref{S:Energy_estimates_linearized}, \ref{S:Estimates_solutions_diagonal}, and \ref{S:Remaining_LWP}:
\begin{itemize}
\item It turns out that $D_t^{2N}(r,v)$ is still not the right variable to consider. We need to correct it by a suitable term, leading to what we call good linear variables. We will, however, continue to talk about estimating $D_t^{2N}(r,v)$ until Sect.~\ref{S:Estimates_solutions_diagonal}, when we will discuss why $D_t^{2N}(r,v)$ is not yet appropriate for the derivation of estimates and how to modify it.
\item The wave operator \eqref{E:Wave_operator_vacuum_bry_leading_order} and the corresponding wave equations \eqref{E:Schematic_wave_equation_vacuum_boundary} are only a schematic representation of the actual wave evolution of the system. It turns out that the correct wave operator needed for the
estimates involves an elliptic part more complicated than $r\Delta$. We wrote $r\Delta$ to illustrate the main important point, namely, that one has a \emph{degenerate (because of $r$) elliptic operator.} In particular, the elliptic estimates one
needs to carry out are degenerate elliptic estimates. See Sect.~\ref{S:Weighted_elliptic_estimates}.
See also Remark \ref{R:Good_variables_material_derivatives}.
\item To prove the elliptic estimates, we need to perform some integration by parts and pay special attention to terms when derivatives fall on the weights, since we need to have norms with the correct weights. For this, we need to observe some subtle cancellations that hold because of the very specific form of the elliptic operators involved (which are not, as mentioned, simply $r \Delta$).
\item We also employ an argument that is reminiscent of the standard ``freezing of coefficients at a point'' used in elliptic equations (see, e.g., \citealt{Gilbarg-Trudinger-Book-2001}). This requires some smallness to absorb top-order terms back into the LHS. When carrying out integration by parts, if the derivatives falling
on the weights are transverse to the distance to the boundary, then one can show that the corresponding terms are small, but they will be of $O(1)$ if the derivatives are approximately normal to the boundary. This can be illustrated by thinking of the boundary as given by $\{ x^3 = 0 \}$, so that $r$ behaves like $x^3$. Then, $\partial_1 r$ and $\partial_2 r$ are small, but $\partial_3 r = O(1)$. However, if we arrange the estimates so that the corresponding coefficient is not $\partial_3 r$ but rather
$\partial_3 r - \mathsf{N}$, where $\mathsf{N} = (0,0,1)$, then this coefficient can be made small by considering a sufficiently small neighborhood of the boundary. This explains the construction of the control norm $A$ in \eqref{E:Control_norm_A}.
\item It is is only the divergence part of $v$ that satisfies a wave equation, thus the above argument only controls the (spatial) divergence of $v$. Full control is obtained with help of vorticity estimates.
See Remark \ref{R:Transport_estimates_implicit}.
\end{itemize}
\end{remark}

Next, we investigate continuation of solutions.

\begin{theorem}[\citealt{Disconzi-Ifrim-Tataru-2022}]
\label{T:Continuation_solutions_vacuum_bry}
For each integer $N \geq 0$, there exists an energy functional 
$E^{2N} = E^{2N}(r,v)$ with the following properties:

\noindent (a) Coercivity: as long as $A$ remains bounded,
\begin{align}
E^{2N} \approx \norm{ (r,v) }^2_{\Hspace^{2N}}.
\nonumber
\end{align}
(b) Energy estimates hold for solutions to \eqref{E:Diagonal_Euler_system},
\begin{align}
\frac{d}{dt} E^{2N} \lesssim_A B \norm{ (r,v) }^2_{\Hspace^{2N}}.
\nonumber
\end{align}
\end{theorem}

As a consequence of Theorem \ref{T:Continuation_solutions_vacuum_bry}, Gr\"onwall's inequality gives
\begin{align}
 \norm{ (r,v) }^2_{\Hspace^{2N}} \lesssim
 e^{\int_0^t C(A) B \, d\tau}  \norm{ (\mathring{r},\mathring{v}) }^2_{\Hspace^{2N}}.
\label{E:Gronwall_vacuum_bry}
\end{align}

\begin{remark}
The energies $E^{2N}$ in Theorem \ref{T:Continuation_solutions_vacuum_bry} are explicitly constructed for integer $N \geq 0$ only, but the arguments in 
\cite{Disconzi-Ifrim-Tataru-2022} show that \eqref{E:Gronwall_vacuum_bry} holds for non-integer $N>0$.
\end{remark}

As a consequence of \eqref{E:Gronwall_vacuum_bry}, we obtain
the following \emph{continuation criterion:}

\begin{corollary}[\citealt{Disconzi-Ifrim-Tataru-2022}]
\label{C:Continuation_criterion}
The unique solutions obtained in Theorem \ref{T:LWP_relativistic_Euler} can be continued as long as $A$ remains bounded and $B \in L_t^1(\Md)$.
\end{corollary}

\begin{remark}
Observe that $B \sim \norm{\bar{\partial} v}_{L^\infty(\Md_t)}
+ \norm{\bar{\partial} r}_{C^\frac{1}{2}(\Md_t)}$, so that
Corollary \ref{C:Continuation_criterion} is very similar to the 
standard continuation criterion for relativistic Euler that comes
from their formulation as a first-order symmetric hyperbolic system. See \cite{Disconzi-Ifrim-Tataru-2022} for details.
\end{remark}

\noindent \emph{Proof of Theorem \ref{T:Continuation_solutions_vacuum_bry}.} Theorem 
\ref{T:Continuation_solutions_vacuum_bry} is obtained from the energy estimates that have been briefly discussed in the proof 
of Theorem \ref{T:LWP_vacuum_bry}. In particular, in order to obtain the estimates always with a factor linear in $B$
we need to use some powerful interpolation Theorems
obtained by Ifrm and Tataru in \cite{Ifrim-Tataru-2024}.
\hfill \qed

\begin{remark}
Local existence and uniqueness for the classical compressible Euler equations with a physical vacuum boundary have been obtained
by \cite{Coutand-Shkoller-2011, Coutand-Shkoller-2012}
and \cite{Jang-Masmoudi-2015,
Jang-Masmoudi-2009,
Jang-Masmoudi-2011,
Jang-Masmoudi-2012} and,
more recently, by \cite{Ifrim-Tataru-2024}, where analogues
to Theorems \ref{T:LWP_vacuum_bry}, \ref{T:Continuation_solutions_vacuum_bry} and Corollary
\ref{C:Continuation_criterion} have been obtained for 
the classical compressible Euler equations.
At the risk of being repetitive, however, we once more stress
that results in the relativistic setting should not be taken
for granted as simple extensions of the non-relativistic case.
\end{remark}

\subsection{Energy estimates for the linearized equations\label{S:Energy_estimates_linearized}}
Consider the linearization of Eqs.~\eqref{E:Diagonal_Euler_system}, where we denote
by $s$ the variable associated to the linearization of $r$ and by $w$ the variable associated to the linearization of $v$. Direct computation\footnote{In fact, the computation is not so direct, see Sect.~\ref{S:Role_weights}.} gives that the linearization of \eqref{E:Diagonal_Euler_system} is
\begin{subequations}{\label{E:Linearized_diagonal_Euler_system}}
\begin{align}
D_t s + \frac{1}{\kappa} (\bar{H}^{-1})^{ij} \partial_i r w_j + r (\bar{H}^{-1})^{ij} \partial_i w_j
+ r a_1 v^i \partial_i s & = f,
\label{E:Linearized_diagonal_Euler_system_sound_speed}
\\
D_t w_i + a_2 \partial_i s & = h_i,
\label{E:Linearized_diagonal_Euler_system_velocity}
\end{align}
\end{subequations}
where $f$ and $h$ are of the form
\begin{subequations}{\label{E:Error_terms_linearized_diagonal}}
\begin{align}
f &= S_1 s + r W_1 w,
\label{E:Error_terms_linearized_diagonal_f}
\\
h & = S_2 s + W_2 w,
\label{E:Error_terms_linearized_diagonal_h}
\end{align}
\end{subequations}
where $S_1, S_2, W_1, W_2$ are linear in $\partial (r,v)$ with coefficients that are smooth functions of $(r,v)$. The functions $f$ and $h$ will be error terms that can be directly controlled by our energy
(see \eqref{E:Weighted_energy_linearized} below). 

\begin{remark}
\label{R:Error_terms_f_h}
In accordance with the spirit of these notes, we could have ignored the terms $f$ and $h$ and written \eqref{E:Linearized_diagonal_Euler_system} with 
$\simeq \dots$ or $= \,\text{L.O.T}$. We will do so in much of what follows. But since the norms we want to control have $r$ weights, it is important to verify that the RHS of \eqref{E:Linearized_diagonal_Euler_system} in fact has the correct powers of $r$ to be bounded by the energy; see \eqref{E:Weighted_energy_linearized}.
\end{remark}

\begin{remark}
\label{R:r_and_v_as_coefficients}
As usual upon considering a linearization, in \eqref{E:Linearized_diagonal_Euler_system} our solution variables are $s$ and $w$, with $r$ and $v$ treated as coefficients. In particular, our energy estimates for the linearized equations in Sect.~\ref{S:Energy_estimates_linearized} hold with bounds depending on the $L^\infty$ norm of\footnote{Some of the derivatives of $r$ and $v$ appearing as coefficients in the energy estimate will be time derivatives coming from lower-order terms involving the material derivative (although not all terms in material derivatives will be error terms, see, e.g., Remark \ref{R:Factoring_dt_moving_domain_r_factors} and Sect.~\ref{S:Estimates_solutions_diagonal}).
Using \eqref{E:Diagonal_Euler_system}, however, we can algebraically solve for $\partial_t (r,v)$ in terms 
of $\bar{\partial}(r,v)$, with coefficients that are smooth functions of $(r,v)$.}
  $\bar{\partial}(r,v)$.
\end{remark}

We make the following observations:
\begin{itemize}
\item As the original system \eqref{E:Diagonal_Euler_system}, Eqs.~\eqref{E:Linearized_diagonal_Euler_system} degenerate on the boundary.
\item The linearized system does not require boundary conditions. This is related to the fact that the solutions in the one-parameter family of solutions used to produce the linearization are not required to have the same domain\footnote{Although some kind of boundary behavior is imposed on $s$ and $w$ through the choice of function spaces.}. 
\item The term $\frac{1}{\kappa} (\bar{H}^{-1})^{ij} \partial_i r w_j $ comes from the linearization of $D_t$, i.e., we obtain a term in $\bar{H}^{-1}$ when computing the linearization of $D_t$. See comments after \eqref{E:Equivalence_linearized_energy_H0_norm} as well.
\item The term $\frac{1}{\kappa} (\bar{H}^{-1})^{ij} \partial_i r w_j $ does not contain derivatives
of $(s,w)$, so at first sight it looks like an error term that should be moved to the RHS. We will see
in Sects.~\ref{S:Role_weights} and \ref{S:Estimates_solutions_diagonal} that this term is not lower order with respect to our energies as it does not contain the right
 weight.
\end{itemize}

In order to derive energy estimates, we will use the following moving domains formula
\begin{align}
\frac{d}{dt} \int_{\Md_t} \mathsf{f} & = 
\int_{\Md_t} D_t \mathsf{f} + \int_{\Md_t} \mathsf{f} \partial_i \left( \frac{v^i}{v^0} \right).
\label{E:Moving_domain_derivative}
\end{align}
Formula \eqref{E:Moving_domain_derivative} is 
 a well-known moving-domain formula for a domain advected by the fluid's velocity
(see, e.g., \citealt[Appendix C.4]{Evans-Book-2010}). To see that such a formula takes the form \eqref{E:Moving_domain_derivative}, we observe that $\frac{v^i}{v^0} = \frac{u^i}{u^0}$ is the the actual physical three-velocity of the fluid particles on the boundary (see \citealt[Sect.~7.1]{Rezzolla-Zanotti-Book-2013}). In particular, this provides further motivation for choosing to work with \eqref{E:Material_derivative_relativistic_v}.

\subsubsection{The role of weights\label{S:Role_weights}} As already mentioned, we need to work with a formalism of weighted
energies and function spaces in order to deal with the degenerate character of the problem. 
Let us explain the role of the weights in deriving energy estimates for the linearized system
\eqref{E:Linearized_diagonal_Euler_system}.

For simplicity, let us consider the case $\kappa=1$, in which case \eqref{E:Linearized_diagonal_Euler_system} becomes
\begin{subequations}{\label{E:Linearized_diagonal_Euler_system_kappa_1}}
\begin{align}
D_t r + (\bar{H}^{-1})^{ij} \partial_i r w_j + r (\bar{H}^{-1})^{ij} \partial_i w_j
+ r a_1 v^i \partial_i s & = f,
\label{E:Linearized_diagonal_Euler_system_sound_speed_kappa_1}
\\
D_t w_i + a_2 \partial_i s & = h_i,
\label{E:Linearized_diagonal_Euler_system_velocity_kappa_1}
\end{align}
\end{subequations}

Suppose we want to control the ``standard'' energy
\begin{align}
E_{\text{standard}} := \frac{1}{2} \int_{\Md_t} s^2 + |w|^2 \, dx.
\nonumber
\end{align}
Multiplying \eqref{E:Linearized_diagonal_Euler_system_sound_speed_kappa_1} by $s$,
contracting \eqref{E:Linearized_diagonal_Euler_system_velocity_kappa_1} with $w^i$, integrating over $\Md_t$, and using \eqref{E:Moving_domain_derivative} and \eqref{E:Error_terms_linearized_diagonal}, we find
\begin{align}
\begin{split}
& \frac{1}{2} \frac{d}{dt} \int_{\Md_t} (s^2 +  |w|^2 )
+ \int_{\Md_t} r (\bar{H}^{-1})^{ij} s \partial_i w_j + 
\int_{\Md_t} a_1 r s v^i \partial_i s
\\
&\ \ \
+ \int_{\Md_t} s (\bar{H}^{-1})^{ij} \partial_i r w_j 
+ \int_{\Md_t} a_2 w^i \partial_i s
=\text{L.O.T.}
\end{split}
\label{E:Attempt_energy_estimate_linearized_standard}
\end{align}
$\text{L.O.T.}$ means terms controlled by the energy.
The third integral on the LHS can be handled with integration by parts,
\begin{align}
\int_{\Md_t} a_1 r s v^i \partial_i s = 
\frac{1}{2}
\int_{\Md_t} a_1 r  v^i \partial_i s^2 = 
- \int_{\Md_t} \partial_i (a_1 r s v^i )  s^2 
\lesssim \int_{\Md_t}  s^2 ,
\nonumber
\end{align}
where there is no boundary term because $r=0$ on $\Fb_t$. 

\begin{remark}
\label{R:No_boundary_terms}
It will be a recurring fact throughout our estimates that integration by parts will produce no boundary terms due to the presence of power of $r$. This is an important aspect of our estimates, as boundary terms could produce uncontrollable terms.
\end{remark}

We need to handle the cross-terms, i.e., to
cancel the terms in
$\bar{\partial} w$ on the second integral on the LHS with the term in $\bar{\partial} s$ on the 
last integral  on the LHS of \eqref{E:Attempt_energy_estimate_linearized_standard} after integrating one of them by parts. This clearly cannot be done because of the coefficient $r (\bar{H}^{-1})^{ij}$ in the $\bar{\partial} w$ term and the coefficient $a_2$ in the $\bar{\partial} s$ term. Thus, we need to 
adjust our energy to balance these coefficients. Multiplying \eqref{E:Linearized_diagonal_Euler_system_sound_speed_kappa_1} by $s$ and contracting \eqref{E:Linearized_diagonal_Euler_system_velocity_kappa_1} with $a_2^{-1}r(\bar{H}^{-1})^{ij} w_j$, we find, after integrating over $\Md_t$ and using \eqref{E:Moving_domain_derivative} (and \eqref{E:Diagonal_Euler_system_sound_speed}, see Remark \ref{R:Factoring_dt_moving_domain_r_factors} below),
\begin{align}
\begin{split}
\frac{1}{2} \frac{d}{dt} 
\int_{\Md_t} (s^2 + \frac{1}{a_2} r (\bar{H}^{-1})^{ij} w_i w_j )
+ 
\int_{\Md_t}  (\bar{H}^{-1})^{ij} \partial_i r w_j s 
\\
\ \ \ 
+
\int_{\Md_t} r (\bar{H}^{-1})^{ij} s \partial_i w_j 
+
\int_{\Md_t} r (\bar{H}^{-1})^{ij} w_i \partial_j s 
= \text{L.O.T.}
\end{split}
\label{E:Computation_energy_estimate_linearized_kappa_1}
\end{align}
(where the term in $a_1 s\partial_i s$ is handled as before and was thus moved to the RHS).
The first integral gives the following weighted energy for the $\kappa=1$ case,
\begin{align}
E^{(\kappa=1)} := \frac{1}{2} \int_{\Md_t} s^2 + \frac{1}{a_2} r | w|_{\bar{H}}^2 \, dx,
\label{E:Weighted_energy_linearized_kappa_1}
\end{align}
where
\begin{align}
| w|_{\bar{H}}^2 := (\bar{H}^{-1})^{ij} w_i w_j 
\nonumber
\end{align}
is simply the norm of $w$ measured with respect to the $\bar{H}$ metric, which is pointwise equivalent to the Euclidean norm; recall \eqref{E:H_bar_equivalent_Euclidean}. Recall also that $a_2  \geq \text{constant}> 0$ 
by \eqref{E:a_2_bounded_from_below},
so $E^{(\kappa=1)}$ is in fact non-negative. 
Observe that \eqref{E:Weighted_energy_linearized_kappa_1} is nothing else than a \emph{weighted} $L^2$ norm (squared) of $s$ and $w$. Note the equivalence between $E^{(\kappa=1)}$ and the $\Hspace^0$-norm squared.

\begin{remark}
\label{R:Energy_H_not_Euclidean}
Despite the equivalence between $\bar{H}^{-1}$ and $\updelta^{-1}$, i.e., Eq.~\eqref{E:H_bar_equivalent_Euclidean}, note in the above calculation that for the energy estimate to close, $E^{(\kappa=1)}$ needs to be defined using $| w|_{\bar{H}}^2$. It is only \emph{after} the energy estimate is carried out that we can invoke the equivalence of $\bar{H}^{-1}$ and $\updelta^{-1}$ and replace  $| w|_{\bar{H}}^2$ with $|w|^2$ if desired\footnote{This is the same situation as in energy estimates for quasilinear wave equations. One defines an energy using the actual metric of the problem in order to carry out energy estimates, but invoke the equivalence of such metric with the Minkowski metric (or the Euclidean metric for the spatial part) to trade the energy by a standard Sobolev norm.}.
\end{remark}

\begin{remark}
\label{R:Factoring_dt_moving_domain_r_factors}
In invoking \eqref{E:Moving_domain_derivative} to factor a $\frac{d}{dt}$, it is important to pay attention to the term in material derivative of $r$, i.e., after contracting \eqref{E:Linearized_diagonal_Euler_system_velocity_kappa_1} with $a_2^{-1}r(\bar{H}^{-1})^{ij} w_j$, we find, after integration over $\Md_t$,
\begin{align}
\label{E:Using_D_t_r_factoring_dt_moving_domain}
\begin{split}
\int_{\Md_t} \frac{1}{a_2} r (\bar{H}^{-1})^{ij} w_j D_t w_i & = 
\frac{1}{2} \int_{\Md_t}  D_t (\frac{1}{a_2} r (\bar{H}^{-1})^{ij} w_i w_j)
 - \frac{1}{2} \int_{\Md_t} r D_t (\frac{1}{a_2} (\bar{H}^{-1})^{ij}) w_i w_j
\\
& \ \ \
-  \frac{1}{2}\int_{\Md_t} D_t r \frac{1}{a_2}  (\bar{H}^{-1})^{ij} w_i w_j.
\end{split}
\end{align}
The first term on the RHS contributes to $E^{(\kappa=1)}$ upon using
\eqref{E:Moving_domain_derivative}. The second term on the RHS is controlled by the energy and is part of the RHS of \eqref{E:Computation_energy_estimate_linearized_kappa_1}. The last term on the RHS seems to be missing a power of $r$ in order to be controlled by the energy. More precisely, 
since $r(t,x) \approx \operatorname{dist}(x,\Fb_t)$, we have $\bar{\partial} r = O(1)$, so that\footnote{Recall Remark \ref{R:Working_neighborhood_boundary}.}
\begin{align}
\label{E:Cannot_bound_energy_without_weights}
\int_{\Md_t} \bar{\partial} r |w|^2 
\approx \int_{\Md_t}  |w|^2 
\cancel{\, \lesssim \,} \int_{\Md_t} r |w|^2.
\end{align}
Had the last term on the RHS of \eqref{E:Using_D_t_r_factoring_dt_moving_domain} been a \emph{generic} derivative of $r$, we would \emph{not} be able to bound the integral of the last 
term on the RHS of \eqref{E:Using_D_t_r_factoring_dt_moving_domain} in terms 
of $E^{(\kappa=1)}$, precisely because of \eqref{E:Cannot_bound_energy_without_weights}.
However, \emph{we can recover a power of $r$ from $D_t r$ upon using \eqref{E:Diagonal_Euler_system_sound_speed},} which gives
\begin{align}
\label{E:D_t_r_gains_r_using_eq_moving_domain}
D_t r = r \bar{\partial} (r,v),
\end{align}
so that the last term on the RHS of \eqref{E:Using_D_t_r_factoring_dt_moving_domain}
belongs to the RHS of \eqref{E:Computation_energy_estimate_linearized_kappa_1}.
Throughout our energy estimates, we will silently use \eqref{E:D_t_r_gains_r_using_eq_moving_domain}
to recover a power of $r$ when factoring a $\frac{d}{dt}$ via \eqref{E:Moving_domain_derivative}.
\end{remark}

Using \eqref{E:Error_terms_linearized_diagonal}, we can check that the RHS of \eqref{E:Computation_energy_estimate_linearized_kappa_1} is controlled by $E^{(\kappa=1)}$ (see Remark \ref{R:Error_terms_f_h}). Thus, in order to close the estimate, it remains to cancel the terms on $\bar{\partial}(s,w)$ on the LHS. The second, third, and fourth terms on the LHS of \eqref{E:Computation_energy_estimate_linearized_kappa_1} can be combined as
\begin{align}
\label{E:Cancelation_linearized_kappa_1}
\begin{split}
& \int_{\Md_t}  (\bar{H}^{-1})^{ij} \partial_i r w_j s 
+
\int_{\Md_t} r (\bar{H}^{-1})^{ij} s \partial_i w_j 
+
\int_{\Md_t} r (\bar{H}^{-1})^{ij} w_i \partial_j s 
\\
& \ \ \ 
= \int_{\Md_t} (\bar{H}^{-1})^{ij} \partial_i (r w_j s)
= -  \int_{\Md_t} \partial_i (\bar{H}^{-1})^{ij}  (r w_j s)
\\
& \ \ \ 
\lesssim \int_{\Md_t} ( s^2 + r|w|^2 ) \lesssim E^{(\kappa=1)},
\end{split}
\end{align}
where we recall Remark \ref{R:r_and_v_as_coefficients} and note that there is no boundary term after the integration by parts since $r$ vanishes on the boundary. 

From the foregoing, we conclude the following energy estimate for \eqref{E:Linearized_diagonal_Euler_system_kappa_1}: 
\begin{align}
E^{(\kappa=1)}(t) \lesssim E^{(\kappa=1)}(0).
\label{E:Energy_estimate_linearized_kappa_1}
\end{align}

In deriving \eqref{E:Energy_estimate_linearized_kappa_1}, we note that 
the term $\int_{\Md_t} s (\bar{H}^{-1})^{ij} \partial_i r w_j$ that \emph{comes from the linearization of $D_t$ is essential to produce the cancelation \eqref{E:Cancelation_linearized_kappa_1}.} Without such a term, \eqref{E:Cancelation_linearized_kappa_1} would give instead
\begin{align}
\begin{split}
& 
\int_{\Md_t} r (\bar{H}^{-1})^{ij} s \partial_i w_j 
+
\int_{\Md_t} r (\bar{H}^{-1})^{ij} w_i \partial_j s 
= \int_{\Md_t} (\bar{H}^{-1})^{ij} r \partial_i ( w_j s)
\\
& \ \ \ 
= -  \int_{\Md_t} \partial_i (\bar{H}^{-1})^{ij}  r w_j s
- \int_{\Md_t}  (\bar{H}^{-1})^{ij}  \partial_i r w_j s.
\end{split}
\nonumber
\end{align}
The last term \emph{cannot be bounded by} $E^{(\kappa=1)}$ in view of \eqref{E:Cannot_bound_energy_without_weights}. We see that the term
$\int_{\Md_t} s (\bar{H}^{-1})^{ij} \partial_i r w_j$ coming from linearizing $D_t$ is key for the energy estimate, even though it has no derivatives of $(s,w)$. This otherwise simple observation has major consequences for the estimates we will present in Sect.~\ref{S:Estimates_solutions_diagonal}.

Having illustrated the energy estimates for \eqref{E:Linearized_diagonal_Euler_system_kappa_1}, let us turn to the general $\kappa$. The argument is very similar to the above. But now, the term coming from the 
linearization of $D_t$ has a $\frac{1}{\kappa}$ factor. i.e., we obtain $\frac{1}{\kappa} (\bar{H}^{-1})^{ij} \partial_i r w_j$. So, in order to get an exact cancellation, we multiply \eqref{E:Linearized_diagonal_Euler_system_sound_speed} by $r^\frac{1-\kappa}{\kappa} s$ and \eqref{E:Linearized_diagonal_Euler_system_velocity} by $\frac{1}{a_2} r^{\frac{1-\kappa}{\kappa}+1} (\bar{H}^{-1})^{ij} w_j$, yielding for the cross-terms
\begin{align}
\begin{split}
& \frac{1}{\kappa} r^\frac{1-\kappa}{\kappa} (\bar{H}^{-1})^{ij} \partial_i r w_j s
+ r^\frac{1}{\kappa} (\bar{H}^{-1})^{ij} \partial_i w_j s 
+ r^\frac{1}{\kappa} (\bar{H}^{-1})^{ij} w_j \partial_i s
\nonumber
\\
\ \ \ & = \partial_i (r^\frac{1}{\kappa}) (\bar{H}^{-1})^{ij}  w_j s
+ r^\frac{1}{\kappa} (\bar{H}^{-1})^{ij} \partial_i w_j s 
+ r^\frac{1}{\kappa} (\bar{H}^{-1})^{ij} w_j \partial_i s
\end{split}
\\
\ \ \ &
= (\bar{H}^{-1})^{ij} \partial_i (r^\frac{1}{\kappa} w_j s ),
\nonumber
\end{align}
which can be integrated by parts (again, no boundary term appears). We see that in the end we control the energy
\begin{align}
E := \frac{1}{2} \int_{\Md_t} r^\frac{1-\kappa}{\kappa}( s^2 + \frac{1}{a_2} r | w|_{\bar{H}}^2) \, dx,
\label{E:Weighted_energy_linearized}
\end{align}
i.e., 
\begin{align}
E(t) \lesssim E(0).
\nonumber
\end{align}
Observe that 
\begin{align}
E \approx \norm{ (s,w) }_{\Hspace^0}^2
\label{E:Equivalence_linearized_energy_H0_norm}
\end{align}
in view of \eqref{E:a_2_bounded_from_below} and \eqref{E:H_bar_equivalent_Euclidean}.

We have one more comment to make about the linearization of $D_t$. We said it produces the term
$\frac{1}{\kappa} (\bar{H}^{-1})^{ij}\partial_i r w_j$. This is not immediate, and require some intentional algebra that we now present.

Linearizing the term $\frac{v^i}{v^0} \partial_i r$ from $D_t r$ and using \eqref{E:v_0_explicit_v_i}, we find,
using $\delta$ to denote the operation of linearization,
\begin{align}
\begin{split}
\delta \left(\frac{v^i}{v^0} \partial_i r \right) & = \delta \left(\frac{v^i}{v^0} \right) \partial_i + 
\frac{v^i}{v^0} \partial_i \delta r
\\
&= \frac{\delta v^i}{v^0} \partial_i r - \frac{v^i}{(v^0)^2} \delta \partial v^0 \partial_i r + 
\frac{v^i}{v^0} \partial_i \delta r.
\end{split}
\nonumber
\end{align}
As $s := \delta r$, the term $\frac{v^i}{v^0} \partial_i \delta r$ belongs to $D_t s$.
But
\begin{align}
\begin{split}
\delta v^0 &= \frac{1}{2 v^0}[ 
(2+\frac{2}{\kappa}) ( 1 + \frac{\kappa r}{\kappa+1})^{1+\frac{2}{\kappa}} \delta r + 2 v^j \delta v_j ]
\end{split}
\nonumber
\end{align}
The term $(2+\frac{2}{\kappa}) ( 1 + \frac{\kappa r}{\kappa+1})^{1+\frac{2}{\kappa}} \delta r$ is linear in $s := \delta r$ and is absorbed into $f$ on the RHS of
\eqref{E:Linearized_diagonal_Euler_system_sound_speed} (see \eqref{E:Error_terms_linearized_diagonal_f}).
Thus, writing $\dots$ for terms that are accounted for through $D_t s$ or $f$,
\begin{align}
\begin{split}
\delta \left(\frac{v^i}{v^0} \partial_i r \right) & =
\left( \frac{\delta v^i}{v^0} - \frac{v^i}{(v^0)^3} v^j \delta v_j \right) \partial_i r + \dots
\\
& = 
\frac{1}{v^0} \left( \updelta^{ij} - \frac{v^i v^j}{(v^0)^2} \right) w_j \partial_i r + \dots
\end{split}
\nonumber
\end{align}
where we used $w_j := \delta v_j$. Next, add and subtract $\frac{1}{\kappa} (\bar{H}^{-1})^{ij} w_j \partial_i r$, recalling the form of $\bar{H}^{-1}$ given by \eqref{E:H_bar_inverse}, to get
\begin{align}
\begin{split}
\delta \left(\frac{v^i}{v^0} \partial_i r \right) & =
\frac{1}{\kappa} \frac{\kappa}{a_0 v^0 }(1 + \frac{\kappa}{\kappa+1} r)\left(\updelta^{ij} - \frac{v^i v^j}{(v^0)^2}\right) w_j \partial_i r
\\
& \ \ \ 
-
\frac{1}{\kappa} \frac{\kappa}{a_0 v^0 }(1 + \frac{\kappa}{\kappa+1} r)\left(\updelta^{ij} - \frac{v^i v^j}{(v^0)^2}\right) w_j \partial_i r
+ 
\frac{1}{v^0} \left( \updelta^{ij} - \frac{v^i v^j}{(v^0)^2} \right) w_j \partial_i r + \dots
\\
& =
\frac{1}{\kappa} (\bar{H}^{-1})^{ij} w_j \partial_i r 
+ \frac{1}{v^0} \left[ -\frac{1}{a_0} (1 + \frac{\kappa r}{\kappa+1} ) + 1 \right] 
\left( \updelta^{ij} - \frac{v^i v^j}{(v^0)^2} \right) w_j \partial_i r + \dots
\end{split}
\nonumber
\end{align}
The term in brackets can be simplified using \eqref{E:a_0}:
\begin{align}
\begin{split}
 -\frac{1}{a_0} (1 + \frac{\kappa r}{\kappa+1} ) + 1 & = \frac{1}{a_0} \left[
 -(1 + \frac{\kappa r}{\kappa+1} ) + a_0 \right]
 \\
 & = \frac{1}{a_0}\left[ -1 - \frac{\kappa r }{\kappa + 1} + 1 - \kappa r \frac{v^i v_i }{(v^0)^2} 
 \right]
  \\
 & = - \frac{1}{a_0}\left[  \frac{\kappa }{\kappa + 1} + \kappa  \frac{v^i v_i }{(v^0)^2} 
 \right] r.
\end{split}
\nonumber
\end{align}
The important fact in this computation is that we obtain a \emph{term proportional to} $r$. Therefore, 
\begin{align}
\begin{split}
\delta \left(\frac{v^i}{v^0} \partial_i r \right) & =
\frac{1}{\kappa} (\bar{H}^{-1})^{ij} w_j \partial_i r 
-\frac{1}{v^0 a_0} \left[  \frac{\kappa }{\kappa + 1} + \kappa  \frac{v^i v_i }{(v^0)^2} 
 \right]
\left( \updelta^{ij} - \frac{v^i v^j}{(v^0)^2} \right) r w_j \partial_i r + \dots
\\
& = \frac{1}{\kappa} (\bar{H}^{-1})^{ij} w_j \partial_i r  + \dots
\end{split}
\nonumber
\end{align}
where in the last step we used that the term linear in $r w$ can be absorbed into $f$ on the RHS of 
\eqref{E:Linearized_diagonal_Euler_system_sound_speed} and thus belongs to the $\dots$ terms.

The computations we presented so far, from the choice of variables $(r,v)$ to some exact cancellations, 
illustrate  the following  
\emph{key idea} in the treatment of the relativistic Euler equations with a physical boundary: \emph{it is crucial to find the right variables to treat the problem.} A right choice of variables will, indeed, be essential for the energy estimates we will derive in the next section.

\subsection{Estimates for solutions\label{S:Estimates_solutions_diagonal}} As we have seen
in Sect.~\ref{S:Energy_estimates_linearized}, the linearized equations possess a good algebraic structure that allows us to derive a basic weighted $L^2$ energy estimate. This is manifest in the cancellation of the cross terms $w \bar{\partial}s$ and $s \bar{\partial}w$, which requires combination
with an undifferentiated (in $(s,w)$) term coming from the linearization of $D_t$. 

In order to derive estimates for $(r,v)$ satisfying \eqref{E:Diagonal_Euler_system}, it is natural to proceed in the usual fashion, differentiating the equations and treating it as a linear system for the top-order derivatives of $(r,v)$, in which case we would appeal to the energy estimate for the linearized equation since $(\bar{\partial} s,\bar{\partial} w)$ satisfies the one-time differentiated system\footnote{So that deriving estimates for the $\ell$-times differentiated system is the same as deriving estimates for the $\ell-1$-times differentiated linearized system.}.

This procedure, however, does not work because differentiating \eqref{E:Diagonal_Euler_system} with respect to arbitrary derivatives \emph{destroys the delicate weighted structure of the equations that is crucial for estimates of the linearized system.} Indeed, when derivatives fall on the coefficients $r$ that play the role of weights we obtain $\partial r = O(1)$ terms. Such terms are not amenable to produce control of the weighted norm $\Hspace^{2N}$ which requires high powers of the weight (recall definitions \eqref{E:Weighted_Sobolev} and \eqref{E:Weighted_space_diagonal}). A lack of weights, naturally, cannot produce 
control of our weighted norms; recall \eqref{E:Cannot_bound_energy_without_weights} and the surrounding discussion.

A way around this difficulty is to differentiate \eqref{E:Diagonal_Euler_system} with respect to material derivatives $D_t$. This preserves the weights because every time a material derivative fall on $r$, we recover it upon using \eqref{E:Diagonal_Euler_system_sound_speed}, which gives $D_t r = r \bar{\partial} (r,v)$ (this was indeed already used in the estimates for the linearized equation, see Remark \ref{R:Factoring_dt_moving_domain_r_factors}). In this way we can, in principle, obtain control of 
$(D_t^{2N} r,D_t^{2N} v)$ in $\Hspace^0$. In order to translate it to control of $\norm{(r,v)}_{\Hspace^{2N}}$, we need to trade the material derivatives by spatial derivatives. We hope to be able to do so with the help of weighted 
elliptic estimates coming from the underlying evolution, as follows.

As mentioned 
in the the text surrounding \eqref{E:Schematic_wave_equation_vacuum_boundary}, 
the wave evolution reads, schematically,
\begin{align}
D_t^{2 N} (r,v) - r \Delta D_t^{2N-2} (r,v) = \dots
\nonumber
\end{align}
where $\dots$ represent error terms. Given control of $\norm{ D_t^{2N} (r,v) }_{\Hspace^0}$ from the energy estimates, we can treat this term as a source, 
\begin{align}
 r \Delta D_t^{2N-2} (r,v) =  D_t^{2 N} (r,v) + \dots
\label{E:Schematic_weighted_elliptic_equation_from_wave_evolution}
\end{align}
Since $r$ vanishes as the distance to the boundary, the operator $ r \Delta$ is degenerate elliptic, so that standard elliptic estimates do not apply. But we can carry out weighted elliptic estimates to obtain
\begin{align}
\norm{ D_t^{2N-2} (r,v) }_{\Hspace^2} \lesssim \norm{ D_t^{2N} (r,v) }_{\Hspace^0} + \dots
\label{E:Schematic_weighted_elliptic_estimate_from_wave_evolution}
\end{align}
We will show \eqref{E:Schematic_weighted_elliptic_estimate_from_wave_evolution} in Sect.~\ref{S:Weighted_elliptic_estimates}, but it is worth noting at this point that the derivative counting matches: Eq.~\eqref{E:Schematic_weighted_elliptic_equation_from_wave_evolution} says that the norm on the LHS involves two more derivatives and an additional power of $r$ than the terms on the RHS, which is precisely the difference between $\Hspace^2$ and $\Hspace^0$ (see Definitions \ref{D:Weighted_Sobolev} and \ref{D:Weighted_space_diagonal}). Proceeding recursively, i.e., controlling $D_t^{2N-4} (r,v)$ in terms of $D_t^{2N-2} (r,v)$, etc.), we finally obtain control of $(r,v)$.

\begin{remark}
The comments of Remark \ref{R:Elliptic_estimates_vacuum_bry}, especially the last bullet point, apply to the above discussion of elliptic estimates.
\end{remark}

Nevertheless, the preceding argument does not work. More precisely, the argument we sketched for the elliptic estimates goes through, as we will show in Sect.~\ref{S:Weighted_elliptic_estimates} and what follows. It is the energy estimates for material derivatives of $(r,v)$ that fails, at least in the form we presented above. Recall that in order to close the energy estimate for the linearized problem \eqref{E:Linearized_diagonal_Euler_system} it was crucial to take into account the term $\frac{1}{\kappa} (\bar{H}^{-1})^{ij} \partial_i r w_j$ that comes from the linearization of $D_t$ in \eqref{E:Diagonal_Euler_system_sound_speed}. Alternatively, from the point of view of the one-time differentiation of Eqs.~\eqref{E:Diagonal_variables}, treating $\partial (r,v)$ as a solution to the linearized equations, the term $\frac{1}{\kappa} (\bar{H}^{-1})^{ij} \partial_i r w_j$ comes from the commutator $[\partial,D_t]$ obtained from differentiating \eqref{E:Diagonal_Euler_system_sound_speed}. \emph{The term
$\frac{1}{\kappa} (\bar{H}^{-1})^{ij} \partial_i r w_j$ is absent upon differentiating \eqref{E:Diagonal_Euler_system} with respect to the material derivative} since $D_t$ obviously commutes with itself.

\subsubsection{Good linear variables} A solution to problem just discussed is found upon introducing a new set of \emph{good linear variables}. Define
\begin{subequations}{\label{E:Good_linear_variables}}
\begin{align}
s_0 &:= r,
\label{E:s_0}
\\
w_0 & := v,
\label{E:w_0}
\\
s_1 & := \partial_t r,
\label{E:s_1}
\\
w_1 & := \partial_t v,
\label{E:w_1}
\\
s_2 & := D_t^2 r + \frac{1}{2} \frac{a_0 a_2}{\kappa (1+ \frac{\kappa r}{\kappa + 1})} (\bar{H}^{-1})^{ij} \partial_i r \partial_j r 
\label{E:s_2}
\\
w_\ell & := D_t^\ell v, \ \ \ell \geq 2,
\label{E:w_ell}
\\
s_\ell & := D_t^\ell r - \frac{a_0}{\kappa(1+ \frac{\kappa r }{\kappa + 1})} (\bar{H}^{-1})^{ij} D_t^{\ell-1} v_j \partial_i r
\label{E:s_ell}
\nonumber
\\
& = D_t^\ell r - \frac{a_0}{\kappa(1+ \frac{\kappa r }{\kappa + 1})} (\bar{H}^{-1})^{ij} (w_{\ell-1})_j \partial_i r
\ \ \ell \geq 3.
\end{align}
\end{subequations}

Leaving aside the case $\ell \leq 2$ for a moment, we observe that $(s_\ell,w_\ell)$ are precisely the variables we hope to control with energy estimates according to the above argument, namely, $(D_t^\ell r, D_t^\ell v)$, except that in the case of $s_\ell$ we introduce a \emph{correction term} in $D_t^\ell r$ to account for aforementioned the fact that the term $\frac{1}{\kappa} (\bar{H}^{-1})^{ij} \partial_i r w_j$ is absent from the commutation of Eq.~\eqref{E:Diagonal_Euler_system_sound_speed} with $D_t$. The correction term in fact could have been chosen to be
\begin{align}
s_\ell^\prime = D_t^\ell r - 
\frac{1}{\kappa} (\bar{H}^{-1})^{ij} (w_{\ell-1})_j \partial_i r, 
\label{E:s_ell_alternative}
\end{align}
since the difference between $s_\ell$ given by \eqref{E:s_ell} and $s^\prime_\ell$ given by \eqref{E:s_ell_alternative} is\footnote{For this, one needs to use the form of $a_0$ given by \eqref{E:a_0}.} a term linear in $r w_{\ell-1}$, 
\begin{align}
s_\ell - s_\ell^\prime = (\dots) r w_{\ell-1}
\label{E:Difference_s_ell_s_ell_prime}
\end{align}
which can be treated as an error term; recall the form of the RHS of \eqref{E:Linearized_diagonal_Euler_system_sound_speed} given by 
\eqref{E:Error_terms_linearized_diagonal_f}. But it is more convenient\footnote{We will be omitting the computations in which such convenience is manifest. In this regard, there would be no harm for out high-level presentation to simply define $s_\ell$ by the RHS of \eqref{E:s_ell_alternative}, but we thought that it would be better to be faithful to the definitions of \cite{Disconzi-Ifrim-Tataru-2022} to avoid any possible confusion.} to work with \eqref{E:s_ell} rather than \eqref{E:s_ell_alternative} as some algebraic manipulations are slightly simpler.

Definition \eqref{E:s_ell} involves $w_{\ell-1}$ instead of $w_\ell$ because we need to compute the evolution equation for $(s_\ell, w_\ell)$ by differentiating them with $D_t$ and then using \eqref{E:Diagonal_Euler_system}. When the material derivative falls on $w_{\ell-1}$ in \eqref{E:s_ell} it produces the desired term $\frac{a_0}{\kappa(1+ \frac{\kappa r }{\kappa + 1})} (\bar{H}^{-1})^{ij} (w_{\ell})_j \partial_i r$ (or equivalently $\frac{1}{\kappa} (\bar{H}^{-1})^{ij} (w_{\ell})_j \partial_i r$, see above) needed for the cancellation of the cross-terms $s_\ell \bar{\partial} w_\ell$ and
$w_\ell \bar{\partial} s_\ell$.

Our goal is then to proceed recursively, using weighted elliptic estimates to control 
$(s_{\ell-2},w_{\ell-2})$ in terms of $(s_\ell, w_\ell)$, as suggested above. At the end of the recursion, however, we want to obtain a bound for $(r,v)$ and not some combination thereof. Hence definitions
\eqref{E:s_0} and \eqref{E:w_0}. But in order to guarantee that the transition from $\ell \geq 3$ down to $\ell=0$ works in our hierarchy of energy and elliptic estimates, we do need to adjust the definition for $\ell=1,2$, as done in \eqref{E:s_1}, \eqref{E:w_1}, and \eqref{E:s_2}. We make the following important remark.

\begin{remark}
\label{R:Good_variables_material_derivatives}
With the previous explanations in mind, we note that from a purely derivative counting point of view, we can think of the good linear variables as $(s_\ell, w_\ell) \sim D_t^\ell (r,v)$. In other words, we should view $D_t^\ell (r,v)$ as a good proxy for $(s_\ell, w_\ell)$ from a schematic point of view.
\end{remark}

\begin{assumption}
\label{A:Modification_small_ell}
We will present arguments that are valid for $2\ell \geq 3$ in what follows. Minor modifications are needed for the remaining cases in view of the different definitions of the good linear variables for $\ell \leq 2$, see \eqref{E:Good_linear_variables}.
\end{assumption}

Next, we apply $D_t$ to $(s_\ell,w_\ell)$, use \eqref{E:Diagonal_Euler_system} and definitions
\eqref{E:Good_linear_variables} to find the following evolution equation
\begin{subequations}{\label{E:Good_linear_variables_system}}
\begin{align}
D_t s_{2\ell} + \frac{1}{\kappa} (\bar{H}^{-1})^{ij} \partial_i r (w_{2\ell})_j + r (\bar{H}^{-1})^{ij} \partial_i (w_{2\ell})_j
+ r a_1 v^i \partial_i s_{2\ell} & = f_{2\ell},
\label{E:Good_linear_variables_system_s}
\\
D_t (w_{2\ell})_i + a_2 \partial_i s_{2\ell} & = (h_{2\ell})_i,
\label{E:Good_linear_variables_system_w}
\end{align}
\end{subequations}

We see that $(s_{2\ell},w_{2\ell})$ satisfy the linearized equation (compare with \eqref{E:Linearized_diagonal_Euler_system}) with suitable source terms. 
We note that \eqref{E:Difference_s_ell_s_ell_prime} has been using in deriving \eqref{E:Good_linear_variables_system_s}.
The structure of the source terms $(f_{2\ell}, h_{2\ell})$
will be investigated in Sect.~\ref{S:Energy_estimates_solutions_diagonal}.

The system \eqref{E:Good_linear_variables} is written for even values $2\ell$ because, as explained,
we will use the underlying wave evolution combined with elliptic estimates to control the good linear variables recursively, with the recursion connecting the different linear variables always at difference of two levels in our hierarchy, i.e., $(s_{\ell-2}, w_{\ell-2})$ with $(s_\ell, w_\ell)$ and so on (recall the previous discussion). Next, we will use Eqs.~\eqref{E:Good_linear_variables_system} to derive estimates for $(r,v)$.

\subsubsection{Energy norm and bookkeeping scheme\label{S:Energy_norm_bookkeeping}} According to the foregoing, our goal is to 
use our estimates for the linearized equation from Sect.~\ref{S:Energy_estimates_linearized} to
control $(s_{2\ell},w_{2\ell})$. We thus introduce the energy
\begin{align}
\label{E:Energy_good_linear_variables}
E^{2N}(r,v) = E^{2N} := \sum_{\ell=0}^N \norm{ (s_{2\ell}, w_{2\ell} ) }_{\Hspace^0}^2,
\end{align}
which is simply the sum of the energies for the linearized variables applied to each pair of
good linear variables $(s_{2\ell},w_{2\ell})$; see \eqref{E:Weighted_energy_linearized} and
\eqref{E:Equivalence_linearized_energy_H0_norm}. Recall that $(s_{2\ell},w_{2\ell})$ are expressions in $(r,v)$ given by \eqref{E:Good_linear_variables}.

Our goal is to estimate the energy \eqref{E:Energy_good_linear_variables} to produce a bound on the norm $\norm{(r,v)}_{\Hspace^{2N}}$. The first task is to show that the energy indeed controls the norm, i.e., 
\begin{align}
E^{2N} \approx \norm{ (r,v) }^2_{\Hspace^{2N}}.
\label{E:Equivalence_energy_norm}
\end{align}
This is statement (a) in Theorem \ref{T:Continuation_solutions_vacuum_bry}. In this regard, we note that, unsurprisingly, the energy estimates in Theorem \ref{T:Continuation_solutions_vacuum_bry} are key for establishing the local well-posedness in Theorem \ref{T:LWP_vacuum_bry}.

\begin{remark}
Strictly speaking, \eqref{E:Energy_good_linear_variables} is not the full energy. It controls only the wave-part of the system and a transport-energy needs to added to \eqref{E:Energy_good_linear_variables}, but we recall that we will not discuss transport estimates here (see Remark \ref{R:Transport_estimates_implicit}).
\end{remark}

In order to establish \eqref{E:Equivalence_energy_norm}, it is convenient to design a bookkeeping scheme that allows us to track how derivatives of $(r,v)$ are paired with powers of the weights $r$ in several multi-linear expressions we will encounter. Such bookkeeping scheme will also be useful to estimate the source terms $(f_{2\ell}, h_{2\ell})$ in \eqref{E:Good_linear_variables_system}. We base such scheme on
the scaling law \eqref{E:Scaling_law}, wherein powers of $\uplambda$ suggest appropriate orders for the terms in our expressions. 

\begin{definition} 
\label{D:Order_multi-linear}
We define the \textdef{order} of a multi-linear expression in $r$, $\bar{\partial}^k r$, and $\bar{\partial}^k v$ with $O(1)$ coefficients, as follows:
\begin{itemize}
\item $r$ and $v$ have order $-1$ and $-\frac{1}{2}$, respectively (we only count $v$ as having order $-\frac{1}{2}$ when it is differentiated. Undifferentiated $v$ has order $0$).
\item $D_t$ and $\bar{\partial}$ have order $\frac{1}{2}$ and $1$, respectively.
\item $\bar{H}^{-1}$, $a_0$, $a_1$, and $a_2$ and, more generally, smooth functions of $(r,v)$ not vanishing at $r=0$, have order $0$. (The order of such smooth functions is defined in terms of the order of the leading term in a Taylor expansion about $r=0$, being order $0$ if the term is constant different than zero).
\item The order of a multi-linear expression is defined as the sum of the order of its factors.
\end{itemize}
\end{definition}

With these conventions, all terms in \eqref{E:Diagonal_Euler_system_sound_speed} have order $-\frac{1}{2}$ except the last term that has order\footnote{The fact that the last terms in \eqref{E:Diagonal_Euler_system_sound_speed} has lower order than the remaining ones is
related to the fact that the last term has enough weight to spare, which is why in our scaling analysis 
in Sect.~\ref{S:Scaling} we treated such a term as a perturbation.} $-1$, and all terms in \eqref{E:Diagonal_Euler_system_velocity} have order zero. Upon successive differentiation of any multi-linear expression with respect to $D_t$ or $\bar{\partial}$, all terms produce the same (highest) order, unless some of these derivatives fall on the coefficients, in which case lower order terms are produced.

The basic idea is that terms of high order in our scheme are the ``dangerous'' ones. This is, in part, because such terms are the ones with more derivatives. As in unweighted estimates, such terms with more derivatives are the ones we have to care the most about. \emph{Unlike} unweighted estimates, however, it is not only the number of derivatives per se that matters but the delicate balence of derivatives and weights (e.g., a terms that is not top order in the number of derivatives but has no weights typically cannot be controlled). More derivatives require more weights, thus powers of $r$ are good and decrease the order of an expression. Hence, a higher order term indicates a term with lots of derivatives, few powers of $r$, or both, and this is why the higher the order of a term the more care is needed to control it. Since more derivatives require more wights, powers of $r$ are good, what is reflected in our scheme since powers of $r$ decrase the order of an expression. We also note that a $D_t$ derivative is better than a $\bar{\partial}$ derivative because, upon using \eqref{E:Diagonal_Euler_system} to 
successively solve for $D_t^k (r,v)$ in terms of $\bar{\partial}$ derivatives we gain powers of $r$ in view of \eqref{E:D_t_r_gains_r_using_eq_moving_domain}. This is why $D_t$ has lower order than $\bar{\partial}$. Finally, $r$ has lower order than $v$ because it requires less weight than $v$ in $\Hspace^{2N}$ (recall Definition \ref{D:Weighted_space_diagonal}).

\begin{remark}
It might not be immediate from \eqref{E:Weighted_Sobolev}, which is used to define $\Hspace^{2N}$ in \eqref{E:Weighted_space_diagonal}, why higher derivatives require more weights, since all terms in
the sum \eqref{E:Weighted_Sobolev} have the same weights. However, terms with fewer derivatives and less weights can be controlled by norms involving higher derivatives and higher powers of the weights, as long as the number of derivatives and powers of the weights allow us to apply Hardy-type weighted embeddings, see \eqref{E:Weighted_embedding}. In fact, in \cite{Disconzi-Ifrim-Tataru-2022}, a norm equivalent to $\Hspace^{2N}$, where terms with fewer derivatives are paired with smaller powers of the weights, is also used. We do not introduce this equivalent norm here for the sake of brevity.
\end{remark}

\begin{remark}
\label{R:Order_different_meaning}
We stress that in what follows, references to the order of a term is in accordance with Definition \ref{D:Order_multi-linear} and \emph{not} in the traditional sense of the order as given by the number of derivatives.
\end{remark}

Using Eqs.~\eqref{E:Diagonal_Euler_system} to successively solve for $D_t(r,v)$ in terms of $\bar{\partial}(r,v)$, we obtain from \eqref{E:Good_linear_variables} that $(s_{2\ell},w_{2\ell})$ is a linear combination of multi-linear expressions in $r,\bar{\partial}^k r, \bar{\partial}^k v$ with zero-order coefficients. It is useful to record here the structure of the top-linear-in-derivatives\footnote{Which refers only to the terms with a maximum number of derivatives and does not include all terms of maximum order in the sense of Definition \ref{D:Order_multi-linear}, see Remark \ref{R:Order_different_meaning}.} terms obtained by this procedure:
\begin{align}
\begin{split}
D_t^{2\ell} r & \approx r^\ell \bar{\partial}^{2\ell} r + r^{\ell+1} \partial^{2\ell} v \approx 
r^\ell \bar{\partial}^{2\ell} r
\\
\text{order: } \,\, \ell-1 & \approx (\ell-1) + (\ell-\frac{3}{2}) \approx \ell -1
\\
D_t^{2\ell} v & \approx r^\ell \bar{\partial}^{2\ell} v + r^{\ell} \partial^{2\ell} vr\approx 
r^\ell \bar{\partial}^{2\ell} v
\\
\text{order: } \,\, \ell-\frac{1}{2} & \approx (\ell-\frac{1}{2}) + (\ell-1) \approx \ell -\frac{1}{2}
\end{split}
\label{E:Orders_top_derivatives}
\end{align}
where below each expression we have written the order (in the sense of Definition \ref{D:Order_multi-linear}) of the corresponding terms, and in the second $\approx$ on each line we kept only the term of highest order.

Taking the $\Hspace^0$-norm and recalling Definition \ref{D:Weighted_space_diagonal}, \eqref{E:Orders_top_derivatives} suggests that
\begin{align}
\norm{D_t^{2\ell} (r,v)}_{\Hspace^0} \approx \norm{(r,v)}_{\Hspace^{2\ell}}.
\label{E:Suggestion_material_derivative_norm}
\end{align}
We already know, however, from the initial discussion in Sect.~\ref{S:Estimates_solutions_diagonal} that 
$D_t^{2\ell} (r,v)$ are not the correct variables for the energy estimates and \eqref{E:Equivalence_energy_norm}, which we will prove next, shows that the correct equivalence with 
$\norm{(r,v)}_{\Hspace^{2\ell}}$ is given by considering $\norm{(s_{2\ell}, w_{2\ell})}_{\Hspace^0}$, 
which includes corrections to $D_t^{2\ell} (r,v)$. We will see this more explicitly in Sect.~\ref{S:Weighted_elliptic_estimates}. Yet, in the spirit of Remark \ref{R:Good_variables_material_derivatives}, the reader should keep \eqref{E:Suggestion_material_derivative_norm} as an heuristic guide for \eqref{E:Equivalence_energy_norm}.

Another ingredient we will need are some powerful weighted interpolation theorems proven in \cite{Ifrim-Tataru-2024}.

\begin{lemma}
\label{L:Weighted_interpolation}
The following inequalities hold\footnote{We state the inequalities in three spatial dimensions, which explains the explicit appearance of $3$ in some of the formulas.}.
\begin{align}
\norm{ r^{\sigma_j} \bar{\partial}^j \mathsf{f} }_{L^{p_j}} \lesssim
\norm{r^{\sigma_0} \mathsf{f} }_{L^{p_0}}^{1-\theta_j}
\norm{r^{\sigma_m} \bar{\partial}^m \mathsf{f} }_{L^{p_m}}^{\theta_j}
\nonumber
\end{align}
for $1\leq p_j, p_m \leq \infty$, $\theta_j = \frac{j}{m}$, $\frac{1}{p_j} = \frac{1-\theta_j}{p_0}
+ \frac{\theta_j}{p_m}$, $\sigma_j = \sigma_0(1-\theta_j) + \sigma_m \theta_j$,
$m-\sigma_m - 3(\frac{1}{p_m} - \frac{1}{p_0}) > - \sigma_0$, $\sigma_j > -\frac{1}{p_j}$, 
$0 < j < m$, $\sigma_0, \sigma_m \in \mathbb{R}$.

\begin{align}
\norm{ r^{\sigma_j} \bar{\partial}^j \mathsf{f} }_{L^{p_j}} \lesssim
\norm{\mathsf{f} }_{L^{\infty}}^{1-\theta_j}
\norm{r^{\sigma_m} \bar{\partial}^m \mathsf{f} }_{L^{2}}^{\theta_j}
\nonumber
\end{align}
for $\theta_j = \frac{j}{m}$, $\frac{1}{p_j} = \frac{\theta_j}{2}$, $\sigma_j = \sigma_m \theta_j$,
$m-\sigma_m - \frac{3}{2} > 0$,  
$0 < j < m$, $\sigma_m > -\frac{1}{2}$.

\begin{align}
\norm{ r^{\sigma_j} \bar{\partial}^j \mathsf{f} }_{L^{p_j}} \lesssim
\norm{ \mathsf{f} }_{\dot{C}^\frac{1}{2}}^{1-\theta_j}
\norm{r^{\sigma_m} \bar{\partial}^m \mathsf{f} }_{L^{p_2}}^{\theta_j}
\nonumber
\end{align}
for $\theta_j = \frac{2j-1}{2m-1}$, $\frac{1}{p_j} = \frac{\theta_j}{2}$, 
$\sigma_j = \sigma_m \theta_j$,
$m-\frac{1}{2}-\sigma_m - \frac{3}{2} > 0$,
$0 < j < m$, $\sigma_m > -\frac{1}{2}$.

\begin{align}
\norm{ r^{\sigma_j} \bar{\partial}^j \mathsf{f} }_{L^{p_j}} \lesssim
\norm{ \mathsf{f} }_{\tilde{C}^\frac{1}{2}}^{1-\theta_j}
\norm{r^{\sigma_m} \bar{\partial}^m \mathsf{f} }_{L^{p_2}}^{\theta_j}
\nonumber
\end{align}
for $\theta_j = \frac{j}{m}$, $\frac{1}{p_j} = \frac{\theta_j}{2}$, 
$\sigma_j = \sigma_m \theta_j-\frac{1}{2}(1-\theta_j)$,
$m-\frac{1}{2}-\sigma_m - \frac{3}{2} > 0$,
$0 < j < m$, $\sigma_m > \frac{m-2}{2}$

Recall that $\tilde{C}^\frac{1}{2}$ is defined in \eqref{E:C_tilde_half_norm} and $\dot{C}^\frac{1}{2}$ is the H\"older semi-norm. 
\end{lemma}

We will also make silent use of the following Hardy-type embedding,
\begin{align}
\label{E:Weighted_embedding}
H^{N_1,\sigma_1} \subset H^{N_2,\sigma_2},
\end{align}
for $N_1 > N_2 \geq 0$, $\sigma_1 > \sigma_2 > -\frac{1}{2}$, and $N_1 - N_2 = \sigma_1 - \sigma_2$.

We are finally ready to establish \eqref{E:Equivalence_energy_norm}. We start with the $\lesssim$ part.

Consider the definition of $(s_{2\ell},w_{2\ell})$ given by \eqref{E:Good_linear_variables}. Using \eqref{E:Diagonal_Euler_system}, we successively solve for material derivatives in terms of spatial derivatives, expressing $(s_{2\ell},w_{2\ell})$ as a linear combination of multi-linear expressions
in $r,\bar{\partial}^k r, \bar{\partial}^k v$, $0\leq k \leq 2\ell$, with coefficients that are smooth functions of $(r,v)$ of order zero. In order to simplify the analysis, we proceed in steps, supposing first that in this procedure we ignore the last term in \eqref{E:Diagonal_Euler_system_sound_speed} (which, we recall, has lower order than the remaining terms). Moreover, let us also first consider 
the case that whenever we commute $D_t$ with $\bar{\partial}$ (which is needed for this procedure of solving for material derivatives in terms of spatial derivatives), all derivatives fall on $v$ and not on $r$ (recall that $D_t$ involves $r$ via $v^0$). Finally, let us also focus at first only on the terms that are top order. In this case, the corresponding multi-linear expressions $s_{2\ell}$ and $w_{2\ell}$
have the following properties:
\begin{itemize}
\item They have orders $\ell-1$, $\ell-\frac{1}{2}$, respectively.
\item They have exactly $2\ell$ derivatives.
\item They contain at most $\ell+1$, $\ell$, factors of $r$, respectively.
\end{itemize}

For $s_{2\ell}$, the corresponding multi-linear expressions with the above properties can be written as
\begin{align}
r^a \prod_{j=1}^J \partial^{n_j} r \prod_{l=1}^L \partial^{n_l} v,
\nonumber
\end{align}
where $n_j, m_l \geq 1$, 
\begin{align}
\sum_{j=1}^J n_j + \sum_{l=1}^L n_l = 2\ell,
\nonumber
\end{align}
and
\begin{align}
a + J + \frac{L}{2} = \ell + 1.
\nonumber
\end{align}
When $J=0$ or $L=0$ the corresponding produce is absent.

With a bit of algebraic manipulations, we can show that these constraints imply that we can choose $b_j$ and $c_l$ such that 
\begin{align}
\begin{split}
& 0 \leq b_j \leq (n_j-1) \frac{\ell}{2\ell -1}, \\
& 0 \leq c_l \leq (m_l-1) \frac{\ell+\frac{1}{2}}{\ell - \frac{1}{2}},\\
& a = \sum_{j=1}^J b_j + \sum_{l=1}^L c_l.
\end{split}
\nonumber
\end{align}
With these choices, a cumbersome but not difficult inspections shows that we can invoke Lemma \ref{L:Weighted_interpolation} to obtain
\begin{align}
\begin{split}
\norm{r^{b_j} \partial^{n_j} r }_{L^{p_j}(r^\frac{1-\kappa}{\kappa})} 
& \lesssim (1+A)^{1-\frac{2}{p_j}} \norm{ (r,v) }^\frac{2}{p_j}_{\Hspace^{2\ell}}
\\
\norm{r^{c_l} \partial^{m_l} r }_{L^{p_l}(r^\frac{1-\kappa}{\kappa})} 
& \lesssim (1+A)^{1-\frac{2}{q_l}} \norm{ (r,v) }^\frac{2}{q_l}_{\Hspace^{2\ell}},
\end{split}
\nonumber
\end{align}
where
\begin{align}
\begin{split}
\frac{1}{p_j} & = \frac{n_j -1 - b_j}{2(\ell-1)},
\\
\frac{1}{q_l} & = \frac{m_l -\frac{1}{2} - c_l}{2(\ell-1)},
\end{split}
\nonumber
\end{align}
and
\begin{align}
\norm{ \mathsf{f} }^p_{L^p(r^\frac{1-\kappa}{\kappa})} := \int_{\Md_t} |\mathsf{f}|^p \,r^\frac{1-\kappa}{\kappa} \, dx.
\nonumber
\end{align}
(Observe that the numerators in $\frac{1}{p_j}$ and $\frac{1}{q_l}$ correspond to the orders of the expressions being estimated and add to $\ell-1$, as needed.) This gives the desired estimate for the top-order terms in $s_{2\ell}$ under the simplifications mentioned above. The remaining terms in $s_{2\ell}$, not considered under the above simplifications, are analyzed in a similar fashion. In fact, they are easier as they are lower-order and thus require less care to apply the Lemma \ref{L:Weighted_interpolation}. A similar analysis can be done for $w_{2\ell}$. This concludes the $\lesssim$ part of the proof of \eqref{E:Equivalence_energy_norm}.

We next move to the $\gtrsim$ part. This part of the proof is more delicate and is done with the help of elliptic estimates. It is also of interest on its own in view of the weighted elliptic estimates we will derive and thus we present it in the next section.

\subsubsection{Weighted elliptic estimates\label{S:Weighted_elliptic_estimates}} In this section we establish the $\gtrsim$ part of \eqref{E:Equivalence_energy_norm}. We begin with the following relations satisfied by $(s_{2\ell},w_{2\ell})$
\begin{subequations}{\label{E:Transition_relations}}
\begin{align}
s_{2\ell} & = L_1 s_{2\ell - 2} + F_{2\ell},
\label{E:Transition_relations_s}
\\
w_{2\ell} &= L_2 w_{2\ell -2 } + H_{2\ell},
\label{E:Transition_relations_w}
\end{align}
\end{subequations}
where 
\begin{subequations}{\label{E:Transition_operators}}
\begin{align}
L_1 s & := a_2 (\bar{H}^{-1})^{ij} ( r \partial_i \partial_j s + \frac{1}{\kappa} \partial_i r \partial_j s),
\label{E:L_1}
\\
(L_2 w)_i &:= a_2 (\bar{H}^{-1})^{pq} ( \partial_i (r \partial_p w_q) + \frac{1}{\kappa} \partial_p r \partial_i w_q ),
\label{E:L_2}
\end{align}
\end{subequations}
and $(F_{2\ell},H_{2\ell})$ are error terms. Let us elaborate on \eqref{E:Transition_relations} and 
\eqref{E:Transition_operators}. Commuting $D_t$ through \eqref{E:Diagonal_Euler_system_sound_speed} and using \eqref{E:Diagonal_Euler_system_velocity},
\begin{align}
D_t^2 r + r (\bar{H}^{-1})^{ij} \partial_i \underbrace{D_t v_j}_{\mathclap{=-a_2 \partial_j r}} & = \dots 
\nonumber
\end{align}
so that, schematically,
\begin{align}
D_t^2 r - r \bar{\partial}^2 r = \dots
\nonumber
\end{align}
This, as discussed in our sketch of the proof of Theorem \ref{T:LWP_vacuum_bry}, is a wave equation for $r$. Taking further material derivatives and proceeding inductively,
\begin{align}
D_t^{2\ell} r - r \bar{\partial}^2 D_t^{2\ell-2} r = \dots
\nonumber
\end{align}
In the spirit of Remark \ref{R:Good_variables_material_derivatives}, we ignore the correction terms in the definition of $s_{2\ell}$ for a moment, in which case we have
\begin{align}
s_{2\ell} - r \bar{\partial}^2  s_{2\ell-2} = \dots
\nonumber
\end{align}
which has the same derivative counting as \eqref{E:Transition_relations_s}, with $s_{2\ell}$ being equal to a second-order spatial operator acting on $s_{2\ell-2}$ up to an error term indicated here by $\dots$. This second-order operator is degenerate due to the $r$ factor.

Considering now the correct form of $s_{2\ell}$ with the correction term given by \eqref{E:s_ell}, taking material derivatives and using \eqref{E:Diagonal_Euler_system}, we arrive at \eqref{E:Transition_relations_s}, with the second-order operator in question given by \eqref{E:L_1} and all the remaining terms collected in $F_{2\ell}$. A similar argument gives \eqref{E:Transition_relations_w}. 

\begin{remark}
Operators \eqref{E:Transition_operators} involve some first-order derivative terms which are not uniquely determined by the above procedure in that we can consider different first-order derivative terms that could in principle be included in their definitions or relegated to the RHS in Eqs.~\eqref{E:Transition_relations}. The choice of first-order-in-derivative terms in these operators is carefully made in order to ensure their ellipticity properties shown below, while also guaranteeing the the terms on the RHS can indeed be proven to be perturbative. In this argument, one also uses that the last term on the LHS of \eqref{E:Diagonal_Euler_system_sound_speed} has order lower by $-\frac{1}{2}$ than the remaining terms. Observe that $L_1, L_2 \sim r \Delta$, in accordance with the outline of the proof of Theorem \ref{T:LWP_vacuum_bry}.
\end{remark}

We have set a up a hierarchy of good linear variables $(s_{2\ell}, w_{2\ell})$, $\ell=0,\dots, N$, which satisfy 
the linearized system with source terms, leading to the energy \eqref{E:Energy_good_linear_variables}. The operators $L_1$ and $L_2$ connect the good linear variables at different levels in our hierarchy, expressing them at level $2\ell$ in terms of their counter-parts at level $2\ell-2$, up to error terms. Because of this property, the operators \eqref{E:Transition_operators} are referred to as ``transition operators'' in \cite{Disconzi-Ifrim-Tataru-2022}.

We are now ready to establish the $\gtrsim$ part of \eqref{E:Equivalence_energy_norm}. The main tool will be the following \emph{weighted elliptic estimates} which we will establish,
\begin{subequations}{\label{E:Weighted_elliptic_estimates}}
\begin{align}
\norm{s}_{H^{2,\frac{1}{2\kappa} + \frac{1}{2}}} & \lesssim
\norm{L_1 s}_{H^{0,\frac{1}{2\kappa} - \frac{1}{2}}} + \norm{s}_{L^2(r^\frac{1-\kappa}{\kappa})},
\label{E:Weighted_elliptic_estimate_s}
\\
\norm{w}_{H^{2,\frac{1}{2\kappa} + 1}} & \lesssim
\norm{L_2 w}_{H^{0,\frac{1}{2\kappa}}} + \norm{w}_{L^2(r^{\frac{1-\kappa}{\kappa}+1})}.
\label{E:Weighted_elliptic_estimate_w}
\end{align}
\end{subequations}
Observe that at this point, \eqref{E:Weighted_elliptic_estimates} establishes general properties of the operators $L_1$ and $L_2$ without appeal to \eqref{E:Good_linear_variables}. After establishing  \eqref{E:Weighted_elliptic_estimates} we show how to apply it to $(s_{2\ell},w_{2\ell})$ to obtain \eqref{E:Equivalence_energy_norm}.

\begin{remark}
Strictly speaking, estimate \eqref{E:Weighted_elliptic_estimate_w} is not correct. Observe that the second-order-in-derivatives part of $L_2$ contains one derivative acting on the divergence of $w$, and as such controls only the divergence-part of $w$. This needs to be complemented by an estimate for the curl of $w$, as in standard div-curl estimates. This is done by estimating the transport part of the system but we do not do so here, see Remark \ref{R:Transport_estimates_implicit}.
\end{remark}

We will focus on \eqref{E:Weighted_elliptic_estimate_s}, with \eqref{E:Weighted_elliptic_estimate_w} proven by similar arguments. First, integration by parts in a standard elliptic fashion yields
\begin{align}
\norm{s}_{H^{2,\frac{1}{2\kappa} + \frac{1}{2}}} & \lesssim
\norm{L_1 s}_{H^{0,\frac{1}{2\kappa} - \frac{1}{2}}} + 
\norm{s}_{H^{1,\frac{1}{2\kappa} - \frac{1}{2}}}.
\label{E:Weighted_elliptic_estimate_weaker_s}
\end{align}
Thus, in view of \eqref{E:Weighted_elliptic_estimate_weaker_s}, it suffices to prove
\begin{align}
\norm{s}_{H^{1,\frac{1}{2\kappa} -\frac{1}{2}}} & \lesssim
\norm{L_1 s}_{H^{0,\frac{1}{2\kappa} - \frac{1}{2}}} + \norm{s}_{L^2(r^\frac{1-\kappa}{\kappa})},
\label{E:Weighted_elliptic_estimate_s_H1}
\end{align}
Compute
\begin{align}
\label{E:Cancellation_weighted_elliptic}
\begin{split}
\int_{\Md_t} r^\frac{1-\kappa}{\kappa} \partial_3 s L_1 s 
&=
\int_{\Md_t} r^\frac{1-\kappa}{\kappa} \partial_3 s
a_2 (\bar{H}^{-1})^{ij} ( r 
\overset{\mathclap{
	\textcolor{blue}{
		\substack{\text{integrate by parts}\\\uparrow}
	}
	}}{\textcolor{blue}{\partial_i}}
 \partial_j s + \frac{1}{\kappa} \partial_i r \partial_j s),
\\
&=
-\int_{\Md_t} r^\frac{1}{\kappa} a_2
\underbrace{\partial_3 \partial_i s (\bar{H}^{-1})^{ij} \partial_j s }_{\mathclap{=\frac{1}{2}(\bar{H}^{-1})^{ij} \partial_3 (\partial_i s \partial_j s)}}
\textcolor{red}{- \int_{\Md_t} \frac{1}{\kappa} r^\frac{1-\kappa}{\kappa} \partial_i r \partial_3 s 
(\bar{H}^{-1})^{ij} a_2 \partial_j w}
\\
& \ \ \ 
- \int_{\Md_t} r^\frac{1}{\kappa} \partial_3 s \partial_i (a_2 (\bar{H}^{-1})^{ij} ) \partial_j s
\textcolor{red}{\,
+ \int_{\Md_t}  r^\frac{1-\kappa}{\kappa} \partial_3 s (\bar{H}^{-1})^{ij} a_2 \frac{1}{\kappa} \partial_i r \partial_j s}
\\
& = 
-\frac{1}{2} \int_{\Md_t} r^\frac{1}{\kappa} a_2
(\bar{H}^{-1})^{ij} \partial_3 (\partial_i s \partial_j s)
- \int_{\Md_t} r^\frac{1}{\kappa} \partial_3 s \partial_i (a_2 (\bar{H}^{-1})^{ij} ) \partial_j s,
\end{split} 
\end{align}
where the terms in \textcolor{red}{red} cancel out and there are not boundary terms due to the presence of $r$ factors. We note that the first term in \textcolor{red}{red} comes from when $\partial_i$ falls on
$r^\frac{1-\kappa}{\kappa} r = r^\frac{1}{\kappa}$. This cancellation is needed because the terms in \textcolor{red}{red} are quadratic in $\bar{\partial} s$ (which we want to control from below by $L_1 s$ in \eqref{E:Weighted_elliptic_estimate_s_H1}) without any sign or smallness. Integrating $\partial_3$ by parts in the term $\partial_3 (\partial_i s \partial_j s)$ and using again that there are no boundary terms,
\begin{align}
\label{E:After_cancellation_weighted_elliptic}
\begin{split}
\int_{\Md_t} r^\frac{1-\kappa}{\kappa} \partial_3 s L_1 s 
& = 
\frac{1}{2\kappa} \int_{\Md_t}a_2 r^\frac{1-\kappa}{\kappa} \partial_3 r
(\bar{H}^{-1})^{ij} \partial_i s \partial_j s
+ \frac{1}{\kappa}\int_{\Md_t}r^\frac{1}{\kappa} \partial_3 (a_2 (\bar{H}^{-1})^{ij}) \partial_i s \partial_j s
\\
& \ \ \
- \int_{\Md_t} r^\frac{1}{\kappa} \partial_3 s \partial_i (a_2 (\bar{H}^{-1})^{ij} ) \partial_j s,
\end{split} 
\end{align}
Recall now that we can assume to be working on a neighborhood of a point $x_0 \in \partial \Md_t$ where
$\bar{\nabla} r(x_0) = \mathsf{N}(x_0)$, so that $|\bar{\nabla} r - \mathsf{N} | \leq A = O(\varepsilon) \ll 1$, 
where $A$ is given by \eqref{E:Control_norm_A}. We can arrange the coordinates such that $N(x_0) = (0,0,1) =: \mathsf{e}_3$, so that 
\begin{align}
\label{E:Key_coercive_term_weighted_elliptic}
\begin{split}
\frac{1}{2\kappa} \int_{\Md_t}a_2 r^\frac{1-\kappa}{\kappa} \partial_3 r
(\bar{H}^{-1})^{ij} \partial_i s \partial_j s
&= 
\frac{1}{2\kappa} \int_{\Md_t}a_2 r^\frac{1-\kappa}{\kappa} ( \partial_3 r - \mathsf{N}_3 )
(\bar{H}^{-1})^{ij} \partial_i s \partial_j s
\\
& \ \ \ 
+ \frac{1}{2\kappa} \int_{\Md_t}a_2 r^\frac{1-\kappa}{\kappa} \mathsf{N}_3 
(\bar{H}^{-1})^{ij} \partial_i s \partial_j s
\\
& \gtrsim - \varepsilon \int_{\Md_t} r^\frac{1-\kappa}{\kappa} |\bar{\partial} s|^2 
+  \int_{\Md_t} r^\frac{1-\kappa}{\kappa} 
(\bar{H}^{-1})^{ij} \partial_i s \partial_j s,
\end{split}
\end{align}
where $\mathsf{N}_3 $ is the third component of $\mathsf{N}$, we used that $\mathsf{N}_3  \gtrsim \text{constant} > 0$ in the neighborhood of $x_0$, and that 
that $a_2$ is bounded from below away from zero (see \eqref{E:a_2_bounded_from_below}).
This takes cares of the first integral on the RHS of \eqref{E:After_cancellation_weighted_elliptic}. For the remaining two integrals, we note that, like the integrals in \textcolor{red}{red} that cancel out in \eqref{E:Cancellation_weighted_elliptic}, they do not have a sign and are quadratic in $\bar{\partial} s$. However, \emph{unlike} those \textcolor{red}{red} terms in \eqref{E:Cancellation_weighted_elliptic}, the last two integrals
in \eqref{E:After_cancellation_weighted_elliptic} carry small terms in that we have a power of $r$ to spare, i.e.
\begin{align}
r^\frac{1}{\kappa} = r^\frac{1-\kappa}{\kappa} r = r^\frac{1-\kappa}{\kappa} O(\varepsilon),
\label{E:Power_of_r_to_spare_weighted_elliptic}
\end{align}
where we recall Remark \ref{R:Working_neighborhood_boundary}. Thus, combining 
\eqref{E:After_cancellation_weighted_elliptic}, \eqref{E:Key_coercive_term_weighted_elliptic}, and \eqref{E:Power_of_r_to_spare_weighted_elliptic},
\begin{align}
\int_{\Md_t} r^\frac{1-\kappa}{\kappa} \partial_3 s L_1 s \gtrsim
- \varepsilon \int_{\Md_t} r^\frac{1-\kappa}{\kappa} |\bar{\partial} s|^2 
+  \int_{\Md_t} r^\frac{1-\kappa}{\kappa} 
(\bar{H}^{-1})^{ij} \partial_i s \partial_j s.
\nonumber
\end{align}
Applying the Cauchy--Schwarz-with-epsilon inequality to the LHS and adding the $\norm{s}_{L^(r^\frac{1-\kappa}{\kappa})}$ to both sides produces \eqref{E:Weighted_elliptic_estimate_s_H1}, and hence \eqref{E:Weighted_elliptic_estimate_s} in view of \eqref{E:Weighted_elliptic_estimate_weaker_s}, as desired. We recall that the proof of \eqref{E:Weighted_elliptic_estimate_w} is obtained with similar ideas.

In order to finish the proof of \eqref{E:Equivalence_energy_norm}, we need to analyze the structure of the error terms $(F_{2\ell},H_{2\ell})$ in \eqref{E:Transition_relations} and show how the weighted elliptic estimates \eqref{E:Weighted_elliptic_estimates} are in practice applyed to $(s_{2\ell},w_{2\ell})$.
We begin with $(F_{2\ell},H_{2\ell})$.

The key estimate we need for the error terms $(F_{2\ell},H_{2\ell})$ is the inequality
\begin{align}
\label{E:Estimate_error_terms_weighted_elliptic}
\norm{(F_{2\ell},H_{2\ell})}_{\Hspace^{2N - 2\ell}} \lesssim \varepsilon 
\norm{(r,v)}_{\Hspace^{2N}}.
\end{align}
To prove \eqref{E:Estimate_error_terms_weighted_elliptic}, 
analyzing the expressions for $(F_{2\ell},H_{2\ell})$ reveals that the terms $F_{2\ell}$ and $H_{2\ell}$ have the following properties:
\begin{itemize}
\item They are linear combinations of multi-linear expressions in $r,\bar{\partial}^k r, \bar{\partial}^k v$ with coefficients that are order-zero smooth functions of $(r,v)$.
\item They have $2\ell$ derivatives and are of order at most $\ell-1$ and $\ell-\frac{1}{2}$, respectively.
\end{itemize}
Moreover, they either:
\begin{itemize}
\item  Have order  $\ell-1$ and $\ell-\frac{1}{2}$, respectively, and contain at most $\ell+1$, $\ell$, factors of $r$, respectively, and their corresponding products contain at least two factors of $\bar{\partial}^{\geq 2} r$ or $\bar{\partial}^{\geq 1} v$,
\end{itemize}
or 
\begin{itemize}
\item Have order strictly less than  $\ell-1$ and $\ell-\frac{1}{2}$, respectively, and contain at most
$\ell+2$, $\ell+1$ factors of $r$, respectively.
\end{itemize}

We will not establish this characterization of $(F_{2\ell},H_{2\ell})$ here, referring to Lemma 5.2 in \cite{Disconzi-Ifrim-Tataru-2022}. We note, however, that these facts about of $(F_{2\ell}, H_{2\ell})$ are not straightforward (for example, they do not follow by a simple derivative counting) and require observing some special cancellations.

With the above description of $(F_{2\ell},H_{2\ell})$ at hand, these terms are then estimated with the help of the weighted interpolation inequalities of Lemma \ref{L:Weighted_interpolation} akin to estimate for the $\lesssim$ part of \eqref{E:Equivalence_energy_norm} carried out in the previous section. In doing so, we need to observe that the interpolations can be carried out in such a way that each term in $(F_{2\ell},H_{2\ell})$ has at least one small factor. This is ensured by the above structure of the the error terms, which guarantees that upon interpolating, one always obtains either a term in $O(A)$ or a term with an extra (with respect to the norm being estimated) power of $r$, in such a way that we always obtained a term $O(\varepsilon)$ as indicated in \eqref{E:Estimate_error_terms_weighted_elliptic}.

In order to conclude the proof, we will show that
\begin{align}
\label{E:Concatenating_estimate}
\norm{ (s_{2\ell-2}, w_{2\ell-2} )}_{\Hspace^{2N-2\ell + 2}} \lesssim
\norm{ (s_{2\ell}, w_{2\ell} )}_{\Hspace^{2N-2\ell}}
+
\varepsilon 
\norm{(r,v)}_{\Hspace^{2N}}
\end{align}
for $1 \leq \ell \leq N$. Concatenating the estimates \eqref{E:Concatenating_estimate} from $\ell=1$ up to $\ell=N$ and using that $\varepsilon$ is small then produces the $\gtrsim$ part of \eqref{E:Equivalence_energy_norm}. (In this concatenation and the use of the above estimates, readers should be reminded of Assumption \ref{A:Modification_small_ell} for completeness.)

Thus, it remains to establish \eqref{E:Concatenating_estimate}. For $\ell=N$, this is simply
\eqref{E:Weighted_elliptic_estimates} and \eqref{E:Estimate_error_terms_weighted_elliptic}
applied to \eqref{E:Transition_relations}.
For $\ell < N$, we commute \eqref{E:Transition_relations} with $r^m \bar{\partial}^n$. For \eqref{E:Transition_relations_s}, this yields
\begin{align}
\label{E:Weighted_elliptic_estimate_commutator_1}
r^m \bar{\partial}^n s_{2\ell} = L_1 (r^m \bar{\partial}^n s_{2\ell -2 } ) +
[r^m \bar{\partial}^n , L_1] s_{2\ell -2 } + r^m \bar{\partial}^n F_{2\ell}.
\end{align}
At this point, one can choose $m$ and $n$ appropriately so that the following properties hold:
\begin{itemize}
\item $r^m \bar{\partial}^n F_{2\ell}$ can be analyzed invoking the structure of $F_{2\ell}$ described above and produces an estimate similar to \eqref{E:Estimate_error_terms_weighted_elliptic}.
\item The commutator term $[r^m \bar{\partial}^n , L_1] s_{2\ell -2 }$ can be decomposed as
\begin{align}
\label{E:Weighted_elliptic_estimate_commutator_2}
[r^m \bar{\partial}^n , L_1] s_{2\ell -2 } = n a_2 (\bar{H}^{-1})^{ij} \partial_i r \partial_j (r^m \bar{\partial}^n s_{2\ell-2}) + \tilde{F}_{2\ell},
\end{align}
where $\tilde{F}_{2\ell}$ is an error term that can be controlled via interpolation and produces an estimate similar to \eqref{E:Estimate_error_terms_weighted_elliptic}.
\end{itemize}

Thus, in view of \eqref{E:Weighted_elliptic_estimate_commutator_1} and \eqref{E:Weighted_elliptic_estimate_commutator_2}, recalling the definition of $L_1$ in \eqref{E:L_1}, and writing $\dots$ for error terms, we have
\begin{align}
r^m \bar{\partial}^n s_{2\ell} = \tilde{L}_1 (r^m \bar{\partial}^n s_{2\ell -2 } ) + \dots,
\nonumber
\end{align}
where
\begin{align}
\tilde{L}_1 s & := a_2 (\bar{H}^{-1})^{ij} \left( r \partial_i \partial_j s + 
( \frac{1}{\kappa} + n) \partial_i r \partial_j s \right).
\nonumber
\end{align}
This is an operator with similar structure to $L_1$ and a similar argument as in the proof of 
\eqref{E:Weighted_elliptic_estimate_s} establish an elliptic estimate giving  \eqref{E:Concatenating_estimate} on the first component (taking into account the suitable values of $m$ and $n$ that are chosen); the argument for the second component (i.e., for $w_{2\ell}$) is similar. (Observe that in the case $\ell = N$ we do not need to commute with $r^m \bar{\partial}^n$, so that $m=n=0$ and $\tilde{L}_1 = L_1$, as expected). This finishes the proof of \eqref{E:Equivalence_energy_norm}.

\begin{remark}
Readers will notice that our use of weighted elliptic estimates here was slightly different than how such estimates were advertised in our sketch of the proof of Theorem \ref{T:LWP_vacuum_bry}. There, we invoked elliptic estimates in the middle of the argument for establishing energy estimates, whereas here they are used to establish the coercivity of the energy \eqref{E:Energy_good_linear_variables}. But this is a mere change of perspective. In both cases, we are interested in controlling an energy constructed out of material derivatives of $(r,v)$ (with corrections terms, which have been ignored in the initial outline of the proof), and \emph{elliptic estimates are invoked to produce a bound from below showing that the energy controls the relevant norms.}
\end{remark}

\subsubsection{Energy estimates\label{S:Energy_estimates_solutions_diagonal}}
Having established \eqref{E:Equivalence_energy_norm}, we turn our attention to the energy estimate 
\begin{align}
\label{E:Energy_estimate_solutions}
\frac{d}{dt} E^{2N} \lesssim_A B \norm{ (r,v) }^2_{\Hspace^{2N}},
\end{align}
i.e., part (b) of Theorem \ref{T:Continuation_solutions_vacuum_bry}.

We have already hinted at the proof of \eqref{E:Energy_estimate_solutions}: we apply the energy estimates for the linearized system from Sect.~\ref{S:Energy_estimates_linearized} to \eqref{E:Good_linear_variables_system}. In order to close the estimate, we need to show that the RHS 
of \eqref{E:Good_linear_variables_system} is perturbative with respect to our norms, i.e., it  satisfies
the estimate
\begin{align}
\norm{(f_{2\ell},h_{2\ell})}_{\Hspace^0} \lesssim_A B \norm{(r,v)}_{\Hspace^{2\ell}}.
\label{E:Perturbative_terms_good_linear_variables_source}
\end{align}
The proof of \eqref{E:Perturbative_terms_good_linear_variables_source} involves ideas reminiscent to the proof of \eqref{E:Equivalence_energy_norm}. First, we need to understand the structure of the the 
terms $(f_{2\ell},h_{2\ell})$. The terms $(f_{2\ell},h_{2\ell})$ satisfy the following properties:
\begin{itemize}
\item They are linear combinations of multi-linear expressions in $(r,\bar{\partial}^k r, \bar{\partial}^k v)$ with coefficients that are smooth functions of $(r,v)$.
\item They have order at most $\ell-\frac{1}{2}$, $\ell$, respectively.
\item Each term in $(f_{2\ell},h_{2\ell})$ contains exactly $2\ell+1$ derivatives.
\item No single factor in $(f_{2\ell},h_{2\ell})$ has order larger than $\ell-1$, $\ell-\frac{1}{2}$, respectively.
\end{itemize} 

As in the characterization of the terms $(F_{2\ell}, H_{2\ell})$ in Sect.~\ref{S:Weighted_elliptic_estimates}, we note that the proof of these properties for $(f_{2\ell},h_{2\ell})$ is not straightforward (e.g., it does not follow by a simple derivative counting) and requires observing some cancellations.

Once the above features of $(f_{2\ell},h_{2\ell})$ have been established, the proof of \eqref{E:Perturbative_terms_good_linear_variables_source} follows by a judicious application of the weighted interpolation inequalities of Lemma \ref{L:Weighted_interpolation} in the same spirit as the proof of the $\lesssim$ part of \eqref{E:Equivalence_energy_norm} in Sect.~\ref{S:Energy_norm_bookkeeping}. We remark that special attention need to be paid when applying these interpolations in order to guarantee that we always obtain a factor linear in $B$, as in \eqref{E:Energy_estimate_solutions}.

\subsection{Remaining arguments\label{S:Remaining_LWP}}

So far we have discussed only energy estimates for solutions to \eqref{E:Diagonal_Euler_system}. Naturally, energy estimates are one of the cornerstones for the local well-posedness stated 
in Theorem \ref{T:LWP_vacuum_bry}. In this section, we briefly mention the remaining arguments needed to go from energy estimates to local well-posedness. Construction of solutions relies on a \emph{time discretization} which involves the following:

\begin{itemize}
\item A regularization step.
\item A transport step, i.e., an iteration of the boundary at each time step (where the boundary is iterated linearly with the given velocity at each step).
\item An application of an Euler's-method-type of iteration.
\end{itemize}

An interesting aspect is that when these steps are taken separately, all of them seem to be unbounded. However, when taken together, there is an extra cancellation that comes to rescue. This cancellation is a direct analogue of the key cancellation we observed for the linearized equations (involving a term coming from the linearization of the material derivative).

We use our energy estimates in order to control the iteration and guarantee convergence to a solution in the limit when the discretization parameter goes to zero. More precisely, the energy estimates will contain extra error terms coming from the fact that the discretized variables are only an approximate solution to the problem, but these error terms can be controlled within our functional framework. More importantly, the energy estimates are obtained for a continuous time variable upon taking $\frac{d}{dt}$ of the energy. Here, we have a discrete time variable. But the energy estimates for a continuous time variable can be translated into estimates for the variables at \emph{fixed time} by reinterpreting the operators $D_t^\ell$ as operators at fixed time which are obtained by reiterating the equations, solving for material derivatives in terms of spatial derivatives. With such control, we can then pass to the continuum limit. Continuous dependence on the data is proven with help of the above regularization.

Let us finally mention how uniqueness is established. Once again, the linearized equations play a crucial role. In order to establish uniqueness, we need to compare different solutions, but distinct solutions are in principle defined on different domains (see Remark \ref{R:Lagrangian_uniqueness} below). Given two solutions $(r_1,v_1)$ and $(r_2,v_2)$ defined on domains\footnote{Shrinking the time intervals if needed, we can assume both solutions to be defined on the same time interval $[0,T]$.} 
\begin{align}
\Md^1 := \bigcup_{0 \leq t < T} \{ t \} \times \Md^1_t
\, \text{ and } \,
\Md^2 := \bigcup_{0 \leq t < T} \{ t \} \times \Md^2_t,
\nonumber
\end{align}
respectively, set 
\begin{align}
\Md_t := \Md^1_t \cap \Md^2_t, \, \Gamma_t := \partial \Md_t.
\nonumber
\end{align}

Define the following distance functional to measure the distance between solutions:
\begin{align}
\label{E:Distance_functional}
\mathcal{D}((r_1,v_1), (r_2, v_2) )  := 
\int_{\Md_t} (r_1 + r_2)^\frac{1-\kappa}{\kappa} \left( 
(r_1 - r_2 )^2 + (r_1 + r_2) |v_1 - v_1|^2 \right) \, dx,
\end{align}
which is directly inspired by the energy for the linearized system, compare with \eqref{E:Weighted_energy_linearized}.

\begin{remark}
\label{R:Distance_functional_Euclidean}
From the estimate for the linearized system in Sect.~\ref{S:Energy_estimates_linearized}, it seems it would be more natural, in fact required, (see Remark \ref{R:Energy_H_not_Euclidean}) to measure the norm of the difference $|v_1 - v_2|$ with respect to the $\bar{H}^{-1}$ metric for $(r_1,v_1)$, $(r_2,v_2)$, or a combination of both. This will not be required, however, because the desired energy estimate will not be carried out directly for $\mathcal{D}((r_1,v_1), (r_2, v_2) )$, see below.
\end{remark}

We observe the following key property of \eqref{E:Distance_functional} for establishing uniqueness,
\begin{align}
\label{E:Distance_functional_zero}
\mathcal{D}((r_1,v_1), (r_2, v_2) )=0  \, \text{ if and only if } \, 
(r_1,v_1) = (r_2, v_2).
\end{align}

In view of \eqref{E:Distance_functional_zero}, uniqueness then follows from the following estimate
\begin{align}
\label{E:Energy_estimate_distance_functional}
\sup_{t\in [0,T]} \mathcal{D}((r_1,v_1)(t), (r_2, v_2)(t) )  \lesssim
\mathcal{D}((r_1,v_1)(0), (r_2, v_2)(0) ) .
\end{align}

We would like to establish \eqref{E:Energy_estimate_distance_functional} by a an energy estimates, computing $\frac{d}{dt}\mathcal{D}((r_1,v_1), (r_2, v_2) )$, canceling terms with help of integration by parts, and Gr\"onwalling. This does not seem possible, however, because the weight $r_1+r_2$ in \eqref{E:Distance_functional} does not necessarily vanish on $\Gamma_t$. Indeed, it will vanish only at points where both boundaries $\Md^1_t$ and $\Md^2_t$ coincide (since
$\Md^{1,2}_t := \{ r_{1,2} > 0 \}$). This difficulty is overcome by introducing another distance functional,
\begin{align}
\label{E:Distance_functional_degenerate}
\begin{split} 
\tilde{\mathcal{D}}((r_1,v_1), (r_2, v_2) )  & := 
\int_{\Md_t} (r_1 + r_2)^\frac{1-\kappa}{\kappa} \big ( 
\alpha (r_1 - r_2 )^2 
\\
& \ \ \ 
+ \beta (a_{21} + a_{22})^{-1} 
(\bar{H}^{-1}_{\text{mid}})^{ij} (v_1-v_2)_i (v_1-v_2)_j \big) \, dx.
\end{split}
\end{align}
In \eqref{E:Distance_functional_degenerate}, $\bar{H}^{-1}_{\text{mid}}$ is an average between $(\bar{H}^{-1})^{ij}$ 
for the solutions $(r_1,v_1)$ and $(r_2, v_2)$, i.e., \eqref{E:H_bar_inverse} with $r$ and $v$ replaced by $\frac{r_1+r_2}{2}$ and $\frac{v_1+v_2}{2}$, respectively; $a_{21}$ and $a_{22}$ are the coefficient $a_2$ corresponding to the solutions $(r_1,v_1)$ and $(r_2, v_2)$, respectively; most importantly, $\alpha$ and $\beta$ and functions of $r_1+r_2$ and $r_1-r_2$ that are \emph{carefully chosen to have the right vanishing rate near $\Gamma_t$ so that integration by parts can be carried out without producing uncontrollable boundary terms.} Readers should compare \eqref{E:Distance_functional_degenerate} with 
\eqref{E:Weighted_energy_linearized} and also note Remark
\ref{R:Distance_functional_Euclidean}.

The key properties of \eqref{E:Distance_functional_degenerate} are its equivalence with \eqref{E:Distance_functional_zero}, 
\begin{align}
\label{E:Distance_functionals_equivalence}
\tilde{\mathcal{D}}((r_1,v_1), (r_2, v_2) ) \approx
\mathcal{D}((r_1,v_1), (r_2, v_2) )
\end{align}
and the fact that it satisfies energy estimates
\begin{align}
\label{E:Distance_functional_degenerate_energy_estimate}
\frac{d}{dt} \tilde{\mathcal{D}}((r_1,v_1), (r_2, v_2) )
\lesssim 
\tilde{\mathcal{D}}((r_1,v_1), (r_2, v_2) ),
\end{align}
where the aforementioned choice of $\alpha$ and $\beta$ is crucial for the proof of \eqref{E:Distance_functional_degenerate_energy_estimate}.

Combining \eqref{E:Distance_functionals_equivalence} and \eqref{E:Distance_functional_degenerate_energy_estimate} then produces 
\eqref{E:Energy_estimate_distance_functional}.

\begin{remark}
Because one is comparing solutions defined in different domains, it seems that at some point an estimate for the distance between their boundaries will be needed. Indeed, the proof of \eqref{E:Energy_estimate_distance_functional} employs the following estimate
\begin{align}
\int_{\Gamma_t} |r_1 + r_2|^{\frac{1}{\kappa} + 2} \lesssim
\mathcal{D}((r_1,v_1), (r_2, v_2) ),
\nonumber
\end{align} 
which can be viewed as a measure of the (weighted) distance between the two boundaries.
\end{remark}

\begin{remark}
The proof of uniqueness relies on some delicate estimates that separate the behavior of solutions near and away from the boundary. Such estimates are borrowed from \cite{Ifrim-Tataru-2024}, where the classical, non-relativistic Euler equations with a physical vacuum boundary are treated.
\end{remark}

\begin{remark}
\label{R:Lagrangian_uniqueness}
A way of avoiding comparing solutions defined in different domains is to work in Lagrangian 
coordinates, as done in the a priori estimates obtained in 
\cite{Jang-LeFloch-Masmoudi-2016,Hadzic-Shkoller-Speck-2019}. 
While this is an advantage of the Lagrangian formalism, the use of Lagrangian coordinates has its own challenges, such as the need to obtain estimates for the co-factor matrix. Additionally, an Eulerian framework seems more suited for coupling to Einstein's equations, see Sect.~\ref{S:Open_problems}.
\end{remark}

\section{Relativistic fluids with viscosity}
\label{S:Relativistic_viscous_fluids}

So far, we have discussed only perfect fluids, i.e., fluids with no viscosity or dissipation\footnote{Strictly speaking, the relativistic Euler equations could exhibit anomalous dissipation \citep{Eyink:2017xtw}, but we do not discuss this here.}. Relativistic fluids with viscosity is a major topic of investigation in physics and we cannot do justice to its importance in this review, where our goal is to focus primarily on mathematical results. Thus, here we will restrict ourselves to simply mention two relativistic systems where inclusion of viscosity is important; a more complete discussion of the relevance of relativistic fluids with viscosity can be found in the references below. They are:
\begin{itemize}
\item The study of the quark-gluon plasma. The quark-gluon plasma is an exotic state of matter that forms in collisions of heavy ions performed at particle accelerators like the Relativistic Heavy Ion Collider (RHIC) at  Brookhaven National Laboratory and the Large Hadron Collider (LHC) at the European Organization for Nuclear Research (CERN). The quark-gluon plasma is the hottest\footnote{For example, the quark-gluon plasma reaches temperatures several order of magnitude higher than the temperature of the Sun.} and densest matter system we currently know. Thus, it offers a rare opportunity for scientists to study properties of matter at extreme conditions.
It is well-attested, through a combination of experiments, data analysis, and numerical simulations, that the quark-gluon plasma behaves as a relativistic liquid with viscosity, see
\cite{Denicol-Rischke-Book-2021,Romatschke:2017ejr,Rezzolla-Zanotti-Book-2013}.
\item Neutron star mergers. Recent state-of-the-art numerical simulations in \cite{Alford:2017rxf} strongly suggest that viscous and dissipative effects can affect the gravitational wave signal produced in collisions of neutron stars. These findings have been corroborated by a detailed study of the microscopic
origins of viscous effects in neutron stars in \cite{Most:2022yhe}. 
Such viscous and dissipative effects are estimated to be within the sensitivity-range of the next
generation of gravitational wave detectors, see \cite{Most:2018eaw,
Most:2019onn,Hammond:2022uua}. See the introduction of \cite{Chabanov:2023blf} for a short summary of the current state of affairs on the matter and 
\cite{Shibata:2017xht,Shibata:2017jyf,Alford:2020lla,Alford:2019kdw,Alford:2019qtm,Alford:2020pld,Most:2021zvc,Ripley:2023qxo,Camelio:2022ljs,Camelio:2022fds}
and references therein for further discussion on the importance of viscosity in neutron star mergers.
\end{itemize}

In particular, there has been recently a great deal of interest in the overlap between viscous effects on neutron star mergers and the quark-gluon plasma. This is leading to efforts to bring together different scientific communities working on these topics and a rich synergy among astrophysicists, nuclear physicists, and high-energy physicists, see, e.g., \citealt{Huth:2021bsp,Dexheimer:2020zzs,Huang:2022mqp}.

\begin{notation}
We will henceforth use the terms viscosity and dissipation interchangeably, e.g., referring to a viscous fluid or a dissipative fluid. This is a common practice in the community.
\end{notation}

\begin{remark}
In the ensuing discussion, we will introduce a significant amount of terminology that is specific to the study of relativistic fluids with viscosity. While we could present much of our mathematical results without referring to such terminology, we think it is useful to introduce it here in order to help readers understand some of the jargon used in the literature covering relativistic fluids with viscosity.
\end{remark}

The first task in the theoretical study of relativistic viscous fluids is to decide on an appropriate model. Unlike the case of a perfect fluid, there is no known Lagrangian for the description of a relativistic viscous\footnote{This is already the case for a classical fluid \cite[Chapter IV]{Serrin-1959}. See also the discussion in \cite{Dubovsky:2011sj}.} fluid, (see in particular the last bullet point at the end of this section). Absent a Lagrangian, one cannot canonically determine an energy-momentum tensor which would give equations of motion upon variation with respect to the metric.
A natural approach in this case is to modify  the perfect fluid energy-momentum tensor and baryon current by adding terms that represent viscous effects, i.e.,
\begin{align}
\label{E:Energy_momentum_viscous}
\mathcal{T}_{\alpha\beta} := (\varrho + \mathscr{R}) u_\alpha u_\beta
+ (p + \mathscr{P}) \proj_{\alpha\beta} + \pi_{\alpha\beta} + \mathscr{Q}_\alpha u_\beta + \mathscr{Q}_\beta u_\alpha,
\end{align}
and
\begin{align}
\label{E:Baryon_density_current_viscous}
\mathcal{J}_\alpha := n u_\alpha + \mathscr{J}_\alpha,
\end{align}
where $\varrho$, $u$, $p$, $n$, and $\proj$ are as in Sect.~\ref{S:Relativistic_Euler}, i.e., they are the fluid's (energy) density, velocity, pressure, baryon density, and projection onto the space orthogonal to $u$, respectively (see, however, the last bullet point at the end of this section).
As in Sect.~\ref{S:Relativistic_Euler}, it is assumed that $u$ is normalized
\begin{align}
g_{\alpha\beta} u^\alpha u^\beta = -1,
\label{E:Velocity_normalization_viscous}
\end{align}
that $p$ is given by an equation of state,
\begin{align}
p = p(\varrho,n),
\nonumber
\end{align}
and that the thermodynamic relations of Sect.~\ref{S:Thermodynamic_properties} hold. In particular, as in Sect.~\ref{S:Relativistic_Euler}, one can choose different thermodynamic quantities (e.g., the entropy, temperature) to be primary variables instead of $\varrho$ or $n$.

In comparison to \eqref{E:Energy_momentum_perfect} and \eqref{E:Baryon_density_current_perfect}, the new quantities
$\mathscr{R}$, $\mathscr{P}$, $\pi$, $\mathscr{Q}$, and $\mathscr{J}$ in \eqref{E:Energy_momentum_viscous} and \eqref{E:Baryon_density_current_viscous} are the \textdef{viscous fluxes,} and they correspond, respectively, to the \textdef{viscous correction to the energy density,} the
\textdef{viscous correction to the pressure,} also known as the \textdef{bulk viscosity}, the 
\textdef{viscous shear stress,} the \textdef{heat flux,} and the \textdef{viscous correction to the baryon current.} When all the viscous fluxes vanish, we recover \eqref{E:Energy_momentum_perfect} and \eqref{E:Baryon_density_current_perfect}.

The physical content of \eqref{E:Energy_momentum_viscous} and \eqref{E:Baryon_density_current_viscous} is that the viscous fluxes describe deviations from \textdef{thermodynamic equilibrium.} In our context, 
thermodynamic equilibrium simply refers to any solution to the relativistic Euler equations\footnote{Often solutions to the Euler equations are called local thermodynamic equilibrium solutions to emphasize the difference from global thermodynamic equilibrium, see Definition \ref{D:Stability}.}. Such solutions do not generate entropy, see \eqref{E:Entropy_transported}\footnote{To be more precise, entropy production is measured by the divergence of the so-called entropy current, which for the relativistic Euler equations is given by $\mathcal{S}_\mu := s n u_\mu$, whose divergence vanishes in view of \eqref{E:Entropy_transported}, \eqref{E:Baryon_density_current_perfect}, and \eqref{E:Relativistic_Euler_eq_full_system_baryon_charge}.}. With the inclusion of viscosity, however, the fluid should dissipate and produce entropy. Because of this, descriptions based on relativistic viscous fluids are also often referred to as \textdef{out of equilibrium} (meaning, out of thermodynamic equilibrium), and ``out of equilibrium'' is also often used interchangeably with viscosity and dissipation.

It is important to remark that in writing \eqref{E:Energy_momentum_viscous} and \eqref{E:Baryon_density_current_viscous} as a sum of would-be thermodynamic equilibrium 
and out-of-equilibrium terms, one is \emph{making the assumption that the dissipative properties of the fluid can be neatly distinguished from those of a perfect fluid} (see also the last bullet point at the end of this section).

\begin{remark}
\label{R:Viscosity_small}
It is often implicitly assumed in theories of relativistic viscous fluids that the viscous fluxes are in a sense small compared to the underlying perfect-fluid contributions, in a sense that one hopes to be able to precisely quantify. On the other hand, it is not uncommon for physicists to be faced with situations where they need to push relativistic viscous fluid dynamics to its limits, applying it to regimes that might be formally outside its regime of validity. See \cite{Romatschke:2017ejr,Heinz:2013th} and references therein.
\end{remark}

\begin{remark}
Based on how one often defines energy and pressure in terms of a given energy-momentum tensor
(see, e.g., Sect.~4.2 of \citealt{Wald:1984rg}) it is customary in the physics literature to refer to $\varrho + \mathscr{R}$ as the density and to $p+\mathscr{P}$ as the pressure, referring to $\varrho$ and $p$ as the equilibrium density and equilibrium pressure, respectively. Similarly, one would refer to $s$, $\uptheta$, as the equilibrium entropy, temperature, etc. Sometimes one also calls 
$\varrho + \mathscr{R}$ and $p+\mathscr{P}$ the total density and pressure, respectively. We will not follow these conventions here.
\end{remark}

In order to define a theory of relativistic viscous fluids, one needs to specify the viscous fluxes\footnote{See the review 
\cite{Rocha:2023ilf} for comprehensive outlook of most theories of relativistic viscous fluids. Here, we review only those theories for which good mathematical properties are known; see the ensuing discussion.
}. 
The equations of motion will then be, as usual, 
\begin{subequations}{\label{E:Divergence_energy_momentum_baryon_current_viscous}}
\begin{align}
\nabla_\alpha \mathcal{T}^\alpha_\beta &= 0,
\label{E:Divergence_energy_momentum_viscous}
\\
\nabla_\alpha \mathcal{J}^\alpha & = 0.
\label{E:Divergence_baryon_current_viscous}
\end{align}
\end{subequations}

The first proposal in this direction was introduced by \cite{Eckart:1940te}, wherein he set
\begin{subequations}{\label{E:Viscous_fluxes_Eckart}}
\begin{align}
\mathscr{R} & := 0,
\label{E:Viscous_fluxes_Eckart_R}
\\
\mathscr{P} & := - \upzeta \nabla_\alpha u^\alpha,
\label{E:Viscous_fluxes_Eckart_P}
\\
\pi_{\alpha\beta} & := - \upeta \proj_\alpha^\mu \proj_\beta^\nu ( \nabla_\mu u_\nu + \nabla_\nu u_\mu - \frac{2}{3} \nabla_\lambda u^\lambda g_{\mu\nu}),
\label{E:Viscous_fluxes_Eckart_pi}
\\
\mathscr{Q}_\alpha & := -\upkappa \uptheta (\proj_\alpha^\mu \nabla_\mu  \ln\uptheta + u^\mu \nabla_\mu u_\alpha ),
\label{E:Viscous_fluxes_Eckart_Q}
\\
\mathscr{J}_\alpha & := 0,
\label{E:Viscous_fluxes_Eckart_J}
\end{align}
\end{subequations}
where $\upzeta = \upzeta(\varrho,n)$, $\upeta = \upeta(\varrho,n)$, and $\upkappa = \upkappa(\varrho, n)$ and the \textdef{coefficients of bulk viscosity, shear viscosity,} and \textdef{heat conductivity, respectively.} (We recall from Sect.~\ref{S:Thermodynamic_properties} that $\uptheta$ is the fluid's temperature.) As the equation of state, their choice depends on the nature of the fluid. They are precise relativistic analogues of the corresponding coefficients for the classical compressible Navier--Stokes-Fourier equations.

A similar proposal for the viscous fluxes was introduced by Landau and Lifshitz in the 1950s (see \citealt{Landau-Lifshitz-Book-Fluids-1987}). They use the same definitions
\eqref{E:Viscous_fluxes_Eckart_R}, \eqref{E:Viscous_fluxes_Eckart_P}, and \eqref{E:Viscous_fluxes_Eckart_pi} as Eckart's but \eqref{E:Viscous_fluxes_Eckart_Q}
and \eqref{E:Viscous_fluxes_Eckart_J} are replaced by
\begin{subequations}{\label{E:Viscous_fluxes_Landau}}
\begin{align}
\mathscr{Q} & : = 0,
\label{E:Viscous_fluxes_Landau_Q}
\\
\mathscr{J}_\alpha & := -\upkappa \frac{\uptheta}{h} (\proj_\alpha^\mu \nabla_\mu \ln \uptheta + u^\mu \nabla_\mu u_\alpha )
\label{E:Viscous_fluxes_Landau_J}
\end{align}
\end{subequations}
where we recall that $h$ is the enthalpy given by \eqref{E:Enthalpy_definition}. The theories of relativistic viscous fluids defined by \eqref{E:Viscous_fluxes_Eckart} and \eqref{E:Viscous_fluxes_Landau} are known as the \textdef{Eckart and the Landau--Lifshitz theories,} respectively.

We will not discuss the physical arguments employed by Eckart and Landau and Lifshitz that lead to the postulates \eqref{E:Viscous_fluxes_Eckart} and \eqref{E:Viscous_fluxes_Landau}, referring to Chapter~6 of \cite{Rezzolla-Zanotti-Book-2013} and the original works \cite{Eckart:1940te,Landau-Lifshitz-Book-Fluids-1987} instead. It suffices to say that Eckart and Landau and Lifshitz were seeking a covariant generalization of the classical Navier--Stokes--Fourier equations (compare \eqref{E:Viscous_fluxes_Eckart_P} and \eqref{E:Viscous_fluxes_Eckart_pi} with the definitions of bulk and shear viscosity in a classical fluid, e.g., Eq.~(61.1) in \cite{Serrin-1959}). In this regard, it is worth mentioning that with
\eqref{E:Viscous_fluxes_Eckart} or \eqref{E:Viscous_fluxes_Landau}, the equations of motion
\eqref{E:Divergence_energy_momentum_baryon_current_viscous} reduce to the 
classical Navier--Stokes--Fourier equations in the non-relativistic limit, see Chapter~6 in \cite{Rezzolla-Zanotti-Book-2013}.

It turns out that the Eckart and Landau--Lifshitz theories suffer from several pathologies that render them inadequate for modeling a relativistic viscous fluid, as shown by \cite{Pichon-1965} and 
\cite{Hiscock:1985zz}. A major problem with such theories is that they violate causality. We recall that \textdef{causality} is a fundamental postulate in relativity theory stating that no signal can propagate faster than the speed of light in vacuum. When a physical system is described by PDEs, causality translates into the mathematical requirement that the characteristics of the PDEs are contained within the lightcone in physical space (see the example in Fig.~\ref{F:Light_cone_and_sound_cones_tangent}) and the system enjoys a domain-of-dependence property. More precisely:

\begin{definition}[Causality]
\label{D:Causality}
Let $(M,g)$ be a globally hyperbolic spacetime\footnote{It is natural to define causality for globally hyperbolic spacetimes since several pathologies, such closed timelike curves, are ruled out for such spaces. We recall that both Minkowski space and spacetimes that arise as a solution to the Cauchy problem for Einstein's equations are globally hyperbolic.}. Let $\Sigma \subset M$ be a Cauchy surface\footnote{Which exists since the spacetime is globally hyperbolic.}. Consider a system of (possibly nonlinear) PDEs on $M$ for a field or collection of fields $\varphi$, which we write as $P\varphi = 0$. We say that the system is causal if the following condition holds. For any $x \in M$ in the future\footnote{Being globally hyperbolic, $M$ can be time oriented and thus the future and past of a set are well defined.} of $\Sigma$, $\varphi(x)$ depends only $J^-(x) \cap \Sigma$, where $J^-(x)$ is the causal past of $x$ (with respect to the metric $g$). In other words, if $\varphi_1$ and $\varphi_2$ are two solutions to $P\varphi = 0$ such that $\left. \varphi_1 \right|_{J^-(x)\cap \Sigma} = 
\left. \varphi_2 \right|_{J^-(x)\cap \Sigma}$, then $\varphi_1(x) = \varphi_2(x)$.
\end{definition}

\begin{remark}
\label{R:Causality_precise_definition}
For related discussion on causality and the domain-of-dependence property, we refer to Sect.~\ref{S:Leray_systems}, particularly in the conclusion of Theorem \ref{T:Leray_Ohya_domain_of_dependence}. We note that the setting of Sect.~\ref{S:Leray_systems} is sufficiently general to cover all the cases of interest where the domain-of-dependence property is invoked in this review, see 
Sect.~\ref{S:Domain_of_dependence_outside_Gevrey}.
\end{remark}

From the equations of motion 
\eqref{E:Divergence_energy_momentum_baryon_current_viscous} with the choices \eqref{E:Viscous_fluxes_Eckart} or \eqref{E:Viscous_fluxes_Landau}, one obtains one set of characteristics given by
\begin{align}
\proj^{\alpha\beta} \xi_\alpha\xi_\beta = 0,
\label{E:Characteristics_Eckart_Landau}
\end{align}
which is not causal\footnote{Equations \eqref{E:Divergence_energy_momentum_baryon_current_viscous} with the choices 
\eqref{E:Viscous_fluxes_Eckart} or \eqref{E:Viscous_fluxes_Landau} form a mixed-order system of PDEs, so its principal part, needed for the computation of the characteristics, is determined with the help of Leray theory (see Appendix \ref{S:Leray_systems}).}.

In addition, the Eckart and Landau-Lifshitz theories are also unstable, see \cite{Hiscock:1985zz}, where stability here refers to mode stability. More precisely, the situation is as follows.

\begin{definition}
\label{D:Stability}
Assume that the spacetime metric is the Minkowski metric.
We define a \textdef{global thermodynamic equilibrium} as a solution to \eqref{E:Divergence_energy_momentum_baryon_current_viscous} where all the viscous fluxes vanish and
$u$, $\varrho$, and $n$ (and consequently the other thermodynamic scalars like $\uptheta$, $s$, and so on, see Sect.~\ref{S:Thermodynamic_properties}) are constant. Note that this gives a constant solution to the relativistic Euler equations. Consider Eqs.~\eqref{E:Divergence_energy_momentum_baryon_current_viscous} linearized about a global theormodynamic equilibrium and consider plane-wave solutions to the linearized equations\footnote{Such solutions can always be determined since the equations are constant-coefficient PDEs.}. We say that the system 
is \textdef{stable} if it does not admit plane-wave solutions that grow in time, and \textdef{unstable} otherwise.
\end{definition}

\begin{remark}
Observe that Definition \ref{D:Stability} applies to any relativistic viscous fluid described by \eqref{E:Energy_momentum_viscous} and \eqref{E:Baryon_density_current_viscous}, i.e., 
relations \eqref{E:Viscous_fluxes_Eckart} or \eqref{E:Viscous_fluxes_Landau} are not assumed.
\end{remark}

A good theory of relativistic viscous fluids should be stable 
in the sense of Definition \ref{D:Stability}. Physically, it is expected that very small deviations from global thermodynamic equilibrium should decay due to dissipation. Such small deviations are believed to be well-modeled by the linearized equations. 

\cite{Hiscock:1985zz} have shown that the Eckart and Landau--Lifshitz theories are unstable in the sense of Definition \ref{D:Stability}. Moreover, they have determined that these instabilities 
are severe in the sense that, for example, deviations from equilibrium in water at room temperature would grow by an $e$ factor in about $10^{-34}$ seconds, a physically unacceptable result\footnote{As Hiscock and Lindblom argue in their paper, ``If, for all astrophysically imaginable conditions, the $e$-folding time for growth of the instability were much longer than the age of the universe, then the instabilities would be only of pedagogical interest.''}.

The acausality and instability of Eckart's and Landau--Lifshitz's theories prompted physicists to seek alternative theories of relativistic fluids. In Sects.~\ref{S:DNMR} and \ref{S:BDNK}, we will describe two different approaches to construct theories of relativistic viscous fluids with good properties, including causality and stability. We stress that by no means these are the only theories of 
relativistic viscous fluids currently available. We focus on them here because they are likely to be the ones of most interest to mathematicians in view of their formal structure and also because they seem to be the ones that have attracted most interest among physicists in recent years (see a sample of references in Sects.~\ref{S:DNMR} and \ref{S:BDNK}). A comprehensive discussion of different theories of relativistic fluids with viscosity can be found in Chapter~2 of \cite{Romatschke:2017ejr}; Chapter~6 of \cite{Rezzolla-Zanotti-Book-2013}; Sect.~II of \cite{Bemfica:2020zjp}, and references therein.

We finish this introduction with the following general observations:
\begin{itemize}
\item While in hindsight the inadequacy of Eckart's and Landau--Lifshitz's theories 
seems straightforward from a computation of their characteristics and an analysis of mode stability, it is important to remark that for many years they were considered ``correct,'' as they are natural generalizations of the Navier--Stokes--Fourier equations to a covariant setting. For example, Weinberg's celebrated book \cite{Weinberg:1972kfs} has a full section, 2.11, dedicated to relativistic viscous fluids 
based on Eckart's approach. In fact, Eckart's and Landau--Lifshitz's theories have been extensively used in the study of viscous cosmological models, see the review \cite{Brevik:2014cxa}.
\item The previous point in conjunction with the aforementioned existence of many different theories of relativistic viscous fluids highlight the following important fact: \emph{there is no universally accepted theory of relativistic viscous fluids.} This should be contrasted with other matter models, wherein generally there is widespread agreement about their applicability. What underlies this state of affairs 
are certain conceptual difficulties that make modeling dissipation in relativity theory a challenging task; see the last bullet point below.
\item As a side note, we remark that it is not possible to decide on a ``correct'' theory of 
relativistic fluids with viscosity from microscopic theory. The derivation\footnote{By which we mean a formal derivation. There are no rigorous results on the derivation of relativistic viscous fluid theories from microscopic theory. For comparison with classical fluids, it is worth mentioning the following rigorous results on derivations from microscopic theory. Limits from the Boltzmann equation to the classical compressible Euler equations have been obtained in \cite{Huang-WangYang-2010} for solutions with contact discontinuities and in \cite{Yu-2005} for weak planar shocks; the incompressible Navier--Stokes and Euler equations and the compressible Euler equations can be derived from Boltzmann's equations, see \cite{Saint-Raymond-Book-2009}, although the compressible Navier-Stokes equations cannot, see \cite{Mott-Smith-1951,Uribe-Velasco-Garcia-Colin-1998,Salomons-Mareschal-1992,Liepmann-Narasimha-Chahine-1962}; finally, for work on the Hilbert expansion, including the case of relativistic perfect fluids, see \cite{Guo-Jang-Jiang-2009,Guo-Jang-2010,Guo-Huang-Wang-2021,Speck-Strain-2011} (see also the more recent works \cite{Kim-La-2024,Kim-Nguyen-2022-arxiv}).} of a fluid model from microscopic theory necessarily involves a coarse-graining procedure and different procedures lead to different results. For example, the Landau--Lifshitz and the theories discussed in Sects.~\ref{S:DNMR} and \ref{S:BDNK} can all be derived from kinetic theory (see \cite{DeGroot:1980dk} and references below).
\item We point out that there are several issues of a conceptual nature regarding the definition of a relativistic viscous fluid, such as, for example, what it means to talk about the definition of temperature or pressure outside thermodynamic equilibrium. We will not discuss these issues here\footnote{It is interesting to note that the ``correct'' definition of the classical Navier--Stokes on Riemannian manifolds is also up for debate, see \cite{Chan-Czubak-Disconzi-2017,Chan-Czubak-Yoneda-2023,Chan-Czubak-2022-arxiv,Czubak-2014} and reference therein.}, save for a few brief remarks in Sect.~\ref{S:BDNK_origins}. Interested readers are referred to Chapter 6 of \cite{Rezzolla-Zanotti-Book-2013}; Sect.~II of \cite{Bemfica:2020zjp}; \cite{Romatschke:2017ejr,Denicol-Rischke-Book-2021}; 
Sect.~2.11 of \cite{Weinberg:1972kfs} and references therein.
\end{itemize}

In the ensuing presentation, we will be always working under the following specific setting.

\begin{assumption}
\label{A:Working_projected_equations_and_normalization}
As in the case of the relativistic Euler equations, the constraint \eqref{E:Velocity_normalization_viscous} will always be assumed as one of the equations of the system, unless stated otherwise.
In practice, however, from the point of view of the Cauchy problem, we will consider Eqs.~\eqref{E:Divergence_energy_momentum_viscous} decomposed in the directions parallel and orthogonal to $u$, i.e., 
\begin{subequations}{\label{E:Divergence_energy_momentum_viscous_projected}}
\begin{align}
u^\beta \nabla_\alpha \mathcal{T}^\alpha_\beta & = 0,
\label{E:Divergence_energy_momentum_viscous_projected_u}
\\
\proj^{\gamma \beta} \nabla_\alpha \mathcal{T}^\alpha_\beta &= 0,
\label{E:Divergence_energy_momentum_viscous_projected_u_perp}
\end{align}
\end{subequations}
as it was done for the relativistic Euler equations, and treat all components
of $u$ as independent, provided that we can show the constraint \eqref{E:Velocity_normalization_viscous} to be propagated, as it was the case for the relativistic Euler equations. This propagation is in general true and can be shown for the specific models of relativistic viscous fluids we will discuss, but it has to be shown in a case-by-case basis, i.e., for each given specific model. Thus, for the sake of brevity, we will not discuss the propagation of \eqref{E:Velocity_normalization_viscous} here.
It will also be always assumed in \eqref{E:Divergence_energy_momentum_baryon_current_viscous}
that an equation of state is given, as well as other functional relations that might vary depending on the model one is studying, such as, for example, the relations for the coefficients of bulk and shear viscosity and heat conductivity defined below as functions of $\varrho$ and $n$. 
We note that we are not restricting ourselves here to the Eckart or Landau--Lifshitz theories. The theories introduced below also involve \eqref{E:Divergence_energy_momentum_baryon_current_viscous} as equations of motion and will have coefficients $\upzeta$, $\upeta$, and $\upkappa$, among others, that are functions of $\varrho$ and $n$.
\end{assumption}

\subsection{The DNMR theory\label{S:DNMR}}

In this section, we present the \textdef{Denicol--Niemi--Molnar--Rischke (DNMR) theory\footnote{For historical reasons, this is also referred to as the M\"uller--Israel--Stewart or also simply Israel--Stewart theory. See Sect.~\ref{S:DNMR_historical}.}}, first introduced in \cite{Denicol:2012cn}. This is the most widely used theory in studies of the quark-gluon plasma, in particular when it comes to implementing numerical simulations, see \cite{Heffernan:2023utr,Heffernan:2023gye,Jeon:2015dfa}. But despite its importance, its mathematical properties are poorly understood, as we will discuss in what follows.

In the DNMR theory, the viscous fluxes are not given in term of $\varrho$, $n$, $u$, and their derivatives, as in the Eckart and Landau--Lifshitz theories. Instead, they are \emph{new variables in the theory to be treated on the same footing as $\varrho$, $n$, and $u$.} The task then becomes finding equations of motion for the viscous fluxes, at which point one needs to make some modeling choices (see Sect.~\ref{S:DNMR_historical} for some background discussion). In the DNMR theory, these equations of motion are derived from kinetic theory via the method of moments that goes back to the work of \cite{Grad-1958}. 

\begin{assumption}
\label{A:DNMR_simplified}
We will present the DNMR theory for the case where $n$, $\mathscr{J}$, and $\mathscr{Q}$ are absent, so in particular all functions of $\varrho,n$ will be functions of $\varrho$ only, since this is the case treated in Theorems \ref{T:Causality_DNMR} and \ref{T:LWP_DNMR}. 
This is also the situation that has been the primary focus of investigation in studies of the quark-gluon plasma, see \cite{Ryu:2017qzn}. We will, nevertheless, also mention some general results valid for the full set of equations derived in \cite{Denicol:2012cn}. Readers interested in the complete DNMR theory are referred to Eqs.\footnote{We warn that \cite{Denicol:2012cn} uses the $+---$ convention for the spacetime metric. Here, we present the DNMR equations in the $-+++$ signature.} (63)--(67) of \cite{Denicol:2012cn}.
We also note that, even in its most general form, the DNMR equations always assume that $\mathscr{R}=0$.
\end{assumption}

The derivation from kinetic theory carried out in \cite{Denicol:2012cn}, which will not be given here, leads to the following equations of motion
\begin{subequations}{\label{E:DNMR}}
\begin{align}
u^\mu \nabla_\mu \varrho + (p + \varrho + \mathscr{P}) \nabla_\mu u^\mu + \pi^\mu_\alpha \nabla_\mu u^\alpha & = 0,
\label{E:DNMR_density}
\\
(\varrho + p + \mathscr{P}) u^\mu \nabla_\mu u_\alpha + c_s^2 \proj_\alpha^\mu \nabla_\mu \varrho
+ \proj_\alpha^\mu \nabla_\mu \mathscr{P} + \proj_\alpha^\mu \nabla_\nu \pi^\nu_\mu & = 0,
\label{E:DNMR_velocity}
\\
\uptau_\mathscr{P} u^\mu \nabla_\mu \mathscr{P} + \mathscr{P} + \upzeta \nabla_\mu u^\mu 
+ \updelta_{\mathscr{P}\mathscr{P}} \mathscr{P} \nabla_\mu u^\mu + \uplambda_{\mathscr{P} \pi} \pi^{\mu\nu} \upsigma_{\mu\nu} & = 0,
\label{E:DNMR_bulk}
\\
\uptau_\pi\hat{\mathsf{\Pi}}_{\mu\nu}^{\alpha\beta} u^\lambda \nabla_\lambda \pi^{\mu \nu} + \pi^{\alpha\beta} + 2 \upeta \sigma^{\alpha\beta} + \updelta_{\pi\pi} \pi^{\alpha\beta} \nabla_\mu u^\mu 
&
\nonumber
\\
+ \uptau_{\pi\pi} \pi_\mu^{\langle\alpha} \sigma^{\beta\rangle \mu}
+ \uplambda_{\pi\mathscr{P}} \mathscr{P} \sigma^{\alpha\beta} & = 0,
\label{E:DNMR_shear}
\end{align}
\end{subequations}
and $\mathscr{R} = 0$, i.e., there is no viscous correction to the energy density in the DNMR theory. Above,
\begin{align}
\hat{\mathsf{\Pi}}_{\mu\nu}^{\alpha\beta} := \frac{1}{2} (\proj^\alpha_\mu \proj^\beta_\nu 
+ \proj^\beta_\mu \proj^\alpha_\nu ) - \frac{1}{3}\proj^{\alpha\beta} \proj_{\mu\nu}
\label{E:Projection_symmetric_trace_free_part}
\end{align}
projects a two-tensor onto its $u$-orthogonal symmetric trace-free part;
\begin{align}
A^{\langle\alpha}_\mu B^{\beta\rangle \mu} : = \hat{\mathsf{\Pi}}^{\alpha\beta}_{\mu\nu} A^{\mu \lambda}B^\nu_\lambda
\nonumber
\end{align}
where $A$ and $B$ are symmetric two-tensors ($\pi$ is symmetric, see \eqref{E:Shear_constraints_symmetry});
\begin{align}
\sigma^{\alpha\beta} := \hat{\mathsf{\Pi}}^{\alpha\beta}_{\mu\nu}\nabla^\mu u^\nu,
\nonumber
\end{align}
is often called the \textdef{shear tensor} (not to be confused with the viscous shear stress $\pi$);
and
\begin{align}
\upzeta, \upeta, \uptau_\mathscr{P}, \uptau_\pi, \updelta_{\mathscr{P}\mathscr{P}},
\uplambda_{\mathscr{P}\pi}, \updelta_{\pi\pi}, \uptau_{\pi\pi}, \uplambda_{\pi\mathscr{P}}
\nonumber
\end{align}
are functions of\footnote{In general, when $n$ is included, they will be functions of $\varrho$ and $n$.} $\varrho$ known as \textdef{transport coefficients,} of which 
$\upzeta$ and $\upeta$ are the \textdef{coefficients of bulk and shear viscosity,} respectively, and 
$\uptau_\mathscr{P}$ and $\uptau_\pi$ are called the \textdef{bulk and shear relaxation times,} respectively (note that \eqref{E:DNMR_bulk} and \eqref{E:DNMR_shear} have the form of relaxation-type equations). In addition to \eqref{E:Velocity_normalization_viscous}, the following constraints hold
\begin{subequations}{\label{E:Shear_constraints}}
\begin{align}
\pi_{\alpha\beta} = \pi_{\beta\alpha},
\label{E:Shear_constraints_symmetry}
\\
u^\mu \pi_{\mu \alpha} = 0, 
\label{E:Shear_constraints_orthogonality}
\\ 
\pi^\mu_\mu = 0.
\label{E:Shear_constraints_trace}
\end{align}
\end{subequations}
We note that we continue to use the definition $c_s^2 = \left. \frac{\partial p}{\partial \varrho}\right|_s$, although, as we will soon see, the corresponding sound cones in the DNMR theory are not given solely in terms of $c_s^2$. Observe that \eqref{E:DNMR_density} and \eqref{E:DNMR_velocity} are simply
\eqref{E:Divergence_energy_momentum_viscous_projected_u} and \eqref{E:Divergence_energy_momentum_viscous_projected_u_perp}, respectively.

\begin{remark}
Observe that \eqref{E:Viscous_fluxes_Eckart_pi} also satisfies \eqref{E:Shear_constraints}. In fact, on physical grounds, \eqref{E:Shear_constraints} is usually required in any theory of relativistic viscous fluids. We also note that a direct calculation shows that $-2\upeta \sigma^{\alpha\beta}$ equals the RHS of \eqref{E:Viscous_fluxes_Eckart_pi}.
\end{remark}

Readers should take note of the complexity of Eqs.~\eqref{E:DNMR}. Despite their complexity, some important results have been obtained for the DNMR theory, namely:
\begin{itemize}
\item Stability (in the sense of Definition \ref{D:Stability}) holds, see \cite{Denicol:2012cn} (see also \cite{Hiscock:1983zz,Olson:1989ey} and Sect.~\ref{S:DNMR_historical}). In fact,
stability was established in the more general case when the DNMR equations include
$n$, $\mathscr{J}$, and $\mathscr{Q}$.
\item When Eqs.~\eqref{E:DNMR} are linearized about global thermodynamic equilibrium, the resulting system is causal.
\item Causality holds in $1+1$ dimensions, see \cite{Denicol:2008ha}, and in rotational symmetry, see \cite{Pu:2009fj,Floerchinger:2017cii}.
\end{itemize}
We remark that these results do not hold unconditionally, but are proved under certain assumptions on the fluid variables and the transport coefficients. Such assumptions are, nevertheless, of general physical significance, see discussions in the above references. On the other hand, they do hold in a more general setting when $n$, $\mathscr{J}$, and $\mathscr{Q}$ are included, which we are not considering here. These results do not address, however, causality of the full set of Eqs.~\eqref{E:DNMR} in $3+1$ dimensions, which are the equations commonly used in numerical simulations of the quark-gluon plasma. This will be the content of Theorem \ref{T:Causality_DNMR}.

\begin{remark}
\label{R:Constraints_propagation_DNMR}
Because of the constraints \eqref{E:Shear_constraints}, only five components of $\pi$ are independent. However, as it is done with \eqref{E:Velocity_normalization} with the relativistic Euler equations (recall Assumption \ref{A:Working_projected_equations_and_normalization}), from the point of view of the Cauchy problem, 
it is more convenient to work with all components of $\pi$, without imposing \eqref{E:Shear_constraints}, which are then shown to be propagated by the flow if imposed on the initial data. Under these conditions. In this setting, Eqs.~\eqref{E:DNMR} form a $22 \times 22$ system of first-order PDEs \emph{without diagonal principal part.}
\end{remark}
\subsubsection{Causality\label{S:Causality_DNMR}}

In order to state our main result of this section, we need to following Notation.

\begin{notation}
\label{N:Diagonalization_pi}
The symmetry and trace-free condition of $\pi$ allows us to diagonalize it, with 
\begin{align}
\pi^\mu_\nu e_A^\nu = \Lambda_A e^\mu_A,\, A=0,\dots,3,
\nonumber
\end{align}
where the eigenvectors $\{ e_A \}_{A=0}^3$ form an orthonormal frame,  
\begin{align}
g_{\mu\nu} e^\mu_A e^\nu_B = m_{AB} = \operatorname{diag}(-1,1,1,1),
\nonumber
\end{align}
$e_0 = u$, $\Lambda_{A=0} = 0$, the eigenvalues $\Lambda_{A=i}$ are real valued, and 
\begin{align}
\Lambda_1 + \Lambda_2 + \Lambda_3 = 0.
\nonumber
\end{align}
We can assume the eigenvalues to be ordered according to $\Lambda_1 \leq \Lambda_2 \leq \Lambda_3$, with $\Lambda_1 \leq 0 \leq \Lambda_3$.
\end{notation}

We can now state a result addressing the causality of Eqs.~\eqref{E:DNMR}. See Remark \ref{R:Sufficient_regular_DNMR} below for the meaning of ``sufficiently regular'' in the statement of Theorem \ref{T:Causality_DNMR}.

\begin{theorem}[\citealt{Bemfica-Disconzi-Hoang-Noronha-Radosz-2021}]
\label{T:Causality_DNMR}
Let $(\varrho,u, \mathscr{P}, \pi)$ be a sufficiently regular solution\footnote{See Remark \ref{R:Sufficient_regular_DNMR}.} to equations
\eqref{E:DNMR} satisfying the constraints \eqref{E:Shear_constraints} and \eqref{E:Velocity_normalization_viscous} and defined in a globally hyperbolic spacetime. Suppose that:
\begin{enumerate}[label=(A.\arabic*),ref=(A.\arabic*)]
\item \label{I:Causality_DNMR_1}
$\uptau_\mathscr{P}, \uptau_\pi > 0$, $\upeta, \upzeta, \updelta_{\mathscr{P}\mathscr{P}}, \uplambda_{\mathscr{P}\pi}, \updelta_{\pi\pi}, \uptau_{\pi\pi}, \updelta_{\pi\mathscr{P}} \geq 0$.
\item \label{I:Causality_DNMR_2}
$\varrho > 0$, $p \geq 0$, $\varrho + p + \mathscr{P} > 0$ (note that $\mathscr{P}$ can be negative).
\item \label{I:Causality_DNMR_3}
$\varrho + p + \mathscr{P} + \Lambda_i > 0$, $i=1,2,3$.
\end{enumerate}
Then, the following are sufficient conditions for causality\footnote{Recall Remark \ref{R:Causality_precise_definition}.} of Eqs.~\eqref{E:DNMR}:
\begin{subequations}{\label{E:Sufficient_causality_DNMR}}
\begin{align}
(\varrho + p + \mathscr{P} - |\Lambda_1| - \frac{1}{2\uptau_\pi}(2\upeta + \uplambda_{\pi\mathscr{P}} \mathscr{P} ) - \frac{\uptau_{\pi\pi}}{2\uptau_\pi}\Lambda_3 & \geq 0,
\label{E:Sufficient_causality_DNMR_a}
\\
2\upeta + \uplambda_{\pi\mathscr{P}} \mathscr{P} - \uptau_{\pi\pi} |\Lambda_1| & > 0,
\label{E:Sufficient_causality_DNMR_b}
\\
\uptau_{\pi\pi} -  6 \updelta_{\pi\pi} & \leq 0,
\label{E:Sufficient_causality_DNMR_c}
\\
\frac{\uplambda_{\mathscr{P}\pi}}{\uptau_\mathscr{P}} + c_s^2 - \frac{\uptau_{\pi\pi}}{12\uptau_\pi} & \geq 0,
\label{E:Sufficient_causality_DNMR_d}
\\
\frac{1}{3\uptau_\pi}[ 4\upeta + 2 \uplambda_{\pi\mathscr{P}} \mathscr{P} + (3\updelta_{\pi\pi} + \uptau_{\pi\pi}) \Lambda_3 ] + \frac{\upzeta + \updelta_{\mathscr{P}\mathscr{P}} \mathscr{P} + \uplambda_{\mathscr{P}\pi} \Lambda_3}{\uptau_\mathscr{P}}  &
\nonumber
\\
+ |\Lambda_1| + \Lambda_3 c_s^2
 + \frac{ \frac{12\updelta_{\pi\pi} - \uptau_{\pi\pi}}{12\uptau_\pi}
	\big( \frac{\uplambda_{\mathscr{P}\pi}}{\uptau_\mathscr{P}} + c_s^2 - \frac{\uptau_{\pi\pi}}{12\uptau_\pi} 
     \big)
     (\Lambda_3 + |\Lambda_1|)^2
 }{ \varrho + p + \mathscr{P} - |\Lambda_1| - \frac{1}{2\uptau_\pi}(2\upeta + \uplambda_{\pi\mathscr{P}} \mathscr{P}) - \frac{\uptau_{\pi\pi}}{2\uptau_\pi} \Lambda_3 } &
 \nonumber
 \\
 - (\varrho + p + \mathscr{P} ) (1-c_s^2) 
  & \leq 0,
  \label{E:Sufficient_causality_DNMR_e}
\\
\frac{1}{6\uptau_\pi}[ 2\upeta + \uplambda_{\pi\mathscr{P}} \mathscr{P} + (\uptau_{\pi\pi} - 6 \updelta_{\pi\pi})|\Lambda_1| ] + \frac{\upzeta + \updelta_{\mathscr{P}\mathscr{P}} \mathscr{P} - \updelta_{\mathscr{P}\pi} |\Lambda_1|}{\uptau_\mathscr{P}} &
\nonumber
\\
+ (\varrho + p + \mathscr{P} - |\Lambda_1|) c_s^2 & \geq 0,
\label{E:Sufficient_causality_DNMR_f}
\\
\frac{ \frac{12\updelta_{\pi\pi} - \uptau_{\pi\pi}}{12\uptau_\pi}
	\big( \frac{\uplambda_{\mathscr{P}\pi}}{\uptau_\mathscr{P}} + c_s^2 - \frac{\uptau_{\pi\pi}}{12\uptau_\pi} 
     \big)
     (\Lambda_3 + |\Lambda_1|)^2
 }{\big[ \frac{1}{2\uptau_\pi}(2\upeta + \uplambda_{\pi\mathscr{P}} \mathscr{P}) - \frac{\uptau_{\pi\pi}}{2\uptau_\pi} |\Lambda_1| \big]^2} & \leq 1,
 \label{E:Sufficient_causality_DNMR_g}
\\
\frac{(\varrho + p + \mathscr{P} + \Lambda_2)(\varrho + p + \mathscr{P} + \Lambda_3)}{3(\varrho + p + \mathscr{P} - |\Lambda_1|)}
\Big\{ 1 + \frac{\frac{1}{\uptau_\pi}(2\upeta + \uplambda_{\pi\mathscr{P}} \mathscr{P}) + \frac{\uptau_{\pi\pi}}{\uptau_\pi} \Lambda_3}{\varrho + p + \mathscr{P} -|\Lambda_1|}
\Big\} & \leq 
\nonumber
\\
\frac{1}{3 \uptau_\pi}[ 4\upeta + 2 \uplambda_{\pi\mathscr{P}}\mathscr{P} - (3\updelta_{\pi\pi} + \uptau_{\pi\pi})|\Lambda_1| ] + \frac{\upzeta + \updelta_{\mathscr{P}\mathscr{P}} \mathscr{P} - \updelta_{\mathscr{P}\pi} |\Lambda_1|}{\uptau_\mathscr{P}} &
\nonumber
\\
+ (\varrho + p + \mathscr{P} - |\Lambda_1|) c_s^2, & 
\label{E:Sufficient_causality_DNMR_h}
\end{align}
\end{subequations}
where condition \eqref{E:Sufficient_causality_DNMR_h} can be dropped if $\uptau_{\pi\pi} = \updelta_{\pi\pi} = 0$.

Moreover, still under assumptions \ref{I:Causality_DNMR_1}, \ref{I:Causality_DNMR_2}, and 
\ref{I:Causality_DNMR_3}, the following are necessary conditions for causality of Eqs.~\eqref{E:DNMR}:
\begin{subequations}{\label{E:Necessary_causality_DNMR}}
\begin{align}
2\upeta + \uplambda_{\pi\mathscr{P}} \mathscr{P} - \frac{1}{2} \uptau_{\pi\pi}|\Lambda_1| & \geq 0,
\label{E:Necessary_causality_DNMR_a}
\\
\varrho+p+\mathscr{P} - \frac{1}{2\uptau_\pi}(2\upeta + \uplambda_{\pi\mathscr{P}}\mathscr{P} )
-\frac{\uptau_{\pi\pi}}{4\uptau_\pi} \Lambda_3 & \geq 0,
\label{E:Necessary_causality_DNMR_b}
\\
\frac{1}{2\uptau_\pi}(2\upeta + \uplambda_{\pi\mathscr{P}}\mathscr{P} )
+ \frac{\uptau_{\pi\pi}}{4\uptau_\pi}(\Lambda_i + \Lambda_j ) & \geq 0, \, i \neq j
\label{E:Necessary_causality_DNMR_c}
\\
\varrho + p + \mathscr{P} + \Lambda_i - \frac{1}{2\uptau_\pi}(2\upeta + \uplambda_{\pi\mathscr{P}}\mathscr{P} ) - \frac{\uptau_{\pi\pi}}{4\uptau_\pi}(\Lambda_i + \Lambda_j ) & \geq 0, \, i \neq j
\label{E:Necessary_causality_DNMR_d}
\\
\frac{1}{2\uptau_\pi}(2\upeta + \uplambda_{\pi\mathscr{P}}\mathscr{P} ) 
+ \frac{\uptau_{\pi\pi}}{2\uptau_\pi} \Lambda_i &
\nonumber 
\\
+ \frac{1}{6\uptau_\pi}[2 \upeta + \uplambda_{\pi\mathscr{P}} \mathscr{P} 
+ (6\updelta_{\pi\pi} - \uptau_{\pi\pi})\Lambda_i ] &
\nonumber
\\
+ \frac{\upzeta + \updelta_{\mathscr{P}\mathscr{P}}\mathscr{P} + \uplambda_{\mathscr{P}\pi} \Lambda_i }{\uptau_\mathscr{P}} + (\varrho + p + \mathscr{P} + \Lambda_i ) c_s^2 & \geq 0,
\label{E:Necessary_causality_DNMR_e}
\\
\varrho + p + \mathscr{P} + \Lambda_i 
-\frac{1}{2\uptau_\pi}(2\upeta + \uplambda_{\pi\mathscr{P}}\mathscr{P} ) 
-\frac{\uptau_{\pi\pi}}{2\uptau_\pi} \Lambda_i &
\nonumber
\\
-\frac{1}{6\uptau_\pi}[2 \upeta + \uplambda_{\pi\mathscr{P}} \mathscr{P} 
+ (6\updelta_{\pi\pi} - \uptau_{\pi\pi})\Lambda_i ] 
 &
\nonumber
\\
- \frac{\upzeta + \updelta_{\mathscr{P}\mathscr{P}}\mathscr{P} + \uplambda_{\mathscr{P}\pi} \Lambda_i }{\uptau_\mathscr{P}}
- (\varrho + p + \mathscr{P} + \Lambda_i ) c_s^2 & \geq 0,
\label{E:Necessary_causality_DNMR_f}
\end{align}
\end{subequations}  
where $i,j=1,2,3$ in \eqref{E:Necessary_causality_DNMR_c}, \eqref{E:Necessary_causality_DNMR_d},
\eqref{E:Necessary_causality_DNMR_e}, and \eqref{E:Necessary_causality_DNMR_f}.
\end{theorem}

\begin{remark}
\label{R:Sufficient_regular_DNMR}
In the statement of Theorem \ref{T:Causality_DNMR}, we use \emph{sufficiently regular} to mean any class of functions for which the domain-of-dependence property can be proven, with the domain of dependence of solutions based on their characteristics. Currently, the Cauchy problem for Eqs.~\eqref{E:DNMR} is known to be solvable only in the Gevrey class (see Sect.~\ref{S:LWP_DNMR}), thus Theorem \ref{T:Causality_DNMR} applies in this case. Nevertheless, if the Cauchy problem is solvable in more general reasonable function spaces, including Sobolev spaces, then one can establish causality in these spaces and the validity of Theorem \ref{T:Causality_DNMR} with the ideas of Sect.~\ref{S:Domain_of_dependence_outside_Gevrey}.
\end{remark}

\begin{remark}
\label{R:Check_sufficient_and_necessary_conditions}
It can be checked that the sufficient conditions \eqref{E:Sufficient_causality_DNMR} are not empty and, as a sanity check that they in fact imply the necessary conditions \eqref{E:Necessary_causality_DNMR}.
It can also be checked that conditions \eqref{E:Necessary_causality_DNMR} are not sufficient for causality, see 
\cite{Bemfica-Disconzi-Hoang-Noronha-Radosz-2021}. 
\end{remark}

\begin{proof}
Modulo verifying some technical aspects, causality boils down to computing the system's characteristics (see Remarks \ref{R:Causality_precise_definition}, \ref{R:Sufficient_regular_DNMR}, and  \ref{R:Causal_structure_A_and_g}, Theorems \ref{T:Leray_Ohya_domain_of_dependence} and \ref{T:Leray_Ohya_diagonalization}, and Sect.~\ref{S:Domain_of_dependence_outside_Gevrey}). Given the complexity of \eqref{E:DNMR}, a brute-force calculation of the characteristic determinant is unlikely to produce useful expressions. In order to obtain an expression for the characteristic determinant that is amenable to an analysis of its roots, we combine the following ideas:
\begin{itemize}
\item We use the causality for the system linearized about global thermodynamic equilibrium
as a useful guide.
\item We borrow from the philosophy of Theorem \ref{T:New_formulation}, adopting a geometric viewpoint and
seeking to identify would-be acoustical metrics associated with the several speeds of propagation of the system (see below for the list of characteristics). We also borrow from Theorem \ref{T:New_formulation} the idea of identifying the right combination of variables that lead to good structures (although, here, unlike in Theorem \ref{T:New_formulation}, we are looking for good \emph{algebraic} structures).
\item We develop several calculation techniques that allow us to carry out some of the most challenging parts of the computation. These techniques rely crucially on the diagonalization of $\pi$, expressing it in terms of its eigenvalues and eigenvectors, as expressed Notation \ref{N:Diagonalization_pi}. I.e., this seemingly inconsequential procedure of expressing $\pi$ in terms of its bases of eigenvectors is in fact one of the \emph{key ideas} in the proof. Without it, the computations become too complicated to produce useful expressions\footnote{Observe the emphasis on \emph{useful}. We can always compute the characteristic determinant with the help of a symbolic math software. But this in general produce pages-long expressions whose roots cannot be analyzed.}.
\end{itemize}
Even after casting the characteristic determinant in suitable form, the analysis of its roots is still challenging. A careful analysis has to be carried out, assumptions \eqref{E:Sufficient_causality_DNMR} in the case of sufficiency or \eqref{E:Necessary_causality_DNMR} in the case of necessity are invoked in order to establish that the roots of the characteristic polynomial have the right properties to ensure causality or lack thereof. 

Once the characteristics have been identified, causality is established with the help of Theorems \ref{T:Leray_Ohya_diagonalization} and \ref{T:Leray_Ohya_domain_of_dependence} in the case of sufficiency.

For necessity, the analysis is as follows. Causality requires that for all $(\xi_1,\xi_2, \xi_3) \in T_x^*M$, $x$ a fixed spacetime point, the roots $\xi_0 = \xi_0(\xi_1,\xi_2, \xi_3)$ of the characteristic determinant are real and the set
\begin{align}
\{ \xi_0 = \xi_0(\xi_1,\xi_2, \xi_3) \, | \, (\xi_1,\xi_2, \xi_3) \in T_x^*M  \}
\nonumber
\end{align}  
lies outside the lightcone in $T_x^*M$ (see Remark \ref{R:Equivalent_definition_hyperbolicity_relativity}; in particular, recall that we require that $(\xi_0(\xi_1,\xi_2, \xi_3),\xi_1,\xi_2, \xi_3)$ lies outside the lightcone because this analysis is on the cotangent space). Thus, in order to violate causality, it suffices to exhibit a single $(\xi_1,\xi_2, \xi_3)$ for which these conditions are not met. We show that if the negation of \eqref{E:Necessary_causality_DNMR} holds then such a $(\xi_1,\xi_2, \xi_3)$  can be constructed, i.e., causality does not hold, and the result follows by the contrapositive. 
\end{proof}

We point out that our high-level presentation of the proof of Theorem \ref{T:Causality_DNMR} should not lead readers to think that its proof is by any means straightforward. As it should be expected from the complexity of Eqs.~\eqref{E:DNMR}, the calculation of the characteristic determinant and analysis of the corresponding roots is challenging and constitute the bulk of \cite{Bemfica-Disconzi-Hoang-Noronha-Radosz-2021}.

The proof of Theorem \ref{T:Causality_DNMR} reveals that the characteristics of Eqs.~\eqref{E:DNMR} are the following:
\begin{itemize}
\item The flow lines, which appear with multiplicity $14$ (i.e., $14$ repeated roots).
\item The sound waves, which appear with single multiplicity (i.e., two distinct roots, corresponding to a cone).
\item Shear waves, which appear as three distinct characteristics of single multiplicity each (i.e., two distinct roots, giving a cone, for each characteristic). More precisely, these are generally distinct characteristics, but they might coincide for specific values of the variables and transport coefficients.
\end{itemize}

Observe that the total number of roots in the above list comes to $22$, see Remark \ref{R:Constraints_propagation_DNMR}. We refer to \cite{Bemfica-Disconzi-Hoang-Noronha-Radosz-2021} for the exact expressions of the polynomials that determine the above characteritics.

\begin{remark}
In comparison to the characteristics of the system linearized about global thermodynamic equilibrium, which considers $n$, $\mathscr{J}$, and $\mathscr{Q}$ as well, Eqs.~\eqref{E:DNMR} do not have the so-called second sound, see \cite{Hiscock:1985zz}, which corresponds to propagation of temperature disturbances. This is because $n$, $\mathscr{J}$, and $\mathscr{Q}$ are not included in \eqref{E:DNMR}.
\end{remark}

Assumptions \ref{I:Causality_DNMR_1}, \ref{I:Causality_DNMR_2}, and 
\ref{I:Causality_DNMR_3} are natural from a physical perspective,
whereas \eqref{E:Sufficient_causality_DNMR} seem physically reasonable, although they 
are too stringent for some applications of interest. Indeed, it seems that one could simply choose initial data satisfying \eqref{E:Sufficient_causality_DNMR} in order to guarantee causality. However in, numerical simulations of the quark-gluon plasma, one \emph{does not get to choose the initial condition.} 
The reason is that the quark-gluon plasma, which is modeled as a relativistic viscous fluid, is only one stage of a multi-stage process. The quark-gluon plasma emerges from the collision of heavy-ions, and after a short-lived period as a fluid, this plasma freezes out into a collection of hadrons. Only the hadrons are measured directly by detectors, with properties of the quark-gluon plasma inferred indirectly
from hadronic measurements with the help of numerical simulations. Such numerical simulations have to account for the pre-fluid phase of matter, wherein the impact of the heavy-ions is modeled using microscopic theory\footnote{More precisely, this is one of two approaches used. The other is to simply parametrize the initial conditions. But this parametrization is constrained by certain structural features of heavy-ion collisions and does not allow one to choose the initial data freely either, see \cite{Moreland:2014oya}.} (kinetic theory, quantum chromodynamics) and produces a fluid state that is then \emph{fed as initial data into the fluid dynamics simulations.} This process is illustrated in Fig.~\ref{F:HIC}. We return to this point in Sect.~\ref{S:Limitations}.

\begin{figure}[ht]
    \centering
    \includegraphics[width=\textwidth]{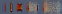}
    \caption{Stages of heavy-ion collisions \citep{Madai}.}
    \label{F:HIC}
\end{figure}

Currently, it is the set of necessary conditions \eqref{E:Necessary_causality_DNMR} that found more direct applications in numerical simulations of the quark-gluon plasma. This is because one can test, at each time step in these simulations, whether inequalities \eqref{E:Necessary_causality_DNMR} are satisfied and, if they are not, this means that causality has been violated in the simulation (of course, if \eqref{E:Necessary_causality_DNMR} holds causality is not guaranteed, see Remark \ref{R:Check_sufficient_and_necessary_conditions}). Recently, two groups of researchers, \cite{Plumberg:2021bme} and \cite{Chiu:2021muk}, tested conditions \eqref{E:Necessary_causality_DNMR} against 
numerical simulations of the quark-gluon plasma in collisions of heavy ions. They found out that \emph{causality is violated} in several stages of the simulations, particular at earlier times. In fact, 20--30\% of the fluid cells violate causality at the initial stages of the simulation, see Fig.~\ref{F:Causality_violations}.

\begin{figure}[ht]
    \centering
    \includegraphics[width=\textwidth]{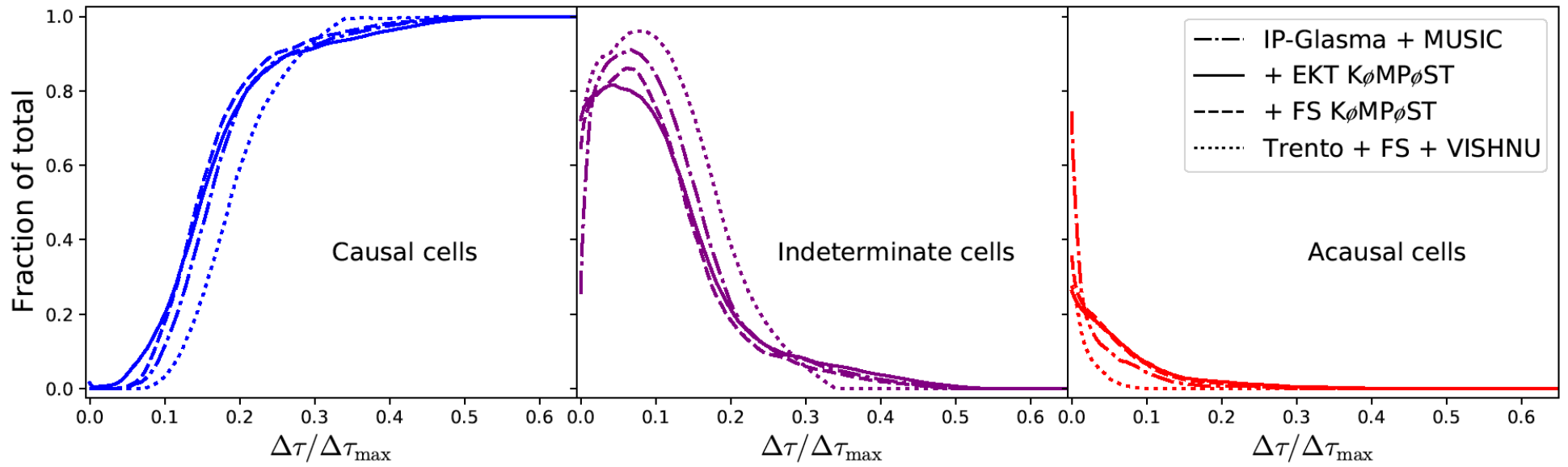}
    \caption{Causality violations from \cite{Plumberg:2021bme}. The graphs show, as a function of time: (left graph): fraction of cells that satisfy conditions \eqref{E:Sufficient_causality_DNMR}, and thus are causal; (right graph): fraction of cells that do not satisfy conditions \eqref{E:Necessary_causality_DNMR}, and thus violate causality; (middle graph): fraction of cells that satisfy \eqref{E:Necessary_causality_DNMR} but do not satisfy \eqref{E:Sufficient_causality_DNMR}, and thus whose causality is unknown. The different curves correspond to different numerical solvers that generate the initial condition fed to the hydrodynamic simulation. Image reproduced with permission from \cite{Plumberg:2021bme}.}
    \label{F:Causality_violations}
\end{figure}

Here, we will not discuss possible consequences of such causality violations for the current understanding of the quark-gluon plasma. Such an analysis would steer us away from our goal of focusing on mathematical aspects. We refer instead to \cite{Plumberg:2021bme} and \cite{Chiu:2021muk}. 
We will restrict ourselves to the following general remarks.

\begin{itemize}
\item Observe that causality violations do not necessarily imply an undefined evolution (in a PDE sense). For instance, it could be that the equations are hyperbolic and have a well-defined evolution, but their characteristics lay outside the lightcones in physical space.
\item It is clear from Fig.~\ref{F:Causality_violations} that causality violations are prominent at early times of the simulations, decreasing and nearly disappearing at later times. One possibility is that the initial data provided by the pre-fluid stage of the heavy-ion collision (see previous discussion) still retains non-fluid features that render the applicability of Eqs.~\eqref{E:DNMR} (which are derived under assumptions that the system is a fluid) questionable (recall Remark \ref{R:Viscosity_small}). On the other hand, the evolution toward a causal regime can in principle be coming form the fact that Eqs.~\eqref{E:DNMR} ``correct'' the initial state, in the sense that dissipation would drive the fluid toward thermodynamic equilibrium, i.e., the dynamics would approach that of a perfect fluid, where causality holds.
\item Another possibility is that the evolution from non-causal to causal states found in \cite{Plumberg:2021bme} and \cite{Chiu:2021muk} and depicted in Fig.~\ref{F:Causality_violations} is an artifact of the numerical algorithms employed. As it is usually the case in numerical simulations of complex systems of PDEs, their implementation is not clear-cut and requires the introduction of several numerical artifacts\footnote{One common example in typical fluid codes, including the case of classical fluids, is the use of artificial viscosity or parabolic regularizations.}. The numerical solvers used in \cite{Plumberg:2021bme} and \cite{Chiu:2021muk} employ certain regulators that 
cut off and reinitialize some variables when they reach certain critical values. It is possible that this process simply resets the variables to values where causality is satisfied.
\item The fact that causality violations nearly disappear at later times raises the natural question of whether such violations matter in practice. A satisfactory answer to this question does not seem to be available. First, one would need to make precise what it means to say that the causality violations do or do not matter. One possibility would be to show that such violations do not influence, in any meaningful way, the values of quantities of interest computed in the simulations and used for comparison with experiments. But it is hard to see how one could do so without having simulations without causality violations to compare with\footnote{In this regard, it is important to remember that, in any setting in which Eqs.~\eqref{E:DNMR} are locally well-posed, the system is deterministic, so that the behavior of solutions at later times is entirely determined by the non-causal initial condition. Thus one cannot simply infer that causality violations do not matter at earlier times from the fact that they no longer occur at later times.}.
\item The interplay between numerical solutions and actual solutions of \eqref{E:DNMR} (assuming they exists, see Sects.~\ref{S:LWP_DNMR} and \ref{S:Open_problems}) is not understood. In fact, there does not exist a result stating that the numerical algorithms used to simulate Eqs.~\eqref{E:DNMR} converge to or reproduce in a meaningful sense mathematical solutions to \eqref{E:DNMR}. In particular, it is not clear what the role of the aforementioned regulators is. An application of the causality conditions of Theorem \ref{T:Causality_DNMR} is only justified if the numerical simulations faithfully reproduce, within acceptable error bounds, actual solutions to \eqref{E:DNMR}. On the other hand, if such reproduction does not hold, it is then not clear what equations of motion are being solved in these simulations. We refer to \cite{Guermond-Marpeau-Popov-2008} and references therein for a discussion
of the delicate relation between local well-posedness and convergence
of numerical schemes.
\end{itemize}

We stress that the above remarks are speculative, save for the last one. It is not yet understood why causality violations happen, decrease over time, nor the interplay between numerical methods and solutions to \eqref{E:DNMR} (see Sects.~\ref{S:Limitations} and \ref{S:Open_problems}).
Nevertheless, we do note that these causality violations should serve at least as a cautionary tale about the \emph{dangers of numerically studying equations whose mathematical properties are poorly understood.}

\subsubsection{Local well-posedness\label{S:LWP_DNMR}}
In this section we address the solvability of the Cauchy problem for the DNMR equations \eqref{E:DNMR}. Our main result is the following.

\begin{theorem}[Bemfica, Disconzi, Hoang, Noronha, Radosz, \cite{Bemfica-Disconzi-Hoang-Noronha-Radosz-2021}]
\label{T:LWP_DNMR}
Consider the Cauchy problem for Eqs.~\eqref{E:DNMR} in Minkowski space, with initial data
$\mathring{\Psi} = (\mathring{\varrho}, \mathring{u}, \mathring{\mathscr{P}}, \mathring{\pi})$ on $\{ t = 0\}$.
Assume that the initial data satisfies the constraints \eqref{E:Velocity_normalization_viscous} and 
\eqref{E:Shear_constraints} and that $\mathring{u}^0 > 0$ (i.e., $\mathring{u}$ is future-pointing). 
Suppose that all transport coefficients in Eqs.~\eqref{E:DNMR} are analytic functions of their arguments.
Assume that the initial data satisfies assumptions \ref{I:Causality_DNMR_1}, \ref{I:Causality_DNMR_2}, and \ref{I:Causality_DNMR_3} of Theorem \ref{T:Causality_DNMR} in strict form, i.e., with $\leq$ replaced by $<$. Assume further that the initial data
satisfies conditions \eqref{E:Sufficient_causality_DNMR} of Theorem \ref{T:Causality_DNMR} in strict form. Finally, suppose that $\mathring{\Psi} \in G^{(s)}(\mathbb{R}^3)$, with $1 \leq s <\frac{20}{19}$, where $G^{(s)}$ is a Gevrey space (see Definition \ref{D:Gevrey_spaces}). Then, there exists a $T>0$ and a unique  solution
$\Psi = (\varrho, u, \mathscr{P}, \pi) \in G^{(s)}([0,T]\times \mathbb{R}^3)$ to Eqs.~\eqref{E:DNMR} such that
$\Psi = \mathring{\Psi}$ on $\{ t = 0\}$. Moreover, $\Psi$ is causal in the sense of Definition \ref{D:Causality}.
\end{theorem}

\begin{proof}
The computation of the characteristics carried out for the proof of Theorem \ref{T:Causality_DNMR} allows us to invoke Theorem \ref{T:Leray_Ohya_diagonalization} to diagonalize Eqs.~\eqref{E:DNMR}, putting them into the form \eqref{E:Diagonalized_Leray}. From the characteristics we can further verify that the intersection of the corresponding cones in co-tangent space has non-empty interior when applied to the initial data. Thus, the diagonalized system is a Leray--Ohya system for which the assumptions of Theorem \ref{T:Leray_Ohya_existence} and are satisfied. More precisely:
\begin{itemize}
\item The diagonliazed system is of higher order than \eqref{E:DNMR} (see Eq.~\eqref{E:Diagonalized_Leray}). Thus, we need to derive higher order initial conditions for the diagonalized system. This is done by sucessively differentiating Eqs.~\eqref{E:DNMR} and algebraically solving for time derivatives in terms of spatial derivatives, which can be done because, in view of the computation of the characteristics applied to the initial data, $\{ t= 0\}$ in non-characteristics.
\item The assumption that the diagonliazed system forms a Leray--Ohya system for any $V \in \mathcal{A}^s(\Sigma,Y)$  (see Definition \ref{D:Iteration_space}) is guaranteed by the fact that we assumed conditions 
\ref{I:Causality_DNMR_1}, \ref{I:Causality_DNMR_2}, \ref{I:Causality_DNMR_3}, and \eqref{E:Sufficient_causality_DNMR} to hold in strict form, so that we obtain a Leray--Ohya system for the given initial data and functions nearby.
\item The Gevrey index of the diagonalized system (see Definition \ref{D:Gevrey_index} and Remark 
\ref{R:Q_Q_minus_1_diagonalization}) is $\frac{19}{20}$.
\end{itemize}
With this, we obtain a solution to the diagonalized system. In order to obtain a solution to the the original equations, we proceed as follows. We approximate the given Gevrey data by analytic data (see \cite{Teofanov-2006}) and solve the corresponding analytic problem by the Cauchy--Kovalevskays theorem (which can be applied since $\{ t = 0\}$ is non-characteristic for our data, see above). These solutions automatically satisfy the diagonalized system. Thus, the Gevrey estimates of \cite{Leray-Ohya-1967} apply to them. We can use these estimates to control the analytic solutions in the Gevrey topology and pass to the limit, obtaining a Gevrey solution to the original equations. See \cite{Disconzi-2019} for details of the argument (which was applied to a different set of equations, but the argument is the same).

Causality follows because the solutions to \eqref{E:DNMR} that we have constructed also satisfy the diagonliazed system for which Theorem \ref{T:Leray_Ohya_domain_of_dependence} can be invoked.
\end{proof}

\begin{remark}
\label{R:Generalization_LWP_DNMR}
We make the following important observations.
\begin{itemize}
\item Theorem \ref{T:LWP_DNMR} generalizes in a more or less straightforward manner to
a globally hyperbolic spacetime and to Eqs.~\eqref{E:DNMR} coupled to Einstein's equations with Gevrey initial data, see \cite{Bemfica-Disconzi-Hoang-Noronha-Radosz-2021}.
\item The requirement of Gevrey initial data is, unfortunately, quite restrictive. We would like to establish local existence and uniqueness in Sobolev and smooth spaces. This seems challenging, see Sect.~\ref{S:Open_problems}.
\item Nevertheless, in the case when $\pi$ is absent (i.e., setting $\pi=0$ in Eqs.~\eqref{E:DNMR_density}, \eqref{E:DNMR_velocity}, \eqref{E:DNMR_bulk}, and dropping Eq.~\eqref{E:DNMR_shear}), the system \eqref{E:DNMR_density}-\eqref{E:DNMR_bulk} is locally well-posed in Sobolev spaces under suitable assumptions on the transport coefficients. This result holds with coupling to Einstein's equations and, in fact, if the baryon current \eqref{E:Baryon_density_current_viscous} is included with $\mathscr{J} = 0$, and if a smooth function of $\varrho$, $n$, and $\mathscr{P}$ (but not their derivatives) is added to \eqref{E:DNMR_bulk}. See \cite{Bemfica-Disconzi-Noronha-2019-2} for details. We note that that in \cite{Bemfica-Disconzi-Noronha-2019-2}, $\upzeta = \upzeta(\varrho,n)$, which is the case often considered in physics, but as also noted in that paper, the result remains true if $\upzeta = \upzeta(\varrho,n,\mathscr{P})$, so that the term $\updelta_{\mathscr{P}\mathscr{P}} \mathscr{P}$ in \eqref{E:DNMR_bulk} can be absorbed into $\upzeta$.
\end{itemize}
\end{remark}

A very natural question once local well-posedness has been established is that of global well-posedness versus breakdown of solutions. This is addressed in the next Theorem, which we state in an informal way for the sake of brevity, referring to \cite{Disconzi-Hoang-Radosz-2023} for a precise statement\footnote{See \cite{Lerman:2023qyc} for an analogue of Theorem \ref{T:Breakdown_DNMR} for a non-relativistic limit of the DNMR equations.}.

\begin{theorem}[\citealt{Disconzi-Hoang-Radosz-2023}]
\label{T:Breakdown_DNMR}
Consier the DNMR equations without shear viscosity and with the Minkowski metric.
Then, there exists an open set of smooth initial data 
for which the corresponding unique smooth solutions to the Cauchy problem
break down in finite time.
\end{theorem}

In Theorem \ref{T:Breakdown_DNMR}:
\begin{itemize}
\item The DNMR equations without shear viscosity are as stated above, i.e., 
setting $\pi=0$ in Eqs.~\eqref{E:DNMR_density}, \eqref{E:DNMR_velocity}, \eqref{E:DNMR_bulk}, and dropping Eq.~\eqref{E:DNMR_shear}, but in Theorem \ref{T:Breakdown_DNMR} it is further assumed that $\updelta_{\mathscr{P}\mathscr{P}}=0$. Equation \eqref{E:Divergence_baryon_current_viscous}, with $\mathscr{J} = 0$ in \eqref{E:Baryon_density_current_viscous}, is also included, with the equation of state and the transport coefficients being now functions of $\varrho$ and $n$. For technical reasons, it turns out that the proof provided in \cite{Disconzi-Hoang-Radosz-2023} does not go through if $n$, and thus Eq.~\eqref{E:Divergence_baryon_current_viscous}, are not considered as apart of the system, see the discussion in 
\cite{Disconzi-Hoang-Radosz-2023}. Finally, we note that the result is in fact a bit more general, as a term proportional to $\mathscr{P}^2$ can be added to Eq.~\eqref{E:DNMR_bulk}, see \cite{Disconzi-Hoang-Radosz-2023}.
\item Although we have not stated a result on the Cauchy problem for the DNMR equations in $C^\infty$, the situation described in the previous bullet point fits into the assumptions that guarantee local well-posedness in Sobolev spaces mentioned in Remark \ref{R:Generalization_LWP_DNMR}. From this and a standard argument one obtains local existence and uniqueness of smooth solutions.
\item The breakdown in Theorem \ref{T:Breakdown_DNMR} happens as a dichotomy. In finite time, either solutions cease to be $C^1$ or they become un-physical, where physical solutions are defined in a precise sense. In particular, one of the ways that solutions can become un-physical is by violation of causality.
\item The initial data constructed in Theorem \ref{T:Breakdown_DNMR} consists of localized (large) perturbations of constant states. 
\end{itemize}

\begin{proof} We will only highlight some aspects of the proof, referring to 
\cite{Disconzi-Hoang-Radosz-2023} for details.
\begin{itemize}
\item The proof follows closely the proof of breakdown of solutions for the relativistic Euler equations by Guo and Tahvildar--Zadeh, see \cite{Guo-Tahvildar-Zadeh-1999}, which in turn builds on the seminal work of Sideris on formation of singularities for the classical compressible Euler equations, see \cite{Sideris-1985}. We assume that solution exists for all time and derive a Riccati-type differential inequality for the fluid variables that contradicts some quantitative assumptions on the initial data. 
\item In order to ensure that that Euler-inspired argument works, we need to obtain some precise control of $\mathscr{P}$. This is done with help of some intricate transport estimates for $\mathscr{P}$.
\end{itemize}
\end{proof}

We remark that the proof of Theorem \ref{T:Breakdown_DNMR}, being a proof by contradiction, does not provide information about the nature of the breakdown. See Sect.~\ref{S:Open_problems}.

\subsubsection{Further context and historical notes\label{S:DNMR_historical}}

Systems where the viscous fluxes are new variables satisfying additional equations of motion are the subject of \textdef{extended irreversible thermodynamics}, also referred to as \textdef{rational extended thermodynamics} (see \citealt{Muller-Ruggeri-Book-1998}). 
``Extended'', in this context, refers to the fact that one is extending the state space by adding extra variables. This idea goes back to the works of \cite{Grad-1963} and \cite{Mueller-1967}, where they adopted this approach to construct models of classical fluids. Grad's obtained equations of motion for the viscous fluxes from the method of moments in kinetic theory, as already mentioned, whereas M\"uller postulated equations for the viscous fluxes based on thermodynamic considerations. 
More precisely, M\"uller introduced an out-of-equilibrium entropy current and tailored the equations of motion for the viscous fluxes to guarantee that entropy production be always non-negative. M\"uller's entropy current contained a correction to the entropy current of a perfect fluid, deemed a second-order correction. This approach has been widely adopted in extended theories and because of it extended theories are also known as \textdef{second-order theories of relativistic viscous fluids.} 

M\"uller's approach was adapted to relativity by 
\cite{Stewart-1977,Israel-Stewart-1979,Stewart-Stewart-1976,Israel:1976tn,Israel:1979wp}, leading to what has become known as the \textdef{M\"uller--Israel--Stewart (MIS) theory,} sometimes referred to simply as \textdef{Israel--Stewart (IS) theory} (see Eqs.~\eqref{E:MIS} below). Another recent example of a second-order theory of relativistic viscous fluids is the \textdef{resumed Baier--Romatschke--Son--Starinets--Stephanov (re-BRSSS) theory} introduced in \cite{Baier:2007ix}, which connects second-order approaches with a gradient expansion.

It has been common practice in the field of relativistic viscous fluids to refer to the original Israel--Stewart theory, the re-BRSSS theory, and the DNMR theory, collectively as MIS or IS theories, or yet as \textdef{MIS-like} or \textdef{IS-like theories.} While such theories share the basic feature that the viscous fluxes are independent variables satisfying additional equations of motion, we note that they are in practice different theories as the equations satisfied by the viscous fluxes are different in each instances. For example, while in the MIS theory the entropy current is always non-negative, in the DNMR theory it is non-negative only up to higher-order terms, where higher-order is measured with respect to a specific power-counting scheme used to organize the moments expansion in the derivation of the DNMR theory. These higher-order terms are negligible in the regime where the derivation from kinetic theory by the method of moments is expected to be valid, thus in the field's jargon one says that entropy production in the DNMR theory is non-negative within the limit of validity of the theory.

On the other hand, because of their similarities, it is often the case that qualitative ideas are treated as interchangeable among these MIS-like theories. For example, while numerical simulations of the quark-gluon plasma are performed using the DNMR theory, the re-BRSSS has been highly influential in setting up the stage for the study of the quark-gluon plasma. We remark, however, that the MIS, re-BRSSS, and DNMR theories agree when linearized about a global thermodynamic equilibrium\footnote{This explains why in our references about the stability of the DNMR theory we cited papers published before the original DNMR work. Those works studied the stability of the MIS theory, which gives the same linearization about global thermodynamic equilibrium as the DNMR theory.}. We refer the reader to Chapter~6 of \cite{Rezzolla-Zanotti-Book-2013}, \cite{Romatschke:2017ejr}, \cite{Denicol-Rischke-Book-2021,Muller-Ruggeri-Book-1998}, and references therein for further discussions of second-order theories of relativistic viscous fluids.

\subsubsection{Viscous shocks in the MIS theory and implications for  the DNMR theory\label{S:Viscous_shocks}}
One topic of much interest in fluid dynamics is that of the formation of shocks, see the beginning of Sect.~\ref{S:Study_of_shocks} and Sects.~\ref{S:Shocks_1d} and \ref{S:Shocks_higher_d}.
The possibility of shocks in the quark-gluon plasma has been investigated in \cite{Bouras:2009nn,Bouras:2009vs,Bouras:2010zz,Fogaca:2012gb,Bouras:2010hm}, whereas the potential for shocks in mergers of neutron stars is discussed in \cite{Radice:2020ddv,Bernuzzi:2020tgt}. 

Given that the DNMR equations form a quasilinear system not in divergence form, it is not immediately clear how to construct a weak formulation that could accommodate jump discontinuities and hence describe shocks. On the other hand, when dissipation is present, the possibility emerges that physical phenomena usually modeled by shocks could be described by solutions that are not discontinuous as in the case of perfect fluids, but rather by the so-called viscous shocks. A \textdef{viscous shock} is a plane-symmetric continuous solution connecting asymptotic thermodynamic equilibrium states which are supersonic in front of the shock and subsonic behind the shock, see below and also the references \cite{Barker-Humpherys-Lafitte-Rudd-Zumbrun-2008,Olson:1991pf,Geroch-Lindblom-1991}. We will now describe some results on viscous shocks for the MIS theory, turning to the DNMR equations at the end of this section (recall from  Sect.~\ref{S:DNMR_historical} the similarities between the MIS and the DNMR theories).

The equations of motion of the MIS theory are, as in the case of the DNMR theory, given by \eqref{E:Divergence_energy_momentum_baryon_current_viscous}, with the energy-momentum tensor and the baryon
current given by
\eqref{E:Energy_momentum_viscous} and \eqref{E:Baryon_density_current_viscous}, respectively, with the viscous fluxes satisfying additional equations of motion. The choice of these equations, however, is different than that of the DNMR theory, and is given by (see \citealt[Chapter 6]{Rezzolla-Zanotti-Book-2013} and \citealt{Olson:1991pf}; compare with \eqref{E:DNMR})
\begin{subequations}{\label{E:MIS}}
\begin{align}
\upbeta_0 u^\mu \nabla_\mu \mathscr{P} + \frac{1}{\upzeta} \mathscr{P} + \nabla_\mu u^\mu 
 -\upalpha_0 \nabla_\mu \mathscr{Q}^\mu 
 -\upgamma_0 \uptheta \mathscr{Q}^\mu \nabla_\mu \left( \frac{\upalpha_0}{\uptheta} \right)
& \nonumber
\\
  + \frac{1}{2} \mathscr{P} \uptheta \nabla_\mu \left( \frac{\upbeta_0 u^\mu}{\uptheta} \right)  & = 0,
\label{E:MIS_bulk}
\\
\Big\langle \upbeta_2 u^\lambda \nabla_\lambda \pi^{\mu \nu} + \nabla^\mu u^\nu  - \upalpha_1 \nabla^\mu \mathscr{Q}^\nu 
 + \frac{1}{2} \pi^{\mu\nu} \uptheta \nabla_\lambda \left( \frac{\upbeta_2 u^\lambda}{\uptheta} \right)
&
\nonumber
\\
 -\upgamma_1 \uptheta \mathscr{Q}^\mu \nabla^\nu \left( \frac{\upalpha_1}{\uptheta} \right) 
+\upgamma_3 \pi_\lambda^\nu \nabla^{[ \mu} u^{\lambda ]} \Big\rangle + \frac{1}{2 \upeta} \pi^{\mu\nu} 
 & = 0,
\label{E:MIS_shear}
\\
\upbeta_1 \proj^{\mu\nu} u^\lambda\nabla_\lambda \mathscr{Q}_\nu + \frac{1}{\upkappa \uptheta} \mathscr{Q}^\mu + \frac{1}{\uptheta}\proj^{\mu\nu}\nabla_\nu \uptheta 
+\proj^{\mu\nu} u^\lambda \nabla_\lambda u_\nu 
-\upalpha_0 \proj^{\mu\nu}\nabla_\nu \mathscr{P} 
&
\nonumber
\\
-\upalpha_1 \proj^{\mu\nu} \nabla_\lambda \pi^{\lambda}_\nu
+ \frac{1}{2} \proj^{\mu\nu}\mathscr{Q}_\nu \uptheta \nabla_\lambda \left( \frac{\upbeta_1 u^\lambda}{\uptheta} 
\right)
-(1-\upgamma_0)\uptheta \mathscr{P} \proj^{\mu\nu}\nabla_\nu\left(\frac{\upalpha}{\uptheta}\right)
&=0.
\label{E:MIS_heat}
\end{align}
\end{subequations}
Above, $\Big \langle \cdot \Big \rangle$ means
\begin{align}
\Big \langle A^{\mu\nu} \Big \rangle := \mathsf{\Pi}^{\mu\nu}_{\alpha\beta} A^{\alpha\beta},
\nonumber
\end{align}
where $\mathsf{\Pi}^{\mu\nu}_{\alpha\beta}$ is given by \eqref{E:Projection_symmetric_trace_free_part}; $[\cdot]$ means anti-symmetrization; the coefficients $\upalpha_i, \upbeta_i, \upgamma_i$ are transport coefficients which are known functions of $\varrho$ and $n$; and all other quantities have the same meaning as in Sects.~\ref{S:Relativistic_Euler} and the beginning of Sect.~\ref{S:DNMR}. Also, in \eqref{E:Energy_momentum_viscous} and \eqref{E:Baryon_density_current_viscous} we have $\mathscr{R}=0,\mathscr{J}=0$. Finally, constraints
\eqref{E:Shear_constraints} continue to hold and, in addition,
\begin{align}
u^\mu \mathscr{Q}_\mu = 0.
\nonumber
\end{align}

\begin{remark}
Unlike the case of Eqs.~\eqref{E:DNMR}, in Eqs.~\eqref{E:MIS} we are considering a dependence on $n$ and the inclusion of $\mathcal{J}$ and $\mathscr{Q}$, compare with Assumption \ref{A:DNMR_simplified}.
\end{remark}

We are now ready to state the problem. We consider plane-symmetric steady solutions and assume the background to be Minkowski. Thus, we can assume all quantities to depend only on the $x^1$ coordinate and write the velocity and the heat flux as $u = \gamma (1,v, 0,0)$, $\mathscr{Q} =  \mathsf{Q}(\gamma v, \gamma,0,0)$, where $v$ and $\mathsf{Q}$ are scalars and $\gamma = (1-v^2)^{-\frac{1}{2}}$ is the standard $\gamma$ factor in Minkowski space. Under these conditions, there is only one non-zero component of $\pi$ which we can take to be the $x^1 x^1$ component and proportional to $\gamma^2$, so we can write $\pi^{11} = \mathbb{B} \gamma^2$ for some $\mathbb{B}$. We prescribe that solutions approach global thermodynamic equilibrium when $x \rightarrow\pm \infty$, i.e., the viscous fluxes vanish when $x \rightarrow\pm \infty$ and the remaining fluid variables approach constant values. In order to describe shocks, we want solutions that are supersonic in front of the shock and subsonic behind it, thus, $v \rightarrow v^\pm$ when $x \rightarrow \pm \infty$, where $v^+ > c_s^+$,  $0 < v^- < c_s^-$ and $c_s^\pm$ are the corresponding asymptotic values of the sound speed. Because the system approaches equilibrium at $\pm \infty$, $v^\pm$ are the perfect-fluid asymptotic speeds before and after the shock\footnote{We note that if we insist that $v^\pm$ are asymptotic states of a shock wave solution to the relativistic Euler equations satisfying appropriate jump conditions (see Sect.~\ref{S:Shocks_1d}), $v^+$ and $v^-$ cannot both be prescribed arbitrarily, see \cite{Geroch-Lindblom-1991}.\label{FN:Asymptotic_states_related}}.
Solutions to \eqref{E:MIS} with the properties just described are called \textbf{viscous shocks.}

\begin{remark}
The description given above can be equally applied to theories with different viscous fluxes. Thus, we can also define viscous shock solutions for other theories of viscous fluids.
\end{remark}

We define the upstream and downstream Mach numbers, denoted $Ma^\pm$, by
\begin{align}
Ma^\pm := \frac{v^\pm}{c_s^\pm}.
\nonumber
\end{align}
Physically, $Ma^\pm$ measures the strength of a shock. A shock is thought of as mild or weak if $Ma^\pm$ is close to one.

The main result concerning viscous shocks for the MIS theory, due to \cite{Geroch-Lindblom-1991} and \cite{Olson:1991pf}, is that there \emph{exists a critical 
value\footnote{$Ma^-$ and $Ma^+$ will, in general, be related, see Footnote \ref{FN:Asymptotic_states_related}. Thus, the result can be alternatively stated in terms of the existence of a critical $Ma^-_*$.} $Ma^+_*$ such that the MIS equations do not admit viscous shock solutions if $Ma^+ > Ma^+_*$.}

The main idea of the proof is the following. Under the conditions defining a viscous shock, the system \eqref{E:MIS} simplifies considerably and $\mathsf{Q}$, $\mathbb{B}$ and $n$ can be written in terms of $\varrho$, $v$, and constants of integration that are determined by prescribing the variables' asymptotic values. The equations simplify to a system of the form
\begin{align}
A(\Psi) \frac{d}{dx^1} \Psi = B(\Psi),
\nonumber
\end{align}
where $\Psi = (\varrho, v, \mathscr{P})$ and $A$ and $B$ are, respectively, a matrix and a vector depending on $\Psi$ (but not its derivative). To have well-defined solutions, we need to avoid the characteristic condition 
\begin{align}
\det A(\Psi) = 0.
\nonumber
\end{align}

The determinant $\det A(\Psi)$ is known at $x = \pm \infty$ since the asymptotic states are prescribed in this limit. Analyzing the determinant evaluated at the asymptotic states and using $v^+ > c_s^+$ and  $0 < v^- < c_s^-$, one can then show that the determinant changes sign and thus, being a continuous function of the fluid variables, must vanish somewhere, if $Ma^+$ is sufficiently large. The argument also uses standard thermodynamic assumptions such as the positivity of several physical quantities. The reader is referred to \cite{Geroch-Lindblom-1991,Olson:1991pf} for the details of the proof.

Recall from Sect.~\ref{S:DNMR_historical} that the MIS and the DNMR theories, albeit different, share many features. Thus, one could expect a similar result as the above for the DNMR theory, see Sect.~\ref{S:Open_problems}. We point out that viscous shocks have also been investigated for the BDNK theory that we discuss in Sect.~\ref{S:BDNK}.

\subsubsection{Potential limitations\label{S:Limitations}}
Despite the widespread use of the DNMR equations and their success in applications to the study of the quark-gluon plasma, it is important to put in perspective some potential limitations of the theory:
\begin{itemize}
\item From a mathematical point of view, the main drawback in the current status of the DNMR theory is a lack of local well-posedness in Sobolev spaces. The only case when local well-posedness in Sobolev spaces does hold, when shear viscosity is absent, is not enough for applications to the quark-gluon plasma, since it is precisely shear viscosity that seems to be the main viscous contribution in this case (see \citealt{Heffernan:2023gye,Heffernan:2023utr}). 
\item In particular, when it comes to establishing results about convergence of numerical schemes (see comments at the end of Sect.~\ref{S:Causality_DNMR}), it seems highly unlikely that Gevrey spaces can be useful.
\item More generally, we would like to establish causality and local well-posedness of the DNMR theory when $\mathscr{J}$, $n$, and $\mathscr{Q}$ are included (recall Remark \ref{A:DNMR_simplified} and see Sect.~\ref{S:Open_problems}). While Eqs.~\eqref{E:DNMR} are enough for current investigations of the quark-gluon that focus on high-energy heavy-ion collisions, in which case $n=0$, low-energy collisions are expected to have $n\neq 0$ (see \citealt{Denicol:2018wdp}).
\item In the study of viscous effect in mergers of neutron stars, where coupling to Einstein's equations needs to be taken into account, local well-posedness in Sobolev spaces is widely regarded as a basic requirement for the implementation of general-relativistic numerical codes (see discussions in
\citealt[Chapter 11]{Baumgarte:2010ndz} and 
\citealt[Chapter 4]{Rezzolla-Zanotti-Book-2013}). The good news in this regard is that 
recent investigations suggest that only bulk viscosity is relevant for these mergers (see \citealt{Alford:2017rxf,Most:2022yhe}). In this case, one could simulate Einstein's equations coupled to the DNMR equations with bulk viscosity as the only viscous flux, in which case local well-posedness in Sobolev spaces has been established
(see Remark \ref{R:Generalization_LWP_DNMR}). The not so good news is that some estimates suggest that shear viscosity can also be relevant (see \citealt{Shibata:2017jyf,Shibata:2017xht,Ripley:2023qxo}), in which case local well-posedness for the DNMR equations is currently unavailable\footnote{Also, as a matter of principle, one would establish local well-posedness in the presence of all viscous fluxes, and show that a smallness condition on the shear viscosity and heat flux is propagated by the flow, in that in reality these quantities will not be exactly zero in mergers, even in the case when bulk is the main contribution.}. This discrepancy in assessing the effects of viscosity in neutron star mergers simply reflects a tentative state of affairs. There has not been yet a single measurement of viscous effects in mergers of neutron stars, with all evidence in their favor currently theoretical, mostly from numerical simulations. Given how much is not understood about neutron star mergers, current estimates about the the role of dissipation and the relative magnitudes of shear and bulk viscosity in such mergers have to be taken as provisional at best. In this regard, it is worth noting that it was not long ago that it was well-accepted that viscosity was not relevant at all for the study of neutron star mergers (see \citealt{Duez:2018jaf,Foucart:2015gaa}), highlighting how future studies can significantly change the current picture about dissipation in neutron star mergers. Given such uncertainties, it is highly desirable to establish local well-posedness and causality under general assumptions and including all viscous fluxes.
\item As discussed in Sect.~\ref{S:Viscous_shocks}, 
\cite{Geroch-Lindblom-1991} and \cite{Olson:1991pf}
showed that MIS theory cannot describe viscous shock solutions in a regime of strong shocks. Due to the similar structure of the DNMR and MIS equations, it is conjectured that the same difficulty is present in the DNMR theory (see Sect.~\ref{S:Open_problems}).
\item The causality violations discussed in Sect.~\ref{S:Causality_DNMR} are not, per se, a limitation of the DNMR theory, given that we established that there are physically reasonable conditions that guarantee the evolution to be causal. But in order to better understand the relevance of these causality violations and their origins, it would be beneficial to find an alternative approach where causality could in principle be enforced throughout the numerical simulations.
\end{itemize}

The remarks should not taken as disapproving of the DNMR theory, which has been highly instrumental in advancing our understanding of the quark-gluon plasma. But they do highlight that there is enough motivation to seek an alternative to the DNRM theory. One possible such alternative will be investigated in the next section.

\subsection{The BDNK theory\label{S:BDNK}}
The \textdef{Bemfica--Disconzi--Noronha--Kovtun (BDNK) theory} is the culmination of the works by
\cite{Bemfica-Disconzi-Noronha-2018,Bemfica-Disconzi-Noronha-2019-1,Bemfica:2020zjp,Kovtun:2019hdm,Hoult:2020eho}, whose goal is to construct a causal, stable, and locally well-posed theory of relativistic viscous fluids \emph{\`a la} Eckart and Landau--Lifshitz, i.e., where the viscous fluxes are given in terms of $\varrho$, $n$, $u$, and their derivatives up to first order.

As we have done with the DNMR theory, we will not discuss physical aspects of the BDNK theory in detail, restricting ourselves to physical aspects of the theory that are of direct relevance for the mathematical results we will present or that help elucidate the current status of the BDNK theory as a research program. Thus, instead of working through the physical arguments and modeling choices that lead to the BDNK theory, we will define it here and work out its mathematical consequences in what follows. But since some of the modeling choices might seem unmotivated, we make some brief remarks about them in Sect.~\ref{S:BDNK_origins}. Readers interested in the physical motivation leading to \eqref{E:Viscous_fluxes_BDNK} below can consult the above references, especially \cite{Bemfica:2020zjp,Kovtun:2019hdm}, where a more thorough discussion is given.

The BDNK theory is defined by taking the viscous fluxes in \eqref{E:Energy_momentum_viscous} and
\eqref{E:Baryon_density_current_viscous} as follows.
\begin{subequations}{\label{E:Viscous_fluxes_BDNK}}
\begin{align}
\mathscr{R} & := \uptau_\mathscr{R} (u^\mu \nabla_\mu \varrho + (p + \varrho) \nabla_\mu u^\mu ),
\label{E:Viscous_fluxes_BDNK_R}
\\
\mathscr{P} & := - \upzeta \nabla_\alpha u^\alpha + \uptau_\mathscr{P} (u^\mu \nabla_\mu \varrho + (p+\varrho) \nabla_\mu u^\mu ),
\label{E:Viscous_fluxes_BDNK_P}
\\
\pi_{\alpha\beta} & := - \upeta \proj_\alpha^\mu \proj_\beta^\nu ( \nabla_\mu u_\nu + \nabla_\nu u_\mu - \frac{2}{3} \nabla_\lambda u^\lambda g_{\mu\nu}),
\label{E:Viscous_fluxes_BDNK_pi}
\\
\mathscr{Q}_\alpha & := \uptau_\mathscr{Q} (p+\varrho) u^\mu \nabla_\mu u_\alpha + \upbeta_\mathscr{Q} \proj_\alpha^\mu \nabla_\mu \varrho + \upbeta_\varrho \proj_\alpha^\mu \nabla_\mu \varrho 
+ \upbeta_n \proj_\alpha^\mu \nabla_\mu n,
\label{E:Viscous_fluxes_BDNK_Q}
\\
\mathscr{J}_\alpha & := 0,
\label{E:Viscous_fluxes_BDNK_J}
\end{align}
\end{subequations}
where
\begin{align}
\upbeta_\varrho & := \uptau_\mathscr{Q} \left. \frac{\partial p}{\partial \varrho}\right|_n 
+ \upkappa \uptheta h \left. \frac{\partial (\upmu/\uptheta)}{\partial \varrho} \right|_n,
\nonumber
\\
\upbeta_n & := \uptau_\mathscr{Q} \left. \frac{\partial p}{\partial n}\right|_\varrho 
+ \upkappa \uptheta h \left. \frac{\partial (\upmu/\uptheta)}{\partial n} \right|_\varrho,
\nonumber
\end{align}
where $\upmu$ is the chemical potential determined by the thermodynamic relation
\begin{align}
\frac{dp}{p+\varrho} = \frac{d\uptheta}{\uptheta} + \frac{n \uptheta}{p+\varrho}d \left(\frac{\upmu}{\uptheta}\right),
\nonumber
\end{align}
$\uptau_\mathscr{R},\uptau_\mathscr{P},\uptau_\mathscr{Q},\upzeta, \upeta, \upkappa$ are transport coefficients that are known functions of $\varrho$ and $n$, with $\uptau_\mathscr{R},\uptau_\mathscr{P},\uptau_\mathscr{Q}$ called relaxation times and $\upzeta, \upeta, \upkappa$ being the coefficients of bulk and shear viscosity and heat conductivity. Finally, one continues to assume that the velocity satisfies the constraint \eqref{E:Velocity_normalization_viscous}. Compare \eqref{E:Viscous_fluxes_BDNK} with \eqref{E:Viscous_fluxes_Eckart}.

Observe that we are taking $\varrho$ and $n$ as the primary thermodynamic scalars, with the remaining ones being functions of $\varrho$ and $n$ determined by the thermodynamic relations of Sect.~\ref{S:Thermodynamic_properties}, whose notation we continue to employ. In particular, all coefficients (the transport coefficients and the $\upbeta$'s) as well as the equation of state are known functions of $\varrho$ and $n$. The equations of motion, the \textdef{BDNK equations,} are, as usual, given by \eqref{E:Divergence_energy_momentum_viscous} and \eqref{E:Divergence_baryon_current_viscous}. Once again, from the point of view of the Cauchy problem, we treat all components of $u$ as independent, imposing the constraint \eqref{E:Velocity_normalization_viscous} on the initial data and propagating it.

\begin{remark}
We are considering $\mathscr{J}=0$, i.e., Eq.~\eqref{E:Viscous_fluxes_BDNK_J}, because this is the case for which stronger and more complete results have been obtained; see Theorems \ref{T:Causality_BNDK} and \ref{T:LWP_BDNK}. The case $\mathscr{J}=0$ is also the primary case of interest in studies of viscous effects in neutron star mergers (see \citealt{Most:2021zvc,Alford:2017rxf}). It is possible, however, to formulate the BDNK theory with $\mathscr{J} \neq 0$ and obtain results similar to Theorem \ref{T:Causality_BNDK} (see \citealt{Hoult:2020eho}).
\end{remark}

\begin{remark}
\label{R:Baryon_current_constraint_BDNK}
Since the BDNK energy-momentum tensor involves first-order derivatives of $\varrho$, $n$ and $u$, Eqs.~\eqref{E:Divergence_energy_momentum_viscous} are second-order PDEs in $\varrho$, $n$ and $u$, and thus $\varrho$, $n$, $u$, and their first-order derivatives need to be prescribed as initial data for the Cauchy problem. Equation \eqref{E:Divergence_baryon_current_viscous}, however, is first-order in $n$ and $u$, and thus it is a constraint that needs to be satisfied by the initial data and then propagated. This is done by considering instead $u^\mu \nabla_\mu$ applied to \eqref{E:Divergence_baryon_current_viscous} as an equation of motion.
\end{remark}

Expanding  $u^\mu \nabla_\mu$ applied to \eqref{E:Divergence_baryon_current_viscous} (see Remark \ref{R:Baryon_current_constraint_BDNK}) and 
\eqref{E:Divergence_energy_momentum_viscous} and projecting the latter on the directions parallel and orthogonal to the velocity, we find
\begin{subequations}{\label{E:BDNK}}
\begin{align}
u^\mu u^\nu \partial_\mu \partial_\nu n + n \updelta^\mu_\lambda u^\nu \partial_\mu \partial_\nu u^\lambda + \tilde{\Err}_n(n,u,g)\partial^2 g &= \Err_n(\partial n, \partial u, \partial g),
\label{E:BDNK_current}
\\
(\uptau_\mathscr{R}u^\mu u^\nu + \upbeta_\varrho \proj^{\mu\nu} ) \partial_\mu \partial_\nu \varrho
+ \upbeta_n \proj^{\mu\nu} \partial_\mu \partial_\nu n
&
\nonumber 
\\
+ (p+\varrho)(\uptau_\mathscr{R} + \uptau_\mathscr{Q}) u^{(\mu} \updelta^{\nu)}_\lambda
\partial_\mu \partial_\nu u^\lambda  
+ \tilde{\Err}_\varrho(\varrho,n, u, g)\partial^2 g
& = \Err_\varrho(\partial \varrho, \partial n, \partial u, \partial g),
\label{E:BDNK_density}
\\
(\upbeta_\varrho + \uptau_\mathscr{P}) u^{(\mu}\proj^{\nu)\alpha}\partial_\mu \partial_\nu \varrho
+\upbeta_n u^{(\mu}\proj^{\nu)\alpha}\partial_\mu \partial_\nu n &
\nonumber
\\
+ \mathcal{C}^{\alpha \mu \nu}_\lambda \partial_\mu \partial_\nu u^\lambda
+ \tilde{\Err}^\alpha_u(\varrho, n, u, g) \partial^2 g 
& = \Err^\alpha_u(\partial \varrho, \partial n, \partial u, \partial g),
\label{E:BDNK_velocity}
\end{align}
\end{subequations}
The notation in Eqs.~\eqref{E:BDNK} is as follows.
The $\Err$ terms are as in Notation \ref{N:Err_terms}, where in addition we have included subscripts ${}_n$ etc. to distinguish among these terms\footnote{Such distinction will not be important here but are being faithful to the notation of \cite{Bemfica:2020zjp}.}. The $\tilde{\Err}(\partial^\ell \phi) \partial^2 g$ terms, in particular, denote terms that are second-order in $g$ with coefficients depending on at most $\ell$ derivatives of $\phi$. The precise form of such terms will not be important here, but we have singled them out because they will contribute to the principal part of the system when the BDNK equations are coupled to Einstein's equations (since Einstein's equations in, say, wave coordinates, are  second-order in $g$). We note that we have expanded all covariant derivatives and written the equations of motion in terms of partial derivatives. In particular, the $\tilde{\Err}(\partial^\ell \phi) \partial^2 g$ terms come from expanding second-order derivatives of $u$ (as they will contain first-order derivatives of the Christoffel symbols).
We have used the standard notation
\begin{align}
A^{(\alpha}B^{\beta)} := \frac{1}{2}( A^\alpha B^\beta + A^\beta B^\alpha),
\nonumber
\end{align}
and defined
\begin{align}
\mathcal{C}^{\alpha \mu \nu}_\lambda := 
(\uptau_\mathcal{P}(p+\varrho) -\upzeta - \frac{1}{3}\upeta)\proj^{\alpha(\mu}\updelta^{\nu)}_\lambda
+ ( (p+\varrho) \uptau_\mathcal{Q} u^\mu u^\nu - \upeta \proj^{\mu\nu})\updelta^\alpha_\lambda.
\nonumber
\end{align}
Finally, recall that $\updelta^\mu_\nu$ is the Kronecker delta.

Equations \eqref{E:BDNK} form a second-order system of PDEs without diagonal principal part, and one should take note of the complexity of these equations. Nevertheless, as we will see in the next section, the BDNK equations seem to possess better mathematical structures than the DNMR equations, as will be able to prove stronger results than those obtained for the DNMR theory.

\subsubsection{Local well-posedness, stability, and causality\label{S:LWP_BDNK}}

In order to state our main results for the BDNK theory, it is convenient to introduce the following notation. We denote by $c_s^2$ corresponding sound speed of a perfect fluid which, in terms of the primary variables $n$ and $\varrho$ and with an equation of state $p = p(\varrho,n)$, read
\begin{align}
c_s^2 := \left. \frac{\partial p}{\partial \varrho} \right|_n + \frac{n}{p+\varrho} 
\left. \frac{\partial p}{\partial n} \right|_\varrho.
\nonumber
\end{align}
We also introduce
\begin{align}
\kappa := \frac{(p+\varrho)^2 \uptheta}{n} \left. \frac{\partial (\upmu/\uptheta)}{\partial \varrho}\right|_n + \uptheta (p+\varrho)\left. \frac{\partial (\upmu/\uptheta)}{\partial n}\right|_\varrho.
\nonumber
\end{align}

\begin{theorem}[\citealt{Bemfica:2020zjp}] 
\label{T:Causality_BNDK}
Let $(\varrho,n,u)$ be a smooth\footnote{See Remark \ref{R:Smooth_assumption_causality_BDNK}.} solution to BDNK equations satisfying the constraint \eqref{E:Velocity_normalization_viscous} and defined in a globally hyperbolic spacetime.
Suppose that
\begin{enumerate}[label=(A.\arabic*),ref=(A.\arabic*)]
\item $p+\varrho, \uptau_\mathscr{R}, \uptau_\mathscr{P}, \uptau_\mathscr{Q} >0, \upzeta, \upeta, \upkappa \geq 0$.
\label{I:Causality_BDNK}
\end{enumerate}
Then, the following are necessary and sufficient conditions for causality\footnote{Recall definition \ref{D:Causality}.} of Eqs.~\eqref{E:BDNK}.
\begin{subequations}{\label{E:Causality_BDNK}}
\begin{align}
(p+\varrho) \uptau_\mathscr{Q} & > \upeta,
\label{E:Causality_BDNK_a}
\\
[ \uptau_\mathscr{R} ( (p+\varrho)c_s^2 \uptau_\mathscr{Q} +\upzeta +\frac{4}{3}\upeta + \upkappa \kappa)
+ (p+\varrho)\uptau_\mathsf{P} \uptau_\mathscr{Q} ]^2 & \geq
\nonumber
\\
4 (p+\varrho) \uptau_\mathscr{R}\uptau_\mathscr{Q}[ \uptau_\mathscr{P}( (p+\varrho) c_s^2 \uptau_\mathscr{Q} + \upkappa \kappa ) - \upbeta_\varrho (\upzeta + \frac{4}{3}\upeta ) ] & \geq 0,
\label{E:Causality_BDNK_b}
\\
2(p+\varrho) \uptau_\mathscr{R} \uptau_\mathscr{Q} > \uptau_\mathscr{R} ( (p+\varrho) c_s^2 \uptau_\mathscr{Q} +\upzeta + \frac{4}{3}\upeta + \upkappa \kappa ) + (p+\varrho) \uptau_\mathscr{R} \uptau_\mathscr{Q} & \geq 0,
\label{E:Causality_BDNK_c}
\\
(p+\varrho) \uptau_\mathscr{R}\uptau_\mathscr{Q} + \upkappa \kappa \uptau_\mathscr{P} > 
\uptau_\mathscr{R} ( (p+\varrho) c_s^2 \uptau_\mathscr{Q} +\upzeta + \frac{4}{3}\upeta + \upkappa \kappa )
&
\nonumber
\\
+(p+\varrho)\uptau_\mathscr{R}\uptau_\mathscr{Q}(1-c_s^2) +\upbeta_\varrho(\upzeta + \frac{4}{3}\upeta).&
\label{E:Causality_BDNK_d}
\end{align}
\end{subequations}
\end{theorem}
\begin{remark}
One can verify that conditions \eqref{E:Causality_BDNK} are not empty.
\end{remark}
\begin{proof}
Like in the case of the proof of Theorem \ref{T:Causality_DNMR}, 
with the help of Theorems \ref{T:Leray_Ohya_domain_of_dependence}, \ref{T:Leray_Ohya_diagonalization}, and ideas from Sect.~\ref{S:Domain_of_dependence_outside_Gevrey}, 
the proof reduces essentially to an analysis of the characteristics. 
We work locally and employ Eqs.~\eqref{E:BDNK}, see Remark \ref{R:Baryon_current_constraint_BDNK}.
Also as in the proof of Theorem \ref{T:Causality_DNMR}, a brute-force calculation does not seem helpful, and we need to borrow from geometric intuition and develop some calculation techniques tailored to the problem. See \cite{Bemfica:2020zjp} for details.
\end{proof}

\begin{remark}
\label{R:Smooth_assumption_causality_BDNK}
We have assumed smoothness in Theorem \ref{T:Causality_BNDK} for simplicity, but the result remains valid for the Sobolev-regular solutions established in Theorem \ref{T:LWP_BDNK}. In fact, it is from Theorem \ref{T:LWP_BDNK} that we obtain, in a standard fashion, smooth solutions from smooth initial.
\end{remark}

Observe that, unlike Theorem \ref{T:Causality_DNMR}, we obtain for the BDNK necessary and sufficient causality conditions. This happens because the characteristic polynomial for the BDNK equations turns out to simpler than that of the DNMR equations. However, different causality conditions than \eqref{E:Causality_BDNK} are not ruled out if \ref{I:Causality_BDNK} is not assumed.
We remark Theorem \ref{T:Causality_BNDK} has been proven in particular cases 
in \cite{Bemfica-Disconzi-Noronha-2019-1, Bemfica-Disconzi-Noronha-2018,Disconzi-2019}.

The analysis of the characteristics in the proof of Theorem \ref{T:Causality_DNMR} reveals that the characteristics of the BDNK system are
\begin{itemize}
\item The flow lines, which appear with multiplicity two (two repeated roots).
\item The sound waves, which appear with single multiplicity (i.e., two distinct roots, corresponding to a cone).
\item The so-called second sound (corresponding to propagation of temperature disturbances),
which appears with single multiplicity (i.e., two distinct roots, corresponding to a cone).
\item Shear waves, with appear as three distinct characteristics of single multiplicity each (i.e., two distinct roots, giving a cone, for each characteristic). More precisely, these are generally distinct characteristics, but they might coincide for specific values of the variables and transport coefficients.
\end{itemize}

Observe that the total number of roots adds to 12, corresponding to the six second-order equations \eqref{E:BDNK}.

Next, we investigate stability of the BDNK equations, in the sense of Definition \ref{D:Stability}.

\begin{theorem}
\label{T:Stability_BDNK}
Under appropriate conditions, the BDNK equations are stable in the sense of Definition \ref{D:Stability}.
\end{theorem}

\begin{remark}
We refer the reader to \cite{Bemfica:2020zjp} for a precise statement of Theorem \ref{T:Stability_BDNK}, remarking that the ``appropriate conditions'' in Theorem \ref{T:Stability_BDNK} consist of a set of inequalities in the spirit of \eqref{E:Causality_BDNK}. We also remark it is possible to simultaneously satisfy the causality and stability conditions of Theorems \ref{T:Causality_BNDK} and \ref{T:Stability_BDNK}, respectively.
\end{remark}

Theorem \ref{T:Stability_BDNK} is proven by appealing to a general stability theorem also established in \cite{Bemfica:2020zjp}, which is of interest on its own. This general theorem states, roughly, that
\begin{align}
\label{E:Stability_theorem_informal}
\begin{split}
\text{mode stability} & \text{ in the LRF} + \text{causality} 
+ \text{certain structural conditions}
\\
& \Rightarrow \text{mode stability in any Lorentz frame.}
\end{split}
\end{align}
Is is assumed that the system of PDEs in question is such that one can define the above notions, e.g., there is a well-defined local rest frame (LRF) and the equations of motion are Lorentz invariant\footnote{Observe that Lorentz invariance does not imply that mode stability is a Lorentz invariant property in that, for two  four-vectors $(\omega,\vec{k})$ and $(\omega^\prime,\vec{k}^\prime)$ related by a Lorentz transformation $\Lambda$, it is not necessarily true that $\omega^\prime = \Lambda \omega$ and $\vec{k}^\prime = \Lambda \vec{k}$, as the Lorentz transformation in general will mix the time and space components.}.

Result \eqref{E:Stability_theorem_informal} was generalized by \cite{Gavassino:2021owo}, who was able to remove the structural conditions assumption in \eqref{E:Stability_theorem_informal} and provide a clear physical interpretation of the result\footnote{See \cite{Heller:2022ejw,Gavassino:2023mad,Gavassino:2023myj,Wang:2023csj,Hoult:2023clg,Gavassino:2023mad} for more on the relation between causality and stability.}. We refer to \cite{Bemfica:2020zjp,Gavassino:2021owo} for the precise statement of these results.

We next turn to the question of local well-posedness. 

\begin{definition}
\label{D:BDNK_initial_data}
On a globally hyperbolic spacetime $(M,g)$ with a Cauchy surface $\Sigma$, \textdef{initial data for the BDNK equations} consists of the values of $\varrho$, $n$, and $u$ along $\Sigma$, such that \eqref{E:Velocity_normalization_viscous} holds, and the values of $\nabla_N \varrho$, $\nabla_N n$, and $\nabla_N u$ along $\Sigma$, where $N$ is the unit future-pointing normal to $\Sigma$ and $\nabla_N u$ is compatible with \eqref{E:Velocity_normalization_viscous}, i.e., it satisfies $g_{\mu\nu} \nabla_N u^\mu u^\nu = 0$. Furthermore, $n$, $u$, and their first-order derivatives are constrained by demanding that \eqref{E:Divergence_baryon_current_viscous} holds on $\Sigma$.
\end{definition}

\begin{remark}
\label{R:BDNK_initial_data}
Observe that in view of the constrains mentioned in Definition \ref{D:BDNK_initial_data}, it suffices to provide the projections of $\nabla_N u$ onto $T\Sigma$. This is more natural when considering coupling to Einstein's equations since only a Riemannian three manifold is initially available. 
In practice, any collection of fields that allows to determine $\varrho$, $n$, and $u$ and their first-order derivatives restricted to $\Sigma$, including derivatives transverse to $\Sigma$, can be taken as initial data for the BDNK equations, as long as they are compatible with the constraints mentioned in 
Definition \ref{D:BDNK_initial_data}. For example, instead of covariant derivatives in the normal direction one could prescribe Lie derivatives in the normal direction.
\end{remark}

\begin{theorem}[\citealt{Bemfica:2020zjp}, 
\citealt{Bemfica-Disconzi-Graber-2021}]
\label{T:LWP_BDNK}
Let $(\mathbb{R}\times \Sigma, g)$ be a globally hyperbolic smooth Lorentzian manifold, where $\Sigma$ is 
a Cauchy surface that is compact smooth three-dimensional manifold without boundary.
Consider initial data
\begin{align}
\begin{split}
(\mathring{\varrho}, \mathring{n}, \mathring{u}) & \in H^{N}(\Sigma) \times H^N(\Sigma) \times H^{N}(\Sigma),
\\
(\mathring{\varrho}^\prime, \mathring{n}^\prime, \mathring{u}^\prime)  &\in H^{N-1}(\Sigma) \times H^{N-1}(\Sigma) \times H^{N-1}(\Sigma),
\end{split}
\nonumber
\end{align}
$N \geq 5$, 
for the BDNK equations, where
$\mathring{\varrho}^\prime$, $\mathring{n}^\prime$, and $\mathring{u}^\prime$  are prescribed data for derivatives of 
$\varrho$, $n$, and $u$ transverse to $\Sigma$, respectively, according to Definition \ref{D:BDNK_initial_data} and Remark \ref{R:BDNK_initial_data}. Assume that 
the equation of state and the transport coefficients in the BDNK equations are analytic functions of their arguments. Suppose the initial data satisfies assumptions \ref{I:Causality_BDNK} and \eqref{E:Causality_BDNK} of Theorem \ref{T:Causality_BNDK} in strict form (i.e., with $\leq$ replaced by $<$).
Then, there exists a $T > 0$ and a unique classical solution
$(\varrho, n, u)$ to the BDNK equations defined on $[-T,T] \times \Sigma$ and taking the given initial data.
Moreover, this solution satisfies
\begin{align}
\begin{split}
(\varrho(t, \cdot), n(t,\cdot), u(t, \cdot) )
& \in 
H^N(\Sigma_t), 
\\
(\nabla \varrho(t, \cdot), \nabla n(t,\cdot), \nabla u(t, \cdot) )
& \in 
H^{N-1}(\Sigma_t),
\end{split}
\nonumber
\end{align}
$-T \leq t \leq T$, where\footnote{See Sect.~\ref{S:Notation_conventions}.} $\Sigma_t := \{ (\tau,x) \in \mathbb{R} \times \Sigma \, | \, \tau = t, x \in \Sigma \}$. Furthermore, solutions depend continuously on the initial data in the $C^0([-T,T], H^N(\Sigma_t))$ topology. Finally, a similar result holds for the BDNK equations coupled to Einstein's equations.
\end{theorem}
\begin{proof}
We outline the main steps of the proof, referring to \cite{Bemfica:2020zjp} and \cite{Bemfica-Disconzi-Graber-2021} for details; \cite{Bemfica-Disconzi-Graber-2021}, in particular, contains the heavy PDE machinery that is employed in the proof\footnote{Readers will notice that \cite{Bemfica-Disconzi-Graber-2021} appeared prior to \cite{Bemfica:2020zjp}, but the PDE techniques in \cite{Bemfica-Disconzi-Graber-2021} are general and apply to large classes of diagonalizable systems.} (see \cite{Disconzi-Shao-2023-arxiv} for a simpler approach).
\begin{itemize}
\item In view of Theorem \ref{T:Causality_BNDK}, it suffices to work locally. We consider Eqs.~\eqref{E:BDNK} (see Remark \ref{R:Baryon_current_constraint_BDNK}). We re-write this second-order system as a system of first order 
\begin{align}
A^\mu (\Phi) \partial_\mu \Phi + B(\Phi) = 0
\label{E:First_order_BDNK_non_diagonal}
\end{align}
by introducing geometric decompositions of derivatives as new variables, e.g., $u^\mu \partial_\mu \varrho$ and $\proj^{\alpha\mu}\partial_\mu \varrho$.
\item Under our assumptions, because we have understood the characteristics from Theorem \ref{T:Causality_BNDK}, the principal symbol of \eqref{E:First_order_BDNK_non_diagonal} can be diagonalized as follows. 
Causality implies that the roots $\xi=(\xi_0,\xi_1,\xi_2,\xi_2)$, $\xi_0 = \xi_0(\xi_1,\xi_2,\xi_2)$, of the characteristic determinant
\begin{align}
\det (A^\mu \xi_\mu ) = 0,
\nonumber
\end{align}
cannot be timelike (since they lie outside the lightcone in cotangent space), thus $A^0$ is invertible. We then write 
\eqref{E:First_order_BDNK_non_diagonal} as
\begin{align}
\partial_t \Phi + \tilde{A}^i(\Phi) \partial_i \Phi + \tilde{B}(\Phi) = 0,
\label{E:First_order_BDNK_non_diagonal_partial_t}
\end{align}
where $\tilde{A}^i(\Phi) :=  (A^0(\Phi))^{-1} A^i(\Phi)$ and $\tilde{B}(\Phi) = (A^0(\Phi))^{-1} B(\Phi)$.
Furthermore, from the description of the characteristics, it also follows that 
and for any spacelike $\zeta$, the eigenvalue problem 
\begin{align}
(\tilde{A}^i \zeta_i - \Lambda \mathbb{I})V = 0,
\nonumber
\end{align}
has only real eigenvalues $\Lambda$ and we can prove that it admits a set of complete eigenvectors $V$ ($\mathbb{I}$ is the identity matrix).
\item  It follows from the above that 
there exist a $\zeta$-dependent and $\Phi$-dependent matrix $\mathcal{S} = \mathcal{S}(\Phi,\zeta)$ and 
a $\zeta$-dependent and $\Phi$-dependent diagonal matrix $\mathcal{D} = \mathcal{D}(\Phi, \zeta)$ such that 
we diagonalize $\tilde{A}^i \zeta_i$,  where $\zeta$ is a co-vector, writing it as
\begin{align}
\mathcal{S} \tilde{A}^i \zeta_i = \mathcal{D} \mathcal{S}.
\label{E:Diagonalization_BDNK_symbol}
\end{align}
\item The diagonalization \eqref{E:Diagonalization_BDNK_symbol} is at the level of symbols.
To pass from the symbols to the corresponding PDEs, we have to introduce pseudo-differential operators
(the entries of the matrices $\mathcal{S}$ and $\mathcal{D}$ are combinations of rational functions and square roots of $\zeta$ in view of the expressions we obtain for the eigenvalues). This allows us to transform \eqref{E:First_order_BDNK_non_diagonal_partial_t} into an evolution equation with diagonal principal part, except that this principal part is now pseudo-differential. (It should be intuitive from \eqref{E:Diagonalization_BDNK_symbol} that some sort of diagonal pseudo-differential equation can be derived from \eqref{E:First_order_BDNK_non_diagonal_partial_t}, but the precise expressions are cumbersome and thus will not be given here).
\item The crucial property of the pseudo-differential diagonal operator obtained is that it satisfies a G\"arding inequality. This relies crucially on the reality of the eigenvalues. From this, we can derive an energy estimate for the diagonal pseudo-differential system, which finally translated into estimates for $\Phi$. We note that, in view of the quasilinear nature of the problem, one has to deal with a pseudo-differential calculus for symbols with limited smoothness.
\item In order to obtain solutions to the original equations, we use a standard approximation by analytic functions. The proof is similar to the argument in Sect.~\ref{S:Complementary_LWP_relativistic_Euler}, which the reader can inspect for more details.
We approximate the initial data by analytic data and construct analytic solutions to \eqref{E:BDNK} (from the characteristics, we know that $\{ t = 0 \}$ is non-characteristic). These analytic solutions automatically satisfy \eqref{E:First_order_BDNK_non_diagonal_partial_t} and thus the above energy estimates hold for them. Moreover, the above procedure to derive energy estimates also allows us to derive estimates for differences of solutions, which we can use to obtain convergence of the analytic solutions in Sobolev spaces, producing a solution to equations\footnote{As usual, convergence happens in a lower norm but the limit will belong to the top-order Sobolev space.} \eqref{E:BDNK} and, finally, the BDNK equations \eqref{E:Divergence_energy_momentum_viscous} and \eqref{E:Divergence_baryon_current_viscous} (recall Remark \ref{R:Baryon_current_constraint_BDNK}). 
\item Uniqueness follows from the aforementioned estimate for the difference of solutions. We also obtain, from standard $L^\infty-L^2$ estimates for nonlinearities, a standard continuation criteria given by derivatives of the variables in $L^\infty$, which we can use to establish local existence and uniqueness of smooth solution. Finally, continuous dependence on the data is obtained by an argument similar to \cite{Kato-1975-1}.
\item We now make a few comments on the proof for the BDNK equations coupled to Einstein's equations. We consider initial data for Einstein's equations satisfying the constraints\footnote{See Sect.~\ref{S:Open_problems}.}, with the metric and the second fundamental form in $H^N$ and $H^{N-1}$, respectively, and work in wave coordinates. The first observation is that the characteristics of the coupled system are the characteristics of the BDNK equations and the lightcones. This is the case despite the fact that Eqs.~\eqref{E:BDNK} have terms in $\partial^2 g$. The principal part of the coupled system has an upper-triangular form with the $\partial^2 g$ contributing only the the upper triangular part that does not enter in the characteristic determinant (see equation (A1) in \cite{Bemfica:2020zjp}).
Writing Einstein's equations in first-order form shows that the coupled system can still be written in the form \eqref{E:First_order_BDNK_non_diagonal}, except that now a little more care has to be taken in defining the fluid variables for the first-order system. For example, one such variables is given by $u^\mu \nabla_\mu u_\alpha$, i.e., it is a combination of first-order partial derivatives of $u$ and first-order derivatives of the metric (hidden in the Christoffel symbols). (This is reminiscent of ideas
from Sects.~\ref{S:New_formulation}, \ref{S:Rough_solutions}, and \ref{S:Vacuum_bry}, where one needs to find the right combination of variables.) We also note that the first-order formulation of Einstein's equations is not the usual first-order symmetric hyperbolic (see, e.g., \citealt{Fischer-Marsden-1972}). Instead, we take as variables $u^\mu \partial_\mu g_{\alpha\beta}$ and 
$\proj^{\lambda \mu} \partial_\mu g_{\alpha\beta}$, which are the decomposition of the spacetime gradient of  $g_{\alpha\beta}$ in directions parallel and perpendicular to $u$ (we recall that upon writing Einstein's equations in coordinates, the components of the metric $g_{\alpha\beta}$ are treated as a collection of scalars from the PDE point of view). With these definitions, the previous arguments applied to \eqref{E:First_order_BDNK_non_diagonal} and \eqref{E:First_order_BDNK_non_diagonal_partial_t} follow with little change.
\end{itemize}
\end{proof}

\begin{remark}
There is no loss of generality in assuming that the spacetime has the topology $\mathbb{R}\times \Sigma$ in Theorem \ref{T:LWP_BDNK}, see Remarks after Theorem \ref{T:LWP_relativistic_Euler}.
As in Theorems \ref{T:LWP_relativistic_Euler} and \ref{T:LWP_Einstein_Euler}, $\Sigma$ is taken compact for simplicity, so that we do not have to discuss conditions at infinity, but the results generalize to typical non-compact situations (e.g., asymptotically flat manifolds).
\end{remark}

Some further mathematical results for the BDNK theory are as follows:
\begin{itemize}
\item Existence of viscous shocks solutions (see Sects.~\ref{S:Viscous_shocks} and \ref{S:Limitations} and  \citealt{Barker-Humpherys-Lafitte-Rudd-Zumbrun-2008,Olson:1991pf} for a definition and discussion of viscous shocks) for the BDNK theory has been established by \cite{Freistuhler-2021} (see also \citealt{Freistuhler-2020}). Freist\"{u}hler's results, in particular, provide a rigorous justification of the behavior of shock solutions observed numerically in \cite{Pandya:2021ief}. Further investigation of shocks in the BDNK theory can be found in \cite{Pellhammer-2023-arxiv}.
\item Global well-posedness for the BDNK equations has recently been established by
Sroczinski, see \cite{Sroczinski-2024}, for perturbations of constant solutions and specific choices of equation of state and transport coefficients. See also the related works 
\cite{Sroczinski-2019,Sroczinski-2020,Freistuhler-Sroczinski-2021,Freistuhler-Reintjes-Sroczinski-2022}.
\end{itemize}

\subsubsection{Applications and connections to physics\label{S:BDNK_applicatons}} 

In Sect.~\ref{S:LWP_BDNK}, we established key mathematical properties for the BDNK theory, namely, causality, local well-posedness, and stability, thus putting the BDNK theory on a sound mathematical footing. This should be contrasted with the results known for the DNMR equations, whose mathematical properties are much less understood (see Sect.~\ref{S:DNMR}). On the other hand, the DNMR theory has been very successful in applications to the quark-gluon plasma, see Sect.~\ref{S:DNMR}, even when one takes into account the potential caveats discussed in Sect.~\ref{S:Limitations}. 

If the good mathematical properties of the BDNK theory are to be relevant for physicists, we need to connect the BDNK theory with known physics. Given its novelty, it is not surprising that fully realistic applications of the BDNK theory have not yet been implemented (for example, numerical simulations of the quark-gluon plasma). But several simple studies of the BDNK theory connecting it to physics have been recently carried out. Here, we restrict ourselves to point out some of these results, referring to the ensuing references for more details.

\begin{itemize}
\item Numerical simulations of the BDNK theory for a $1+1$-dimensional conformal fluid have been carried out by \cite{Pandya:2021ief}, and \cite{Bantilan:2022ech}. They considered some standard test-cases, like initially stationary energy configurations and the Riemann problem. They also compared their results with the DNMR theory, finding that both theories essentially agree for small values of viscosity (in this case, small transport coefficients), but give different results when viscosity is large. This situation is illustrated in Fig.~\ref{F:BDNK_vs_IS_1d_conformal}. We recall that small viscosity is the primary regime where one expect relativistic viscous fluid dynamics to be applicable, although one often would like to push its application to large viscosity values as well, see Remark \ref{R:Viscosity_small}.
\item Simulations of of a $2+1$ conformal fluid in the BDNK formalism have been carried out by \cite{Pandya:2022pif}, where again several test-cases are studied. In particular, the authors consider the evolution of Kelvin-Helmholtz unstable initial data and show that viscosity has the expected stabilizing effect on the dynamics, see Fig.~\ref{F:BDNK_Kelvin_Helmholtz}.
\item \cite{Pandya:2022sff} considered the BDNK theory with the equation of state of an ideal gas.
\item We remark that the above numerical simulations have all been carried out in Minkowski background.
We refer the reader to the above references and \cite{Pandya:2023alj} for more results on numerical simulations of the BDNK theory, including a discussion of the numerical algorithms used in these implementations.
\item Some further test-cases numerical results for the BDNK theory in the so-called Gubser flow \citep{Gubser:2010ze}, and Bjorken flow \citep{Bjorken:1982qr}, which are some simplified scenarios providing toy models for the quark-gluon plasma, can be found in \cite{Bemfica-Disconzi-Noronha-2018}.
\end{itemize}

\begin{figure}[ht]
    \centering
    \includegraphics[width=\textwidth]{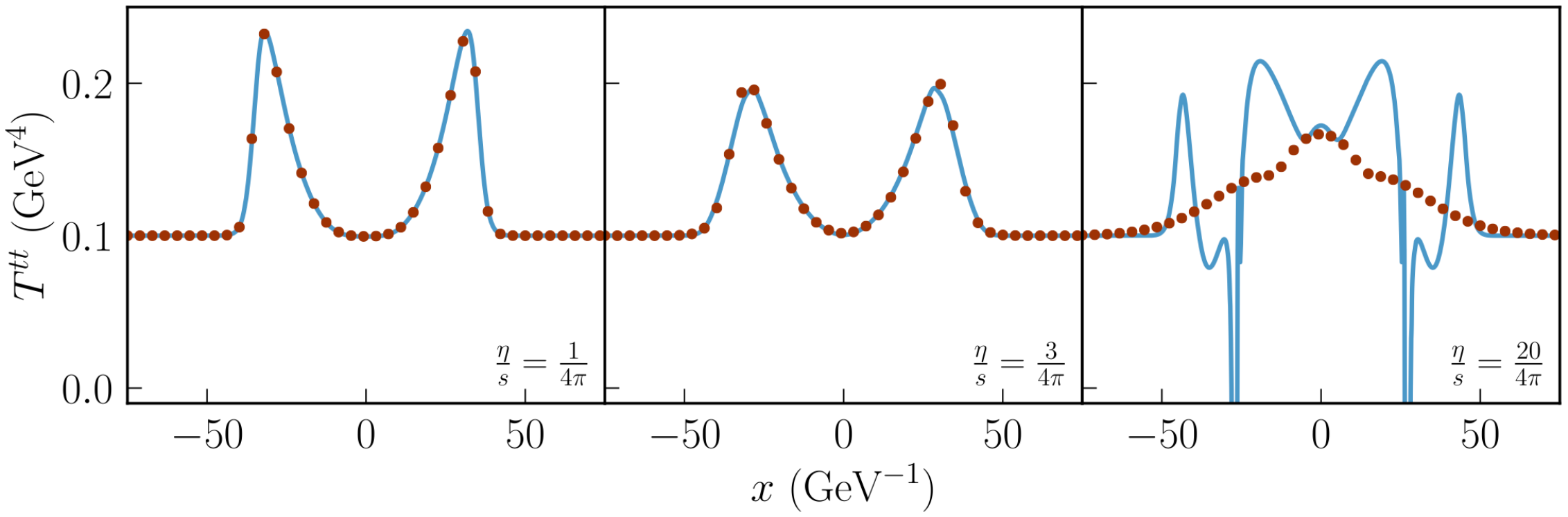}
    \caption{Snapshot of evolution of the density $tt$ component of the energy-momentum tensor as a function of space, for different values of the shear viscosity coefficient\protect\footnotemark \, $\upeta$ (denoted $\eta$ in the graphs). The solid lines represent the BDNK theory and the dots the DNMR theory.  Source: \cite{Pandya:2021ief}. }
    \label{F:BDNK_vs_IS_1d_conformal}
\end{figure}
\footnotetext{It is more appropriate to use $\upeta/s$, where $s$ is the entropy, instead of $\upeta$ as the former is related to the Kudsen number.}

\begin{figure}[ht]
    \centering
    \includegraphics[scale=0.3]{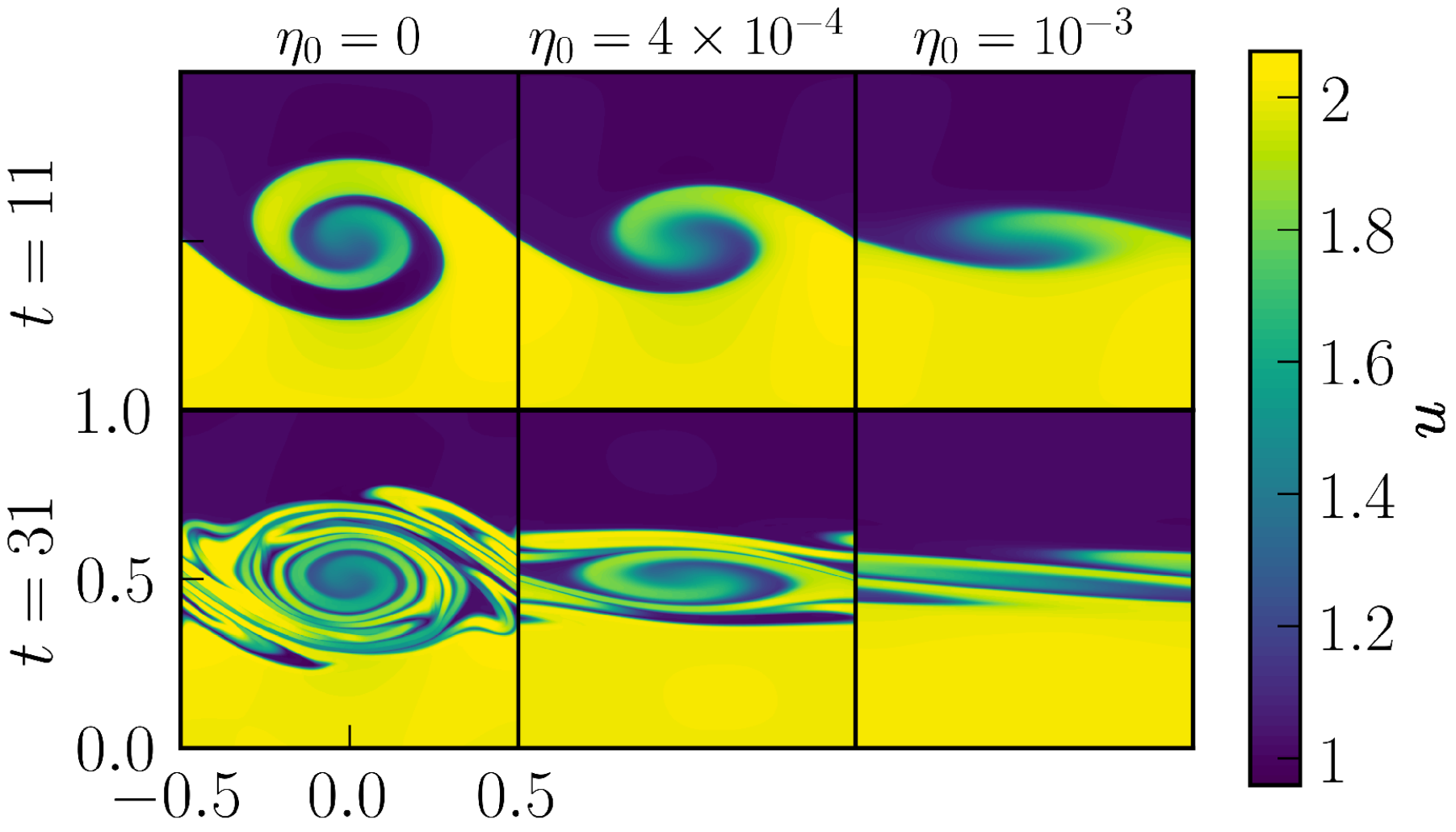}
    \caption{Snapshot of a BDNK evolution of unstable Kelvin-Helmholtz initial data at two different times, with different columns indicating different viscosity values (the first column has no viscosity and corresponds to a perfect fluid). Viscosity tends to damp the formation of vorticies. 
See the supplementary material for a video of the corresponding numerical simulation.    
   Image reproduced with permission from \cite{Pandya:2022pif}, copyright by APS.}
    \label{F:BDNK_Kelvin_Helmholtz}
\end{figure}

Above, we discussed numerical simulations because they provide the only likely avenue for applications of the BDNK theory to realistic systems. While simulations of the quark-gluon plasma or neutron star mergers with the BDNK theory have yet to be implemented, the above results show that the BDNK theory does reproduce known physics or expected results in simple situation. In addition, it is possible to establish results of physical significance for the BDNK theory by theoretical means, namely:

\begin{itemize}
\item The BDNK theory can be derived from kinetic theory, see
\cite{Bemfica-Disconzi-Noronha-2018,Bemfica-Disconzi-Noronha-2019-1,Hoult:2021gnb,Rocha:2023hts}.
\item Hoult and Kovtun have also derived the BDNK theory from the so-called holographic principle, see
\cite{Hoult:2021gnb}. 
\item Entropy production is non-negative for the BDNK theory within the regime of validity of the theory, i.e., up to higher-order gradients that are to be neglected upon the derivation of the BDNK theory (see below), see \cite{Kovtun:2019hdm}. This is a similar situation to the non-negativity of entropy production in the DNMR theory, see Sect.~\ref{S:Limitations}.
\item The validity of the BDNK and DNMR theories in comparison with results from microscopic theory
(for a specific choice of microscopic theory, the so-called $\lambda \phi^4$ of weakly-interacting particles), have been obtained by Brito, Denicol, and Rocha, see \cite{Rocha:2023hts}.
\end{itemize}

In sum, aside from its sound mathematical foundations, the BDNK theory seems to possess many physically appealing features. Nevertheless, we caution the reader that \emph{only with studies of the BDNK theory directed to realistic systems, what inevitably requires full scale $3+1$ numerical simulations, will one be able to asses whether the BDNK theory provides a satisfactory model of relativistic viscous fluids found in nature.}

\subsubsection{Origins of the BDNK tensor and historical notes\label{S:BDNK_origins}}
Having discussed mathematical and physical implications of the BDNK theory, we now turn to its origins.
We start noting that the Eckart and Landau-Lifshitz choices \eqref{E:Viscous_fluxes_Eckart} and \eqref{E:Viscous_fluxes_Landau}, respectively, do not provide the most general form of an energy-momentum tensor for a viscous fluid where the viscous fluxes are given in terms of $\varrho$, $n$, $u$, and their derivatives up to first order. As we mentioned, the choices \eqref{E:Viscous_fluxes_Eckart} and \eqref{E:Viscous_fluxes_Landau} were motivated by the search for a covariant generalization of the Navier-Stokes-Fourier theory. But other modeling choices, partly inspired by the theory of perfect fluids, have also been made in \eqref{E:Viscous_fluxes_Eckart} and \eqref{E:Viscous_fluxes_Landau}. To see this, take, for example, \eqref{E:Viscous_fluxes_BDNK_R}. This choice can be understood by recalling that, in a perfect fluid, the energy density satisfies
\begin{align}
\varrho = u^\mu u^\nu \mathcal{T}_{\mu\nu}
\nonumber
\end{align}
so that $\varrho$ is the energy density measured by an observer locally at rest with respect to the fluid. Postulate \eqref{E:Viscous_fluxes_BDNK_R} says that the same principle holds for a viscous fluid, i.e., an observer locally at rest with respect to the fluid \emph{will not see viscous corrections to the energy density\footnote{Where, as it was commented at the beginning of our discussion of relativistic viscous fluids, it is also assumed that one can neatly distinguish between viscous and non-viscous contributions to the fluid.}.}  The same assumption is made in the DNMR theory, see comments in Assumption \ref{A:DNMR_simplified}. Even if one thinks that this 
hypothesis is simple and natural, it is nonetheless a \emph{modeling choice that does not follow from well-established physical principles}. Observe that when it comes to the pressure, a different reasoning is adopted, wherein a observer locally at rest with the fluid will see the pressure split into two parts, a viscous and a non-viscous contribution 
\begin{align}
p + \mathscr{P} = \frac{1}{3} \proj^{\mu\nu} \mathcal{T}_{\mu\nu},
\nonumber
\end{align}
with $\mathscr{P}$ given by \eqref{E:Viscous_fluxes_Eckart_P} in Eckart's and Landau-Lifshitz's theories and $\mathscr{P}$ satisfying \eqref{E:DNMR_bulk} in the DNMR theory (see \cite[Section 4.2]{Wald:1984rg}
for the above expression for the pressure in a continuum medium).

The \emph{key idea} of the BDNK theory is that one should let the \emph{fundamental principle of causality dictate which terms are allowed in the theory, rather than making what are otherwise ``natural'' modeling choices and only afterwards investigating causality.} In practice, this means starting with the most general energy-momentum tensor and baryon current compatible with the symmetries of a fluid and investigating the characteristics of the corresponding equations of motion. 

More precisely, in practice, at its face value this would lead us to intractable expressions, so one still needs to invoke physical arguments and modeling choices to simplify the energy-momentum tensor. But such choices should be kept to a \emph{bare minimum.}
 For example, to simplify the analysis of the characteristics, one would like to have quasilinear equations that have no derivatives as coefficients in the principal part, leading to viscous fluxes that are linear in first-order derivatives. As another example, one can check that the most general energy-momentum tensor that is linear in first-order derivatives will have more than the six transport coefficients $\uptau_\mathscr{R},\uptau_\mathscr{P},\uptau_\mathscr{Q},\upzeta, \upeta, \upkappa$
in \eqref{E:Viscous_fluxes_BDNK}. For example, the terms $u^\mu \nabla_\mu \varrho$ and $(p+\varrho)\nabla_\mu u^\mu$ in \eqref{E:Viscous_fluxes_BDNK_R} do not need to have, a priori, the same coefficient $\uptau_{\mathscr{R}}$.
Physical arguments are then invoked to establish that some of the transport coefficients are proportional to each other, reducing the total number of transport coefficients to six, see \cite{Hoult:2020eho}.
While one can argue whether assumptions of this form are being faithful to the philosophy of keeping 
modeling choices to a bare minimum, they are substantially less restrictive than setting an entire viscous flux to zero, as in the Eckart and Landau-Lifshitz\footnote{Such a restrictive choice in the DNMR theory is compensated by the introduction of new variables and new equations of motion to model the viscous fluxes.}.

It is apparent at this point that there is a great deal of freedom in the construction of relativistic theories of viscous fluids, as one can postulate the viscous fluxes in different manners. This is a consequence of an intrinsic ambiguity in the description of relativistic dissipative phenomena, namely, \emph{out of thermodynamic equilibrium, there is no first-principles microscopic definition of thermodynamic quantities such as temperature, energy density, and so on.} The quantities that \emph{do have a first-principles microscopic definitions are the energy-momentum tensor and the baryon current.} What one calls temperature, energy density, etc. in  a relativistic viscous fluid are choices of variables that \emph{parametrize} the energy-momentum tensor and the baryon current. Such choices are restricted only by the requirement that they reduce to their counter-parts in thermodynamic equilibrium, when temperature etc. can be defined from first-principles. See \cite{Kovtun:2019hdm}, 
\cite[Chapter 6]{Rezzolla-Zanotti-Book-2013}, \cite[Section II]{Bemfica:2020zjp}, \cite{Romatschke:2017ejr,Denicol-Rischke-Book-2021}, \cite[Section 2.11]{Weinberg:1972kfs} for a detailed discussion.

\begin{remark}
By making a choice for the viscous fluxes and thus paramatrizing the energy-momentum tensor and baryon current, one is effectively providing a definition for thermodynamic quantities out of thermodynamic equilibrium. For historical reasons, each such choice is known as a \textdef{hydrodynamic frame.}
\end{remark}

A more systematic discussion of the underlying physical arguments leading to the BDNK theory is beyond the scope of this mathematical review. We refer the reader to \cite{Kovtun:2019hdm}, \cite[Section II]{Bemfica:2020zjp}, and references therein. 

Theories of relativistic viscous fluids where the viscous fluxes are given in terms of first-order derivatives of\footnote{As usual, different thermodynamic scalars can be used as primary variables.} $\varrho$, $n$, and $u$, being linear in the derivatives, such as the Eckart, Landau-Lifshitz, and BDNK theories, are known as \textdef{first-order theories of relativistic viscous fluids} (see \cite[Chapter 6]{Rezzolla-Zanotti-Book-2013} for a history of the terminology). 

The instability and acausality results of Hiscock and Lindblom, see \cite{Hiscock:1985zz}, and Pichon, see \cite{Pichon-1965}, that we mentioned at the beginning of Sect.~\ref{S:Relativistic_viscous_fluids} as ruling out the Eckart and Landau-Lifshitz theories are applicable, in fact, to large classes of first-order theories. Essentially, theories where $\mathscr{R} = 0$ and a few further mild assumptions on the viscous fluxes hold will be acausal and unstable. But we see from Theorem \ref{T:Causality_BNDK} that causality of the BDNK theory requires $\mathscr{R} \neq 0$ since $\uptau_\mathscr{R} > 0$. Because $\mathscr{R} = 0$ was deemed a ``natural'' assumption (see the above discussion), for a long time it was thought that first-order theories of relativistic fluids are inevitably unstable and acausal.
(We note that Theorem \ref{T:Causality_BNDK} in fact requires $\uptau_\mathscr{T}, \uptau_\mathscr{P}, \uptau_\mathscr{Q} > 0$; the first-order theories treated by Hicock and Lindblom, on the other hand, explicitly exclude the terms in the $\uptau$'s in \eqref{E:BDNK}.)

The first indications that a causal, stable, and locally well-posed theory of relativistic viscous fluids was possible were obtained by the author in \cite{Disconzi-2014-1} and Freist\"{u}hler and Temple in \cite{Freistuhler-Temple-2014}. 

In \cite{Disconzi-2014-1}, the author considered a theory of relativistic viscous fluids introduced by Lichnerowicz in the '50s, see \cite{Lichnerowicz-Book-1955}. Lichnerowicz theory satisfies $\mathscr{R}=0$, but it does not fall under the assumptions of Hicock and Lindblom's or Pichon's results because it does not satisfy some of the further assumptions in their works. In a nutshell, in Lichnerowicz's theory the viscous fluxes are given in terms of derivatives of the enthaply current (see Definition \ref{D:Relativistic_vorticity}) instead of derivatives of the velocity, as assumed in those works.  Causality of the Lichnerowicz theory is established in \cite{Disconzi-2014-1} for the case of an irrotational fluid (when $\Omega$ given in Definition \ref{D:Relativistic_vorticity} vanishes). This result was extended by Czubak and the author \cite{Czubak-Disconzi-2016} to the case when the vorticity satisfies certain constraints. We stress that the results \cite{Disconzi-2014-1,Czubak-Disconzi-2016} are more of a pedagogical nature, serving as motivation to explore first-order theories of relativistic viscous fluids not satisfying the assumptions of Hiscock and Lindblom and Pichon, since\footnote{Although, curiously, Lichnerowicz's theory has found interesting applications in cosmology, see \cite{Disconzi-Kephart-Scherrer-2017,Disconzi-Kephart-Scherrer-2015,Acquaviva:2018rqi,Montani:2016hmf}.} (a) neither the vorticity-free condition in \cite{Disconzi-2014-1} nor the constraints in \cite{Czubak-Disconzi-2016} are propagated by the flow, except for some very special and uninteresting initial data, see \cite{Lichnerowicz-Book-1955}; and (b) more importantly, Lichnerowicz theory is unstable\footnote{This does not seem to be written in the literature, but it follows by an analysis similar to that carried out by Hiscock and Lindblom in \cite{Hiscock:1985zz}.}.

Freist\"{u}hler and Temple's theory introduced in \cite{Freistuhler-Temple-2014} is more akin to the BDNK theory, i.e., it considers viscous fluxes with $\mathscr{R} \neq 0$ and other assumptions not fitting the hypotheses in the works of Hiscock and Lindblom and Pichon. The main goal of \cite{Freistuhler-Temple-2014} was the mathematical study of shocks in relativistic viscous fluids, and thus a complete and and systematic study of causality and stability was not carried out by the authors. Freist\"{u}hler and Temple extended their results in \cite{Freistuhler-Temple-2017,Freistuhler-Temple-2018}, but again focusing primarily on the study of shocks.

Some earlier hints that causality and stability was possible for first-order theories of relativistic viscous fluids were provided by constructions of Choquet-Bruhat \cite[Chapter IX]{Choquet-Bruhat-Book-2009} and Van and Biro \cite{Van:2011yn}.

\emph{The first unambiguous construction of a first-order causal and stable relativistic viscous fluid theory} was obtained by Bemfica, Noronha, and the author in \cite{Bemfica-Disconzi-Noronha-2018}, where the BDNK theory was constructed for the case of a conformal fluid, which is probably the simplest possible case of interest\footnote{Conformal fluids are often used as a toy model for the quark-gluon plasma, see \cite{Romatschke:2017ejr}.}. Coupling to gravity was also considered in this work, but local well-posedness was established only in Gevrey spaces. Local well-posedness in Sobolev spaces was later obtained by Bemfica, Rodriguez, Shao, and the author in
\cite{Bemfica-Disconzi-Rodriguez-Shao-2021}. We note that the PDE techniques in \cite{Bemfica-Disconzi-Rodriguez-Shao-2021} were generalized and refined by Bemfica, Graber, and the author in \cite{Bemfica-Disconzi-Graber-2021}, and this latter reference provides the key PDE techniques 
needed for the proof of local well-posedness in Theorem \ref{T:Causality_BNDK}. See also
\cite{Disconzi-Shao-2023-arxiv} for a simpler and self-contained exposition of these techniques.

Once it was realized that first-order theories are not inevitably acausal and unstable, several works contributed to what culminated in the BDNK theory here presented, see \cite{Bemfica:2020zjp,Bemfica-Disconzi-Noronha-2019-1,Bemfica-Disconzi-Noronha-2018,Kovtun:2019hdm,Hoult:2020eho}. 

The BDNK theory has spurred a revival of interest in first-order theories of relativistic viscous fluids, see
\cite{Das:2020gtq,
Das:2020fnr,
Grozdanov:2018fic,
Romenski:2019qzs,
Gavassino:2020ubn,
Dore:2020jye,
Poovuttikul:2019ckt,
Shokri:2020cxa,
Taghinavaz:2020axp,
Erschfeld:2020blf,
Celora:2020pzs,
Mitra:2021ubx,
Freistuhler-2020,
Biswas:2022hiv,
Salazar:2022yud}, references therein, and the above works on numerical simulations and shocks. This renewed interest includes extensions to relativistic viscous magnetohydrodynamics, see
\cite{Armas:2022wvb,Hattori:2022hyo},
extensions to viscous spin hydrodynamics, see
\cite{Weickgenannt:2023btk,Abboud:2023hos,Xie:2023gbo},
extensions to anisotropic hydrodynamics, see \cite{Bemfica:2023res}, and higher-order fluids\footnote{See \cite{deBrito:2023tgb} for a second-order, DNMR-like, approach to higher-order fluids.}, see \cite{Diles:2023tau},
studies of the non-relativistic limit and post-Newtonian approximations, see
\cite{HegadeKR:2023glb,Ripley:2023qxo}, connections to kinetic theory, see
\cite{Rocha:2023hts,Rocha:2022ind,Garcia-Perciante-Mendez-2023}, cosmology, see \cite{Bemfica:2022dnk}, 
and the fluid-gravity correspondence and Carrollian fluids, see \cite{Ciambelli:2023mvj}. More generally, the discovery of the BDNK theory highlighted how much remains to be understood about relativistic viscous fluids, thus leading to  
renewed efforts in the construction of causal and stable theories of relativistic fluids with viscosity, see, e.g., \cite{Hoang-2024-arxiv,Andersson-Celora-Comer-Hawke-2024} and references therein\footnote{We also note the recent work by \cite{Figueras-Aaron-Kovacs-2024-arxiv}. There, a formalism to construct causal and locally well-posed theories of gravity beyond Einstein is developed. The authors suggest that their methods can be applied to fluids as well, possibly leading to a generalization of the BDNK theory.}.

Interestingly, the BDNK theory has also led to new developments in the study of second-order theories. Once it became clear that a proper choice of hydrodynamic frame is needed for causality and stability of first-order theories, it was not a huge leap to ask how different choices of hydrodynamic frames would affect second-order theories, leading to the so-called second-order theory in a general hydrodynamic frame, see \cite{Noronha:2021syv}. It then follows that the \emph{BDNK theory can be obtained as a particular limit of such a second-order theory in a general hydrodynamic frame,} see \cite{Noronha:2021syv}.

\section{Some further important results in relativistic fluids}
\label{S:Further_important_results}

As said in the Introduction, our goal in this review is to help non-specialists to understand some recent exciting ideas and techniques in the field of relativistic fluids. Our goal is not to provide a comprehensive account of the field. But given that there are many other important developments in relativistic hydrodynamics beyond the ones we have discussed so far, in this section we will provide a glimpse of some of them. 

The works we discuss below fall broadly into two classes. In Sects.~\ref{S:Further_Singularities_Einstein_Euler_symmetry} and \ref{S:Further_connection_self_gravitating}, we discuss results related to the Einstein--Euler system under symmetry assumptions. While a general goal in the mathematical analysis of equations of motion coming from physics is to establish results for ``generic'' data or solutions, in particular away from symmetry, symmetric cases provide a valuable setting where one can understand in detail several aspects of the dynamics without having to deal with many of the complications that arise away from symmetry. The symmetric case also often allows one to pursue questions that are otherwise out of reach (see, for example, the contrast in our understanding of shock solutions in the $1+1$ case versus $2+1$ and $3+1$ cases in Sects.~\ref{S:Shocks_1d} and \ref{S:Shocks_higher_d}). In Sects.~\ref{S:Further_turning_point},  \ref{S:Further_NR_limit}, \ref{S:Further_GWP_expanding}, and \ref{S:Further_big_bang}, we discuss the question of stability of symmetric solutions. (More precisely, in Sect.~\ref{S:Further_NR_limit} we discuss the non-relativistic limit, hence the question of stability in that case is one of stability of the dynamics in certain asymptotic regimes. But the reason to investigate it is similar to that underlying stability of symmetric solutions.)
The problem of stability of symmetric solutions is crucial for many applications in physics. It is often the case that physicists rely on special solutions that are derived under symmetry assumptions. Realistic models, however, are never perfectly symmetric. Thus, one hopes that perturbations away from symmetry remain close to symmetric solutions under the evolution, so that conclusions derived from the symmetric solutions would be good approximations for the true, non-symmetric system. For example, the Friedmann–Lema\^itre–-Robertson–-Walker (FLRW)  
\citep{Friedmann-1922,
Friedmann-1924,
Lemaitre-1931,
Lemaitre-1933,
Robertson-1935,
Robertson-1936-1,
Robertson-1936-2,
Walker-1937} family of solutions describing a homogeneous and isotropic universe is a cornerstone of modern cosmological models (see \citealt{Weinberg-Book-2008} for a discussion of current cosmological models and Sect.~5.1 of \cite{Wald:1984rg} for a precise definition of homogeneity and isotropy). Yet, the universe is only approximately homogeneous and isotropic. Thus, it is important to establish that solutions to the Einstein--Euler system with initial data close to a FLRW solution remain close to that FLRW solution. 

\begin{remark}
It should be apparent from the above discussion that by stability of solutions we mean nonlinear stability, i.e., we are considering solutions of the complete nonlinear system of PDEs for initial data close to some given symmetric initial data. Observe that in the case where the symmetric solution is defined globally in time, this in particular requires establishing a global existence result for the evolution of the nonlinear problem.
\end{remark}

Both themes referred above, the study of symmetric solutions to the Einstein--Euler system and the problem of stability of symmetric solutions, are broad and important topics that deserve each a review of their own, and we could not do justice to these topics in this section. 
Therefore, our presentation is far from exhaustive and, unfortunately, omits many other interesting results. Moreover, we will restrict ourselves to summarize results, without discussing their proofs.

\subsection{The turning point principle for relativistic stars}
\label{S:Further_turning_point}

In spherical symmetry and in the static case, the Einstein--Euler system simplifies considerably. The resulting system, known as the Tolman--Oppenheimer--Volkoff (TOV) equations \citep{Oppenheimer:1939ne,Tolman:1934za,Tolman:1939jz}, has been used to describe stars in hydrostatic equilibrium, as follows. One considers a spherically symmetric static star of radius $R$ governed by Einstein--Euler system. The region exterior to the star is described by the vacuum Einstein equations and, being spherically symmetric, is given by the Schwarzschild metric in view of Birkhoff's theorem \citep{Birkhoff-Book-1923,Birkhoff-Book-reprint-2013,Veblen-1924}. On imposes that on the boundary of the star the fluid's pressure vanishes (see Eq.~\eqref{E:Bry_condition_pressure} and surrounding discussion) and the metric is continuous (thus the interior metric matches the outside Schwarzschild metric). This is the static version of the free-boundary Einstein--Euler system that was discussed in a fixed background but without symmetry assumptions in Sect.~\ref{S:Vacuum_bry} (see also Problem \ref{Pr:LWP_Einstein_Euler_vacuum_bry} in Sect.~\ref{S:Open_problems}).

Under reasonable assumptions on the equation of state, an analysis of the TOV equations reveals that there are relations constraining the allowed values of a star radius for a given mass. This is the well-known \emph{mass-radius diagram} illustrated in Fig.~\ref{F:Mass-radius}. 

\begin{figure}[ht]
\centering
  \includegraphics[scale=0.5]{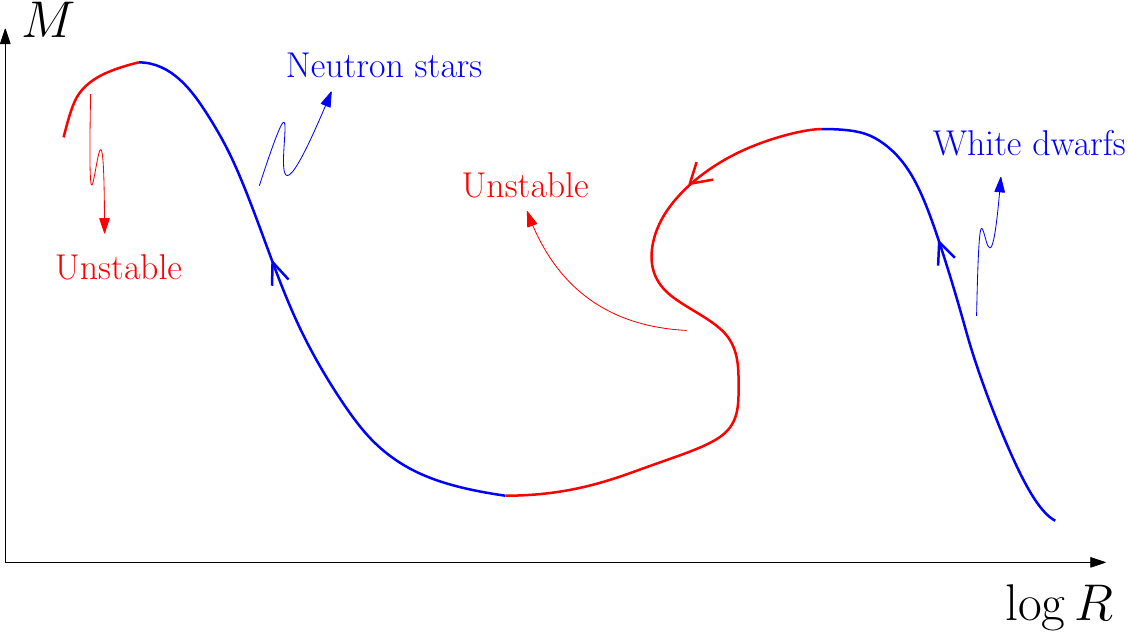}
  \caption{Illustration of the mass-radius diagram. The arrows indicate directions of increasing density at the center of the star. The red curves represent configurations believed to be unstable.}
  \label{F:Mass-radius}
\end{figure}

The details of the mass-radius diagram depend on many assumptions about the nature of the star, see Sect.~6.2 of \cite{Wald:1984rg}. Despite such assumptions and the static spherically symmetric regime, it is believed that the mass-radius diagram captures general qualitative features of \emph{realistic} stars. For example, some estimates show that deviations from spherical symmetry caused by magnetic fields are small for typical field strengths (e.g., excluding magnetars) and more general estimates in axial symmetry lead to a qualitatively similar mass-radius diagram (wherein one needs to define what is meant by the ``radius'' of a non-spherically symmetric star), see \cite{Ozel:2016oaf}; Chapter~14 of \cite{Baumgarte:2010ndz} and references therein for more background and discussion.

A rigorous mathematical justification of the mass-radius diagram, i.e., a proof that its qualitative features are present in solutions of the Einstein--Euler system away from symmetry and under assumptions appropriate for a realistic description of star evolution, is currently out of reach. Simply posing the question would already  require understanding the dynamics of the Einstein--Euler system with a free boundary, a problem that in itself is very challenging (see Problem \ref{Pr:LWP_Einstein_Euler_vacuum_bry} in Sect.~\ref{S:Open_problems}). Specific aspects of the problem, on the other hand, are more amenable to a mathematical treatment. 

One clear interesting feature of the mass-radius diagram is the change from stable to unstable behavior indicated by points of local maximum and minimum of the curve in Fig.~\ref{F:Mass-radius}. In order to study this phenomena, \cite{Zeldovich-Podurets-1966} considered spectral properties of the spherically symmetric Einstein--Euler system linearized about solutions to the TOV equations parametrized by values of the central density, i.e., the density at the center of the star. They conjectured that the linearized system's spectral stability properties, including the number of corresponding growing modes for the system, could only change at extrema of the mass function. This conjecture is known as the \emph{turning point principle.}

In what seems to be the first rigorous result toward the problem of providing a full mathematical justification that the mass-radius curve captures general qualitative features of the Einstein--Euler system describing realistic stars, \cite{Hadzic-Lin-2021} have recently proven the turning point principle for equations of state of the form $p(\varrho) = K \varrho^\gamma + O(\varrho^{\gamma+1})$, where $K$ and $\gamma$ are constants within some specified range. We refer the reader to \cite{Hadzic-Lin-2021} for a precise formulation of the turning point principle and its proof.

\subsection{Singularities and shocks in the Einstein--Euler system under symmetry assumptions\label{S:Further_Singularities_Einstein_Euler_symmetry}}

The problem of shock formation for the relativistic Euler equations in Minkowski space remains open, see Sect.~\ref{S:Some_context_shocks} and Problem \ref{Pr:Shock_relativistic_Euler} in Sect.~\ref{S:Open_problems}. Such a problem is naturally much more challenging upon coupling to Einstein's equations. However, under symmetry assumptions, important results have been obtained (compare with the discussion in Sect.~\ref{S:Shocks_1d}).

\cite{Rendall-Stahl-2008} considered the Einstein--Euler system in plane symmetry. They proved that for large classes of smooth initial data and a large number of equations of state, solutions break down in finite time. Recalling the connection between breakdown and distributional or weak solutions in Sect.~\ref{S:Study_of_shocks}, weak solutions to the Einstein--Euler system on $\mathbb{T}^3$ with Gowdy symmetry have been constructed by \cite{LeFloch-Rendall-2011}. These solutions admit shock waves and the so-called impulsive gravitational waves, which are solutions with Dirac-delta curvature
singularities. These results have been extended by \cite{Grubic-LeFloch-2015}.
Along similar lines, for the Einstein--Euler system in spherical symmetry, \cite{Burtscher-LeFloch-2014}
constructed a large class of spherically symmetric BV solutions for the Einstein--Euler system (compare with Sect.~\ref{S:Shocks_1d}). They also proved that these solutions exhibit formation of trapped surfaces.

\subsection{The non-relativistic limit\label{S:Further_NR_limit}}
A very natural question is whether solutions to the relativistic Euler equations converge to those of the classical compressible Euler equations when one takes the non-relativistic limit. We can similarly ask whether solutions to the Einstein--Euler system converge to those of the Euler--Poisson system in the non-relativistic limit. Formally, it is not difficult to see that these limits hold, see Sect.~3.5 of \cite{Rezzolla-Zanotti-Book-2013} and Sects.~4.2, 4.4 of \cite{Wald:1984rg}. A rigorous mathematical proof, on the other hand, presents the usual difficulties one encounters when studying singular limits. This is a problem of clear interest for applications since the Euler--Poisson system is often used as an approximation to the Einstein--Euler system in regimes where relativistic corrections are expected to be small, see \cite{Buchert:2011sx,Saari-1971,Marchal-Saari-1976,Gibbons-Patricot-2003,Saari-2005,Ellis-Gibbons-2014} and references therein. Thus, it is important to justify that solutions to the Einstein--Euler system can be approximated by solutions of the Euler--Poisson equations in a controlled fashion.

The non-relativistic limit discussed above was rigorously established for large classes of solutions by \cite{Oliynyk-2010-1,Oliynyk-2010-2,Oliynyk-2012-2} and by \cite{Liu-Oliynyk-2018-1, Liu-Oliynyk-2018-2, Oliynyk-2015} in a cosmological setting.

\subsection{Global future stable solutions in an expanding spacetime\label{S:Further_GWP_expanding}}
We have already highlighted above the importance of understanding stability properties of the FLRW family of solutions to the Einstein--Euler system. In this section we report on results on the (nonlinear) future stability of this family. To understand the meaning of \emph{future} stability, it is important to highlight that the construction of FLRW solutions employs a conveniently chosen time coordinate $t$ in which the solution exhibits a singularity\footnote{This is a true curvature singularity, not an artifact of the choice of coordinates.} when $t \rightarrow -1^+$ but remains well-defined and regular when $t \rightarrow \infty$ (the choice of $t=-1$ as the ``time of the singularity'' is just a choice of normalization and any other $t$ values can be chosen), see \cite[Section 5.2]{Wald:1984rg} for details. By future stability we meant that for initial data sufficiently close to initial data for a FLRW solution, solutions to the Einstein--Euler system exist for all $t>0$ and remain close to the corresponding FLRW solution. Here, $t$ is a suitable time function related to the aforementioned time coordinate used in the FLRW solution and we normalize it to have the data given at $t=0$. See the references below for precise statements.

The question of future stability of the FLRW family has been addressed to satisfaction in the case of a linear equation of state, i.e., $p(\varrho) = c_s^2 \varrho$ with constant $c_s$, and with a positive cosmological constant in Einstein's equations.
The latter is the case of primary interest for applications according to current cosmological models in view of the accelerated expansion of the universe, see \cite{Weinberg-Book-2008}. The results, always under the positive cosmological constant assumption, are the following.  \cite{Rodnianski-Speck-2013} established the result for an irrotational fluid and when $0<c_s^2 <\frac{1}{3}$. Speck removed the irrotational assumption in \cite{Speck-2012}. The cases $c_s^2=0$ and $c_s^2 = \frac{1}{3}$ have been treated by  \cite{Hadzic-Speck-2015} and \cite{Lubbe:2011kz}, respectively.
Interestingly, with heuristic arguments and numerical simulations, \cite{Rendall-2004} and \cite{Beyer-Marshall-Oliynyk-2023}, respectively, found evidence that stability might fail in the case $\frac{1}{3}< c_s^2 < 1$.

We note that these works are in part inspired by an earlier result of \cite{Speck-2013} on small-data global stability of the relativistic Euler equations on a background spacetime that undergoes accelerated expansion and in the case $0 \leq c_s^2 \leq \frac{1}{3}$. \cite{Oliynyk-2021} obtained a similar result for $\frac{1}{3}< c_s^2 < \frac{1}{2}$, later improved by \cite{Marshall-Oliynyk-2023} to the range $\frac{1}{3}< c_s^2 < 1$.
More recently, Speck's result has been improved by \cite{Fajman-Oliynyk-Wyatt-2021}, who considered a background whose expansion is not accelerated.

\subsection{Stability of big-bang singularities\label{S:Further_big_bang}}
In Sect.~\ref{S:Further_GWP_expanding} we discussed the future stability of the FLRW family of solutions. Here we comment on results regarding their \emph{past} stability, thus in particular stability of the big-bang singularity. Without getting into details, we should remark that this is a rather delicate matter in that one needs to make sense of closeness for solutions that are blowing up.

A satisfactory picture of stability of the big-bang singularity has been obtained in a series of works by 
\cite{Rodnianski-Speck-2018-1,
Rodnianski-Speck-2018-2,
Fournodavlos-Rodnianski-Speck-2023}. The authors established stability of the big-bang singularity in the case of an irrotational stiff fluid, i.e., $c_s^2 =1$. We note that the result is in fact more general: they establish stability of singularities arising from a larger family of solutions than FLRW, the so-called Kasner family. See the previous references for details and precise statements, and also \cite{Speck-2018-2} for a related result.

\subsection{Works in connection with self-gravitating fluids\label{S:Further_connection_self_gravitating}}
Understanding the free-boundary Einstein--Euler system is an important open problem in the mathematical theory of relativistic stars. See Sect.~\ref{S:Open_problems} and Problem \ref{Pr:LWP_Einstein_Euler_vacuum_bry}. 
Important results have been obtained in the static case going back to the 1980s, see the references 
\cite{Rendall-Schmidt-1991,Lindblom-1988,Makino-1992,Makino-1998,Makino-2016,Makino-2017,Makino-2018,Makino-2019-2-arxiv,Makino-2019-arxiv} and 
section 1.6 of \cite{Disconzi-Ifrim-Tataru-2022} for a more detailed discussion; see also the related discussion in Sect.~\ref{S:Further_turning_point}. 

The case of a star wherein one considers matter confined to a compact spatial region is only one example of a self-gravitating fluid. One can also consider the situation when matter extends in space, decaying to zero at infinity. In this case, it is of interest to provide a description of the system's asymptotic behavior. This has been recently done in \cite{Andersson-Burtscher-2019} for certain power-law equations of state (including the linear case).

In the study of isolated bodies such as stars, it makes sense to consider a zero cosmological constant in that effects from the universe's accelerated expansion are negligible at such local scales. Yet, a more complete description should account for the accelerated expansion. This point of view has been recently investigated by \cite{Beheshti-Kroon-2022}, who studied global-in-time properties a system of self-gravitating balls of dust in an expanding Universe. We refer the reader to their paper for precise statements and background discussion.

\section{Open problems and future directions of research\label{S:Open_problems}}

We conclude this review listing some open problems that follow from or are directly connected to the results we have presented. As mentioned in the Introduction, our list of open problems will vary from precise conjectures to tantalizing speculations. We list the problems in an order that more or less mirrors our presentation in the text. It goes without saying that all problems listed below are \emph{conjectures,} so, when we state these problem in the form ``establish P...'' we are relying on an educated guess that P is true.

\begin{innercustomthmproblem}
Solve the problem of shock formation for the relativistic Euler equations.
\label{Pr:Shock_relativistic_Euler}
\end{innercustomthmproblem}

Problem \ref{Pr:Shock_relativistic_Euler} asks for establishing a constructive proof of stable shock formation without symmetry assumptions in more than one spatial dimension, as described in Sect.~\ref{S:Study_of_shocks}, for the relativistic Euler equations. More precisely, we would like to establish for the relativistic Euler equations a result like Theorems 4.2 and 4.3 of \cite{Luk-Speck-2024}. 

In Sect.~\ref{S:Some_context_shocks}, we discussed how some of the key tools needed for a resolution of Problem \ref{Pr:Shock_relativistic_Euler} are available. We also note that a detailed road map for solving Problem \ref{Pr:Shock_relativistic_Euler} is given in \cite{Abbrescia-Speck-2023}. After a resolution of Problem \ref{Pr:Shock_relativistic_Euler}, we can ask for:

\begin{innercustomthmproblem}
Give a description of the maximal development of the initial data for shock-forming solutions of the relativistic Euler equations.
\label{Pr:Maximal_development_Euler}
\end{innercustomthmproblem}

In problem \ref{Pr:Maximal_development_Euler}, which is significantly more challenging than \ref{Pr:Shock_relativistic_Euler}, one is asking for a generalization to the relativistic setting of the result \cite{Abbrescia-Speck-2022-arxiv,Shkoller-Vicol-2024} obtained for the classical compressible Euler equations, see Sect.~\ref{S:Shocks_higher_d}.

Problems \ref{Pr:Shock_relativistic_Euler} and \ref{Pr:Maximal_development_Euler} deal with the relativistic Euler equations in a fixed background. A much more interesting, albeit much more challenging, problem is:

\begin{innercustomthmproblem}
Solve the problem of shock formation for the Einstein--Euler system.
\label{Pr:Shock_Einstein_Euler}
\end{innercustomthmproblem}

We recall that by the problem of shock formation we are in particular referring to solutions away from symmetry (see Sect.~\ref{S:Study_of_shocks}). Shock solutions for the Einstein--Euler system in symmetry classes have been studied in \cite{Rendall-Stahl-2008,Groah-Smoller-Temple-Book-2007}.

It is likely that genuinely new ideas will be needed to address Problem \ref{Pr:Shock_Einstein_Euler}. As we reviewed in Sect.~\ref{S:Study_of_shocks}, the techniques currently available for the study of shock formation are \emph{adapted to the fluid's characteristics,} more precisely to the sound cones, whose collapse characterizes the onset of shock formation. But the sound cones are timelike with respect to the spacetime metric, and one cannot derive energy estimates along timelike surfaces\footnote{Unacquainted readers can see the comments on the energy formalism in Sects.~\ref{S:Detour_Minkowski} and \ref{S:Estimates_curl_Omega_along_sound_cones}.}. Thus, proper control of the gravity part cannot be obtained by known methods. This difficulty is a manifestation of the multi-characteristics nature of the problem. It is interesting to notice, however, that if the sound cones were spacelike with respect to the spacetime metric, i.e., if the fluid traveled at speeds greater than the speed of light, then current techniques of shock formation would be applicable, see \cite{Speck-2018-1}.

\begin{innercustomthmproblem}
Establish low-regularity solutions for the relativistic Euler equations with no H\"older assumption on the data and 
$(\hat{h}, u, \curl u) \in H^{2+\varepsilon}(\Sigma_0)\times H^{2+\varepsilon}(\Sigma_0)\times H^{2+{\varepsilon^\prime}}(\Sigma_0)$, $0< \varepsilon^\prime < \varepsilon$.
\label{Pr:Rough_solutions}
\end{innercustomthmproblem}

Recall from the discussion after Theorem \ref{T:Rough_classical_Euler} and from Sect.~\ref{S:Lowering_regularity_classical_Euler} that \cite{Wang-2022} and \cite{Zhang-Andersson-arxiv-2022}, established 
a version of Problem \ref{Pr:Rough_solutions} for the classical compressible Euler equations. Thus, Problem \ref{Pr:Rough_solutions} asks about the generalization of Wang's and Andersson and Zhang's results to the relativistic setting, improving Theorem \ref{T:Rough_relativistic_Euler}. The works by \cite{Wang-2022,Zhang-Andersson-arxiv-2022} involve some very \emph{delicate} estimates that rely on \emph{special structures} of the (new formulation\footnote{Recall Sect.~\ref{S:New_formulation_classical_schematic}.} of the) classical compressible Euler equations (see Sect.~\ref{S:Lowering_regularity_classical_Euler}), thus it is not clear how easy it would be to generalize them to the relativistic setting.

\begin{innercustomthmproblem}
Establish local well-posedness for the relativistic Euler equations with a physical vacuum boundary and a general equation of state.
\label{Pr:LWP_vacuum_bry_general}
\end{innercustomthmproblem}

Problem \ref{Pr:LWP_vacuum_bry_general} generalizes Theorem \ref{T:LWP_vacuum_bry} for an equation of state $p = p(\varrho,n)$ (or $p=p(\varrho, s)$, etc.; as usual, one can choose different thermodynamic scalars as primary variables), in which case Eq.~\eqref{E:Relativistic_Euler_eq_full_system_baryon_charge} is also included in the system. It is implicitly assumed that the equation of state is compatible with the behavior of a gas\footnote{For example, not every equation of state will allow $p$ and $\varrho$ to vanish simultaneously on the boundary.}. Note that upon the introduction of new variables such as $n$ or $s$, further decay conditions compatible with a finite non-zero acceleration of the boundary have to be imposed, similarly to what was done for $c_s^2$.

For the equation of state of an ideal gas
\begin{align}
\label{E:Ideal_gas_EoS}
p(\varrho,n) = (\varrho - n)( \gamma - 1),
\end{align}
where $\gamma >1$ is a constant,
a priori estimates for the linearized relativistic Euler equations with a physical vacuum boundary have been recently obtained by \cite{Luczak-in_prep}. In view of the role played by the analysis of the linearized problem in the proof of Theorem \ref{T:LWP_vacuum_bry}, these a priori estimates constitute a key step in a resolution of problem \ref{Pr:LWP_vacuum_bry_general} for the case of the equation of state
\eqref{E:Ideal_gas_EoS}. We note that fluids with the equation of state \eqref{E:Ideal_gas_EoS} are widely in study in physics \cite[Sect.~2.4 of][]{Rezzolla-Zanotti-Book-2013}, hence a resolution of Problem \ref{Pr:LWP_vacuum_bry_general} in this case is important. But naturally we 
would like a resolution of Problem \ref{Pr:LWP_vacuum_bry_general} for a general equation of state.

\begin{innercustomthmproblem}
Establish local well-posedness for the Einstein--Euler system with a physical vacuum boundary.
\label{Pr:LWP_Einstein_Euler_vacuum_bry}
\end{innercustomthmproblem}

In Problem \ref{Pr:LWP_Einstein_Euler_vacuum_bry}, one is considering a more realistic model of a general-relativistic star, wherein we solve the Einstein--Euler system inside the fluid region $\{ \varrho > 0 \}$ and the vacuum Einstein equations outside\footnote{For readers not versed in relativity, we stress that vacuum means a solution to Einstein's equations where the energy-momentum tensor on the RHS vanishes. This does not imply, however, that ``vacuum'' should be understood as ``trivial,'' i.e., solutions with no dynamics or uninteresting mathematically.}. Thus, differently then the case of a fixed background, we do need to consider the exterior of the fluid region in a non-trivial manner. In this case, Einstein's equations in wave coordinates will read
\begin{align}
\label{E:Einstein_Euler_wave_coordinates_free_bry}
g^{\mu\nu} \partial_\mu \partial_\nu g_{\alpha\beta} + \mathcal{N}(g,\partial g) 
= \chi_{\Md_t}( \mathcal{T}_{\alpha\beta} -\frac{1}{2} g^{\mu\nu} \mathcal{T}_{\mu\nu} g_{\alpha\beta} )
\end{align}
where $\chi_{\Md_t}$ is the characteristic function of the moving domain at time $t$, $\mathcal{T}$ is the energy-momentum tensor for a perfect fluid, Eq.~\eqref{E:Energy_momentum_perfect}, and 
$\mathcal{N}(g,\partial g)$ is the standard quadratic nonlinearity for Einstein's equations in wave coordinates (a cosmological constant is not included for simplicity). Equation \eqref{E:Einstein_Euler_wave_coordinates_free_bry} expresses that one solves Einstein's equations coupled to the fluid within the fluid region $\Md_t$ and vacuum Einstein outside.

Compared to the case of a fixed background, Problem \ref{Pr:LWP_Einstein_Euler_vacuum_bry} presents many new challenging features that will likely require genuinely new ideas for its solution. One difficulty is conceptually similar to that of Problem \ref{Pr:Shock_Einstein_Euler}, namely, the fluid's free boundary is timelike with respect to the spacetime metric, and no estimates along timelike manifolds are generally available. Note, in this respect, that there are \emph{no boundary conditions} for the metric along $\Md_t$. Another difficulty is that, as we saw in Sect.~\ref{S:Vacuum_bry}, we need to derive weighted estimates for the fluid variables. Estimates for the metric, consequently, need to be compatible with the weights. But it is not clear that this will be the case since the weights are well-adapted to the fluid part, and not to the gravity part, of system \eqref{E:Einstein_Euler_wave_coordinates_free_bry}. The origin of these difficulties is the fact that \eqref{E:Einstein_Euler_wave_coordinates_free_bry} forms a system with multiple characteristic speeds.

Another challenge is that problem \eqref{E:Einstein_Euler_wave_coordinates_free_bry} is \emph{highly singular.} We need to control the metric across full spacetime slices $\Sigma_t$ and not only across slices $\Md_t$ contained within the fluid region.
In order to derive estimates for $g$, one needs to differentiate \eqref{E:Einstein_Euler_wave_coordinates_free_bry} and immediately face the problem of taking derivatives of the characteristic function. More precisely, the RHS of \eqref{E:Einstein_Euler_wave_coordinates_free_bry} is not discontinuous in view of the fact that $\varrho$ and $p$ vanish on the boundary $\Gamma_t = \partial \Md_t$. For example, for a quadratic equation of state $p(\varrho) = \varrho^2$, we have that $\varrho(t,x) \sim c_s^2(t,x) \approx \operatorname{dist}(x,\Gamma_t)$, so that the RHS of \eqref{E:Einstein_Euler_wave_coordinates_free_bry} is in this case Lipschitz, but no better. Thus, one could differentiate \eqref{E:Einstein_Euler_wave_coordinates_free_bry} once, but no more than this, which is far from enough to close the problem\footnote{Consider, for illustration, the problem of deriving estimates for a quasilinear wave equation with a Lipschitz source term.}.

A natural idea to overcome this difficulty is to commute \eqref{E:Einstein_Euler_wave_coordinates_free_bry} with derivatives tangent to $\Gamma_t$. But using only tangent \emph{spatial} derivatives is not enough in that tangent spatial derivatives will not produce control over derivatives transverse to $\Gamma_t$. Thus, we also need to commute the equation with material derivatives $D_t$. This would in fact be compatible with estimates for the fluid part, since this is roughly\footnote{Recall Remark \ref{R:Good_variables_material_derivatives}.} what we have done in Sect.~\ref{S:Vacuum_bry}. But the following problem arises. In order to estimate $D_t^N g$, we need to know the initial values 
$\left. D_t^N g \right|_{\Sigma_0}$. These must be obtained in terms of the prescribed initial data
$\left. (g,\partial_t g) \right|_{\Sigma_0}$ by using \eqref{E:Einstein_Euler_wave_coordinates_free_bry} to successively solve for $D_t g$ in terms of spatial derivatives. In doing so, we quickly see that, in view of the purely spatial $g^{ij}\partial_i \partial_j g_{\alpha\beta}$, that we will have to differentiate the characteristic function (in non-tangential directions) as soon as we try to compute $D_t^4 g$ at $t=0$.

One possible approach to deal with the problem of obtaining data $\left. D_t^N g \right|_{\Sigma_0}$ was proposed in \cite{Andersson-Oliynyk-Schmidt-2016} and also employed in \cite{Miao-Shahshahani-2024}. In this approach, one fine tunes the data $\left. (g,\partial_t g) \right|_{\Sigma_0}$ to precisely kill the bad terms that would arise from differentiation of the characteristic function. Data of this form was said to satisfy \emph{compatibility conditions} in \cite{Andersson-Oliynyk-Schmidt-2016}. The drawback of this approach is that while one can certainly choose compatible initial data for a free wave equation, for Einstein's equations the initial data must satisfy the constraints, and it is \emph{not clear that there exists data satisfying the constraint equations and the compatibility conditions} (except for the trivial case of static solutions). 

(Existence of such data was not proven in \cite{Andersson-Oliynyk-Schmidt-2016,Miao-Shahshahani-2024}. We note that these works do not deal with the Einstein--Euler system with a physical vacuum boundary, dealing instead with the free-boundary Einstein--Elasticity system in \cite{Andersson-Oliynyk-Schmidt-2016} and the free-boundary Einstein--Euler system for a stiff liquid in \cite{Miao-Shahshahani-2024}, but in both cases the model involves a RHS with a characteristic function. We also note that in the case where the sound speed decays faster than the distance to the boundary, then the RHS of \eqref{E:Einstein_Euler_wave_coordinates_free_bry} is more regular, and the problem can then be treated, see \cite{Rendall-1992}.)

Perhaps it is the case that such compatibility conditions are necessary for the local well-posedness of the equations, and proving so would be an important step toward a resolution of Problem \ref{Pr:LWP_Einstein_Euler_vacuum_bry}. On the other hand, absent such a proof and looking at the problem from a different perspective, even if one is able to construct compatible data, it appears such data is not general enough. From the foregoing discussion, what the presence of the characteristic function on the RHS of \eqref{E:Einstein_Euler_wave_coordinates_free_bry} is suggesting is that the metric cannot be very regular, so that some derivatives of the metric will necessarily be singular. This singularity will then \emph{propagate} along the characteristics of the metric, i.e., along the lightcones, \emph{into the interior of $\Md_t$.} In other words, singularities of $g$ \emph{will not, in general, stay localized\footnote{Incidentally, this shows how singular problem \eqref{E:Einstein_Euler_wave_coordinates_free_bry} is.}.} The entire artifact of compatibility conditions, on the other hand, is to allow the construction of data where the singularity is tame, so that it is not allowed to propagate to the interior of the moving domain. But a simple exercise with D'Alembert's and Duhamel's formula for the linear wave equation in $1+1$ dimensions with a characteristic function as source term explicitly shows that singularities will propagate along the characteristics. Thus, solutions arising from compatible initial data do not seem to capture the general behavior of the system \eqref{E:Einstein_Euler_wave_coordinates_free_bry}. Note that, for the fluid variables, a similar problem does not occur. Any potential singularity in the fluid stays localized to the boundary since the sound cones degenerate to the flow lines on the boundary.

The above informal comments are only a few of the difficulties present in the study of Problem \ref{Pr:LWP_Einstein_Euler_vacuum_bry}. Such difficulties nonetheless, it is clear that Problem \ref{Pr:LWP_Einstein_Euler_vacuum_bry} is mathematically very rich, aside from being of clear physical importance for the description of general-relativistic stars\footnote{Since numerical simulations of stars are routinely performed by astrophysicists, one might wonder if the difficulties we discussed are present in numerical simulations. The short answer is yes. Current simulation will not in fact allow for the density to vanish outside the star, introducing instead a small ``artificial atmosphere,'' see \cite{Wu-Tang-2017}, as simulations would not run otherwise. It is interesting to speculate if a mathematical framework that leads to a solution of Problem \ref{Pr:LWP_Einstein_Euler_vacuum_bry} could also lead to better numerical algorithms that would dispense with an artificial atmosphere.}.

In some problems that follow, we will refer to the DNMR equations with only bulk viscosity. We recall
from Sect.~\ref{S:LWP_DNMR}, Remark \ref{R:Generalization_LWP_DNMR}, that this refers to setting $\pi=0$ in Eqs.~\eqref{E:DNMR_density}, \eqref{E:DNMR_velocity}, \eqref{E:DNMR_bulk}, and dropping Eq.~\eqref{E:DNMR_shear}. The reason for considering this case in particular is that we will list problems 
whose investigation would be premature absent local well-posedness in Sobolev spaces.
We recall from Sect.~\ref{S:LWP_DNMR} that the DNMR equations are known to be locally well-posed in Gevrey spaces only, but local well-posedness in Sobolev spaces holds when bulk viscosity is the only viscous flux. However, 
if local well posedness of Eqs.~\eqref{E:DNMR} in Sobolev spaces is demonstrated, then 
the problems stated below for the DNMR theory with only bulk viscosity can also be addressed for the full system \eqref{E:DNMR}. Alternatively, one can also consider such problems in Gevrey spaces, but this would be of limited mathematical and physical interest.

\begin{innercustomthmproblem}
Establish local well-posedness for the DNMR equations with a physical vacuum boundary.
\label{Pr:LWP_DNMR_vacuum_bry}
\end{innercustomthmproblem}

Along the same lines:

\begin{innercustomthmproblem}
Establish local well-posedness for the BDNK equations with a physical vacuum boundary.
\label{Pr:LWP_BDNK_vacuum_bry}
\end{innercustomthmproblem}

Since free-boundary problems provide basic models of isolate starts, Problems \ref{Pr:LWP_DNMR_vacuum_bry} and \ref{Pr:LWP_BDNK_vacuum_bry} are important in view of the interest in viscous effects on mergers of neutron stars (see Sect.~\ref{S:Relativistic_viscous_fluids}). More precisely, in order to develop a mathematical theory of such mergers and their corresponding dissipative phenomena, significantly stronger results would be needed, including local well-posedness of the DNMR and BDNK equations with a physical vacuum boundary coupled to Einstein's equations. But as seen in the discussion of Problem \ref{Pr:LWP_Einstein_Euler_vacuum_bry}, already in the case of a perfect fluid severe difficulties appear, so it is not clear if the case with viscosity, for either DNMR or BDNK, is within reach.

According to the above comments, in Problem \ref{Pr:LWP_DNMR_vacuum_bry}, it makes sense to consider the case with bulk viscosity only, as this is the only case where Eqs.~\eqref{E:DNMR} are known to be locally well-posed in Sobolev spaces. A priori estimates for the linearized\footnote{We stress that these are the equations linearized about an arbitrary solution, and not only about global thermodynamic equilibrium as in the discussion of stability in Sect.~\ref{S:DNMR}.} DNMR equations with only bulk viscosity and with a physical vacuum boundary have been recently obtained by \cite{Zhong-in_prep} under suitable assumptions on the decay rates for $\upzeta$, $\uptau_\mathscr{P}$, and $\mathscr{P}$ near the free-boundary and an equation of state $p(\varrho) = \varrho^{\kappa+1}$ as in Sect.~\ref{S:Vacuum_bry}. These estimates open the door for a solution to problem \ref{Pr:LWP_DNMR_vacuum_bry} under these assumptions. It is also natural to consider a general equations of state and other decay assumptions for $\upzeta$, $\uptau_\mathscr{P}$, and $\mathscr{P}$.

Regarding Problem \ref{Pr:LWP_BDNK_vacuum_bry}, the first step would be define what a physical vacuum boundary is, i.e., what conditions are needed to guarantee that the boundary can accelerate with a finite non-zero acceleration. This is not entirely clear given that the BDNK equations are second-order, so the  argument used in Sect.~\ref{S:Vacuum_bry} that leads to \eqref{E:Physical_vacuum_bry_condition} needs to be modified.

Problems \ref{Pr:LWP_vacuum_bry_general}, \ref{Pr:LWP_Einstein_Euler_vacuum_bry}, \ref{Pr:LWP_DNMR_vacuum_bry}, and \ref{Pr:LWP_BDNK_vacuum_bry} consider local-in-time questions. A natural global-in-time problem for relativistic fluids with a physical vacuum boundary is the following.

\begin{innercustomthmproblem}
Establish global well-posedness for the relativistic Euler equations with a physical vacuum boundary in the case of an expanding gas.
\label{Pr:GWP_vacuum_bry_expanding}
\end{innercustomthmproblem}

By an expanding gas, we mean a configuration where the initially compactly supported fluid is expanding outward. If the initial rate of expansion is sufficiently large, one can intuitively expect that the fluid will continue to expand indefinitely and no singularity will form. For the classical compressible Euler equations with a physical vacuum boundary, this was proven to be indeed the case, first by \cite{Sideris-2017} under symmetry assumptions, wherein the equations reduce to ODEs, and subsequently by
\cite{Hadzic-Jang-2018} without symmetry assumptions but for initial data close
to the symmetric solutions found by Sideris. This result was later extended in various directions by 
several authors \citep{Hadzic-Jang-2019,Shkoller-Sideris-2019,Rickard-2021-1,Rickard-2021-2,Parmeshwar-Hadzic-Jang-2021,Rickard-Hadzic-Jang-2019}. In view of these results in the classical setting and with local well-posedness of the relativistic Euler equations with a physical vacuum boundary established via Theorem \ref{T:LWP_vacuum_bry}, Problem \ref{Pr:GWP_vacuum_bry_expanding} is a natural question. The first task would be to find a relativistic analogue of Sideris' ODE solutions.

\begin{innercustomthmproblem}
Establish local well-posedness of the DNMR equations in Sobolev spaces.
\label{Pr:LWP_DNMR_Sobolev}
\end{innercustomthmproblem}

Given the importance of the DNMR theory for applications (see Sect.~\ref{S:DNMR}), problem \ref{Pr:LWP_DNMR_Sobolev} is one of the most pressing mathematical problems in relativistic viscous fluids. In view of the complexity of Eqs.~\eqref{E:DNMR}, it is unclear how laborious Problem \ref{Pr:LWP_DNMR_Sobolev} is.

We recall that Eqs.~\eqref{E:DNMR} are not the most general form of the DNMR equations. For the full DNMR system, which are Eqs.~(63)--(67) of \cite{Denicol:2012cn}, we can thus ask

\begin{innercustomthmproblem}
Establish causality for the full set of DNMR equations.
\label{Pr:Causality_full_DNMR}
\end{innercustomthmproblem}

In comparing to Eqs.~\eqref{E:DNMR}, we  warn that \cite{Denicol:2012cn} uses the $+---$ convention for the spacetime metric. 

\begin{innercustomthmproblem}
Establish the nature of the breakdown in Theorem \ref{T:Breakdown_DNMR}.
\label{Pr:Nature_breakdown_DNMR}
\end{innercustomthmproblem}

This is a natural question, given that Theorem \ref{T:Breakdown_DNMR} does not provide any information about the nature of the breakdown. In particular, it is interesting to investigate whether the breakdown could be due to the formation of a shock; a initial investigation should first consider the $1+1$-dimensional case. More generally, recall that for the relativistic Euler equations, the appropriate framework for the study of shock formation in more than one spatial dimension involves the new formulation of Theorem \ref{T:New_formulation}. Hence, one might imagine that something similar could be required for the study of shocks in the DNMR equations:

\begin{innercustomthmproblem}
Seek a new formulation of the DNMR equations akin to the new formulation of the relativistic Euler equations in Theorem \ref{T:New_formulation} for the case when only bulk viscosity is present.
\label{Pr:New_formulation_DNMR}
\end{innercustomthmproblem}

The intuition behind Problem \ref{Pr:New_formulation_DNMR} is the following. When bulk viscosity is the only viscous flux in Eqs.~\eqref{E:DNMR}, the equations possess only the flow lines and the sound cones as characteristics. Moreover, the sound speed is now given by\footnote{See comments on the last bulltet point of Remark \ref{R:Generalization_LWP_DNMR} about absorbing $\updelta_{\mathscr{P}\mathscr{P}}\mathscr{P}$ into $\upzeta$. Alternatively, we can take $\updelta_{\mathscr{P}\mathscr{P}}=0$ for simplicity.} 
\begin{align}
(c_s^2)_\text{viscous} = \tilde{c}_s^2 + \frac{\upzeta}{\uptau_\mathscr{P}(p+\varrho+\mathscr{P})},
\nonumber
\end{align}
where
\begin{align}
\tilde{c}_s^2 = \frac{\partial p}{\partial \varrho} + \frac{n}{\varrho + p + \mathscr{P}} \frac{\partial p}{\partial n}
\end{align}
reduces to the sound speed for a perfect fluid, $\frac{\partial p}{\partial \varrho} + \frac{n}{\varrho + p } \frac{\partial p}{\partial n}$, computed from an equation of state $p = p(\varrho,n)$, when $\mathscr{P}=0$, see \cite{Bemfica-Disconzi-Noronha-2019-2}. Thus, an acoustical metric can be defined using $(c_s^2)_\text{viscous}$, whose null hypersurfaces will be the sound cones. Equations \eqref{E:DNMR_density} and \eqref{E:DNMR_velocity} are similar to the relativistic Euler equations, essentially replacing $p$ by $p+\mathscr{P}$ in the Euler system (recall that $\pi=0$ here). Thus, there is hope that a formulation similar to that of Theorem \ref{S:New_formulation} can be constructed. The main question will be that of the role played by Eq.~\eqref{E:DNMR_bulk}.

We recall that Theorem \ref{T:Breakdown_DNMR} holds for large perturbations of constant states. For small perturbations, it is not unreasonable to ask

\begin{innercustomthmproblem}
Investigate whether the DNMR equations with only bulk viscosity are globally well-posed for initial data close to constant states.
\label{Pr:GWP_DNMR}
\end{innercustomthmproblem}

Numerical simulations can be useful in helping motivate Problem \ref{Pr:GWP_DNMR}. We also recall that Theorem \ref{T:Breakdown_DNMR} required $n\neq 0$, thus the inclusion of Eq.~\eqref{E:Divergence_baryon_current_viscous} is needed. This is not the situation treated in most studied of the quark-gluon plasma, when $n$ is taken to be zero. We thus ask:

\begin{innercustomthmproblem}
Establish a version of Theorem \ref{T:Breakdown_DNMR} with $n=0$.
\label{Pr:Breakdown_DNMR_n}
\end{innercustomthmproblem}

We mentioned in Sect.~\ref{S:LWP_BDNK} that \cite{Sroczinski-2024} established global well-posedness of the BDNK equations for perturbations of constant solutions and specific choices of equation of state and transport coefficients. It is natural to extend his results:

\begin{innercustomthmproblem}
Establish global well-posedness for the BDNK equations for initial data close to constant and a different set up than that of \cite{Sroczinski-2024}.
\label{Pr:GWP_BNDK}
\end{innercustomthmproblem}

On the flip side of Problem \ref{Pr:GWP_BNDK}, it is not unreasonable to think that the BDNK equations, being a complicated system of quasilinear equations, may not admit global solutions for large data. Thus:

\begin{innercustomthmproblem}
Establish a breakdown result for the BDNK equations in the spirit of Theorem \ref{T:Breakdown_DNMR}.
\label{Pr:Breakdown_BNDK}
\end{innercustomthmproblem}

We mentioned in Sects. \ref{S:Viscous_shocks} and \ref{S:Limitations} that the MIS cannot describe viscous shock solutions in a regime of strong shocks and conjectured that the same is true for the DNMR theory.

\begin{innercustomthmproblem}
Show that the DNMR cannot describe strong viscous shocks in the sense of \cite{Geroch-Lindblom-1991,Olson:1991pf}.
\label{Pr:Viscous_shocks_DNMR}
\end{innercustomthmproblem}

The non-relativistic limit of the DNMR and BDNK theories was studied in detail in \cite{HegadeKR:2023glb}. These results, however, are formal, i.e., they have not been rigorously established.

\begin{innercustomthmproblem}
Establish rigorously the non-relativistic limits of the DNMR and BDNK equations formally established in 
\cite{HegadeKR:2023glb}.
\label{Pr:NR_limit}
\end{innercustomthmproblem}

In problem \ref{Pr:NR_limit}, it is understood that the convergence to a non-relativistic solution should happen in the function space where the relativistic theories have been showed to be locally well-posed, i.e., Sobolev spaces for the DNMR theory with only bulk viscosity and the BDNK theory, and Gevrey spaces for the DNRM equations \eqref{E:DNMR}.

As a more open-ended project, numerical analysts might be interested in the following question.

\begin{innercustomthmproblem}
Investigate the convergence and mathematical properties of the numerical schemes used in studies of the quark-gluon plasma based on simulations of the DNMR equations.
\label{Pr:Convergence_numerical_schemes}
\end{innercustomthmproblem}

See, in particular, comments in Sect.~\ref{S:Causality_DNMR} regarding a possible relation between causality violations and numerical schemes.

Theorems \ref{T:LWP_BDNK} and \ref{T:LWP_BDNK} establish local well-posedness of the DNMR and BDNK equations in suitable spaces, including when coupling to Einstein's equations is considered. However, in these theorems it is assumed that an initial-data set for the Einstein--DNMR and Einstein--BDNK system is \emph{given}. The existence of these initial-data sets has not been addressed, but is obviously very important if these fluid models are to be coupled to Einstein's equations.

\begin{innercustomthmproblem}
Establish existence of initial-data sets for the Einstein--DNMR and Einstein--BDNK systems by proving existence of solutions for the constraint equations.
\label{Pr:Constraints}
\end{innercustomthmproblem}

Problem \ref{Pr:Constraints} can be studied under many different assumptions, including asymptotically flat manifolds, manifolds with trivial topology (i.e., diffeomorphic to $\mathbb{R}^3$), and so on. In the particular case of a conformal fluid and for the Einstein-BDNK equations, Problem \ref{Pr:Constraints} has recently been solved in \cite{Disconzi-Isenberg-Maxwell-2024-arxiv}.

\medskip

We have not discussed kinetic-theoretical aspects of relativistic fluid dynamics in this review. The interested reader can consult the references provided in the Introduction. But it seems appropriate in this section to cite one open problem related to kinetic theory. \cite{Elskens-Kiessling-2020} consider the microscopic foundations of kinetic theory for a relativistic plasma in Minkowski space and provide
a road-map to go from the exact equations of motion to the relativistic
Vlasov--Maxwell equations. In the authors' words, ``The purpose of this work is not
to supply an entirely rigorous vindication, but to lay down a conceptual road map for the
microscopic foundations of the kinetic theory of special-relativistic plasma, and to emphasize
that a rigorous derivation seems feasible.'' What makes this work relevant for the discussion of this section is that the authors are very intentional about pointing out where proofs are lacking in their proposal for the derivation, thus providing interested readers with a set of concrete mathematical problems.

\appendix

\section{Complementary proofs}

The goal of this section is to provide proofs that were omitted above but which are still of interest. Some of these
provide further details on arguments discussed in the main text. Some formulas below are also cited in the main text.

\subsection{Proof of Lemma \ref{L:Lemma_enthalpy_current_evolution}\label{S:Lemma_enthalpy_current_evolution}}
We start with the Hodge Laplacian\footnote{Not really a Laplacian because $g$ is Lorentzian.} of $w$,
\begin{align}
\label{E:Lemma_enthalpy_current_evolution_Hodge_Laplacian}
\square_H w := (d d^* + d^* d ) w = d d^* w + d^* \Omega,
\end{align}
where $d^*$ is the formal adjoint of $d$ with respect to $g$. Using $d^* w = - \nabla_\alpha w^\alpha$ and
\eqref{E:Projected_relativistic_Euler_eq_full_system_continuity} 
we find
\begin{align}
\label{E:Lemma_enthalpy_current_evolution_divergence_w}
\begin{split}
d^*w &= 
-\nabla_\alpha w^\alpha = -\nabla_\alpha(h u^\alpha)
= -w^\alpha \nabla_\alpha h - 
h \underbrace{\nabla_\alpha u^\alpha}_{\mathclap{= -\frac{u^\alpha \nabla_\alpha n}{n}}}
\\
&
= 
-u^\alpha \nabla_\alpha h + \frac{h}{n} u^\alpha \nabla_\alpha n
\\
&=
-w^\alpha (\frac{\nabla_\alpha h}{h} - \frac{\nabla_\alpha n}{n} )
\\
&= \iota_w dF,
\end{split}
\end{align}
where $F = \ln \frac{n}{h}$. Applying $d$
to \eqref{E:Lemma_enthalpy_current_evolution_divergence_w}
 we find
\begin{align}
\label{E:Lemma_enthalpy_current_evolution_Lie_derivative_dF}
d d^* w = d(\iota_w dF) = \mathcal{L}_w dF,
\end{align}
where we used \eqref{E:Lie_derivative_iota_d}.

We need a good expression for $dF$ in 
\eqref{E:Lemma_enthalpy_current_evolution_Lie_derivative_dF}.
It will be convenient to introduce $\tilde{h} = h^2$ and consider $F$ as a function of $\tilde{h}$ and $s$. Then, since $w^\alpha w_\alpha = -h^2$ (recall \eqref{E:Enthalpy_current_normalization}),
using \eqref{E:Vorticity_definition} and
\eqref{E:Lichnerowicz_equation}, we find
\begin{align}
\label{E:Lemma_enthalpy_current_evolution_dF}
\begin{split}
\nabla_\alpha F 
&= 
\frac{\partial F}{\partial \tilde{h}} \nabla_\alpha \tilde{h} + \frac{\partial F}{\partial s} \nabla_\alpha s
= 
-\frac{\partial F}{\partial \tilde{h}} \nabla_\alpha (w^\beta w_\beta) + \frac{\partial F}{\partial s} \nabla_\alpha s
\\
&=
-2 \frac{\partial F}{\partial \tilde{h}}  w^\beta \nabla_\alpha w_\beta + \frac{\partial F}{\partial s} \nabla_\alpha s
\\
&=
-2 \frac{\partial F}{\partial \tilde{h}}  
	w^\beta
	(
	\Omega_{\alpha\beta} + \nabla_\beta w_\alpha
	)	
+ \frac{\partial F}{\partial s} \nabla_\alpha s
\\
&=
-2 \frac{\partial F}{\partial \tilde{h}}  
w^\beta\nabla_\beta w_\alpha
+2 \frac{\partial F}{\partial \tilde{h}}
	\underbrace{
		w^\beta \Omega_{\beta\alpha} 
	}_{\mathclap{=h\uptheta \nabla_\alpha s}}
+ \frac{\partial F}{\partial s} \nabla_\alpha s
\\
&=
-2 \frac{\partial F}{\partial \tilde{h}}  
w^\beta\nabla_\beta w_\alpha
+ 
	(
	2 \frac{\partial F}{\partial \tilde{h}} h \uptheta
	+
	\frac{\partial F}{\partial s}
	) \nabla_\alpha s
\end{split}
\end{align}
Using 
\eqref{E:Lemma_enthalpy_current_evolution_dF} and
a standard formula for the Lie derivative in terms of covariant derivatives, we can write
\begin{align}
\label{E:Lemma_enthalpy_current_evolution_Lie_derivative_dF_long}
\begin{split}
(\mathcal{L}_w dF)_\gamma 
&= w^\beta \nabla_\beta  \nabla_\gamma F
+ \nabla_\gamma w^\beta \nabla_\beta F
\\
&= 
-2 \frac{\partial F}{\partial \tilde{h}}  
w^\alpha w^\beta \nabla_\alpha \nabla_\beta w_\gamma
+ 
	(
	2 \frac{\partial F}{\partial \tilde{h}} h \uptheta
	+
	\frac{\partial F}{\partial s}
	) w^\alpha \nabla_\alpha \nabla_\gamma s
\\
	& \ \ \ 
+ \Err_\gamma(\partial g, \partial s, \partial w).
\end{split}
\end{align}
But \eqref{E:Projected_relativistic_Euler_density_entropy_full_system_entropy} gives
\begin{align}
\label{E:Lemma_enthalpy_current_evolution_partial_2_s}
\begin{split}
w^\alpha \nabla_\alpha \nabla_\gamma s
& = w^\alpha \nabla_\gamma \nabla_\alpha  s
= \nabla_\gamma( 
	\underbrace{
	w^\alpha \nabla_\alpha s)
	}_{\mathclap{=0}}
- \nabla_\gamma w^\alpha \nabla_\alpha s
\\
&= \Err_\gamma(\partial g, \partial s, \partial w),
\end{split}
\end{align}
so 
\eqref{E:Lemma_enthalpy_current_evolution_partial_2_s}
and 
\eqref{E:Lemma_enthalpy_current_evolution_Lie_derivative_dF_long}
yield
\begin{align}
\label{E:Lemma_enthalpy_current_evolution_Lie_derivative_dF_short}
(\mathcal{L}_w dF)_\gamma 
=
-2 \frac{\partial F}{\partial \tilde{h}}  
w^\alpha w^\beta \nabla_\alpha \nabla_\beta w_\gamma
+ 
\Err_\gamma(\partial g, \partial s, \partial w).
\end{align}
Next, recall the following formula 
for the Hodge Laplacian,
\begin{align}
\label{E:Lemma_enthalpy_current_evolution_Hodge_Laplacian_curvature}
(\square_H w)_\gamma 
= -g^{\alpha \beta} \nabla_\alpha \nabla_\beta w_\gamma
+ \Ric_{\gamma \alpha} w^\alpha,
\end{align}
where $\Ric$ is the Ricci curvature.
Combining \eqref{E:Lemma_enthalpy_current_evolution_Hodge_Laplacian},
\eqref{E:Lemma_enthalpy_current_evolution_Lie_derivative_dF_short}, and
\eqref{E:Lemma_enthalpy_current_evolution_Hodge_Laplacian_curvature}, 
we have
\begin{align}
\label{E:Lemma_enthalpy_current_evolution_wave_plus_curvature}
\begin{split}
-g^{\alpha \beta} \nabla_\alpha \nabla_\beta w_\gamma
+ \Ric_{\gamma \alpha} w^\alpha
&=
-2 \frac{\partial F}{\partial \tilde{h}}  
w^\alpha w^\beta \nabla_\alpha \nabla_\beta w_\gamma
+ (d^* \Omega)_\gamma
\\
& \ \ \ 
+ 
\Err_\gamma(\partial g, \partial s, \partial w).
\end{split}
\end{align}
Recalling the definition of $F$ above, compute
\begin{align}
\label{E:Lemma_enthalpy_current_evolution_F_computation}
\begin{split}
2 \frac{\partial F}{\partial \tilde{h}}  
&= 
2 \frac{\partial F}{\partial h} 
\underbrace{\frac{\partial h}{\partial \tilde{h}}}_{\mathclap{=\frac{1}{2h}}}
\\
&=
\frac{1}{h} \frac{\partial}{\partial h}
\ln\frac{n}{h} = \frac{1}{h} (\frac{1}{n} \frac{\partial n}{\partial h} - \frac{1}{n} )
\\
&= -\frac{1}{h^2}( 1 - \frac{h}{n} \frac{\partial n}{\partial h}).
\end{split}
\end{align}
Using \eqref{E:Lemma_enthalpy_current_evolution_F_computation} into 
\eqref{E:Lemma_enthalpy_current_evolution_wave_plus_curvature} and rearranging the terms,
we find
\begin{align}
\label{E:Lemma_enthalpy_current_evolution_second_order_w}
\begin{split}
[ -g^{\alpha\beta} - \frac{1}{h^2}( 1 - \frac{h}{n} \frac{\partial n}{\partial h}) \frac{w^\alpha w^\beta}{h^2} ] \nabla_\alpha \nabla_\beta w_\gamma
& = 
-\Ric_{\gamma \alpha} w^\alpha
+ (d^* \Omega)_\gamma
\\
& \ \ \
+ 
\Err_\gamma(\partial g, \partial s, \partial w).
\end{split}
\end{align}
We will apply $w^\mu \nabla_\mu$ to
\eqref{E:Lemma_enthalpy_current_evolution_second_order_w}.
Hence, using that 
$(d^* \Omega)_\gamma = \nabla_\nu \Omega\indices{^\nu_\gamma}$ and \eqref{E:Vorticity_evolution}, we compute,
\begin{align}
\label{E:Lemma_enthalpy_current_evolution_transport_div_Omega}
\begin{split}
w^\mu \nabla_\mu (d^* \Omega)_\gamma
&=
w^\mu \nabla_\mu \nabla_\nu \Omega\indices{^\nu_\gamma}\\
&= 
\underbrace{
	w^\mu \nabla_\nu \nabla_\mu  
	\Omega\indices{^\nu_\gamma}
	}_{\mathrlap{=\nabla_\nu
		(
		w^\mu\nabla_\mu \Omega\indices{^\nu_\gamma}
		)
		-\nabla_\nu w^\mu \nabla_\mu \Omega\indices{^\nu_\gamma}
	}}
\overbrace{
	-w^\mu \Riem\indices{_\mu_\nu_\lambda^\nu}
	\Omega\indices{^\lambda_\gamma}
	+ w^\mu \Riem\indices{_\mu_\nu_\gamma^\lambda} 	
	\Omega\indices{^\nu_\lambda}
	}^{\mathclap{=\Err_\gamma(\partial^2 g, w, \Omega)}}
	\\
&=	
\nabla_\nu
	\underbrace{
	(
	w^\mu\nabla_\mu \Omega\indices{^\nu_\gamma}
	)
	}_{\mathclap{=\Err_\gamma(\partial g, \partial w, \partial s, \partial h, \Omega)}
	}
	-\nabla_\nu w^\mu \nabla_\mu
	\Omega\indices{^\nu_\gamma}
	+
	\Err_\gamma(\partial^2 g, w, \Omega)
\\
&=
\Err_\gamma(\partial^2 g, \partial^2 w, \partial^2 s, \partial^2 h, \partial \Omega)
\end{split}
\end{align}
where $\Riem$ is the Riemann tensor.

Thus, 
applying $u^\mu \nabla_\mu$ to
\eqref{E:Lemma_enthalpy_current_evolution_second_order_w}
and using 
\eqref{E:Lemma_enthalpy_current_evolution_transport_div_Omega}
we finally obtain
\begin{align}
[g^{\alpha\beta} + (1-\frac{h}{n} \frac{\partial n}{\partial h} ) \frac{w^\alpha w^\beta}{h^2} ]
w^\mu \nabla_\alpha \nabla_\beta w_\gamma
=
\Err_\gamma(\partial^2 g, \partial^2 w, \partial^2 s, \partial^2 h, \partial \Omega).
\nonumber
\end{align}
Multiplying this expression by $c_s^2$ and 
using that the sound speed is equivalently given by (see \cite[Section 2.4]{Rezzolla-Zanotti-Book-2013})
\begin{align}
\frac{1}{c_s^2} = \frac{h}{n} 
\left.\frac{\partial n}{\partial h}\right|_s
\nonumber
\end{align}
produces the result.

\subsection{Finishing the proof of Theorem \ref{T:LWP_relativistic_Euler}\label{S:Complementary_LWP_relativistic_Euler}}
In this section, we provide the remaining arguments in the proof of Theorem \ref{T:LWP_relativistic_Euler}.
We resume from the end of the proof provided in the main text, continuing with the same notation and conventions. The proof follows known arguments, but 
we provide it mostly because it is only briefly mentioned
in \cite{Lichnerowicz-Book-1967}.

\medskip

Under the assumptions of the Theorem, if we consider the relativistic Euler equations written as \eqref{E:Projected_relativistic_Euler_enthalpy_entropy_full_system},
we can solve for $\partial_t(s, h, u)$ at $t=0$ in terms of the initial data.
Differentiating \eqref{E:Projected_relativistic_Euler_density_entropy_full_system} with respect to $\partial_t$ and proceeding inductively, we can algebraically solve for
$\partial_t^k(s, h, u)$ at $t=0$, $k=2, 3, \dots$, in terms of the initial data and their spatial 
derivatives. This allows us to (a) construct initial data for the system 
\eqref{E:LWP_relativistic_Euler_system} and (b) construct analytic solutions to 
\eqref{E:Projected_relativistic_Euler_enthalpy_entropy_full_system} provided that the initial
data is analytic (in other words, we can verify that $\{t=0\}$ is non-characteristic 
for \eqref{E:Projected_relativistic_Euler_enthalpy_entropy_full_system} and apply the Cauchy--Kovalevskaya theorem). Such an analytic solution satisfies the system 
\eqref{E:LWP_relativistic_Euler_system} with $w = h u$, $\Omega = dw$, and $u^\alpha u_\alpha = -1$, so in particular $w^\alpha w_\alpha = - h^2$.

(We observe that the standard construction of analytic solutions applies to unconstrained system. Thus, we consider all components of $u$ as independent and apply the Cauchy--Kovalevskay theorem
to \eqref{E:Projected_relativistic_Euler_enthalpy_entropy_full_system}  without assuming \eqref{E:Velocity_normalization} and with $\proj$ given by
\eqref{E:Projection_u_orthogonal_non_normalized}. 
Without the normalization \eqref{E:Velocity_normalization}, it is not guaranteed
that $\{t=0\}$ is non-characteristic for 
\eqref{E:Projected_relativistic_Euler_enthalpy_entropy_full_system}. However, \eqref{E:Velocity_normalization} does hold at $\{t=0\}$, so the non-characteristic character is guaranteed. Once the analytic solution is constructed, the constraint
\eqref{E:Velocity_normalization} is propagated as in \eqref{E:Velocity_constraint_propagation} and the surrounding discussion.)

Given initial data as in the statement of the Theorem, consider a sequence  
$(\mathring{s}_\ell, \mathring{h}_\ell, \mathring{u}_\ell)$, $\ell \in \mathbb{N}$, of analytic
data converging to $(\mathring{s}, \mathring{h}, \mathring{u})$ in $H^N\times H^N \times H^N$. For each
$\ell$, let $(s_\ell,h_\ell,u_\ell)$ be the corresponding analytic solution for the analytic data
$(\mathring{s}_\ell, \mathring{h}_\ell, \mathring{u}_\ell)$ discussed above. From the proof of the Cauchy--Kovalevskaya theorem it follows that there exists a $T^\prime > 0$ independent\footnote{See, for instance, the discussion in Chapter 5.3.c of \cite{John-Book-1982}.} of $\ell$ such that 
for all $\ell$, $(s_\ell,h_\ell,u_\ell)$ 
is defined on $[0,T^\prime]$. 

Let $(s_\ell,\Omega_\ell,w_\ell)$ be the analytic solution to
\eqref{E:LWP_relativistic_Euler_system} constructed out 
of $(s_\ell,h_\ell,u_\ell)$, as explained above.
Applying \eqref{E:LWP_relativistic_Euler_energy_estimates_compact} to 
$(s_\ell,\Omega_\ell,w_\ell)$
we obtain that the sequence $(s_\ell,\Omega_\ell,w_\ell)$ is bounded in $L^\infty([0,T],H^N(\mathbb{T}^3)\times H^{N-1}(\mathbb{T}^3)\times H^N(\mathbb{T}^3))$ for some
$0<T\leq T^\prime$. For each $\ell$, we also have that 
$\partial_t w_\ell \in L^\infty([0,T],H^{N-1}(\mathbb{T}^3))$ and, using equations
\eqref{E:LWP_relativistic_Euler_system_s} and \eqref{E:LWP_relativistic_Euler_system_Omega} also that 
$\partial_t (s_\ell,\Omega_\ell) \in L^\infty([0,T],H^{N-1}(\mathbb{T}^3)\times H^{N-2}(\mathbb{T}^3))$. These bounds can then be combined with the fundamental
theorem of calculus in the $t$-variable to estimate the difference 
$(s_\ell, \Omega_\ell,w_\ell)(t_1,\cdot) -  
(s_\ell,  \Omega_\ell,w_\ell)(t_2,\cdot)$, $0\leq t_1, t_2 \leq T$ and, upon further application of
\eqref{E:LWP_relativistic_Euler_energy_estimates_compact} 
and a standard interpolation, we obtain\footnote{See Sect.1.5 of \cite{Disconzi-Luo-2020} for an example of how
the fundamental theorem of calculus is applied to obtain continuity.} that 
the sequence $(s_\ell, \Omega_\ell,w_\ell)$ is bounded
in $C^0([0,T], H^{N-\delta}(\mathbb{T}^3)\times H^{N-1-\delta}(\mathbb{T}^3)\times H^{N-\delta}(\mathbb{T}^3))$ for any $0 < \delta < N-3 -\frac{3}{2}$. Similarly we obtain that
$\partial_t(s_\ell, \Omega_\ell,w_\ell)$ is bounded
in $C^0([0,T], H^{N-1-\delta}(\mathbb{T}^3)\times H^{N-2-\delta}(\mathbb{T}^3) \times H^{N-1-\delta}(\mathbb{T}^3))$.

Next, we consider the difference $(s_\ell, w_\ell, \Omega_\ell) - (s_{\ell^\prime}, w_{\ell^\prime}, \Omega_{\ell^\prime})$ between two solutions. Subtracting 
\eqref{E:LWP_relativistic_Euler_system_w} 
for $\ell$ and $\ell^\prime$
we obtain\footnote{We do not write the $h$ dependence for $G^{-1}$ because $h=\sqrt{-w^\alpha w_\alpha}$.}
\begin{align}
\label{E:LWP_relativistic_Euler_difference_w}
\begin{split}
&(G^{-1}(s_\ell, w_\ell))^{\alpha\beta} w_\ell^\gamma \partial_\alpha \partial_\beta \partial_\gamma [ (w_\ell)_\delta - (w_{\ell^\prime})_\delta ]
\\
&=
[G^{-1}(s_{\ell^\prime}, w_{\ell^\prime}))^{\alpha\beta} w_{\ell^\prime}^\gamma 
-
(G^{-1}(s_\ell, w_\ell))^{\alpha\beta} w_\ell^\gamma
]
\partial_\alpha \partial_\beta \partial_\gamma  (w_{\ell^\prime})_\delta 
\\
& \ \ \
+\Err_\delta (\partial^2 s_\ell, \partial^2 w_\ell,\partial \Omega_\ell)
-
\Err_\delta (\partial^2 s_{\ell^\prime}, \partial^2 w_{\ell^\prime},\partial \Omega_{\ell^\prime}),
\end{split}
\end{align}
where we omitted the dependence on $g$ since it does not depend on $\ell$.
The fundamental theorem of calculus implies that the RHS of \eqref{E:LWP_relativistic_Euler_difference_w} can be written as an expression linear 
in the difference $(s_\ell, \Omega_\ell, w_\ell) - (s_{\ell^\prime}, \Omega_{\ell^\prime}, w_{\ell^\prime})$ and their derivatives. Indeed, denoting
by $\xi$ the argument of $\Err_\delta$ we have
\begin{align}
\label{E:LWP_relativistic_Euler_difference_w_FTC}
\begin{split}
\Err_\delta(\xi_\ell) - \Err_\delta(\xi_{\ell^\prime}) 
&=
\int_0^1 \frac{d}{dr} 
[\Err_\delta( r \xi_\ell + (1-r)\xi_{\ell^\prime}) ]\, dr
\\
&=
\int_0^1 D_\xi \Err_\delta( r \xi_\ell + (1-r)\xi_{\ell^\prime})  ) \cdot (\xi_\ell - \xi_{\ell^\prime}) \, dr
\\
&=
\Err_\delta
(\partial^2 s_\ell, \partial \Omega_\ell, \partial^2 w_\ell, \partial^2s_{\ell^\prime}, \partial\Omega_{\ell^\prime},  \partial^2 w_{\ell^\prime}
) 
\\
& \ \ \ \cdot \big(
\partial^2 (s_\ell-s_{\ell^\prime}),
\partial(\Omega_\ell, - \Omega_{\ell^\prime} ),
\partial^2 (w_\ell-w_{\ell^\prime})
\big),
\end{split}
\end{align}
where $D_\xi \Err_\delta$ is the derivative of $\Err_\delta$ and 
$(
\partial^2 (s_\ell-s_{\ell^\prime}),
\partial^2 (w_\ell-w_{\ell^\prime}),
\partial(\Omega_\ell, - \Omega_{\ell^\prime} )
)$ represents the vector whose entries are 
$s_\ell-s_{\ell^\prime}$ and its derivatives up to second order, $\Omega_\ell, - \Omega_{\ell^\prime}$ and its derivatives up to first order, 
and $w_\ell-w_{\ell^\prime}$ and its derivatives up to second order.
An argument similar to \eqref{E:LWP_relativistic_Euler_difference_w_FTC}
can be applied to the  term in brackets on the RHS of 
\eqref{E:LWP_relativistic_Euler_difference_w}.
A similar argument using Eqs.~\eqref{E:LWP_relativistic_Euler_system_s} and
\eqref{E:LWP_relativistic_Euler_system_Omega} gives that the difference 
$(s_\ell, \Omega_\ell, w_\ell) - (s_{\ell^\prime}, \Omega_{\ell^\prime}, w_{\ell^\prime})$ satisfies the system
\begin{subequations}{\label{E:LWP_relativistic_Euler_system_difference}}
\begin{align}
& w_\ell^\alpha \partial_\alpha (s_\ell - s_{\ell^\prime}) 
= 
\nonumber
\\
& \ \ \
\Err(
w_\ell,\partial s_{\ell^\prime},w_{\ell^\prime})
\cdot
\big(
 (s_\ell-s_{\ell^\prime}),
 (w_\ell-w_{\ell^\prime})
\big), 
\label{E:LWP_relativistic_Euler_system_difference_s}
\\
& 
w_\ell^\mu \partial_\mu 
((\Omega_\ell)_{\alpha \beta}
-(\Omega_{\ell^\prime})_{\alpha\beta} ) 
=
\nonumber
\\
& \ \ \
\Err_{\alpha\beta}(
\partial s_\ell, \partial w_\ell, \Omega_\ell,
\partial s_{\ell^\prime}, \partial w_{\ell^\prime}, \partial \Omega_{\ell^\prime}
)
\cdot
\big(
\partial (s_\ell-s_{\ell^\prime}),
(\Omega_\ell, - \Omega_{\ell^\prime} ),
\partial (w_\ell-w_{\ell^\prime})
\big), 
\label{E:LWP_relativistic_Euler_system_difference_Omega}
\\
&
(G^{-1}(s_\ell,w_\ell))^{\alpha\beta} w_\ell^\gamma \partial_\alpha \partial_\beta \partial_\gamma ((w_\ell)_\delta - (w_{\ell^\prime})_\delta )
=
\nonumber
\\
& \ \ \ 
 \Err_\delta
(\partial^2 s_\ell, \partial \Omega_\ell, \partial^2 w_\ell, \partial^2s_{\ell^\prime}, \partial\Omega_{\ell^\prime},  \partial^3 w_{\ell^\prime}
) 
\nonumber
\\
& \ \ \ 
\cdot
\big(
\partial^2 (s_\ell-s_{\ell^\prime}),
\partial(\Omega_\ell, - \Omega_{\ell^\prime} ),
\partial^2 (w_\ell-w_{\ell^\prime})
\big).
\label{E:LWP_relativistic_Euler_system_difference_w}
\end{align}
\end{subequations}
System \eqref{E:LWP_relativistic_Euler_system_difference}
is a linear system for 
$(s_\ell, \Omega_\ell, w_\ell) - (s_{\ell^\prime}, \Omega_{\ell^\prime}, w_{\ell^\prime})$ with coefficients depending on derivatives of
$s_\ell, \Omega_\ell, w_\ell, s_{\ell^\prime}, \Omega_{\ell^\prime}, w_{\ell^\prime}$ and
the same principal part as \eqref{E:LWP_relativistic_Euler_system}. 
Observe, however, that while the RHS of equations
\eqref{E:LWP_relativistic_Euler_system_s},
\eqref{E:LWP_relativistic_Euler_system_Omega},
and
\eqref{E:LWP_relativistic_Euler_system_w},
depend on up to zero derivatives of $s$, zero derivatives of $\Omega$, and two derivatives of
$w$, respectively, the 
the RHS of
\eqref{E:LWP_relativistic_Euler_system_difference_s},
\eqref{E:LWP_relativistic_Euler_system_difference_Omega},
and 
\eqref{E:LWP_relativistic_Euler_system_difference_w}
depend on up to one derivative of $s$, one derivative of $\Omega$, and three derivatives
of $w$. This extra derivative comes from 
the quasilinear nature of \eqref{E:LWP_relativistic_Euler_system}.
Thus, since the sequence is bounded in the norm
\eqref{E:LWP_relativistic_Euler_norm}, we cannot close \eqref{E:LWP_relativistic_Euler_system_difference}
in $H^N$ but only in $H^{N-1}$. More precisely,
defining 
\begin{align}
\begin{split}
\tilde{\mathcal{N}}_{N-1}(t) & := 
\norm{s_\ell - s_{\ell^\prime}}_{H^{N-1}(\Sigma_t)}
+
\norm{\Omega_\ell-\Omega_{\ell^\prime}}_{H^{N-2}(\Sigma_t)}
\\
&
\ \ \
+
\norm{\partial_t^2 w_\ell - \partial_t^2 w_{\ell^\prime}}_{H^{N-3}(\Sigma_t)}
+ 
\norm{\partial_t w_\ell- \partial_t w_{\ell^\prime} }_{H^{N-2}(\Sigma_t)}
\\
& \ \ \
+
\norm{w_\ell-w_{\ell^\prime}}_{H^{N-1}(\Sigma_t)},
\end{split}
\nonumber
\end{align}
arguing as in the proof of \eqref{E:LWP_relativistic_Euler_energy_estimates}
but with an estimate for the $H^{N-1}$ norms 
instead of $H^N$, 
using \eqref{E:LWP_relativistic_Euler_energy_estimates},
and that \eqref{E:LWP_relativistic_Euler_system_difference} is a linear system for the difference, we obtain
\begin{align}
\tilde{\mathcal{N}}_{N-1}(t) \lesssim
\tilde{\mathcal{N}}_{N-1}(0), \, 0 \leq t \leq T,
\nonumber
\end{align}
and
\begin{align}
\label{E:LWP_relativistic_Euler_energy_estimate_lower_continuous}
\begin{split}
\norm{
(s_\ell - s_{\ell^\prime},
\Omega_\ell - \Omega_{\ell^\prime},
w_\ell - w_{\ell^\prime})
}_{C^0([0,T],H^{N-1-\delta}(\mathbb{T}^3)\times H^{N-2-\delta}(\mathbb{T}^3) \times H^{N-1-\delta}(\mathbb{T}^3))}
\lesssim
\\
\norm{
(\mathring{s}_\ell - \mathring{s}_{\ell^\prime},
\mathring{\Omega}_\ell - \mathring{\Omega}_{\ell^\prime},
\mathring{w}_\ell - \mathring{w}_{\ell^\prime})
}_{H^N(\mathbb{T}^3)\times H^{N-1}(\mathbb{T}^3) \times H^N(\mathbb{T}^3)},
\end{split}
\end{align}
for any $0 < \delta < N - \frac{3}{2} - 2$, 
where the $\mathcal{N}(0)$ norm has been subsumed 
under $\lesssim$.

By construction, the sequence $(\mathring{s}_\ell, \mathring{\Omega}_\ell, \mathring{w}_\ell)$ is Cauchy in $H^N(\mathbb{T}^3)\times H^{N-1}(\mathbb{T}^3) \times H^N(\mathbb{T}^3)$, thus 
by \eqref{E:LWP_relativistic_Euler_energy_estimate_lower_continuous},
$(s_\ell, \Omega_\ell, w_\ell)$ is a Cauchy sequence
in 
$C^0([0,T],H^{N-1-\delta}(\mathbb{T}^3)\times H^{N-2-\delta}(\mathbb{T}^3) \times H^{N-1-\delta}(\mathbb{T}^3))$. Denote its limit
by $(s_\infty, \Omega_\infty, w_\infty)$. Fix $t_* \in (0,T]$. The sequence 
$(s_\ell, \Omega_\ell, w_\ell) (t_*, \cdot)$ converges
to $(s_\infty, \Omega_\infty, w_\infty)(t_*,\cdot)$
in $H^{N-1-\delta}(\mathbb{T}^3)\times H^{N-2-\delta}(\mathbb{T}^3)\times H^{N-1-\delta}(\mathbb{T}^3)$ and is bounded
in $H^{N}(\mathbb{T}^3)\times H^{N-1}(\mathbb{T}^3)\times H^{N}(\mathbb{T}^3)$ by \eqref{E:LWP_relativistic_Euler_energy_estimates_compact}, thus the limit is in $H^{N}(\mathbb{T}^3)\times H^{N-1}(\mathbb{T}^3)\times H^{N}(\mathbb{T}^3)$. Hence
$(s_\infty, \Omega_\infty, w_\infty) \in L^\infty([0,T],
H^{N}(\mathbb{T}^3) \times H^{N-1}(\mathbb{T}^3)  \times H^{N}(\mathbb{T}^3) )$.

Using $h_\ell = \sqrt{-(w^\alpha_\ell) (w_\ell)_\alpha}$ and 
$u_\ell = h^{-1} w_\ell$, the Sobolev embedding theorem
applied to $(s_\ell,h_\ell,u_\ell)$ and their time derivatives (which can be obtained from \eqref{E:Projected_relativistic_Euler_enthalpy_entropy_full_system}), we see that we can pass to the limit in 
\eqref{E:Projected_relativistic_Euler_enthalpy_entropy_full_system}, 
and obtain that $s_\infty$, $h_\infty = \sqrt{-w_\alpha w_\alpha}$, and $u_\infty = -h_\infty^{-1} w_\infty$ is a classical solution to \eqref{E:Projected_relativistic_Euler_enthalpy_entropy_full_system},  with the regularity stated in the Theorem. We obtain a solution for $-T \leq t < 0$ by time reversal. Uniqueness follows from the estimate for the difference of two solutions.
\hfill \qed

\section{Leray systems}
\label{S:Leray_systems}

In this appendix, we recall some results about Leray--Ohya systems, also called weakly hyperbolic systems (and called by Leray and Ohya non-strictly hyperbolic systems),
that have been invoked in the main text, particularly in Sect.~\ref{S:DNMR}. These results have been established by
\cite{Leray-Ohya-1964,Leray-Ohya-1967} for the case of systems with diagonal 
principal part, and extended by \cite{Choquet-Bruhat-1966} to more general systems.
These
works build upon the seminal work of Leray on hyperbolic differential equations, see \cite{Leray-Book-1953}. 
The reader can consult these references for the  proofs of the results stated below. Further discussion
can be found (without proofs) in \cite{Choquet-Bruhat-Book-2009, Czubak-Disconzi-2016,Disconzi-2014-1}. 
Related results can also be found in \cite{Mizohata-Book-1985}.

We start by recalling some standard notions and fixing the notation that will
be used throughout\footnote{We believe that this appendix can be useful on its own to readers. Thus, in order to make it as self-contained as possible, we repeat some definitions and notations already introduced in the main text.}. 
Given $T>0$, let 
$X = [0,T] \times  \mathbb{R}^n$.
By  $\partial^k$ we shall denote any $k^{\text{th}}$ order derivative. We shall denote coordinates on $X$ by
$\{ x^\alpha \}_{\alpha = 0}^n$, thinking of $x^0 \equiv t$ as the time-variable.
We use the multi-index notation to write
\begin{align}
\begin{split}
\partial^\alpha 
\equiv
\frac{\partial^{|\alpha |}}{\partial x_0^{\alpha_0} \partial x_1^{\alpha_1} \partial x_2^{\alpha_2} \cdots \partial x_n^{\alpha_n} } 
\equiv \partial^{\alpha_0}_{x^0} \partial^{\alpha_1}_{x^1} \partial^{\alpha_2}_{x^2} \cdots \partial^{\alpha_n}_{x^n},
\end{split}
\nonumber
\end{align}
where $|\alpha | = \alpha_0 +  \alpha_1 + \alpha_2 + \cdots + \alpha_n$.
We will often indicate vector-valued quantities in abbreviate form, e.g., 
$z = (z^I)$, $I=1,\dots,N$, to mean $z = (z^1, \dots, z^N)$.

\subsection{Gevrey spaces\label{S:Gevrey_spaces}}

In this section we review the definition of Gevrey spaces. Roughly speaking,
a function is of Gevrey class if it obeys inequalities similar, albeit weaker,
than those satisfied by analytic functions. One of the crucial properties of Gevrey spaces
for their use in general relativity is that they admit compactly supported functions.

\begin{definition} 
\label{D:Gevrey_spaces}
Let $s \geq 1$.
We say that $f:  \mathbb{R}^n  \rightarrow \mathbb{C}$ belongs to the 
\textdef{Gevrey space} $G^{(s)}(\mathbb{R}^n)$ if 
\begin{align}
\sup_\alpha  \frac{1}{(1+|\alpha|)^s} \norm{\partial^\alpha f}_{L^2(\mathbb{R}^n)}^\frac{1}{1+|\alpha|} 
< \infty.
\nonumber
\end{align}
Let $K\subset \mathbb{R}^n$ be the cube of unit side. We say that $f$ belongs to the \textdef{local Gevrey space}\footnote{The terminology ``local Gevrey space'' is not standard.} $G^{(s)}_{loc}(\mathbb{R}^n)$ if 
\begin{align}
\sup_\alpha  \frac{1}{(1+|\alpha|)^s} \left(\sup_K \norm{\partial^\alpha f}_{L^2(K)}\right)^\frac{1}{1+|\alpha |} 
< \infty,
\nonumber
\end{align}
where $\sup_K$ is taken over all side one cubes $K$ in $\mathbb{R}^n$.
\end{definition}

We note that the case $s=1$, i.e., $G^{(1)}(\mathbb{R}^n)$, corresponds to the space of analytic functions. We refer to \cite{Rodino-Book-1993,Mizohata-Book-1985} for equivalent definitions or variations on the definition of Gevrey spaces.

We next introduce the space of maps defined on $X$ whose derivatives up to order
$m$ belong to $G^{(s)}(\{x^0=t\})$, $0\leq t \leq T$.

\begin{definition}
\label{D:Gevrey_derivative}
On $X$, denote $\Sigma_t = \{ x^0 = t\}$. Let $s\geq 1$, and let $m\geq 0$ be an integer.
We denote by $\overline{\alpha}$ a multi-index $\alpha=(\alpha_0, \dots, \alpha_n)$ for which $\alpha_0 = 0$.
We define $G^{m,{(s)}}(X)$ as the set of maps $f:X \rightarrow \mathbb{C}$ such that
\begin{align}
\sup_{\overline{\alpha}, \, |\beta| \leq m, \, 0 \leq t \leq T}
\frac{1}{( 1 + |\overline{\alpha} |)^s} \norm{\partial^{\beta+\overline{\alpha}} f }_{L^2(\Sigma_t)}^\frac{1}{1+|\overline{\alpha}|} < \infty.
\nonumber
\end{align}
Let $Y$ be an open set of $\mathbb{R}^d$. We define $G^{m,{(s)}}(X\times Y)$ as the set of maps $f:X\times Y \rightarrow \mathbb{C}$ such that
\begin{align}
\sup_{\overline{\alpha}, \, \gamma, \, |\beta| \leq m, \, 0 \leq t \leq T}
\frac{1}{( 1 + |\overline{\alpha} | + |\gamma| )^s} 
\norm{\sup_{y \in Y} \left| \partial_x^{\beta+\overline{\alpha}} \partial_y^\gamma f \right| }_{L^2(\Sigma_t)}^\frac{1}{1+|\overline{\alpha}| + |\gamma|} < \infty.
\nonumber
\end{align}
Let $K_t \subset \Sigma_t$ be the cube whose sides have unit length. 
The spaces $G^{m,{(s)}}_{loc}(X)$ and $G^{m,{(s)}}_{loc}(X\times Y)$ are defined as the set of maps $f:X \rightarrow \mathbb{C}$
and  $f:X\times Y \rightarrow \mathbb{C}$, respectively, such that
\begin{align}
\sup_{\overline{\alpha}, \, |\beta| \leq m, \, 0 \leq t \leq T}
\frac{1}{( 1 + |\overline{\alpha} |)^s} \left(\sup_{K_t} \norm{\partial^{\beta+\overline{\alpha}} f }_{L^2(K_t)} \right)^\frac{1}{1+|\overline{\alpha}|} < \infty,
\nonumber
\end{align}
and
\begin{align}
\sup_{\overline{\alpha}, \, \gamma, \, |\beta| \leq m, \, 0 \leq t \leq T}
\frac{1}{( 1 + |\overline{\alpha} | + |\gamma| )^s} 
\left( \sup_{K_t} \norm{\sup_{y \in Y} \left| \partial_x^{\beta+\overline{\alpha}} \partial_y^\gamma f \right| }_{L^2(K_t)} \right)^\frac{1}{1+|\overline{\alpha}| + |\gamma|} < \infty,
\nonumber
\end{align}
where $\sup_{K_t}$ is taken over all cubes of side one within $\Sigma_t$.
\end{definition}

\begin{remark}
Definitions \ref{D:Gevrey_spaces} and \ref{D:Gevrey_derivative} are easily generalized
to vector and tensor fields in $\mathbb{R}^n$ and $X$, and to open subsets of $\mathbb{R}^n$ and $X$.
In particular, replacing $\mathbb{R}^n$ by an open set $\Omega$ and $X$ by $[0,T]\times \Omega$
in the above definitions we obtain the corresponding spaces for $\Omega$. This allows one to define
Gevrey spaces on manifolds. If $M$ is a differentiable manifold, we say that $f: M \rightarrow \mathbb{C}$
belongs to $G^{(s)}(M)$ if for every $p \in M$ there exists a coordinate chart $(x,U)$ about
$p$ such that $f\circ x^{-1} \in G^{(s)}(\Omega)$, where $\Omega = x(U)$.
This definition generalizes for vector and tensor fields.
\end{remark}

\begin{remark}
The reason to treat $X$ and $Y$ differently in definitions of $G^{(s)}(X\times Y)$ 
and $G^{m,{(s)}}(X\times Y)$ is that, in the theorems of Sect.~\ref{S:Cauchy_problem_Leray},
 we need
to distinguish between the regularity with respect to the spacetime $X$ and the regularity with respect to the parametrization of the initial data.
\end{remark}

\begin{remark}
We could similarly define for manifolds the analog of the other Gevrey spaces introduce above. 
However, this can be somewhat cumbersome and not always natural. In particular, the spaces
$G^{m,(s)}$ require a distinguished coordinate that plays the role of time. This can always be 
done locally, and it can be done for globally hyperbolic manifolds if we fix a particular 
foliation in terms of spacelike slices, although it is debatable
how canonical this is. Ultimately, for the type of equations we are interested, one can 
always localize the problem, invoke coordinate charts, and rely on results valid in $\mathbb{R}^n$.
\end{remark}

\subsection{The Cauchy problem for Leray systems\label{S:Cauchy_problem_Leray}}

Let  $a = a(x,\partial^k)$, $x \in X$, be a linear differential operator of order $k$. We can write
\begin{align}
a(x,\partial^k) = \sum_{ |\alpha | \leq k } a_\alpha(x) \partial^\alpha,
\nonumber
\end{align}
where $\alpha = (\alpha_0, \alpha_1, \alpha_2, \dots, \alpha_n)$ is a multi-index.
Let $p(x,\partial^k)$ be the principal 
part of $a(x,\partial^k)$, i.e., 
\begin{align}
p(x,\partial^k) = \sum_{ |\alpha | = k } a_\alpha(x) \partial^\alpha.
\nonumber
\end{align}
At each point $x \in X$ and for each co-vector $\xi \in T_x^* X$, where $T^*X$ is the cotangent bundle of $X$,
we can associate a polynomial of order $k$ in the cotangent space
$T_x^* X$ obtained by replacing the derivatives by $\xi \in T_x^* X$. More precisely, for each 
$k^{\text{th}}$ order derivative in $a(x,\partial^k)$, i.e., 
\begin{align}
\partial^\alpha 
=
\frac{\partial^{|\alpha |}}{\partial x_0^{\alpha_0} \partial x_1^{\alpha_1} \partial x_2^{\alpha_2} \cdots \partial x_n^{\alpha_n} } 
\nonumber
\end{align}
$|\alpha | = k$,
we associate the polynomial
\begin{align}
\xi^\alpha \equiv \xi_0^{\alpha_0} \xi_1^{\alpha_1} \xi_2^{\alpha_2} \cdots \xi_n^{\alpha_n},
\nonumber
\end{align}
where $\xi = (\xi_0, \xi_1,\xi_2,\dots, \xi_n) \in T_x^* X$, forming in this way the polynomial
\begin{align}
p(x,\xi) = \sum_{ |\alpha | = k } a_\alpha(x) \xi^\alpha.
\nonumber
\end{align}
Clearly, $p(x,\xi)$ is a homogeneous polynomial of degree $k$. It is called the 
\textdef{characteristic polynomial} (at $x$) of the operator $a$.

The \textdef{cone} (at $x$) $V_x(p)$ of $p$ is the set in
$T_x^* X$ defined by the equation
\begin{align}
p(x,\xi) = 0,
\nonumber
\end{align}
i.e., 
\begin{align}
V_x(p) := \{ \xi \in T_x^*X \, | \, p(x,\xi) = 0 \}.
\nonumber
\end{align}

\begin{definition}
\label{D:Hyperbolic_polynomial}
With the above notation, $p(x,\xi)$ is called a \textdef{strictly hyperbolic polynomial} or simply \textdef{hyperbolic polynomial\footnote{The terminology ``hyperbolic'' is not uniform across the literature and is sometimes applied to different definitions}} (at $x$) if there exists  $\zeta \in T_x^*X$ such 
that every straight line through $\zeta$ that does not contain the origin intersects 
the cone $V_x(p)$ at $k$ real distinct points (see also Remarks \ref{R:Equivalent_definition_hyperbolicity} and \ref{R:Equivalent_definition_hyperbolicity_relativity}).  The differential operator $a(x,\partial^k)$
is called a \textdef{(strictly) hyperbolic operator} (at $x$) if $p(x,\xi)$ is hyperbolic.
\end{definition}

\cite{Leray-Book-1953}
 that  (if $X$ is at least three-dimensional)  if $p(x,\xi)$ is hyperbolic at $x$, then 
 the set of points $\zeta$ satisfying the condition of Definition \ref{D:Hyperbolic_polynomial}
forms the interior
of two opposite  half-cones $\Gamma_x^{*,+}(a)$, $\Gamma_x^{*,-}(a)$, with
$\Gamma_x^{*,\pm}(a)$ non-empty, with boundaries that belong to 
$V_x(p)$ .

\begin{remark}
\label{R:Equivalent_definition_hyperbolicity}
Another way of stating Definition \ref{D:Hyperbolic_polynomial} is as follows (see \citealt[Chapter~VI]{Courant-Hilbert-Book-1989-2}). Given
$\zeta \in T_x X$, consider a non-zero vector $\theta$ that is not parallel to $\zeta$ and
form the line $\lambda \zeta + \theta$, where $\lambda \in \mathbb{R}$ is a parameter. We then require
this line to intersect the cone $V_x(p)$ at $k$ distinct real points.
Algebraically, this means that the polynomial equation in $\lambda$
\begin{align}
p(x,\lambda \zeta + \theta)=0,
\end{align}
has $k$ distinct real roots.
\end{remark}

\begin{remark}
\label{R:Equivalent_definition_hyperbolicity_relativity}
In relativity, the situation is often such that one has picked coordinates adapted to a problem (e.g., wave coordinates) and is interested in situations where $V_x(p)$ lies outside the light-cone in $T_x^*X$ (see Remark \ref{R:Causal_structure_A_and_g}). In this situation, we can take $\zeta = (1,0,\dots,0)$ and the
definition of hyperbolic polynomials reduces to the following:
$p(x,\xi)$ is hyperbolic at $x$ if for each non-zero $\xi=(\xi_0,\dots, \xi_n) \in T^*_x X$, the equation
$p(x,\xi)=0$ has $k$ distinct real roots $\xi_0 = \xi_0(\xi_1,\dots,\xi_n)$.
\end{remark}

With applications to systems in mind, we next consider the $N \times N$ diagonal linear differential operator matrix
\begin{align}
A(x,\partial) = 
\left(
\begin{matrix}
a^1(x,\partial^{k_1}) & \cdots & 0 \\
\vdots & \ddots & \vdots \\
0 & \cdots & a^N(x,\partial^{k_N})
\end{matrix}
\right).
\nonumber
\end{align}
Each $a^J(x,\partial^{k_J})$, $J = 1, \dots, N$ is a linear differential operator
of order $k_J$. 

\begin{definition}
\label{D:Leray_Ohya_hyperbolic}
The operator  $A(x,\partial)$ is called 
\textdef{Leray--Ohya  hyperbolic} (at $x$) if:

(i) The characteristic polynomial $p^J(x,\xi)$ of each $a^J(x,\partial^{k_J})$ is a product\footnote{See Remark \ref{R:Strictly_hyperbolic}.} 
of  hyperbolic
polynomials, i.e.
\begin{align}
p^J(x,\xi) = p^{J,1}(x, \xi) \cdots p^{J,r_J}(x,\xi),\, J=1,\dots, N,
\nonumber
\end{align}
where each $p^{J,q}(x,\xi)$, $q=1,\dots,r_J$, $J=1,\dots,N$,
is a hyperbolic polynomial.

(ii) The two opposite
convex half-cones,
\begin{align}
\Gamma_x^{*,+}(A) = \bigcap_{J=1}^N \bigcap_{q=1}^{r_J} \Gamma_x^{*,+}(a^{J,q}),
\, \text{ and } \,
\Gamma_x^{*,-}(A) = \bigcap_{J=1}^N \bigcap_{q=1}^{r_J} \Gamma_x^{*,-}(a^{J,q}),
\nonumber
\end{align}
have a non-empty interior. Here, $\Gamma_x^{*,\pm}(a^{J,q})$ are the
half-cones associated with the hyperbolic polynomials $p^{J,q}(x,\xi)$,
$q=1,\dots,r_J$,
$J=1,\dots,N$.
\end{definition}

\begin{remark}
\label{R:Constant_multiplicity}
When the hyperbolicity properties (i) and (ii) in Definition \ref{D:Leray_Ohya_hyperbolic} hold for every $x$, we call the corresponding
operators Leray--Ohya hyperbolic (we can also talk about hyperbolicity in an open set, a certain region, etc.).
When we say that an operator is Leray--Ohya hyperbolic on the whole space (or in an open set, etc.),
this means not only that Definition \ref{D:Leray_Ohya_hyperbolic} applies for every
$x$, but also that the numbers $r_J$ and the degree of the polynomials 
$p^{J,q}(x,\xi)$, $q=1,\dots,r_J$, $J=1,\dots,N$, do not change with $x$.
\end{remark}

\begin{definition}
\label{D:Causal_structure_A}
We define the \textdef{dual convex half-cone} $C_x^+(A)$ at $T_x X$ as the
set of $v \in T_x X$ such that $\xi(v) \geq 0$ for every
$\xi \in \Gamma_x^{*,+}(A)$; $C_x^-(A)$ is analogously defined, and 
we set  $C_x(A) = C_x^+(A) \cup C_x^-(A)$. 
If the convex cones $C_x^+(A)$ and
$C_x^-(A)$ can be continuously distinguished with respect to $x \in X$,
then $X$ is called \textdef{time-oriented} (with respect to the hyperbolic form
provided by the operator $A$). If $X$ is time-oriented, a path in $X$ is called
\textdef{future (past) time-like} with respect to $A$ if its tangent at each point 
$x \in X$ belongs to $C_x^+(A)$ ($C_x^-(A)$), and \textdef{future (past) 
causal} if its tangent at each point  $x \in X$ belongs or is tangent 
to $C_x^+(A)$ ($C_x^-(A)$). A regular surface $\Sigma$ is 
called \textdef{spacelike} with respect to $A$ if  $T_x\Sigma$ ($ \subset T_x X$) is
exterior to $C_x(A)$ for each $x \in \Sigma$.  It follows that for a time-oriented
$X$,
 the concepts of \textdef{causal past, 
future, domains of dependence and influence} of a set can be defined 
in the same way one does when the manifold is endowed with a Lorentzian 
metric. We refer the reader to \cite{Leray-Book-1953} for details. Here we need only 
the following: the \textdef{past domain of dependence}\footnote{In accordance with the language of Lorentzian manifolds, it would be more appropriate to call $J^-(x)$ the causal past of $x$. But we prefer to call it the past domain of dependence in order to emphasize that it is determined by the operator $A$ rather than a Lorentzian metric. This also avoids confusion with the notaion of causal past when a Lorentzian metric is involved.} $J^-(x)$ of a point $x \in X$ is the set of points that can
be joined to $x$ by a past causal curve. 
\end{definition}

\begin{remark}
\label{R:Causal_structure_A_and_g}
The definitions in Definition \ref{D:Causal_structure_A} endow $X$ with a
causal structure provided by 
the operator $A$.
Despite the similar terminology, however, it should be noticed that all of the above
definitions depend only on the structure of the operator $A$, and do not require
an a priori  Lorentzian metric on $X$. In particular, 
the timelike and causal directions defined by the operator $A$ in Definition 
\ref{D:Causal_structure_A} may not agree with a pre-specified time direction in a given physical problem.
The case of interest in relativity, however, is when
the causal structure of the space-time is connected with that of $A$. In this regard, the following observation is useful.
Suppose that $X$ has a Lorentzian metric $g$.
For solutions 
of the systems of equations here described (see Theorem \ref{T:Leray_Ohya_domain_of_dependence} below) to be causal in the sense 
of general relativity, one needs that (a) they satisfy the domain-of-dependence property, see Theorem \ref{T:Leray_Ohya_domain_of_dependence} and Sect.~\ref{S:Domain_of_dependence_outside_Gevrey}; and (b) for all $x \in X$, $C_x^\pm(A) \subseteq K_x^\pm$, where 
$K_x^\pm$ are the two halves of the light-cone $\{ g_{\mu \nu} \xi^\mu \xi^\nu \leq 0 \}$.
By duality, this is equivalent to saying that in the cotangent spaces we have $K_x^{*,\pm} \subseteq \Gamma_x^{*,+}(A) $,
where  $K_x^{*,\pm}$ are the two halves of the dual light-cone $\{ g^{\mu \nu} \xi_\mu \xi_\nu \leq 0 \}$.
\end{remark}

\begin{remark}
We stress that although we have been working with coordinates $\{ x^\alpha \}_{\alpha=0}^n$ in $[0,T]\times \mathbb{R}^n$, the sets $V_x(p)$, $\Gamma_x(a)^{*,\pm}$, 
$\Gamma_x(A)^{*,\pm}$, and $C_x^{\pm}(A)$ are invariantly defined, i.e., they do not depend on coordinates choices (see \citealt[Chapter~VI]{Courant-Hilbert-Book-1989-2}). Similarly for the notion of hyperbolicity. 
\end{remark}

Next, we consider the following quasi-linear system of differential 
equations
\begin{align}
\label{E:Quasilinear_system}
A(x,U,\partial) U = B(x, U),
\end{align}
where $A(x,U, \partial)$ is the $N \times N$ diagonal matrix
\begin{align}
\label{E:Matrix_A_Leray_system}
A(x,U, \partial) = 
\left(
\begin{matrix}
a^1(x,U, \partial^{k_1}) & \cdots & 0 \\
\vdots & \ddots & \vdots \\
0 & \cdots & a^N(x,U, \partial^{k_N})
\end{matrix}
\right),
\end{align}
with $a^J(x,U,\partial^{k_J})$, $J = 1, \dots, N$ differential operators
of order $k_J$. $B(x,U)$ is the vector
\begin{align}
B(x,U) = (b^J(x,U)),\, J=1,\dots, N,
\nonumber
\end{align}
and the vector 
\begin{align}
U(x) = (U^I(x)), \, I = 1, \dots, N
\nonumber
\end{align}
is the unknown. Notice that because $a_J$ is allowed to depend on $U$, the above system is in general 
nonlinear.

\begin{definition}
\label{D:Leray_system}
The system $A(x,U,\partial) U = B(x, U)$ is called a 
\textdef{Leray system}
if it is possible to attach to each unknown $u^I$ an integer $m_I \geq 0$,
and to each equation $J$ of the system an integer $n_J \geq 0$,
such that:

(i) $k_J = m_J - n_J$, $J=1,\dots, N$;

(ii) the functions $b^J$ and the coefficients of the differential operators
$a^J$ are\footnote{The regularity required for the coefficients
$a^J$ and $b^J$ depends on particular applications and context. For instance,
for Theorem \ref{T:Leray_Ohya_existence} the required regularity
is specified. Similarly, in Definition \ref{D:Leray_Cauchy_problem}, one needs
to take derivatives of these quantities up to order $n_J$, thus they need to be
at least as many times differentiable.} functions of $x$, of $u^I$, and of the 
derivatives of $u^I$ of order at most $m_I - n_J - 1$, 
$I, J =1\dots, N$. If for some $I$ and some $J$, $m_I - n_J< 0$, then
the corresponding $a^J$ and $b^J$ do not depend on $u^I$.
\end{definition}

\begin{remark} The indices $m_I$ and $n_J$ in Definition \ref{D:Leray_system}
are defined up to an additive integer.
\end{remark}

\begin{definition}
\label{D:Leray_Ohya_system}
A \textdef{Leray--Ohya system (with diagonal principal part)} is a Leray system where the matrix $A$ is Leray--Ohya
hyperbolic. In the quasilinear case, since the operators $a$ depend on $U$,
we need to specify a function 
$U$ that  is  plugged into $A(x,U,\partial)$ in order to compute the characteristic polynomials. 
In this case we talk about
a Leray--Ohya system for the function $U$. The primary case of interest
is when $U$ assumes the values of the given Cauchy data.
\end{definition}

When considering a quasilinear system, we write $p(x,U,\xi)$ and similar expressions to indicate 
the dependence on $U$. 
 
We now formulate the Cauchy problem for Leray systems.

\begin{definition}
\label{D:Leray_Cauchy_problem}
Let $\Sigma$ be a regular hypersurface in $X$, which we assume for simplicity to be given by $\{ x^0 = 0 \}$. 
The \textdef{Cauchy data} on $\Sigma$ for a Leray 
system in $X$ consists of the values
of $U=(u^I)$ and their derivatives up to order $m_I - 1$ on $\Sigma$,
i.e., the Cauchy data is $U_0 = (u^{I,\alpha})$, $I=1,\dots, N$, $|\alpha| \leq m_I - 1$,
where $u^{I,\alpha}$ is the prescribed $\alpha$ derivative\footnote{It is customary to prescribe only derivatives transverse to $\Sigma$ as Cauchy data. But since derivatives of $U$ tangent to $\Sigma$ are known when $\left. U \right|_\Sigma$ is known, there is no loss of generality in considering all derivatives of $u^I$ up to order $m_I - 1$ as data. This simplifies the notation as we to not have to distinguish among $|\alpha| \leq m_I -1$ those that correspond solely to transverse derivatives.} of $u^I$, so that
$\left. \partial^\alpha u^I \right|_{\Sigma} = u^{I,\alpha}$, when $U$ is a solution.
The Cauchy data is required to satisfy the following compatibility conditions. If
$V=(v^I)$ is an extension of the Cauchy data defined in a neighborhood  
of $\Sigma$, i.e.
$\left. \partial^\alpha v^I \right|_{\Sigma} =  u^{I,\alpha}$,  $I=1,\dots, N$, $|\alpha| \leq m_I -1$,
then, for each $J$ such that $n_J >0$, the difference $a^J(x,V, \partial^{m_J-n_J})V - b^J(x,V)$
and its derivatives of order less than $n_J$ vanish on $\Sigma$, for
$J=1,\ldots, N$. When to a Leray system
 $A(x,U, \partial) U = B(x, U)$
we prescribe initial data satisfying these conditions, we say that we have 
a \textdef{Cauchy problem} for $A(x,U, \partial) U = B(x, U)$. 
\end{definition}

\begin{remark}
Because the operators $a^J$ in the matrix $A$ in \eqref{E:Matrix_A_Leray_system}
have order $m_J - n_J$, one might think that one should prescribe as 
data only $u^I$ and their derivatives up to order $m_I - n_I -1$. However, this will in general not determine a unique solution, even for analytic data and assuming that $\Sigma$ is non-characteristic.
To see this, recall that in the equations $a^J(x,U, \partial^{m_J-n_J})u^J - b^J(x,U) = 0$, $a^J$ and $b^J$ depend on
derivatives of $u^I$ up to order $m_I - n_J - 1$, $I,J = 1,\dots, N$. The coefficients $a^J$, $b^J$ need to be determined by the data along $\Sigma$ for an application of the Cauchy--Kovalevskaya theorem, which then requires
\begin{align}
\label{E:Indices_inequality_compatibility_conditions}
m_I - n_J - 1 \leq m_I - n_I - 1,
\end{align}
for $I,J= 1,\dots, N$ if only derivatives up to order $m_I - n_I -1$ of $u^I$ were prescribed as data. But \eqref{E:Indices_inequality_compatibility_conditions} can only be satisfied if $n_J=0$, $J=0,\dots, N$.
If, instead, derivatives of $u^I$ up to order $m_I - 1$ are prescribed as data, as in Definition \ref{D:Leray_Cauchy_problem}, then, instead of \eqref{E:Indices_inequality_compatibility_conditions}, the requirement that the coefficients be known along $\Sigma$ from the data becomes
\begin{align}
m_I - n_J - 1 \leq m_I - 1,
\nonumber
\end{align}
i.e., $n_J \geq 0$, in accordance with definition \eqref{D:Leray_system}. Note, also, that if $n_J > 0$ for some $J$, then in the corresponding equation
\begin{align}
\label{E:n_J_greater_than_1}
a^J(x,U, \partial^{m_J-n_J})u^J - b^J(x,U) = 0 
\end{align}
all terms are known along $\Sigma$ since $u^J$ has its derivatives up to order $m_J -1$ are prescribed on $\Sigma$. Thus, compatibility conditions have to be satisfied by the data, and we see that since 
$\left. \partial^{\alpha} u^J \right|_\Sigma$, $|\alpha| \leq m_J -1$ is prescribed, all derivatives up to order $n_J -1$ of \eqref{E:n_J_greater_than_1} are determined by the data, which explains the definition of the compatibility conditions in Definition \ref{D:Leray_Cauchy_problem}. 
See \cite[Chapter 6]{Choquet-Bruhat-DeWitt-Morette-Dillard-Bleick-Book-1977} for further discussion on the Cauchy problem for Leray systems.
\end{remark}

Notice that by definition,
the Cauchy data for a Leray system satisfies the aforementioned compatibility
conditions. We also introduce the following notions related to the
Cauchy problem for a Leray system. 

\begin{assumption}
\label{A:Structure_operators}
Consider the Cauchy problem for a Leray system $A(x,U,\partial) U = B(x,U)$. 
Let $Y$ be an open set of $\mathbb{R}^L$, where $L$ equals the number of derivatives 
of $u^J$ of order less or equal to $\max_{I} m_I - n_J$,  $J=1,\dots,N$, and
such that $Y$ contains the closure of the values taken by the Cauchy data on $\Sigma$.
It is convenient to consider  $A(x,U,\partial)$ as a differential operator defined over $X \times Y$, 
as follows. We shall assume
that there exists a differential operator $\widetilde{A}(x,y,\partial)$
defined over $X\times Y$ with the following property.  If $(x,y) \in X\times Y$ and 
$V = (v^J)$ is a sufficiently regular
function on $\Sigma$ such that $y = (\partial^{\max_{I} m_I - n_J } v^J (x))_{J=1,\dots,N}$, 
then $A(x,V(x),\partial) = \widetilde{A}(x,y,\partial)$. We shall write $A(x,y,\partial)$ for 
$\widetilde{A}(x,y,\partial)$.
\end{assumption}

\begin{definition}
\label{D:Iteration_space}
Consider the Cauchy problem for a Leray system $A(x,U,\partial) U = B(x,U)$.
Let $\Sigma$ and $Y$ be as in Definition \ref{D:Leray_Cauchy_problem} and
Assumption \ref{A:Structure_operators}, respectively.
Denote by $\mathcal{A}^s(\Sigma,I)$ the set of $V = (v^J) \in G^{(s)}(\Sigma)$, $J=1,\dots,N$,
such that $ (\partial^{\max_{I} m_I - n_J } v^J (x))_{J=1,\dots,N} \in Y$ for all $x \in \Sigma$ (see Remark 
\ref{R:Iteration_space}).
\end{definition}

We are now ready to state the results of this appendix.
We use the above notation and definitions in the statement of the theorems below. 
 
\begin{theorem}[Existence and uniqueness, \citealt{Leray-Ohya-1967}]
\label{T:Leray_Ohya_existence}
Consider the Cauchy problem for (\ref{E:Quasilinear_system}).  Suppose that the Cauchy data 
is in $G^{(s)}(\Sigma)$, and that 
\begin{align}
a^J( \cdot,\cdot, \partial^{k_J}) \in G_{loc}^{n_J,(s)}(X \times Y), \, \text{ and } \, b^J(\cdot,\cdot) \in G^{n_J,(s)}(X \times Y).
\nonumber
\end{align}
Suppose that for any $V \in \mathcal{A}^s(\Sigma,Y)$
the system is Leray--Ohya hyperbolic with indices $m_I$ and $n_J$;
thus for all $x \in \Sigma$, each $p^J(x,V, \xi)$ is the product of 
$r_J$ hyperbolic polynomials,
\begin{align}
p^J(x,V,\xi) = p^{J,1}(x, V,\xi) \cdots p^{J, r_J}(x, V, \xi),\, J=1,\dots, N.
\nonumber
\end{align}
Suppose that each $p^{J,q+1}(x, V, \xi)$, $q=0, \dots, r_J-1$, depends on 
at most $m_I - m_{J,q} - r_I +q$ derivatives of $v^I$, $I=1,\dots,N$, where
\begin{align}
m_{J,q} = n_J + \operatorname{deg}(p^{J,1}) + \cdots +\operatorname{deg}(p^{J,q}), \, m_{J, r_J} = m_J, \, m_{J,0} = n_J. 
\nonumber
\end{align}
Above, $\operatorname{deg}(p^{J,q})$ is the degree, in $\xi$, of the polynomial 
$p^{J,q}(x, V,\xi)$. 

Denote by $a^J_{q+1}(x,y,\partial)$ the differential operator associated with 
$p^{J,q+1}$. Assume that
\begin{align}
a^J_{q+1}(\cdot,\cdot,\partial) \in G_{loc}^{m_{J,q} - q, (s)}(X\times Y).
\nonumber
\end{align}
Let $0 \leq g_I \leq r_I$ be the smallest integers such that $a^J(x,V,\partial^{m_J-n_J})$
and $b^J(x,V)$ depend on at most $m_I - n_J - r_I + g_I$ derivatives of 
$v^I$, $I=1,\dots, N$, $J=1,\dots, N$.
Finally, assume that 
\begin{align}
1 \leq s \leq \frac{r_J}{g_J} \, \text{ and } \, \frac{n}{2} + r^J < n_J, \, J=1,\dots, N.
\nonumber
\end{align}
Then, there exists a $T^\prime > 0$ and a solution $U=(u^I)$ to the Cauchy problem for (\ref{E:Quasilinear_system}) and defined on $[0,T^\prime) \times \mathbb{R}^n \subseteq X$. The solution satisfies
\begin{align}
u^I \in G^{m_I,(s)}([0,T^\prime) \times \mathbb{R}^n), \, I=1,\dots, N.
\nonumber
\end{align}
Furthermore, the solution is unique in this regularity class.
\end{theorem}

\begin{theorem}[Domain-of-dependence property, \citealt{Leray-Ohya-1967}] 
\label{T:Leray_Ohya_domain_of_dependence}
Assume the same hypotheses
of Theorem \ref{T:Leray_Ohya_existence}, and suppose further that 
\begin{align}
1 \leq s < \frac{r_J}{g_J}, \, J=1,\dots, N.
\nonumber
\end{align}
Let $T^\prime$ and $U$ be as in the conclusion of  Theorem \ref{T:Leray_Ohya_existence}. 
Then, if $T^\prime$ is sufficiently small, the operator $A(x,U,\partial)$
is Leray--Ohya hyperbolic (thus the past domain of dependence of a point is well-defined, see Remark \ref{D:Causal_structure_A}),
and for each $x \in [0,T^\prime) \times \mathbb{R}^n$, $U(x)$ depends only on
$\left. U_0 \right|_{J^-(x) \cap \Sigma}$, where $U_0$ is the Cauchy data.
\end{theorem}

\begin{definition}
When solutions to a system of PDEs satisfy the conclusion of Theorem \ref{T:Leray_Ohya_domain_of_dependence}, we say that the system has the \textdef{domain-of-dependence property.}
\end{definition}

\begin{remark}
\label{R:Iteration_space}
Theorem \ref{T:Leray_Ohya_existence} assumes that the system is Leray--Ohya hyperbolic
for $V \in \mathcal{A}(\Sigma,Y)$, which is essentially the space of values near the initial data.
Theorem \ref{T:Leray_Ohya_existence} is proven with an iteration (see below), thus we need to guarantee that, at each step, the iterates form a Leray--Ohya system with the right properties, which is guaranteed by staying within 
$\mathcal{A}(\Sigma,Y)$.
(Naturally, it would not make sense to require the system to be Leray--Ohya hyperbolic
for the yet to be proven to exist solution $U$.) Once $U$ is constructed, one can then
ask whether the system is Leray--Ohya hyperbolic for $U$. This will be the case if $T^\prime$
is small, since in this case the values of $U$ will be close to those of the initial data
by continuity, guaranteeing that $U \in \mathcal{A}(\Sigma,Y)$.
\end{remark}

\begin{remark}
\label{R:Remark_depends_only_Leray}
We recall the meaning of ``depends only on $\left. U_{0} \right|_{J^-(x) \cap \Sigma}$"
in Theorem \ref{T:Leray_Ohya_domain_of_dependence}.
Let $U_{0,1}$ and $U_{0,2}$ be two initial data for \eqref{E:Quasilinear_system} and $U_1$ and $U_2$ the corresponding solutions given by Theorem \ref{T:Leray_Ohya_existence}, which we can assume to be defined on the same time interval $[0,T^\prime]$ by shrinking one of the time intervals for $U_1$ or $U_2$ if necessary. Let $J^-_1(x)$ be the past domain of dependence of $x$ with respect to the solution $U_1$.
Suppose that $\left. U_{0,1} \right|_{J_1^-(x) \cap \Sigma} = \left. U_{0,2} \right|_{J_1^-(x) \cap \Sigma}$ Then, $U_1(x) = U_2(x)$ and, in fact, $U_1 = U_2$ in $J_1^-(x)$. It also follows that $J_1^-(x) = J_2^-(x)$, where $J^-_2(x)$ be the past domain of dependence of $x$ with respect to the solution $U_2$.
\end{remark}

See Sect.~\ref{S:Domain_of_dependence_outside_Gevrey} for further applications of the domain-of-dependence property, including the case of more general function spaces.

We now consider a system whose principal part is not necessarily diagonal.
The definition of a Leray system depends only on the existence of the indices $m_I$ and
$n_J$ with the stated properties, and thus can be extended to non-diagonal systems.

\begin{definition}
Consider a system of $N$ partial differential 
equations and $N$ unknowns in $X$,
and
denote the unknown as 
$U=(u^I)$, $I=1,\dots, N$. 
The system is a \textdef{(not necessarily diagonal in the principal part) Leray system} if
it is possible 
to attach to each unknown $u^I$ a non-negative integer $m_I$ and to
each equation a non-negative integer $n_J$, such that the system reads
\begin{align}
\label{E:General_Leray_system}
h^J_I(x,\partial^{m_K - n_J -1} u^K, \partial^{m_I - n_J}) u^I
+ b^J(x, \partial^{m_K - n_J - 1} u^K) = 0, \, J=1, \dots, N.
\end{align}
Here, 
$h^J_I(x,\partial^{m_K - n_J -1} u^K, \partial^{m_I - n_J})$
is a homogeneous differential operator of order $m_I - n_J$ (which can
be zero), whose coefficients depend on at most
$m_K - n_J -1$ derivatives of $u^K$, $K=1,\dots N$, and there is a sum over $I$ in $h^J_I(\cdot) u^I$. 
The remaining terms, 
$b^J(x,\partial^{m_K - n_J - 1} u^K)$, also depend on at most 
$m_K - n_J -1$ derivatives of $u^K$,  $K=1,\dots N$. 
As before, these indices are defined only up to an
overall additive integer.
\end{definition}

As done above, for a given sufficiently regular $U$,   
 $h^J_I(x,\partial^{m_K - n_J -1} U^K, \partial^{m_I - n_J})$  are well-defined linear
 operators, and we can ask about their hyperbolicity properties. The case of interest will be, again, when
 we evaluate these operators at some given Cauchy data.
 
Write (\ref{E:General_Leray_system}) in matrix form as
\begin{align}
\label{E:Leray_system_matrix}
H(x,U,\partial) U = B(x,U).
\end{align}
 
\begin{definition} 
The \textdef{characteristic determinant} of (\ref{E:Leray_system_matrix}) 
at $x \in X$ and for a given $U$
 is the polynomial $p(x,\xi)$ in the co-tangent space $T^*_x X$, $\xi \in T_x^* X$, 
given by 
\begin{align}
\label{E:Leray_characteristic_det}
p(x,U,\xi) = \det( H(x,U,\xi)).
\end{align}
\end{definition}
Note that $p$ is a homogeneous polynomial of degree
\begin{align}
 \ell := \sum_{I=1}^N m_I - \sum_{J=1}^N n_J.
\nonumber
\end{align}
Under appropriate conditions, (\ref{E:Leray_system_matrix}) can be transformed 
into a Leray--Ohya system of the form (\ref{E:Quasilinear_system}), i.e., with diagonal principal part.
More precisely, we have the following.

\begin{theorem}[Diagonalization, \citealt{Choquet-Bruhat-1966}]
\label{T:Leray_Ohya_diagonalization}
Consider (\ref{E:Leray_system_matrix}).
Suppose that the characteristic determinant (\ref{E:Leray_characteristic_det}) at a given $U$ is not identically zero, and it is the product
of $Q$ hyperbolic polynomials, i.e., 
\begin{align}
p(x,U,\xi) = p_1(x,U,\xi) \cdots p_Q(x,U,\xi).
\nonumber
\end{align}
Let $d_q$ be the degree of $p_q(x,U,\xi)$, $q=1,\dots, Q$, and suppose
that 
\begin{align}
\max_q d_q \geq \max_I m_I - \min_J n_J.
\nonumber
\end{align}
Finally, assume that 
\begin{align}
\ell \geq \max_I m_I - \min_J n_J.
\nonumber
\end{align}
Then, there exists a $N\times N$ matrix $C(x,U,\partial)$ of differential operators whose coefficients depend on $U$, such that
\begin{align}
C(x,U,\partial) H(x,U,\partial) U =\mathbb{I} \, p(x,U,\partial) U + \widetilde{B}_1(x,U),
\nonumber
\end{align}
and 
\begin{align}
C(x,U,\partial) B(x,U) = \widetilde{B}_2(x,U),
\nonumber
\end{align}
where $\mathbb{I}$ is the $N\times N$ identity matrix, $p(x,U,\partial)$ is the differential operator associated with 
$p(x,U,\xi)$, and $\widetilde{B}_1(x,U)$ and $ \widetilde{B}_2(x,U)$ depend on at most $\ell -1 $ derivatives of $U$, as do the coefficients
of the operator $p(x,U,\xi)$. Furthermore, there is a choice of indices that makes the system
\begin{align}
\label{E:Diagonalized_Leray}
\mathbb{I} \, p(x,U,\partial) U =  \widetilde{B}_2(x,U) - \widetilde{B}_1(x,U)
\end{align}
into a Leray system. In particular,
if the intersections $\cap_q \Gamma_x^{*,+}(a^q)$ and $\cap_q \Gamma_x^{*,-}(a^q)$, where
$\Gamma_x^{*,\pm}(a^q)$ are the half-cones associated with the hyperbolic polynomials $p_q(x,U,\xi)$,
have non-empty interiors, then 
 (\ref{E:Diagonalized_Leray}) is a Leray--Ohya system with diagonal principal part 
in the sense of definition \ref{D:Leray_Ohya_system}.
\end{theorem}

Theorem \ref{T:Leray_Ohya_diagonalization} is proven in \cite{Choquet-Bruhat-1966}.

\begin{definition}
\label{D:Gevrey_index}
Under the hypotheses of Theorem \ref{T:Leray_Ohya_diagonalization}, the number $\frac{Q}{Q-1}$ is 
called the \textdef{Gevrey index} of the system.  
\end{definition}

\begin{remark}
\label{R:Q_Q_minus_1_diagonalization}
Suppose that (\ref{E:Diagonalized_Leray}) forms a Leray--Ohya system in the sense
of definition \ref{D:Leray_Ohya_system}, i.e., the half-cones have non-empty interiors
as stated in Theorem \ref{T:Leray_Ohya_diagonalization}. 
It can then be shown, see \cite{Choquet-Bruhat-1966}, that a value of $s$ sufficient to apply
Theorems \ref{T:Leray_Ohya_existence} and \ref{T:Leray_Ohya_domain_of_dependence}
is $1 \leq s < \frac{Q}{Q-1}$.
\end{remark}

Let us make a brief comment about the proofs of the above results. Theorem \ref{T:Leray_Ohya_existence}
is proven as follows. First, one solves the associated linear problem. This is done by a method of majorants reminiscent 
of the Cauchy--Kowalevskaya theorem (see \citealt{Leray-Ohya-1964} for a treatment of the linear problem). One uses the fact that Gevrey functions admit a formal series expansion that provides
a consistent way of constructing successive approximating solutions to the problem. The non-linear problem is then treated
via a fixed point argument, upon solving successive linear problems. Theorem \ref{T:Leray_Ohya_domain_of_dependence} is obtained
by a Holmgren type of argument. We remark that the assumption that 
$p^{J,q+1}(x,V, \xi)$, $q=0, \dots, r_J-1$, depends on 
at most $m_I - m_{J,q} - r_I +q$ derivatives of $v^I$, $I=1,\dots,N$,
ensures that the coefficients of the associated differential operators 
$a^{J,q+1}(x,U, \partial)$ do not depend on too many derivatives of $U$,
as it should be in the treatment of quasi-linear equations.

Theorem \ref{T:Leray_Ohya_diagonalization} is based on the following identity:
\begin{align}
c^T a = \det(a) \mathbb{I},
\end{align}
where $a$ is an $N\times N$ invertible matrix and $c^T$ the transpose of the co-factor matrix. At the level of differential operators,
this identity produces the lower order terms $\widetilde{B}_1$. One then needs to match the order of the resulting differential operators
and lower order terms with appropriate indices satisfying the definition of a Leray system. This is possible under the conditions
on $d_q$ and $\ell$ stated in the theorem.

When the polynomials $p^J$, $J=1,\dots, N$ in Definition \ref{D:Leray_Ohya_hyperbolic} are hyperbolic polynomials whose corresponding cones have non-empty intersection (part (b) of Definition \eqref{D:Leray_Ohya_hyperbolic}), then the corresponding operator $A(x,\partial)$ is called a \textdef{strictly hyperbolic, or simply hyperbolic, operator (with diagonal principal part)}. The above definitions can be adjusted to this case and in particular a Leray system $A(x,U,\partial) U = B(x,U)$ is then called a (strictly) hyperbolic system. These are the types of operators and systems treated by Leray in his celebrated monograph \citep{Leray-Book-1953}. In this situation, local existence and uniqueness can be established in suitable Sobolev spaces and the domain of dependence property can also be demonstrated. Similarly, if 
the characteristic determinant of \eqref{E:Leray_system_matrix} is a hyperbolic polynomial, then the diagonalized system \eqref{E:Diagonalized_Leray} will be (strictly) hyperbolic.

\begin{remark}
\label{R:Strictly_hyperbolic}
We note a certain subtlety in the definition of Leray--Ohya hyperbolicity. In part (a) of Definition \ref{D:Leray_Ohya_hyperbolic}, $p^J(x,\xi)$ is a product
\begin{align}
p^J(x,\xi) = p^{J,1}(x, \xi) \cdots p^{J,r_J}(x,\xi)
\nonumber
\end{align}
with each $p^{J,q}(x,\xi)$, $q=1,\dots,r_J$, $J=1,\dots,N$, being a (strictly) hyperbolic polynomial. But this does not preclude $p^J(x,\xi)$ from being itself a hyperbolic polynomial. In fact, if the polynomials 
$p^{J,q}(x,\xi)$ do not have roots in common, then all the roots of the polynomial $p^{J,1}(x, \xi) \cdots p^{J,r_J}(x,\xi)$ will be real and distinct, thus $p^J(x,\xi)$ will be a hyperbolic polynomial. In this case the operator $A(x,\partial)$ will be strictly hyperbolic and the corresponding Leray system $A(x,U,\partial) U = B(x,U)$ will be a strictly hyperbolic system. Of course, Definition \ref{D:Leray_Ohya_hyperbolic} still applies, which is simply to say that every strictly hyperbolic system is in particular a Leray--Ohya system or, in alternative terminology, that every hyperbolic system is in particular a weakly hyperbolic system. But naturally, the results we can obtain if we apply to a strictly hyperbolic system techniques of Leray--Ohya system are significantly weakened, since the latter will provide local existence and uniqueness in Gevrey spaces whereas one can establish local existence and uniqueness for strictly hyperbolic systems in Sobolev spaces, as mentioned above.
\end{remark}

We see from Remark \ref{R:Strictly_hyperbolic} that the primary situation of interest in the study of Leray--Ohya system is then when one or more of the polynomials $p^J(x,\xi)$ has \emph{repeated roots}. More precisely, strict hyperbolicity can also fail if one or more of the numbers $r_J$ change with $x$ or if one or more of the polynomials $p^{J,q}(x,\xi)$ change their degree with $x$. This is a situation that we are not considering here (see Remark \ref{R:Constant_multiplicity}). Readers interested in this case are referred to 
\cite{Reissig-Schulze-Book-2005,Garetto-Jah-Ruzhansky-2018,Garetto-Jah-Ruzhansky-2020} and references therein for more details.

\subsubsection{More on domains of dependence\label{S:Domain_of_dependence_outside_Gevrey}}
Here, we will provide a template for how the domain-of-dependence property can be established outside the Gevrey class. It is often the case that this can be established by direct energy estimates. But for some complicated systems, like the ones dealt with in Sect.~\ref{S:Relativistic_viscous_fluids}, such direct estimates may not be available. For example, the estimates used in Theorem \ref{T:LWP_BDNK} are carried out in a pseudo-differential framework that makes a standard integration by parts argument more difficult.

The specific details of how this template can be applied will depend on particular aspects of the problem being studied, such as structural properties of the equations, the function spaces being used, and so on. Thus, we will inevitably be somewhat imprecise in some instances, referring to generic terms like ``appropriate conditions,'' ``sufficiently regular,'' etc. We expect that readers should have no difficulty in identifying what these should be in particular problems of interest.

We begin noticing that in view of Remark \ref{R:Remark_depends_only_Leray}, upon considering the system satisfied by the difference of two solutions, the domain-of-dependence property can be rephrase in terms of vanishing initial data. More precisely, it states that upon considering the system satisfied by the difference of two solutions with initial data $U_0$,  if $\left. U_0 \right|_{J^-(x)\cap \Sigma} = 0$, then 
$U = 0$ in $J^-(x)$, where $U$ is a solution for the system satisfied by the difference of two solutions. 

Suppose we are considering equations 
of the form \eqref{E:Leray_system_matrix}
that are locally well-posed in a function space $\mathsf{X}$ that is more general than Gevrey spaces, Sobolev spaces being the primary case of interest. Consider initial data $U_0 \in \mathsf{X}$ and let $U$ be the corresponding solution obtained by local well-posedness in $\mathsf{X}$.
Suppose that Gevrey functions are dense in $\mathsf{X}$, which will be the case for many function spaces of interest, including smooth spaces and Sobolev spaces (see, e.g., \citealt{Teofanov-2006} and \citealt[Example 1.4.9]{Rodino-Book-1993}). Take a sequence $\{ U_{0,i} \}$ of Gevrey initial data converging in $\mathsf{X}$ to $U_0$. Since $U_{0,i} \in \mathsf{X}$, local well-posedness in $\mathsf{X}$ yields solutions\footnote{More precisely, if the coefficients and nonlinearities of the equations are not of appropriate Gevrey regularity, we have to approximate them as well (assuming, say, that they are smooth). This is done in the usual way wherein one approximates them by analytic functions. See, e.g., Chapter~6, \S~10 of \cite{Courant-Hilbert-Book-1989-2}.} $U_i$ in $\mathsf{X}$. Under appropriate conditions, we can guarantee that the solutions $U_i$ are all defined on the same time interval.

If the method of establishing local well-posedness for \eqref{E:Leray_system_matrix}
in $\mathsf{X}$ produces sufficiently good estimates (as it will often be the case for hyperbolic systems where energy estimates are employed), then we can apply such estimates to the solutions $U_i$ and to differences $U_i - U_j$ and obtain that $U_i$ converges to $U$, typically not in $\mathsf{X}$ but in some weaker topology which we denote $\mathsf{X}^\prime$. For example, if $\mathsf{X}=H^s$ then $\mathsf{X}^\prime = H^{s-1}$, where $H^s$ is a Sobolev space.

Assume that the characteristic determinant of system \eqref{E:Leray_system_matrix} satisfies the assumptions of Theorem \ref{T:Leray_Ohya_diagonalization}. Then we can, upon appropriate regularity assumptions, diagonalize the system, transforming into a system of the form \eqref{E:Diagonalized_Leray}. Observe that, again under suitable regularity assumptions, $U$ and $U_i$ are solutions to \eqref{E:Diagonalized_Leray}, respectively. 

Let $J^-(x)$ and $J^-_i(x)$ be the past domain of dependence of $x$ associated with the solutions $U$ and $U_i$, respectively. Notice that these past domains of dependence are defined with the diagonal principal part of \eqref{E:Diagonalized_Leray}, since we have not assumed that the system \eqref{E:Leray_system_matrix} has diagonal principal part, and it is only for such systems that we have defined a causal structure, see Definition \ref{D:Causal_structure_A}. Let us use the standard notation $C^k$ for the space of $k$-times continuously differentiable functions.
If $\mathsf{X}^\prime$ has good embedding properties into $C^k$ for some appropriate $k$, then $J^-(x) \rightarrow J^-_i(x)$ as $i\rightarrow\infty$. Here, the convergence $J^-_i(x) \rightarrow J^-(x)$ is simply	convergence of sets pointwise, where we measure the distance between points along each $\Sigma_t$ with the Euclidean\footnote{If the spacetime has a metric, we could take the metric induced on $\Sigma_t$.} metric in $\mathbb{R}^n$. 

Assume now that $\left. U \right|{J^-(x)\cap \Sigma} = 0$. We can assume without loss of generality that the approximating initial data $U_{0,i}$ vanishes in the interior of ${J^-(x)\cap \Sigma}$.
Since $J^-_i(x) \rightarrow J^-(x)$, given any point $y$ in the interior of $J^-(x)$, we have that $y \in J^-_i(x)$ and $J^-_i(y) \subset J^-(x)$ for all $i$ sufficiently large. For the Gevrey solutions $U_i$, the domain-of-dependence property
holds in view of Theorems \ref{T:Leray_Ohya_existence} and \ref{T:Leray_Ohya_domain_of_dependence}. It follows that $U_i = 0$ in $J^-_i(y)$. Passing to the limit and using the above embedding of $\mathsf{X}^\prime$ into $C^k$, we conclude that $U(y) = 0$. This establishes the domain-of-dependence property for $\mathsf{X}$-solutions to \eqref{E:Leray_system_matrix} since $y$ is an arbitrary interior point of $J^-(x)$.

\newpage
\bmhead{Acknowledgements} 

This review article grew out of a
series of lectures the author delivered at the \href{https://cmsa.fas.harvard.edu}{Center
for Mathematical Sciences and Applications (CMSA)} at Harvard
University as part of the semester program \href{https://cmsa.fas.harvard.edu/event/general-relativity-program/}{General Relativity} in Spring 2022, and at the \href{http://www.fullerton.edu/ires-uz/asi/index.php}{Advanced Studies Institute in Mathematical Physics} at Urgench State University in Uzbekistan, as part of the \href{http://www.fullerton.edu/ires-uz/}{USA--Uzbekistan Collaboration for Research, Education and Student Training} in July--August 2022.
The author is grateful to the organizers of both events for the invitation and for their hospitality. The author is also thankful 
to all participants for stimulating discussions. The author would like to thank Leonardo Abbrescia, Brian Luczak, and Jean-Fran\c cois Paquet, for reading an earlier version of this manuscript and providing valuable feedback, and Jared Speck for feedback on an updated version of the manuscript. The author would like to thank Alex Pandya for producing a video of a simulation presented in Sect.~\ref{S:BDNK_applicatons}. Finally, the author would like to thank the referees for a careful reading of the manuscript and for providing thoughtful feedback that led to an improvement of the article.

\section*{Declarations}

\begin{itemize}

\item Funding
The author gratefully acknowledges support from NSF grants DMS-2406870 and DMS-2107701, from DOE grant DE-SC0024711,
from a Chancellor's Faculty Fellowship, and a Vanderbilt's Seeding Success grant.

\end{itemize}

\phantomsection
\addcontentsline{toc}{section}{References}
\bibliography{References.bib}

\end{document}